\documentclass[12pt]{article}
\usepackage{amsmath,amssymb,amsthm,mathtools,mathrsfs,latexsym}
\usepackage{graphicx}
\usepackage{xcolor}
\usepackage{tikz}
\usepackage[linkcolor=blue,colorlinks=true,urlcolor=red,bookmarksopen=true]{hyperref}
\usepackage{newclude}
\usepackage{imakeidx}
\makeindex
\usepackage{indentfirst}
\usepackage{bm}

\allowdisplaybreaks

\title{On global classical and weak solutions with arbitrary large initial data to the multi-dimensional viscous Saint-Venant system and compressible Navier-Stokes equations subject to the BD entropy condition under spherical symmetry }
\date{}
\author{
	\bf\large Xiangdi Huang$^{a}$\thanks{E-mail addresses: xdhuang@amss.ac.cn (X. Huang); mengweili@amss.ac.cn (W. Meng); xyzhang05@163.com (X. Zhang). }, Weili Meng$^a$, Xueyao Zhang$^b$\\
	\small a. State Key Laboratory of Mathematical Sciences, Academy of Mathematics and Systems Science,\\
	\small Chinese Academy of Sciences, Beijing 100190, China;\\
	\small b. School of Mathematics and Statistics, Yulin University, Yulin 719000, China;\\
     }

\newcommand{\norm}[1]{\left\Vert#1\right\Vert}
\let\div\relax
\DeclareMathOperator*{\div}{div}

\usepackage{hyperref}
\hypersetup{hypertex=true,
	colorlinks=true,
	linkcolor=blue,
	anchorcolor=blue,
	citecolor=blue}

\usepackage{color}
\newtheorem{thm}{Theorem}[section]

\newtheorem{lema}{Lemma}[section]
\newtheorem{prop}{Proposition}[section]
\newtheorem{defi}{Definition}[section]
\newtheorem{rmk}{Remark}[section]

\allowdisplaybreaks

\begin{document}
	\maketitle %显示标题
	\begin{abstract}
In 1871, Saint-Venant introduced the renowned shallow water equations. Since then, for the two-dimensional viscous or inviscid shallow water equations, the global existence of smooth solutions with arbitrarily large initial data has remained a challenging and long-standing open problem. In this paper, we provide an affirmative resolution to the viscous problem under the assumption of two-dimensional radial symmetry. Specifically, we establish the global existence of smooth solutions for the two-dimensional radially symmetric viscous shallow water equations with arbitrary smooth initial data. To achieve this goal, our approach relies crucially on overcoming two major obstacles: first, treating the viscous Saint-Venant system as the endpoint case of the BD entropy condition for the compressible Navier-Stokes equations; and second, addressing the critical embedding imposed by the spatial dimension, which currently holds only in two dimensions. However, the same result can be extended to three dimension for the compressible Navier-Stokes equations satisfying general BD entropy conditions excluding the endpoint case. Indeed, under the same symmtric framework, we also prove the global existence of smooth solutions for arbitrarily large initial data for both the two- and three-dimensional compressible Navier-Stokes equations subject to the BD entropy condition. It is particularly noteworthy that the aforementioned shallow water equations precisely correspond to the endpoint case of the compressible Navier-Stokes equations satisfying the BD entropy condition. More precisely, assuming that for the  two- and three-dimensional compressible Navier-Stokes equations with viscosity coefficients satisfying the BD entropy condition which take the specific form:
		\[
	\mu(\rho)=\rho^\alpha,\quad \lambda(\rho)=(\alpha-1)\rho^\alpha,\quad \alpha\ge \frac{N-1}{N}.
		\]
For \(\alpha\le 1\), to address the singularity at the center of the high-dimensional sphere, we introduce a new parameter-continuation method and a carefully designed r-weight function for the nearby density. This approach enables us to prove, for arbitrarily large initial data far from vacuum, the global existence and large-time behavior of classical solutions. Particularly, in the two-dimensional scenario, only $\alpha \in (0.5,1]$ is required, and its lower bound is optimal due to physical constraints. Most notably, the case of \(\alpha=1\) covers the shallow-water equations, where the key to the proof lies in obtaining a positive lower bound for the density by estimating the \(L^\infty\) norm of the so-called effective velocity and the velocity. For \(\alpha\ge 1\), we construct a global weak solution containing vacuum with bounded density. This weak solution possesses higher regularity than those available in previous works [Guo-Jiu-Xin, SIAM J. Math. Anal 39 (5):1402-1427, 2008; Li-Xin, arXiv:1504.06826, 2015; Vasseur-Yu, Invent. math. 206:935-974, 2016; Bresch-Vasseur-Yu, J. Eur. Math. Soc. 24:1791-1837, 2022]. Moreover, we establish a new time-independent \(L^4\) estimate for both the velocity field and the density gradient, proving that the vacuum state of such global weak solutions will disappear within a finite time, thereby extending the existence of the weak solutions with bounded density and vacuum-vanishing phenomenon in one-dimensional cases [Li-Li-Xin,  Comm. Math. Phys 281(2):401-444, 2008] to the higher dimensions.  For $\alpha=1 $, assuming that the initial density is far from the vacuum and has additional local regularities, we introduce some new suitable cut-off functions to derive a family of approximation-independent uniform estimates, thereby obtaining a class of global weak solutions in which any possible vacuum appears only near the origin, and proving that the weak solution is actually a strong solution on certain subsets of the non-vacuum fluid regions. All results concerning weak solution obtained in this article are also applicable to the shallow-water equations for describing shallow water waves.
\end{abstract}
	
	\tableofcontents %生成目录

	\section{Introduction}
	
	The system of compressible isentropic Navier-Stokes equations with density-dependent viscosity coefficients in ${\mathbb R}^N$, $N=2,3,$ can be read as
	\begin{equation}\label{0}
		\begin{cases}
			\partial_t{\rho} + \operatorname{div}(\rho \mathbf{u}) = 0,\\
			\partial_t(\rho \mathbf{u}) + \operatorname{div}(\rho \mathbf{u} \otimes \mathbf{u})+ \nabla P(\rho ) - \operatorname{div}\big(\mu(\rho) \mathbb{D}\mathbf{u}\big) - \nabla \big(\lambda(\rho) \operatorname{div}\mathbf{u}\big)
			= 0,
		\end{cases}
	\end{equation}
	where $\rho(x,t)$, $\mathbf{u}(x,t)$ and $P(\rho)=\rho^\gamma$($\gamma>1$) are the fluid density,
	velocity and pressure, respectively, and $\mathbb{D}\mathbf{u} = \frac{{\nabla \mathbf{u} + {\nabla \mathbf{u}}^T}}{2}$. The viscosities $\mu(\rho)$ and $\lambda(\rho)$ satisfy
	$$\mu(\rho)\geq0,\ \mu(\rho)+N\lambda(\rho)\geq0.$$
	
	Throughout the process of studying the solutions of $(\ref{0})$, a critical problem that one has to face is the possible appearance of vacuum. As it is well-known, the formation and dynamics of vacuum states are key issues in the studies of the existence, regularity, and long time behaviors of solutions for viscous compressible fluids, see \cite{Huang-Li-Xin-2012, jiang-2006, jiu-2014, Li-2008, Liu-Xin-Yang-1998, Matsumura-Nishida-1980,Xin-1998} and the references therein.
	
	In particular, Hoff-Smoller \cite{hoff-2001} proved that weak solutions of the Navier–Stokes equations for compressible fluid flow in one space dimension do not exhibit vacuum states, provided that no vacuum states are present initially. For the one-dimensional case where the viscosity coefficient depends on density, interesting phenomena such as the vanishing of vacuum and the blow-up of solutions were found by Li-Li-Xin \cite{Li-2008}, after that, the weak solution transforms into a strong solution and tends toward a non-vacuum equilibrium state. However, in general, for compressible fluids, it remains unknown whether a vacuum will emerge within a finite time, even if the initial state is far from vacuum. For instance, Xin-Yuan \cite{xin-2006} extended such a result in \cite{hoff-2001} to the spherically symmetric case on which a sufficient condition on the regularity of the velocity to ensure non-formation of vacuum is given, and it is shown that the separate two initial vacuum states shall not meet together in a finite time. Moreover, for the spherically symmetric case, as initial data away from vacuum states, Hoff \cite{Hoff-1992} first constructed the global weak solutions of (\ref{0}) with constant viscosities for $\gamma=1$, and proved that such solutions may develop vacuum regions near the symmetry center where the boundaries of these vacuum regions are H{\"o}lder continuous in space-time. Hoff-Jenssen \cite{Jensen-2004} generalized these results and proved the global existence of weak solutions for the nonbarotropic flow, the analysis allowed for the possibility that a vacuum state emerges at the origin or axis of symmetry, and the equations hold in the sense of distributions in the ``Fluid region" where the density is positive. Moreover, Huang-Matsumura \cite{Ma-2015} considered the initial boundary value problem, and it is proved that the classical solution which is spherically symmetric loses its regularity in a finite time either the density concentrates or vanishes around the center. Therefore, for the multi-dimensional spherically symmetric compressible flow, one important problem is whether the initial non-vacuum states retain all the time or the initial vacuum states vanish in finite time. The motivation of this paper is to give some positive answers to such problems about compressible flow, which is also crucial to the global existence of classical solution.

    In the case where the viscosity coefficients $\mu(\rho)$ and $\lambda(\rho)$ are constants, the global well-posedness theory of the compressible Navier-Stokes equations has been extensively studied. For the one-dimensional case, when the initial data is far from vacuum, Kanel \cite{Kale-1968} first established the existence of global solutions under smooth initial data. Serre \cite{Serre-1986-1,Serre-1986-2} and Hoff \cite{Hoff-1987} further obtained global solutions for discontinuous initial data. Kazhikhov-Shelukhin \cite{Kazhikhov-Shelukhin-1977} and Kawashima-Nishida \cite{Kawashima-Nishida-1981} respectively proved the existence of global strong solutions on bounded and unbounded domains. For the multi-dimensional case, by a standard Banach fixed point argument, Nash \cite{Nash-1962} established the local existence and uniqueness of classical solutions to the Cauchy problem under the condition that the initial data does not contain vacuum. The result of global solution was first obtained by Matsumura-Nishida \cite{Matsumura-Nishida-1980}, when the initial data close to a non-vacuum equilibrium state, and they established the existence of global classical solutions in some Sobolev space \( H^s(\mathbb{R}^3) (s>\frac52)\). later, Hoff \cite{hoff-1995} studied the problem with discontinuous initial data and introduced a new type of a priori estimates on the material derivative of the velocity field. Furthermore, Cho-Choe-Kim \cite{Cho-Choe-Kim-2004} and Cho-Kim \cite{Cho-Kim-2006} proposed initially an important compatibility condition, by using linearization and iterative methods, and proved the local well-posedness of strong solutions for the vacuum case in three dimensions. Subsequently, Luo \cite{Luo-2012} obtained the local classical solutions in two dimensions. The main breakthrough on global well-posedness is duo to Lions \cite{Lions-1998} by using the methods of weak convergence and renormalized solutions, where he established the global existence of weak solutions with finite energy and large initial data allowing vacuum as \( \gamma \ge \frac{3N}{N+2} (N=2,3)\). Later, Feireisl-Novotn\'{y}-Petzeltov\'{a} \cite{FNP-2001} extended this result to \( \gamma > \frac{N}{2} \) by introducing
an appropriate method of truncation. For the three-dimensional spherically symmetric case, Jiang-Zhang \cite{Jiang-Zhang-2001} extended this threshold to \( \gamma > 1 \). However, the regularity and uniqueness of such renormalized weak solutions remain important open problems. Indeed, in inhomogeneous Sobolev
space, Xin \cite{Xin-1998} showed that any classical solutions to the compressible Navier–Stokes equations with finite energy cannot exist globally in time since it may blow up in finite time, provided that the density is compactly supported. In addition, Xin-Yan \cite{Xin-Yan-2013} proved that any classical solutions of viscous non-isentropic compressible
fluids without heat conduction will blow up in finite time, if the initial data has an
isolated mass group. In particular, for initial data containing vacuum even have compact support, Huang-Li-Xin \cite{Huang-Li-Xin-2012} established, for the first time, the existence and uniqueness of global classical solutions of Cauchy problem with small energy and large oscillations in three-dimension. Then, Li-Xin \cite{li-2019} proved the global well-posedness and large time asymptotic behavior of strong and classical solutions in two dimensions with vacuum as far field density. Recently, Li-Wang-Xin \cite{Li-Wang-Xin-2019} proved that the one-dimensional classical solution with finite energy does not exist in the inhomogeneous Sobolev space for any short time with vacuum.

In the theory of gas dynamics, by some physical considerations, Liu-Xin-Yang \cite{Liu-Xin-Yang-1998} introduced the modified compressible Navier-Stokes equations with density-dependent viscosity coefficients for isentropic fluids. The case in which viscosities depend on the density has received
	a lot of attention recently. When the shear viscosity $\mu>0$ is a positive constant and the bulk viscosity $\lambda$ is density-dependent, namely,
	\begin{align*}
		\mu=const.>0,\ \lambda(\rho)=\rho^{\beta},
	\end{align*}
	the system (\ref{0}) was first studied by Vaigant-Kazhikhov \cite{vaigant-1995}, where the global well-posedness of the strong solution to the 2D periodic problem was established for arbitrarily large and non-vacuum initial data under the restriction $\beta>3$. If the initial values may contain
	vacuum states, Jiu-Wang-Xin \cite{jiu-2014} proved the global well-posedness of the classical solution to the 2D periodic problem but still under the restriction $\beta>3$. Later, Huang-Li \cite{huang-1205} relaxed the power index $\beta$ to be $\beta>\frac43$ and studied the large time behavior of the solution. The global well-posedness of 2D Cauchy problem with the vacuum states at far-fields was studied by Jiu-Wang-Xin \cite{jiu-2013} and Huang-Li \cite{huang-1207}. For 2D Cauchy problem with non-vacuum far-fields, the global well-posedness was proved by Jiu-Wang-Xin \cite{jiu-pre}, in particular, it is shown that the solution will not develop the vacuum states in any finite time for non-vacuum large initial data.  By using the pull-back Green’s function method to get over the difficulties brought by boundary, Fan-Li-Li \cite{fan-2022} solved the Navier-slip boundary problem. For $\beta>1$, the 2D global strong solution of the free boundary problem with spherically symmetric data was obtained by Li-Zhang \cite{li-2016}. Huang-Meng-Ni \cite{hmn} obtained the global spherically symmetric strong solution of the free boundary problem as $\beta=1$. Furthermore, Guo-Wang-Wang \cite{wang-2018} considered the 3D spherically symmetric cases, and proved the global existence of weak solutions with $0\leq\beta\leq\gamma$.
	
	Moreover, when the viscosity coefficients both depend on the density and satisfying the B-D relation as follows
	\begin{align*}
		\lambda(\rho)=\rho\mu'(\rho)-\mu(\rho),
	\end{align*}
a new mathematical entropy estimate was obtained by Bresch-Desjardins-Lin \cite{bd-2003, bd-2004}. Later, by obtaining a new priori estimate on smooth approximate solutions, Mellet-Vasseur \cite{Mellet-2007} studied the stability of (\ref{0}). Based on these, Guo-Jiu-Xin \cite{Guo-2008} first showed the existence of global weak solutions with arbitrarily large and spherically symmetric initial data which contain vacuum states, further Guo-Li-Xin \cite{guo-2012} extended it to the free boundary case and obtained the detailed regularity and Lagrangian structure of this weak solution. For the general multi-dimensional initial data, some crucial
progresses on the global well-posedness of the weak solutions were achieved by Li-Xin \cite{Li-2015} and Vasseur-Yu \cite{Vasseur-2016}, where they independently obtained the global existence of weak solutions for the compressible system with $\mu(\rho)=\rho,\lambda(\rho)=0$ for arbitrarily large data allowing vacuum. Note that, Li-Xin \cite{Li-2015} also considered a broader range of viscosity coefficients satisfying the BD relation, where the viscous term take the form $-\operatorname{div}(\mu(\rho) \mathbb{D}\mathbf{u} )-\nabla(\lambda(\rho)\operatorname{div}\mathbf{u})$ in two dimensions and $-\operatorname{div}(\mu(\rho)\nabla \mathbf{u})-\nabla(\lambda(\rho)\operatorname{div}\mathbf{u})$ in two and three dimensions respectively. Later, in the three-dimensional space case, a major breakthrough was obtained by Bresch-Vasseur-Yu \cite{va-2022}, the authors extended the results on global existence of weak solutions obtained by Li-Xin \cite{Li-2015} to a physical symmetric viscous stress tensor $-\operatorname{div}(\mu(\rho) \mathbb{D}\mathbf{u} )-\nabla(\lambda(\rho)\operatorname{div}\mathbf{u})$ and more general viscosities. More recently, for the spherically symmetric data, Zhang \cite{zhang-2025} first proved the result of the global existence of the non-vacuum strong solution with arbitrarily large data. Guo-Xu-Zhang \cite{Guo-2025} further constructed the global strong solution of the Cauchy problem. Most notably, a special case of viscosity coefficients satisfying the B-D relation is $\mu(\rho)=\rho$, $\lambda(\rho)=0$, which covers the Saint-Venant model for the motion of shallow water. This important physical
system has attracted extensive attentions recently, and some important achievements both
on weak and strong solutions have been obtained. For the non-vacuum initial data, Haspot \cite{Haspot-2018} and Haspot-Burtea \cite{Haspot-Burtea} proved the existence of a unique global strong solution for the one-dimensional Cauchy problem. Later, Cao-Li-Zhu \cite{Cao-Li-Zhu-1} proved the existence of global regular solutions for one-dimensional  Cauchy problem with large data and far field vacuum. Cao-Li-Zhu \cite{Cao-Li-Zhu-2} further considered the initial-boundary value problem in the domain exterior to a ball in $\mathbb{R}^d(d=2,3)$, and proved the global existence of the unique
spherically symmetric classical solution for large initial data with far field vacuum. Some interesting progress on the vacuum free boundary problem are achieved by Xin-Zhang-Zhu \cite{Xin-Zhang-Zhu} and Li-Wang-Xin \cite{Li-Wang-Xin}. In \cite{Xin-Zhang-Zhu}, the authors established the global-in-time well-posedness of classical
solutions to the vacuum free boundary problem of the one-dimensional case with large data. Li-Wang-Xin \cite{Li-Wang-Xin} further given the local-in-time well-posedness of classical solutions to the vacuum free boundary problem in two dimensions. Indeed, for the high-dimensional Saint-Venant model, there still exist no results about the global existence of classical solutions with large initial data.

In addition, to a class of density-dependent viscosity coefficients but independent of the BD-entropy relation
	\begin{align*}
		\mu(\rho)=\alpha\rho^{\delta},\ \lambda(\rho)=\beta\rho^{\delta},
	\end{align*}
	based on the observations that the degeneracies of the time evolution and the viscosity can be transferred to the possible singularity of the special source term, as $\delta=1$, Li-Pan-Zhu \cite{ya-2014} proved the existence of the unique local regular solution in two-dimension space, see Zhu \cite{zhu-20152} for the three-dimension case. Later, via introducing a proper class of solution space, as $1<\delta\leq\min\{3,\frac{\gamma+1}{2}\},$ the same authors \cite{ya-2015, zhu-2015} gave the existence of local regular solution. Xin-Zhu \cite{xin-2021} further proved the local existence of regular solution as $0<\delta<1$. In particular, for $\delta>1$, under the smallness assumptions on $\Vert{\rho_0}^{\frac{\gamma-1}{2}}\Vert_{H^3}+\Vert{\rho_0}^{\frac{\delta-1}{2}}\Vert_{H^3}$, by making use of the “quasi-symmetric
	hyperbolic” and “degenerate elliptic” coupled structure to control the behavior of the fluid velocity, Xin-Zhu \cite{xin-2018} proved the global well-posedness of regular solution to the three-dimension Cauchy problem with vacuum and smooth initial data. As initial data away from vacuum, Guo-Song \cite{song-2019} obtained the global strong solution under the smallness assumptions on $\Vert \mathbf{u}_0\Vert_{H^1}+\Vert\rho_0-\tilde{\rho}\Vert_{L^2}$.  However, except for the one-dimension problems, there are
	still only few results on the global well-posedness theory of strong solution to (\ref{0}) with density-dependent viscosities because of the
	possible degeneracy, please refer to \cite{song-2019,xin-2018}. 
	
	In this article, as the shear viscosity coefficient and the bulk viscosity coefficient satisfy the BD relation 
	\begin{align*}
		\mu(\rho)=\rho^\alpha,\quad \lambda(\rho)=(\alpha-1)\rho^\alpha,
	\end{align*}we study the global existence of large classical solutions away from vacuum, the global existence of large weak solutions that allow vacuum, and the phenomenon of Vanishing of Vacuum States and Non-formation of Vacuum States.
	
	Assume that \(\Omega\) is a bounded domain in \(\mathbb R^N\) with \(N=2,3\).  We consider equation \eqref{0} with initial data
	\begin{align}\label{0-1}
		\rho(x,0)=\rho_0(x),\quad 
		\rho\mathbf{u}(x,0)=\mathbf{m}_0(x),\quad x\in\Omega,
	\end{align}
	and the boundary condition 
	\begin{align}\label{0-2}
		\rho\mathbf{u}(x,\cdot)=0,\quad x\in\partial\Omega. 
	\end{align}
	In the radially symmetric setting, the domain, density field, and velocity field have the following representations:
	\begin{align*}
		\Omega=B_R,\quad \rho(x,t)=\rho(r,t),\quad \mathbf{u}(x,t)=u(r,t)\frac{\mathbf{x}}{r}.
	\end{align*}
	The initial boundary value problem \eqref{0}-\eqref{0-2} is transformed into 
	\begin{align}\label{Equ 2}
		\left\{ 
		\begin{array}{ll}
			\rho_t+(\rho u)_r+\frac{N-1}{r}\rho u=0,\\
			(\rho u)_t+(\rho u^2)_r+(\rho^\gamma)_r+\frac{N-1}{r}\rho u^2-\left(\frac{\alpha}{r^{N-1}}\rho^\alpha(r^{N-1}u)_r\right)_r+\frac{N-1}{r}(\rho^\alpha)_r u=0.
		\end{array}
		\right.
	\end{align}
	with the initial data
	\begin{align}\label{inital data2}
		\rho(r,0)=\rho_0(r),\quad \rho u(r,0)=m_0(r),
	\end{align}
	and the boundary condition
	\begin{align}\label{boundary condition2}
		\rho u(0,t)=\rho u(R,t)=0.
	\end{align}

    We now outline the organization of the rest of this paper. In Sect. 2, we state our main results: Theorem \ref{Thm1} establishes the global existence of classical solutions for large initial data away from vacuum as $\alpha<1$. Theorem \ref{Thm2} gives the large-time behavior of such classical solution. Theorem \ref{Thm SV 2d} establishes the global existence of classical solutions with large initial data away from vacuum, as well as the large-time behavior of the solutions, in the two-dimensional endpoint case $\alpha = 1$, which includes the viscous Saint‑Venant model describing shallow water motion for $\gamma = 2$. Theorem \ref{Thm3} and Theorem \ref{Thm4} respectively assert the global existence of the weak solutions with large initial data allowing for vacuum in two-dimensional and three-dimensional settings as $\alpha\ge 1$. Theorem \ref{Thm5} demonstrates the finite-time vacuum-vanishing phenomenon for such weak solutions. Theorem \ref{T_2} yields a class of global weak solutions where any potential vacuum only appears near the origin as $\alpha=1$. Moreover, we present the main strategies and a brief overview for the proofs of these main results. Sect. 3 is dedicated to proving Theorem \ref{Thm1} and Theorem \ref{Thm2}, and the primary goal here is to derive the upper and lower bounds of the density as well as some uniform-in-time estimates. Sect. 4 is devoted to the proof of Theorem \ref{Thm SV 2d}, with the key step being the derivation of a positive lower bound for the density in the two-dimensional endpoint case. Sect. 5 focuses on the proofs of Theorem \ref{Thm3}–Theorem \ref{Thm5}. To this end, we introduce artificial viscosities into the system away from the center, and investigate the global solvability of the approximate system. Further we employ a careful compactness argument to take the limits, thereby obtaining global weak solutions with bounded density. Sect. 6 is devoted to proving Theorem \ref{T_2}, where the approximate system contains no artificial viscosity terms, and the weak solutions are obtained as limits of solutions in annular regions as the inner radius tends to zero and exhibit high regularity away from the symmetry center. Finally, Sect. 7 lists the appendix which includes several lemmas frequently utilized throughout the paper.

	\section{Main results}
	Before stating our main results, we first clarify the following notations and conventions. For a scalar function \(f\), define the material derivative of \(f\) as
	\begin{align*}
		\dot{f}=\frac{D}{Dt}f=\partial_t f+\mathbf{u}\cdot \nabla f,
	\end{align*}
	and the average value of the integral as
	\begin{align*}
		\bar{f}=\frac{1}{|\Omega|}\int_\Omega fdx.
	\end{align*}
    We call a scalar function \(f\) or a vector-valued function $\mathbf{v}$ radially symmetric if  
    \begin{align*}
        f(x)=f(|x|),\quad \mathbf{v}(x)=v(|x|)\frac{x}{r}.
    \end{align*}
	
	\begin{defi}
		Define
	\begin{align*}
		M_{\text{set}}=\left\{1+\frac{s}{2k+1}\Big| s\in\mathbb{N}^+, k\in\mathbb{N}\right\}.
	\end{align*}
	\end{defi}
	It is readily verified that \(M_{\mathrm{set}}\) is dense in \((1,\infty)\).
	
	\begin{defi}\label{Def alpha 2}
	For $n>1$ define
	\begin{align}  
		&\alpha_{2,-}(n)=1-\frac{n\sqrt{2n-1}-2n+1}{n^2-2n+1},\\
		&\alpha_{2,+}(n)=1+\frac{n\sqrt{2n-1}+2n-1}{n^2-2n+1}.  
	\end{align}     
	\end{defi}
	Here $\alpha_{2,-}(\cdot)$ is strictly increasing on $(1,\infty)$ with $\lim\limits_{n\to 1^+}\alpha_{2,-}(n)=\frac12$ and $\lim\limits_{n\to\infty}\alpha_{2,-}(n)=1$, while $\alpha_{2,+}(\cdot)$ is strictly decreasing on $(1,\infty)$ with $\lim\limits_{n\to 1^+}\alpha_{2,+}(n)=\infty$ and $\lim\limits_{n\to\infty}\alpha_{2,+}(n)=1$. For the sake of continuity, the proofs of the properties of \(\alpha_{2,\pm}\) are given in Lemma \ref{Lem alpha2} of the Appendix.
	
	\begin{defi}
		Define $n_2(\cdot):(\frac12,\infty)\to(1,\infty]$ as follows.  For $\alpha\in(\frac12,1)$, let $n_2(\alpha)$ be the unique solution of $\alpha_{2,-}\bigl(n_2(\alpha)\bigr)=\alpha$; for $\alpha=1$, set $n_2(1)=+\infty$; and for $\alpha\in(1,\infty)$, let $n_2(\alpha)$ be the unique solution of $\alpha_{2,+}\bigl(n_2(\alpha)\bigr)=\alpha$.  
	\end{defi}
	
	\begin{defi}\label{Def alpha 3}
	For $n>1$ define
		\begin{align}  
			&\alpha_{3,-}(n)=1-\frac{\sqrt{4n(4n^2-n-1)+1}-6n+3}{4n^2-8n+4},\\
			&\alpha_{3,+}(n)=1+\frac{\sqrt{4n(4n^2-n-1)+1}+6n-3}{4n^2-8n+4}.
		\end{align}  
	\end{defi}
	Here $\alpha_{3,-}(\cdot)$ is strictly increasing on $(1,\infty)$ with $\lim\limits_{n\to 1^+}\alpha_{3,-}(n)=\frac23$ and $\lim\limits_{n\to\infty}\alpha_{3,-}(n)=1$, whereas $\alpha_{3,+}(\cdot)$ is strictly decreasing on $(1,\infty)$ with $\lim\limits_{n\to 1^+}\alpha_{3,+}(n)=\infty$ and $\lim\limits_{n\to\infty}\alpha_{3,+}(n)=1$. The proofs of the properties of \(\alpha_{3,\pm}\) can be found in Lemma \ref{Lem alpha3} of the Appendix.
	
	\begin{defi}
	Define $n_3(\cdot):(\frac23,\infty)\to(1,\infty]$ as follows. For $\alpha\in(\frac23,1)$, let $n_3(\alpha)$ solve $\alpha_{3,-}\bigl(n_3(\alpha)\bigr)=\alpha$; for $\alpha=1$, set $n_3(1)=+\infty$; and for $\alpha\in(1,\infty)$, let $n_3(\alpha)$ solve $\alpha_{3,+}\bigl(n_3(\alpha)\bigr)=\alpha$.
	\end{defi}
	
	\subsection{Global classical solution away from vacuum $(\frac{N-1}{N}<\alpha<1)$}
	We start with the definition of the global classical solution to \eqref{0}.
	\begin{defi}[Global classical solution]\label{D_1}
		A radially symmetric pair \((\rho,\mathbf{u})\) with $\rho>0$ is called a global classical solution of the initial-boundary-value problem \eqref{0}-\eqref{0-2} if, for any \(T>0\),
		\begin{align}
			\left\{
			\begin{array}{l}
				\rho\in  C\big([0,T];H^3(\Omega)\big),\ \rho_t\in C\big([0,T];H^2(\Omega)\big),\nonumber\\
				u\in  C\big([0,T];H_0^1(\Omega)\cap H^3(\Omega)\big)\cap L^2\big(0,T;H^4(\Omega)\big),\nonumber\\
				u_t\in L^{2}\big(0,T;H^{2}(\Omega)\big)\cap L^{\infty}\big(0,T;H_0^1(\Omega)\big),\nonumber\\
				u_{tt}\in L^{2}\big(0,T;L^{2}(\Omega)\big).\nonumber
			\end{array}
			\right.
		\end{align}
	\end{defi}
	
	The first result on the global existence of classical solutions with large initial data away from vacuum is stated as follows:
	\begin{thm}[Global classical solution]\label{Thm1}
		Let $N=2$ or $N=3$. Assume that \((\alpha,\gamma)\) satisfies
		\begin{align}
			&N=2,\quad 0.5<\alpha<1,\quad \gamma>1;\label{C 2d alpha}\\
			&N=3,\quad 0.686<\alpha<1,\quad 1<\gamma<6\alpha-3+\frac{3-5\alpha}{2n_3(\alpha)},\label{C 3d alpha}
		\end{align}
		and that the radially symmetric initial data \((\rho_0,\mathbf{u}_0)\) satisfies
		\begin{align}\label{C initial data}
			0<\underline{\rho_0}\leq \rho_0\leq \overline{\rho_0},\quad (\rho_0, \mathbf{u}_0)\in H^3(\Omega),\quad \mathbf{u}_0|_{\partial \Omega}=0,
		\end{align}
		where $\underline{\rho_0}$ and $\overline{\rho_0}$ are positive constants. Then the initial-boundary-value problem \eqref{0}–\eqref{0-2} admits a unique global radially symmetric classical solution $(\rho,\mathbf{u})$ satisfying, for any \(T>0\) and any \((x,t)\in\Omega\times [0,T]\),
		\begin{align}
			(C(T))^{-1}\le \rho(x,t)\le C(T),
		\end{align}
		where the constant \(C(T)>0\) depends on the initial data and \(T\).
	\end{thm}
	
	\begin{rmk}
		In the two-dimensional case, we improve the existence result in \cite{zhang-2025} from the range \(0.8\le\alpha<1\) and $\gamma>1$ to \(0.5<\alpha<1\) and $\gamma>1$.  The restriction \(\alpha<1\) arises from the derivation of the positive lower bound on the density.  The physical requirement on the viscosity coefficients \(\mu+2\lambda\ge0\) is equivalent to \(\alpha\ge0.5\). However, the endpoint case $\alpha=0.5$ is extremely difficult because the dissipative estimate of the velocity field becomes unavailable, so our range appears to be optimal.
	\end{rmk}

\begin{rmk}
    In the three-dimensional case, the physical requirement on the viscosity coefficients \(\mu+3\lambda\ge0\) is equivalent to \(\alpha\ge\tfrac23\).  The range \(\alpha\in(\tfrac23,0.686]\) is therefore left for future work.
\end{rmk}
    
	\begin{rmk}
		The range of \(\gamma\) in \eqref{C 3d alpha} is non-empty.  Observe that  for any $\alpha>\frac{2}{3}$,
		\[
		1<6\alpha-3+\frac{3-5\alpha}{2n_3(\alpha)}
		\Longleftrightarrow
		n_3(\alpha)>\frac{5\alpha-3}{12\alpha-8}.
		\]  
		For \(\alpha\in[\frac57,1)\) the inequality holds automatically because \(n_3(\alpha)>1\).  
		For \(\alpha\in(\frac23,\frac57)\) it becomes  
		\[
		\alpha>\alpha_{3,-}\Bigl(\frac{5\alpha-3}{12\alpha-8}\Bigr),
		\]  
		which is satisfied whenever \(0.686<\alpha<\frac57\).
	\end{rmk}
	
	\begin{rmk}
		In the three-dimensional case, we extend the result of \cite{zhang-2025} from the ranges \(0.875\le\alpha<1\) and \(1<\gamma<9\alpha-6\) to \(0.686<\alpha<1\) and \(1<\gamma<6\alpha-3+\frac{3-5\alpha}{2n_3(\alpha)}\). In fact, for $0.875\leq \alpha<1$,  it holds that
		\begin{align*}
			9\alpha-6<6\alpha-3+\frac{3-5\alpha}{2n_3(\alpha)}\Longleftrightarrow\alpha>\alpha_{3,-}\left(\frac{5\alpha-3}{6-6\alpha}\right).
		\end{align*}
	\end{rmk}
	
	We now state the second result concerning the large-time behavior of the classical solution.
	\begin{thm}[Global classical solutions and large-time behavior]\label{Thm2}
			Let $N=2$ or $N=3$. Assume that \((\alpha,\gamma)\) satisfies
		\begin{align}
			&N=2,\quad 0.54<\alpha<1,\quad \gamma>1;\label{C 2d unialpha}\\
			&N=3,\quad 0.689<\alpha<1,\quad 1<\gamma<3\alpha-1,\label{C 3d unialpha}
		\end{align}
		and that the radially symmetric initial data $(\rho_0,\mathbf{u}_0)$ satisfies \eqref{C initial data}. Then the initial-boundary-value problem \eqref{0}–\eqref{0-2} admits a unique global radially symmetric classical solution $(\rho,\mathbf{u})$ satisfying, for any \(T>0\) and any $(x,t)\in\Omega\times[0,T]$, 
		\begin{align*}
			C^{-1}\le \rho(x,t)\le C,
		\end{align*}
		where the constant \(C>0\) depends on the initial data but is independent of \(T\). Moreover, the following large-time behavior holds:
		\begin{align}
			\lim\limits_{t\to\infty}\norm{\rho(t)-\overline{\rho_0}}_{C(\overline{\Omega})}=0
		\end{align}
		and
		\begin{align}\label{large time behavior2}
			\lim\limits_{t\to\infty}(\norm{\nabla\rho(t)}_{L^2(\Omega)}+\norm{\nabla\mathbf{u}(t)}_{L^2(\Omega)}+\norm{\nabla^2\mathbf{u}(t)}_{L^2(\Omega)})=0.
		\end{align}
	\end{thm}
	\begin{rmk}
		When \eqref{C 2d unialpha} holds, the global existence of classical solutions in Theorem \ref{Thm2} is already guaranteed by Theorem \ref{Thm1}.  However, when \eqref{C 3d unialpha} holds, the global existence asserted in Theorem \ref{Thm2} is no longer covered by Theorem \ref{Thm1}, because the validity of \eqref{C 3d alpha} does not imply that of \eqref{C 3d unialpha}. Indeed, we observe that
		\begin{align*}
			3\alpha-1<6\alpha-3+\frac{3-5\alpha}{2n_3(\alpha)}\Longleftrightarrow \alpha>\alpha_{3,-}\left(\frac{5\alpha-3}{6\alpha-4}\right),
		\end{align*}
		which \textbf{fails} for \(0.689<\alpha<0.7\).
	\end{rmk}
	\begin{rmk}
		The ranges of \(\alpha\) in \eqref{C 2d unialpha} and \eqref{C 3d unialpha} are determined by the condition \(2<n_N(\alpha)\), which guarantees a time-independent estimate for \(\int_\Omega\rho|\mathbf{u}|^4 dx\).
	\end{rmk}
	
	We now comment on the analysis of the proof of Theorem \ref{Thm1}. Taking the two-dimensional case as an example, we extend the global existence result for classical solutions away from vacuum in \cite{zhang-2025} from the range \(0.8\le\alpha<1\) and $\gamma>1$ to \(0.5<\alpha<1\) and $\gamma>1$. The requirement \(\alpha<1\) is essential for deriving a positive lower bound on the density.  
	
The local well-posedness of classical solutions to the initial-boundary-value problem \eqref{0}-\eqref{0-2} away from vacuum can be established by adapting the arguments in \cite{Cho-Kim-2006,ya-2015,zpx-2014,zhang-2025}. As noted in \cite{zhang-2025}, the key to continuing this local solution globally is to obtain both an upper bound and a positive lower bound for the density. Fix \(T>0\) prior to the maximal existence time \(T^*\) of the local classical solution. First, we perform the standard energy estimate to obtain bounds on \(\|\rho^{\frac{1}{2}}\mathbf u\|_{L^\infty(0,T; L^2(\Omega))}\) and \(\|\rho\|_{L^\infty(0,T; L^\gamma(\Omega))}\). We then apply the well-known B-D entropy estimate, following \cite{bd-2003,bd-2004}, to obtain a bound on \(\|\nabla\rho^{\alpha-\frac{1}{2}}\|_{L^\infty(0,T;L^2(\Omega))}\). These estimates not only yield a bound on \(\|\rho\|_{L^\infty(0,T;L^s(\Omega))}\) for any finite \(s\), but also provide, in the radially symmetric case, the crucial \(r\)-weighted \(L^\infty\) bound \(\|\rho^{\alpha-\frac{1}{2}} r^\xi\|_{L^\infty(0,T;L^\infty(\Omega))}\) for any sufficiently small \(\xi>0\). This \(r\)-weighted estimate on the density is then used to obtain higher integrability of both the velocity field and the density gradient. To this end, for any \(\alpha\) and \(\gamma\) satisfying \eqref{C 2d alpha}, we choose a constant \(1<k\in M_{\mathrm{set}}\) depending on \(\alpha\) such that
\begin{align*}
	1-\frac{k\sqrt{2k-1}-2k+1}{k^2-2k+1}<\alpha<1-\frac{1}{2k}.
\end{align*}
The left-hand side of the inequality is equivalent to \(k<n_2(\alpha)\), which ensures that the bound \(\|\rho^{\frac{1}{2k}}\mathbf{u}\|_{L^\infty(0,T;L^{2k}(\Omega))}\) can be obtained. Next, for technical reasons, we introduce the set \(M_{\mathrm{set}}\) to obtain a bound on \(\|\nabla\rho^{\alpha-1+\frac{1}{2k}}\|_{L^\infty(0,T;L^{2k}(\Omega))}\). The \(L^{2k}\) estimate for the density gradient immediately gives an upper bound for the density. Finally, instead of estimating the lower bound of the density in Eulerian coordinates as in \cite{zhang-2025}, we derive a positive lower bound of the density in Lagrangian coordinates by utilizing the bound on \(\|\nabla\rho^{\alpha-1+\frac{1}{2k}}\|_{L^\infty(0,T;L^{2k}(\Omega))}\)with \(\alpha<1-\frac{1}{2k}\), thereby avoiding the difficulty of estimating \(\norm{\rho^{-1}}_{L^\infty(0,T;L^p(\Omega))}\) for some $p\ge1$ in Eulerian coordinates.
	
Another key point in continuing this local solution to a global one is to obtain higher-order estimates for the density and the velocity. Indeed, the estimates that were established while deriving the upper and lower bounds of the density, namely
\begin{align*}
	\sup_{0\leq t\leq T}\Bigl(\norm{\mathbf{u}}_{L^{2k}(\Omega)}&+\norm{\rho}_{L^\infty(\Omega)}+\norm{\rho^{-1}}_{L^\infty(\Omega)}+\norm{\nabla\rho}_{L^2(\Omega)}+\norm{\nabla\rho}_{L^{2k}(\Omega)}\Bigr)\\
	&+\int_0^T\left(\norm{\nabla\rho}_{L^2(\Omega)}^2+\norm{\nabla\rho}_{L^{2k}(\Omega)}^{2k}+\norm{\nabla\mathbf{u}}_{L^2(\Omega)}^2\right)dt\leq C(T),
\end{align*}
already suffices to close the higher-order estimates of both the density and the velocity fields. We can therefore apply a standard continuation argument to conclude that \(T^*=\infty\), which implies that the classical solution exists globally in time.

We now comment on the analysis of the proof of Theorem \ref{Thm2}. Taking the two-dimensional case as an example again, the global existence of the classical solution is guaranteed by Theorem \ref{Thm1}, and the key to obtaining the large-time behavior of the solution is to derive time-independent upper and lower bounds for the density. To this end, we first obtain a time-independent bound for \(\|\rho^{\frac14}\mathbf{u}\|_{L^\infty(0,T;L^4(\Omega))}\) under the condition \(2<n_2(\alpha)\), which is guaranteed by \eqref{C 2d unialpha}. We then use the dissipation estimate of the velocity field to obtain time-independent bounds on \(\|\nabla\rho^{\alpha-\frac{3}{4}}\|_{L^\infty(0,T;L^4(\Omega))}\) and \(\|\nabla\rho^{\frac{\gamma+3\alpha-3}{4}}\|_{L^4(\Omega\times(0,T))}\), which directly yields a uniform upper bound for the density. As shown in \cite{Li-2008}, this uniform \(L^4\) estimate for the density gradient, together with the \(L^4\) space-time dissipation estimate, suffices to guarantee that the density converges to its equilibrium state. Hence the density possesses a uniform positive lower bound. With the uniform estimate
\begin{align*}
	\sup_{0\leq t<\infty}\Bigl(\norm{\mathbf{u}}_{L^4(\Omega)}&+\norm{\rho}_{L^\infty(\Omega)}+\norm{\rho^{-1}}_{L^\infty(\Omega)}+\norm{\nabla\rho}_{L^2(\Omega)}+\norm{\nabla\rho}_{L^4(\Omega)}\Bigr)\\
	&+\int_0^\infty\left(\norm{\nabla\rho}_{L^2(\Omega)}^2+\norm{\nabla\rho}_{L^{4}(\Omega)}^{4}+\norm{\nabla\mathbf{u}}_{L^2(\Omega)}^2\right)dt\leq C
\end{align*}
at hand, repeating the higher-order estimates, we obtain time-independent bounds on \(\|\nabla\mathbf {u}\|_{L^\infty(0,T;L^2(\Omega))}\) and \(\|\sqrt\rho\dot{\mathbf u}\|_{L^2(\Omega\times(0,T))}\). Taking a step further, we show that the bounds on \(\|\sqrt{\rho}\dot{\mathbf{u}}\|_{L^\infty(0,T;L^2(\Omega))}\), \(\|\mathbf{u}\|_{L^\infty(0,T;H^2(\Omega))}\), and \(\|\nabla\dot{\mathbf{u}}\|_{L^2(\Omega\times(0,T))}\) are independent of time. These estimates directly yield the large-time behavior \eqref{large time behavior2}.
	
The proofs of Theorem \ref{Thm1} and Theorem \ref{Thm2} in three dimensions are essentially similar to those in two dimensions, but require more involved calculations and a finer analysis of the exponents.

\subsection{Global classical solution away from vacuum including viscous shallow water equations \((\alpha=1)\)}
In this subsection, we consider the global existence and large‑time behavior of classical solutions away from vacuum to the initial‑boundary value problem \eqref{0}–\eqref{0-2} for the endpoint case \(\alpha = 1\) with arbitrarily large initial data. In particular, our results cover the physically important viscous Saint‑Venant model that describes shallow water motion. 

We present the result concerning the global existence and large‑time behavior of classical solutions in two dimension.
\begin{thm}[Global classical solutions and large-time behavior]\label{Thm SV 2d}
			Let $N=2$. Assume that \((\alpha,\gamma)\) satisfies
		\begin{align}
			\alpha=1,\quad \gamma\ge\frac{3}{2}\label{SV-2gamma}
		\end{align}
		and that the radially symmetric initial data $(\rho_0,\mathbf{u}_0)$ satisfies \eqref{C initial data}. Then the initial-boundary-value problem \eqref{0}–\eqref{0-2} admits a unique global radially symmetric classical solution that possesses all the properties stated in Theorem \ref{Thm2}.
	\end{thm}
\begin{rmk}
    Theorem \ref{Thm SV 2d} covers the viscous Saint-Venant model for the shallow water motion with $N=2, \alpha=1$ and $\gamma=2$. 
\end{rmk}
\begin{rmk}
    The global existence of classical solutions with arbitrarily large initial data for the 2D viscous Saint‑Venant system in the whole space, under both far‑field vacuum and non‑vacuum conditions, will be reported in the future work \cite{HMZ-SV-2025}.
\end{rmk}

We now comment on Theorem \ref{Thm SV 2d}. Theorem \ref{Thm SV 2d} establishes the global existence of classical solutions with large initial data away from vacuum, as well as the large-time behavior of the solutions, for the initial-boundary value problem \eqref{0}--\eqref{0-2} in the two-dimensional endpoint case $\alpha= 1$. In particular, for $\gamma=2$, this is the first global classical solution for the viscous Saint‑Venant model in higher dimensions for arbitrarily large initial data. Indeed, the difficulty in the endpoint case lies in obtaining a positive lower bound for the density. To this end, we first utilize the special structure of two-dimensional radial symmetry to obtain an estimate for $\|\sqrt{\rho}\mathbf{u}\|_{L^2(0,T;L^\infty(\Omega))}$. Using this estimate, we derive the $L^\infty_{t,x}$ bound for the so-called effective velocity $\mathbf{w} = \mathbf{u}+\nabla\log\rho$.
Then, instead of using velocity field estimates to bound the effective velocity as in previous works, we obtain an \(L^\infty_{t,x}\) bound on the velocity field itself from the \(L^\infty_{t,x}\) bound on the effective velocity. This immediately yields an \(L^\infty_{t,x}\) bound on \(\nabla\log\rho\), which is the key to obtaining a positive lower bound for the density in Lagrangian coordinates. After obtaining both the upper and positive lower bounds on the density, the global existence of classical solutions and their large-time behavior follow by exactly the same arguments as for the case \(\alpha<1\).

It is worth noting that for $\alpha < 1$, global classical solutions with large initial data can be established in both two and three dimensions, and in the two-dimensional case, only $\gamma > 1$ is required. In contrast, for the critical case $\alpha = 1$, the result is only available in two dimensions, and the constraint on the adiabatic exponent $\gamma$ is tightened to $\gamma \ge 3/2$. The fundamental reason lies in the fact that obtaining a positive lower bound for the density is considerably more challenging in the critical case compared to non-endpoint cases. In our approach, a critical embedding in the two-dimensional spherically symmetric setting is employed, which constitutes the key obstacle preventing the extension of this endpoint result to three dimensions.
	
	\subsection{Global weak solutions allowing vacuum \((\alpha\ge1)\)}
    As shown in the previous subsection, the condition \(\alpha<1\) in Theorem \ref{Thm1} arises from estimating the positive lower bound of the density, which guarantees the regularity of the velocity field. A natural question therefore arises: for \(\alpha\ge1\), can one still obtain global weak solutions with bounded density that allow vacuum? This is precisely the goal of the present subsection.
    
		In this subsection we consider the global existence of radially symmetric weak solutions that allow vacuum, as well as the finite-time vanishing of vacuum for such weak solutions.  We prescribe radially symmetric initial data \(\rho_0(x)=\rho_0(r)\), \(\mathbf{m}_0(x)=m_0(r)\frac{\mathbf{x}}{r}\) satisfying
	\begin{align}\label{2d weak initial}
		\begin{split}
			&0\leq \rho_0\in L^\gamma(\Omega),\quad\int_\Omega \rho_0 dx>0,\quad \nabla\rho_0^{\alpha-1+\frac{1}{2q}}\in L^{2q}(\Omega),\\ &\frac{|\mathbf{m}_0|^2}{\rho_0}\in L^1(\Omega),\quad \frac{|\mathbf{m}_0|^{2p}}{\rho_0^{2p-1}}\in L^1(\Omega),\quad \mathbf{m}_0=0, \text{ a.e. on }\left\{\rho_0=0\right\},
		\end{split}
	\end{align}
	where \(1<p\in\mathbb{R}\) and \(q\in M_{set}\).
	Next, we give the definition of a weak solution to the initial-boundary-value problem \eqref{0}-\eqref{0-2}.
	\begin{defi}[Global weak solution]\label{Def weak sol}
		A radially symmetric pair \((\rho,\mathbf{u})\) is called a global weak solution to the initial-boundary-value problem \eqref{0}-\eqref{0-2} if
		\begin{enumerate}
			\item \(\rho\ge 0\) a.e. and
			\begin{align}\label{weak regularity}
				\begin{split}
					&\rho\in L^\infty(\Omega\times(0,T)),\quad \nabla \rho^{\alpha-\frac{1}{2}}\in L^\infty(0,T;L^{2}(\Omega)),\quad \nabla \rho^{\alpha-1+\frac{1}{2q}}\in L^\infty(0,T;L^{2q}(\Omega)),\\
					&\sqrt{\rho}\mathbf{u}\in L^\infty(0,T;L^2(\Omega)),\quad \rho^\alpha\nabla \mathbf{u}\in L^2(0,T;W^{-1,1}_{\rm{loc}}(\Omega));
				\end{split}
			\end{align}
			\item for any \(t_2\ge t_1\ge 0\) and any \(\zeta\in C^1(\overline{\Omega}\times[t_1,t_2])\), the continuity equation holds in the following sense:
			\begin{align}\label{mass equ weak}
				\int_\Omega \rho\zeta dx\left|_{t_1}^{t_2}\right.=\int_{t_1}^{t_2}\int_\Omega (\rho\zeta_t+\sqrt{\rho}\sqrt{\rho}\mathbf{u}\cdot\nabla\zeta)dxdt;
			\end{align}
			\item for any \(t_2\ge t_1\ge 0\), any \(\phi\in C^1(\overline{\Omega}\times[t_1,t_2])\), and any $b\ge \alpha$, the renormalized continuity equation holds in the following sense:
			\begin{align}\label{rnm mass equ}
				\begin{split}
					\int_\Omega \rho^b\phi dx\left|_{t_1}^{t_2}\right.=\int_{t_1}^{t_2}\int_\Omega\left( \rho^b\partial_t \phi +\frac{2b(b-1)}{2b-1}\sqrt{\rho}\mathbf{u}\cdot\nabla \rho^{b-\frac{1}{2}}\phi+b\rho^{b-\frac{1}{2}}\sqrt{\rho}\mathbf{u}\cdot\nabla\phi\right) dxdt;
				\end{split}
			\end{align}
			\item furthermore, it holds that
			\begin{align}
				\rho \in C(\overline{\Omega}\times[0,T]);
			\end{align}
			\item for any \(\bm{\psi}=(\psi_i)_{i=1}^N\) with \(\psi_i\in C^2(\overline{\Omega}\times [0,T])\), satisfying \(\bm{\psi}(x,t)|_{\partial\Omega}=0\) and \(\bm{\psi}(\cdot, T)=0\), the momentum equation holds in the following sense:
			\begin{align}\label{mom equ weak}
				\begin{split}
					\int_{\Omega}\mathbf{m}_0\cdot\bm{\psi}(\cdot,0)dx+\int_0^T\int_\Omega[\sqrt{\rho}\sqrt{\rho}\mathbf{u}\cdot\partial_t\bm{\psi}+\sqrt{\rho}\mathbf{u}\otimes\sqrt{\rho}\mathbf{u}:\nabla\bm{\psi}]dxdt\\
					+\int_0^T\int_\Omega \rho^\gamma\div\bm{\psi} dxdt-\langle \rho^\alpha \nabla \mathbf{u},\nabla\bm{\psi}\rangle-(\alpha-1)\langle \rho^\alpha\div\mathbf{u}, \div\bm{\psi}\rangle=0.
				\end{split}
			\end{align}
			The diffusion terms are defined by
			\begin{align}\label{diff term}
				\begin{split}
					\langle \rho^\alpha \nabla\mathbf{u},\nabla\bm{\psi}\rangle&:=-\int_0^T\int_\Omega\left( \rho^{\alpha-\frac{1}{2}}\sqrt{\rho}\mathbf{u}\cdot\Delta\bm{\psi} +\frac{2\alpha}{2\alpha-1}\nabla \rho^{\alpha-\frac{1}{2}}\cdot\nabla\bm{\psi}\cdot\sqrt{\rho}\mathbf{u}\right) dxdt,\\
					\langle \rho^\alpha \div\mathbf{u},\div\bm{\psi}\rangle&:=-\int_0^T\int_\Omega\left( \rho^{\alpha-\frac{1}{2}}\sqrt{\rho}\mathbf{u}\cdot\nabla\div\bm{\psi} +\frac{2\alpha}{2\alpha-1}\nabla \rho^{\alpha-\frac{1}{2}}\cdot\sqrt{\rho}\mathbf{u} \div\bm{\psi}\right) dxdt;
				\end{split}
			\end{align}
		\end{enumerate}
	\end{defi}
	\begin{rmk}
		As explained in \cite{Guo-2008}, we need to clarify in what sense our definition of weak solution satisfies boundary condition \eqref{0-2}. We rewrite the weak form \eqref{mass equ weak} of the continuity equation as follows:
		\begin{align}\label{yy1}
			\int_0^R \rho\zeta r^{N-1} dr|_{t_1}^{t_2}=\int_{t_1}^{t_2}\int_0^R(\rho\zeta_t+\rho u\zeta_r)r^{N-1}drdt,
		\end{align}
		where $\zeta\in C^1([0,R]\times[t_1, t_2])$. Define \(\tilde{\zeta}(r,t)=\zeta_1(t)\zeta_2(r)\), where \(\zeta_1(t)=1\) on \([t_1,t_2]\) and \(\zeta_2(r)\) is given by
		\begin{align*}
			\zeta_2(r)=\left\{
			\begin{array}{ll}
				1,& r\in[0,R-\delta],\\
				1-\frac{1}{\delta}(r-(R-\delta)),& r\in[R-\delta,R].
			\end{array}
			\right.
		\end{align*}
		Indeed, \eqref{yy1} also holds for Lipschitz continuous functions \(\zeta\). Substituting \(\tilde{\zeta}\) into \eqref{yy1} gives
		\begin{align*}
			\frac{1}{\delta}\left|\int_{t_1}^{t_2}\int_{R-\delta}^R \rho u r^{N-1}drdt\right|\leq \left|\int_{R-\delta}^R\rho(r,t_2)(\zeta_2(r)-1)r^{N-1}dr-\int_{R-\delta}^R\rho(r,t_1)(\zeta_2(r)-1)r^{N-1}dr\right|.
		\end{align*}
		Letting \(\delta\to 0\) shows that \(\rho u(R,t)=0\) in the sense of trace.
	\end{rmk}
	
	We next present the global existence result for weak solutions that allow vacuum.
	\begin{thm}[Global weak solution in 2D]\label{Thm3}
		Let $N=2$. Assume that $(\alpha,\gamma)$ satisfies
		\begin{align}\label{WI 2d alpha}
			\alpha\ge1,\quad \gamma>\max\left\{1,\alpha-\frac{1}{2}\right\}
		\end{align}
		and that the radially symmetric initial data \((\rho_{0},\mathbf{m}_{0})\) satisfies \eqref{2d weak initial}
		where \(p\) and \(q\) satisfy
		\begin{align}\label{2d pq}
			1<q<p<n_2(\alpha).
		\end{align}
		Moreover, assume that
		\begin{align}\label{WI 2d gamma}
			\gamma\ge \max\left\{\left(1-\frac{1}{2p}\right)\alpha,\alpha-1+\frac{q}{p}\right\}.
		\end{align}
		Then the initial-boundary-value problem \eqref{0}-\eqref{0-2} admits a global radially symmetric weak solution.
	\end{thm}
	
		\begin{thm}[Global weak solution in 3D]\label{Thm4}
		Let $N=3$. Assume that $(\alpha,\gamma)$ satisfies
		\begin{align}\label{WI 3d alpha}
			1\leq \alpha<11.7,\quad \gamma>\max\left\{1,\alpha-\frac{1}{2}\right\}
		\end{align}
		and that the radially symmetric initial data \((\rho_{0},\mathbf{m}_{0})\) satisfies \eqref{2d weak initial}
		where \(p\) and \(q\) satisfy
		\begin{align}\label{3d pq}
			1.5<q<p<n_3(\alpha).
		\end{align}
		Moreover, assume that
		\begin{align}\label{WI 3d gamma}
				2\alpha-1\leq \gamma <3\alpha-1+\frac{\alpha-1}{2p}+\min\left\{2\alpha-1,(3\alpha-2)\left(1-\frac{1}{q}\right)\right\}.
		\end{align}
		Then the initial-boundary-value problem \eqref{0}-\eqref{0-2} admits a global radially symmetric weak solution.
	\end{thm}
	\begin{rmk}
		Theorem \ref{Thm3} covers the viscous Saint-Venant model for the shallow water motion with $N=2, \gamma=2$ and $\alpha=1$. 
	\end{rmk}
	\begin{rmk}
		Since \(n_2(\alpha)>1\) for any \(1\leq \alpha<\infty\) and \(n_3(\alpha)>1.55\) for any $1\leq\alpha<11.7$, we see that $\eqref{2d pq}$ and $\eqref{3d pq}$  are meaningful. 
	\end{rmk}
	\begin{rmk}
		The conditions in Theorem \ref{Thm3} and \ref{Thm4} imply that \(\rho_{0}\in L^{\infty}(\Omega)\). Since \(\gamma>\alpha-\frac{1}{2}\) and $q>1$, it follows that $\rho_0^{\alpha-1+\frac{1}{2q}}\in L^1(\Omega)$. Together with \(\nabla\rho_0^{\alpha-1+\frac{1}{2q}}\in L^{2q}(\Omega)\), the standard Sobolev embedding gives \(\rho_{0}\in L^{\infty}(\Omega)\).
	\end{rmk}
	\begin{rmk}\label{RMK 2d}
		Since the initial density is bounded, smaller \(p\) and \(q\) correspond to weaker requirements on the initial data. Indeed, when the initial data possess higher regularity, for example \(q>n_N(\alpha)\), the existence of weak solutions is guaranteed by Theorem \ref{Thm3} and \ref{Thm4}. However, the corres4ponding high regularity $\norm{\nabla\rho^{\alpha-1+\frac{1}{2q}}}_{L^\infty(0,T; L^{2q}(\Omega))}$ need not persist.
	\end{rmk}
	\begin{rmk}
		Indeed, the global existence of radially symmetric weak solutions to the initial-boundary-value problem \eqref{0}--\eqref{0-2} under the assumptions
        \[
        N=3,\quad \alpha=1,\quad 1<\gamma<3,\quad \sqrt{\rho_0}\in H^1(\Omega),\quad \frac{|\mathbf{m}_0|^2}{\rho_0}\in L^1(\Omega),
        \]
        along with additional integrability on the velocity field and natural compatibility conditions, has been established in \cite{Guo-2008}.
        Compared with their result, we impose stronger conditions on the initial data and obtain solutions with higher regularity. For instance, the densities of our weak solutions are bounded, whereas theirs are not.
	\end{rmk}
		\begin{rmk}
		For the general high-dimensional initial data,  Li-Xin \cite{Li-2015} and Vasseur-Yu \cite{Vasseur-2016} independently obtained the global existence of weak solutions with $\mu(\rho)=\rho,\lambda(\rho)=0$. Moreover, for a non-symmetric viscous term of the form $-\operatorname{div}(\mu(\rho)\nabla \mathbf{u})-\nabla(\lambda(\rho)\operatorname{div}\mathbf{u})$, the work \cite{Li-2015} also considers viscosities satisfying the BD relation $\mu(\rho)=\rho^\alpha$ and $\lambda(\rho)=(\alpha-1)\rho^\alpha$, under the following conditions on $\alpha$ and $\gamma$: for $N=2$, $\alpha > \frac{1}{2}$, $\gamma>1$ and $\gamma \ge 2\alpha - 1$; for $N=3$, $\frac{3}{4} \le \alpha < 2$ and
\[
\gamma \in 
\begin{cases} 
(1, 6\alpha - 3), & \text{if } \alpha \in [3/4, 1], \\ 
[2\alpha - 1, 3\alpha - 1], & \text{if } \alpha \in (1, 2).
\end{cases}
\]
Later, Bresch-Vasseur-Yu \cite{va-2022} considered a physical symmetric viscous term $-\operatorname{div}(\mu(\rho) \mathbb{D}\mathbf{u} )-\nabla(\lambda(\rho)\operatorname{div}\mathbf{u})$ and obtained the global existence of entropy weak solutions in a periodic domain $\mathbb{T}^3$, which cover the case of the viscosities satisfy the BD relation with $\mu(\rho)= \rho^\alpha$ where $\frac23<\alpha<4$. Further, \cite{ab-2020} remove the restriction about the upper bound of $\alpha$.
	\end{rmk}
The following theorem shows that the vacuum region vanishes in finite time for the weak solution.
		\begin{thm}[Vanishing of vacuum states]\label{Thm5}
		Let \(N=2\) or \(N=3\). Assume that $(\alpha,\gamma)$ satisfies 
		\begin{align}
			&N=2, \quad 1\leq \alpha<7.46,\quad\gamma>1,\quad\gamma\ge 2\alpha-1;\label{2d van}\\
			&N=3,\quad 1\leq \alpha<5.81,\quad  \gamma>1,\quad2\alpha-1\leq \gamma<3\alpha-1 \label{3d van}
		\end{align}
		and that  the radially symmetric initial data \((\rho_{0},\mathbf{m}_{0})\) satisfies \eqref{2d weak initial}, where \(p\) and \(q\) satisfy
		\begin{align}
			p=2,\quad \frac{N}{2}<q<2.
		\end{align}
		Then the global radially symmetric weak solution obtained in Theorem \ref{Thm3} and \ref{Thm4} has a uniformly bounded density and the vacuum state vanishes in finite time. Specifically, there exist positive constants \(\rho^+,\;T_0,\;\rho^->0\) depending only on the initial data such that
		\begin{align*}
			&\rho(x,t)\leq \rho^+,\quad(x,t)\in\Omega\times(0,\infty), \\ &\rho(x,t)\ge \rho^-, \quad(x,t)\in\Omega\times[T_0,\infty).
		\end{align*}
	\end{thm}
		\begin{rmk}
		Theorem \ref{Thm5} covers the viscous Saint-Venant model for the shallow water motion with $N=2, \gamma=2$ and $\alpha=1$. 
	\end{rmk}
		\begin{rmk}
		For the one-dimensional case, Li-Li-Xin \cite{Li-2008} obtained the global existence of weak solutions to \eqref{0}-\eqref{0-2} with viscous term $-(\rho^\alpha u_x)_x$ $(\alpha > 1/2)$ and proved that for any global entropy weak solution, any vacuum state must vanish within finite time. Theorem \ref{Thm5} extend such result to the high-dimensional case.
	\end{rmk}
    
	We now comment on the analysis of the proof of Theorem \ref{Thm3}. First, we introduce the approximate system from \cite{Guo-2008}:
	\begin{align*}
		\left\{
		\begin{array}{l}
			\partial_t\rho_\eta+\div(\rho_\eta\mathbf{u}_\eta)=0,\\
			\partial_t(\rho_\eta \mathbf{u}_\eta)+\div(\rho_\eta\mathbf{u}_\eta\otimes\mathbf{u}_\eta)+\nabla(\rho_\eta^\gamma)=\div(\tilde{\mu}(\rho_\eta)\nabla \mathbf{u}_\eta)+\nabla(\tilde{\lambda}(\rho_\eta)\div\mathbf{u}_\eta),\\
			(\rho_{\eta},\mathbf{u}_{\eta})|_{t=0}=(\rho_{\eta,0},\mathbf{u}_{\eta,0}),\\
			\mathbf{u}_{\eta}|_{\partial\Omega_\eta}=0
		\end{array}
		\right.
	\end{align*}
	where $t\ge0$ and $x\in \Omega_\eta=\Omega-B_\eta$. The viscosity coefficients satisfy
	\begin{align*}
		\tilde{\mu}(\rho_\eta)=\rho_\eta^\alpha+\eta\rho^\delta_\eta, \quad
		\tilde{\lambda}(\rho_\eta)=(\alpha-1)\rho_\eta^\alpha+\eta(\delta-1)\rho_\eta^\delta,
	\end{align*}
	where \(\delta\) is chosen as follows:
		\begin{align*}
			1-\frac{p\sqrt{2p-1}-2p+1}{p^2-2p+1}<\delta<1-\frac{1}{2p}.
	\end{align*}
	The initial data \((\rho_{\eta,0},\mathbf{u}_{\eta,0})\) of the approximate system satisfy
\begin{align*}
	\begin{split}
		&\rho_{\eta,0}\rightarrow \rho_0 \text{ in }L^\gamma(\Omega),\quad \nabla\rho_{\eta,0}^{\alpha-\frac{1}{2}}\rightarrow\nabla\rho_0^{\alpha-\frac{1}{2}} \text{ in }L^2(\Omega),\quad \rho_{\eta,0}\ge C\eta^{\frac{1}{\alpha-\delta}},\\
		&\nabla\rho_{\eta,0}^{\alpha-1+\frac{1}{2q}}\rightarrow  \nabla\rho_{0}^{\alpha-1+\frac{1}{2q}}\text{ in }L^{2q}(\Omega), \quad |\mathbf{m}_{\eta,0}|^{2p}/\rho_{\eta,0}^{2p-1}\rightarrow |\mathbf{m}_0|^{2p}/\rho_0^{2p-1} \text{ in }L^1(\Omega),\\
		&|\mathbf{m}_{\eta,0}|^{2}/\rho_{\eta,0}\rightarrow |\mathbf{m}_0|^{2}/\rho_0 \text{ in }L^1(\Omega), \quad \mathbf{u}_{\eta,0}|_{\partial\Omega_\eta}=0,
	\end{split}
\end{align*}
where \(C>0\) is a generic constant and $\mathbf{m}_{\eta,0}=\rho_{\eta,0}\mathbf{u}_{\eta,0}$. 

Next, we investigate the global solvability of the approximate system. Since the initial data $\rho_{\eta,0}$ are strictly positive and the approximate system stays away from the origin, the problem becomes essentially one-dimensional, and the local well-posedness of classical solutions can be established by standard arguments. As shown in \cite{Guo-2008}, the key to extending the local classical solution globally is to obtain both an upper bound and a positive lower bound for \(\rho_\eta\). Let \(T^*\) denote the maximal existence time of the classical solution and fix \(0<T<T^*\).  Below we devote ourselves to obtaining an \(\eta\)-independent upper bound and an \(\eta\)-dependent positive lower bound for \(\rho_\eta\).  We first apply the standard energy estimate and the B-D entropy estimate to obtain \(\eta\)-independent bounds for \(\|\rho_\eta\|_{L^\infty(0,T;L^\gamma(\Omega_\eta))}\) and \(\|\nabla\rho_\eta^{\alpha-\frac{1}{2}}\|_{L^\infty(0,T;L^2(\Omega_\eta))}\). Applying the standard Sobolev embedding on \(\Omega_\eta\), we can obtain a bound on \(\|\rho_\eta\|_{L^\infty(0,T;L^s(\Omega_\eta))}\) for any finite \(s\ge1\). However, this bound will depend on the domain \(\Omega_\eta\) and thus on \(\eta\). Moreover, since the annular domain does not possess the same scaling properties as a ball, we are unable to explicitly write down how the embedding constant depends on \(\eta\). To overcome this difficulty, inspired by \cite{Guo-2008}, we extend \(\rho_\eta\) constantly to the whole domain \(\Omega\) and prove that \(\rho_\eta^{\alpha-\frac{1}{2}}(\eta,t)\leq C\eta^{-\zeta}\) for some \(\zeta<2\). This yields an \(\eta\)-independent bound on \(\|\rho_\eta^{\alpha-\frac{1}{2}}\|_{L^\infty(0,T;L^1(\Omega))}\). Consequently, applying the Sobolev embedding on \(\Omega\), we obtain \(\eta\)-independent bounds on \(\|\rho_\eta\|_{L^\infty(0,T;L^s(\Omega))}\) for any finite \(s\ge1\) and the $r$-weighted $L^\infty$ estimate $\|\rho_\eta^{\alpha-\frac{1}{2}} r^\xi\|_{L^\infty(0,T;L^\infty(\Omega_\eta))}$ for any sufficiently small $\xi>0$. With these estimates at hand, we use the fact that \(p<\min\{n_2(\alpha),n_2(\delta)\}\) to obtain an \(\eta\)-independent bound on \(\|\rho_\eta^{\frac{1}{2p}}\mathbf{u}_\eta\|_{L^\infty(0,T;L^{2p}(\Omega_\eta))}\). We then obtain an \(\eta\)-independent estimate for \(\|\nabla\rho_\eta^{\alpha-1+\frac{1}{2q}}\|_{L^\infty(0,T;L^{2q}(\Omega_\eta))}\), which directly yields an \(\eta\)-independent upper bound for \(\rho_\eta\). As for the lower bound on the density, we follow the approach in \cite{Guo-2008} to derive an \(\eta\)-dependent estimate for \(\|\nabla\rho_\eta^{\delta-1+\frac{1}{2p}}\|_{L^\infty(0,T;L^{2p}(\Omega_\eta))}\). Combining this with the fact that \(\delta<1-\frac{1}{2p}\) then yields a positive lower bound for \(\rho_\eta\) that depends on \(\eta\). Therefore, standard arguments then yield higher-order estimates for the approximate solution, and a continuity argument shows that the approximate solution exists globally in time.

Finally, we investigate the stability of the global approximate solutions \((\rho_\eta, \mathbf{u}_\eta)\) defined on \(\Omega_\eta \times [0, \infty)\). We extend \(\rho_\eta\) and \(\mathbf{u}_\eta\) constantly to the whole domain \(\Omega \times [0, \infty)\) and recall the \(\eta\)-independent estimates:
\begin{align*}
	\sup_{0\le t\le T}\Big(\norm{\rho_\eta}_{L^\infty(\Omega)}&+\norm{\nabla\rho_\eta^{\alpha-\frac12}}_{L^2(\Omega)}+\norm{\nabla\rho_\eta^{\alpha-1+\frac{1}{2q}}}_{L^{2q}(\Omega)}\Big)\\
	&+\int_0^T\Big(\norm{\nabla\rho_\eta^{\frac{\gamma+\alpha-1}{2}}}_{L^2(\Omega)}^2+\norm{\nabla\rho_\eta^{\alpha-1+\frac{\gamma-\alpha+1}{2q}}}_{L^{2q}(\Omega)}^{2q}\Big)dt\leq C(T),\\
	\sup_{0\le t\le T}\int_{\Omega}\Big(\rho_\eta|\mathbf{u}_\eta|^2&+\rho_\eta|\mathbf{u}_\eta|^{2p}\Big)dx+\int_0^T\int_\Omega\Big(\rho_\eta^{\alpha}|\nabla\mathbf{u}_\eta|^2+\eta\rho_\eta^{\delta}|\nabla\mathbf{u}_\eta|^2\Big)dxdt\leq C(T).
\end{align*}
These \(\eta\)-independent estimates are stronger than those obtained in \cite{Guo-2008}. Hence, by applying the standard compactness results from \cite{Guo-2008,Mellet-2007}, we obtain a global weak solution \((\rho,\mathbf{u})\) of the original system that allows for vacuum.

We now comment on Theorem \ref{Thm4}. Indeed, the restriction \(1\leq\alpha<11.7\) stems from the limitations of the approximate system. Specifically, to close the positive lower bound on \(\rho_\eta\), we need to obtain an estimate for \(\|\nabla\rho_\eta^{\delta-1+\frac{1}{2n}}\|_{L^\infty(0,T;L^{2n}(\Omega_\eta))}\) with \(\delta<1-\frac{1}{2n}\). To achieve this, we further require the bound on \(\|\rho_\eta^{\frac{1}{2n}}\mathbf{u}_\eta\|_{L^\infty(0,T;L^{2n}(\Omega_\eta))}\), which in turn demands \(n<\min\{n_3(\alpha),n_3(\delta)\}\). The constraints on \(n\) and \(\delta\) thus reduce to
\begin{align*}
	1-\frac{\sqrt{4n(4n^2-n-1)+1}-6n+3}{4n^2-8n+4}<\delta<1-\frac{1}{2n}.
\end{align*}
From the inequalities above, one can compute that the range of \(\delta\) is non-empty provided \(n\ge1.55\). Consequently, to ensure that a suitable artificial viscosity can be added, \(\alpha\) must satisfy \(1.55<n_{3}(\alpha)\), a condition guaranteed by \(1\le\alpha<11.7\).

We now comment on Theorem \ref{Thm5}. As in the estimates for the global classical solution, we first show that \(\|\rho_\eta^{\frac{1}{4}}\mathbf{u}_\eta\|_{L^\infty(0,T; L^4(\Omega_\eta))}\) is bounded independently of \(T\) and \(\eta\), provided that \(2<\min\{n_N(\alpha),n_N(\delta)\}\). We then exploit the dissipation estimate of the velocity field to obtain bounds on \(\|\nabla\rho_\eta^{\alpha-1+\frac{1}{2q}}\|_{L^\infty(0,T;L^{2q}(\Omega_\eta))}\) and \(\|\nabla\rho_\eta^{\alpha-1+\frac{\gamma-\alpha+1}{2q}}\|_{L^{2q}(\Omega_\eta\times(0,T))}\) that are independent of both \(T\) and \(\eta\), which directly yields a uniform upper bound on \(\rho_\eta\) independent of \(T\) and \(\eta\). Therefore, standard compactness arguments show that \(\rho\) has a uniform upper bound and that the bounds on \(\|\nabla\rho^{\alpha-1+\frac{1}{2q}}\|_{L^\infty(0,T;L^{2q}(\Omega))}\) and \(\|\nabla\rho^{\alpha-1+\frac{\gamma-\alpha+1}{2q}}\|_{L^{2q}(\Omega\times(0,T))}\) are independent of $T$. Since \(2q>N\), arguments similar to those in \cite{Li-2008} show that the uniform \(L^{2q}\) estimate for the density gradient, together with the corresponding \(L^{2q}\) space-time dissipation estimate, suffice to guarantee that the density converges to its equilibrium state. Consequently, the vacuum state of the weak solution vanishes in finite time.

\subsection{Global weak solutions with possible vacuum at the origin $(\alpha=1)$}
In this section, we investigate whether the weak solution to the initial-boundary-value problem \eqref{0}–\eqref{0-2} with $\alpha = 1$ develops a vacuum in finite time, given that the initial data is away from vacuum. The following Theorem \ref{T_2} is our main result.
\begin{thm}\label{T_2}
				Let \(N=2\) or \(N=3\). Assume that $(\alpha,\gamma)$ satisfies 
	\begin{align}\label{wii alpha}
		\alpha=1,\quad \gamma>1,
	\end{align}
	and that the radially symmetric initial data \((\rho_{0},\mathbf{u}_{0})\) satisfies, for any $n\in\mathbb{N}^+$, 
	\begin{equation}
		\begin{cases}
			0<\underline{\rho_0}\leq \rho_0(x)\leq\overline{\rho_0},\ \ \rho_0\in W^{1,2q}(\Omega),\ \ \nabla\rho_0\in L^{\infty}(\Omega\backslash B_{\frac{1}{n}}) ,\\
			\mathbf{u}_0\in L^{2p}(\Omega)\cap H^1(\Omega\backslash B_{\frac{1}{n}}),\ \ \mathbf{u}_0|_{\partial\Omega}=0,\label{1'}
		\end{cases}
	\end{equation}
	where \(\overline{\rho_0}\) and \(\underline{\rho_0}\) are positive constants and \(p\) and \(q\) satisfy
	\begin{align*}
		\frac{N}{2}<q<p<\infty,\quad q\in M_{set}.
	\end{align*}
	Moreover, assume that
	\begin{align}\label{wii gamma}
		\begin{split}
			&N=2,\quad 1+\frac{1}{p}\leq \gamma;\\
			&N=3,\quad 1+\frac{1}{p}\leq \gamma<3-\frac{1}{q}.
		\end{split}
	\end{align}
	Then the initial-boundary-value problem \eqref{0}-\eqref{0-2} possesses a global radially symmetric weak solution \((\rho,\mathbf{u})\) in the sense of Definition \ref{Def weak sol} that satisfies the following:
	\begin{enumerate}
		\item There exists an upper semicontinuous curve $\underline{r}(\cdot):[0,\infty)\to[0,\infty)$ such that, for any $t\ge 0$,
		\begin{align*}
			\rho(x,t)r^{N-1} = 0,\quad 0\leq |x|\leq \underline{r}(t);  \\
			\rho(x,t) > 0,\quad \underline{r}(t)<|x|\leq R.
		\end{align*}
		\item Let $\mathcal{F}$ be the set
		\begin{align*}
			\mathcal{F}:=\{(x,t):t\ge 0 \text{ and }\underline{r}(t)<|x|\leq R\}.
		\end{align*}
		 For any $(x,t)\in\mathcal{F}$, there exist small positive constants $\tilde{r}$ and $\tilde{t}$ such that
		\begin{align*}
		\Omega\backslash B_{|x|-\tilde{r}}\times[(t-\tilde{t})_+, t+\tilde{t}] \subset\mathcal{F}
		\end{align*}
		and the pair $(\rho,\mathbf{u})$ is H\"{o}lder continuous on $\Omega\backslash B_{|x|-\tilde{r}}\times[(t-\tilde{t})_+, t+\tilde{t}]$ and satisfies
		\begin{align}
			\begin{split}
				&\rho\in L^\infty((t-\tilde{t})_+,t+\tilde{t}; W^{1,\infty}(\Omega\backslash B_{|x|-\tilde{r}}));\\
				& \partial_t\rho\in L^\infty((t-\tilde{t})_+,t+\tilde{t}; L^2(\Omega\backslash B_{|x|-\tilde{r}}));\\
				&\mathbf{u}\in L^\infty((t-\tilde{t})_+,t+\tilde{t}; H^1(\Omega\backslash B_{|x|-\tilde{r}}))\cap L^2((t-\tilde{t})_+,t+\tilde{t}; H^2(\Omega\backslash B_{|x|-\tilde{r}}));\\
				&\partial_t \mathbf{u}\in L^2((t-\tilde{t})_+,t+\tilde{t}; L^2(\Omega\backslash B_{|x|-\tilde{r}})).
			\end{split}
		\end{align}
		The Navier-Stokes equations \eqref{0} hold almost everywhere on \(\Omega\backslash B_{|x|-\tilde{r}}\times[(t-\tilde{t})_+, t+\tilde{t}]\).
		\item It holds that
		\begin{align}
			\underline{r}(t)\leq \left(R^N-\frac{M_0^{\frac{\gamma}{\gamma-1}}}{\omega_N((\gamma-1)E_0)^{\frac{1}{\gamma-1}}}\right)^{\frac{1}{N}},
		\end{align}
		where \(\omega_N\) denotes the measure of the unit ball in \(\mathbb R^N\) and
		\begin{align*}
			M_0:=\int_\Omega \rho_0(x)dx,\quad E_0:=\frac{1}{2}\int_\Omega\rho_0(x)|\mathbf{u}_0(x)|^2dx+\frac{1}{\gamma-1}\int_\Omega\rho_0^{\gamma}(x)dx.
		\end{align*}
	\end{enumerate}
	\end{thm}

    \begin{rmk}
        Theorem \ref{T_2} applies to a large class of parameters and initial data, for example \((N,\alpha,\gamma,p,q)=(2,1,2,4,2)\) and \((\rho_0,\mathbf u_0)\) given by
        \begin{align*}
            \rho_0(x)=|x|^{\frac{2}{3}}+1,\quad \mathbf{u}_0(x)=|x|^{-\frac{1}{100}}\chi(|x|)\frac{x}{r},
        \end{align*}
        where \(\chi\) is a smooth cut-off function equal to \(1\) on \([0,\frac{R}{2}]\) and equal to \(0\) on \([\frac{2}{3}R,R]\).
    \end{rmk}

    \begin{rmk}
        Due to the limitations of the approximate system, in Theorem \ref{T_2} we can only treat the case \(\alpha=1\). Indeed, when \(\alpha>1\) the added artificial viscosity prevents us from establishing a positive lower bound on the approximate density that is independent of the approximation parameter.
    \end{rmk}

    Theorem \ref{T_2} shows that, under the assumption that the initial data possess sufficiently high local regularity and that the initial density is strictly positive (i.e., away from vacuum), there exists a global weak solution in which the vacuum may appear only inside a certain semi-continuous curve. Moreover, the solution is strong on suitable subsets of the flow region lying outside this semi-continuous curve. Now, we give a brief description of the proof of Theorem \ref{T_2}. We first construct an approximate system on an annulus $\Omega_\iota=\Omega\backslash B_\iota$ away from the origin.  Unlike the approximate system for weak solutions in the previous subsection, the one used here does not contain any artificial-viscosity terms. Consequently, it is easier to adapt the arguments from the previous subsection to obtain \(L^{2p}\) bounds on the velocity field, \(L^{2q}\) bounds on the density gradient, and \(\iota\)-independent upper bounds on the density. We then follow the approach of \cite{guo-2012} to derive a positive lower bound on the density that depends on \(\iota\). This ensures that the approximate solutions exist globally in time.  Combining the established \(\iota\)-independent estimates with standard compactness arguments yields the global existence of weak solutions.
	
	Next, we adapt the method of \cite{Hoff-1992} to give a detailed description of the weak solution, including the distribution of the vacuum states and the regularity of the solution in the fluid region. To this end, we introduce the following particle path:
\begin{equation*}
	0<h=\int_\iota^{r_h^{\iota} (t)}\rho_{\iota}(r,t)r^{N-1}dr.
\end{equation*}
	On the right-hand side of each particle path \(r_h^\iota(t)\) determined by \(h\), we establish an \(\iota\)-independent lower bound on the density,  
\begin{align*}
	0<(C(h/2,T))^{-1}\leq\rho_\iota(r,t)\leq C(T),\ (r,t)\in\big[r_h^{\iota} (t),R\big]\times[0,T]
\end{align*}
	and the \(\iota\)-independent higher-order estimate  
\begin{align*}
	\sup_{t\in[0,T]}\int_{r_h^{\iota} (t)}^R\Big((\partial_{r} u_\iota)^2+(\partial_t \rho_\iota)^2 \Big)dr&+\sup_{t\in[0,T]}\norm{\partial_{r}\rho_\iota}_{L^\infty((r_h^{\iota} (t),R))}\\
	&+\int_0^{T} \int_{r_h^\iota
		(t)}^R{ \Big((\partial_t u_\iota)^2+(\partial_{r r}u_\iota)^2 \Big)}drdt\leq C(h/4,T).
\end{align*}
On the one hand, these \(\iota\)-independent estimates guarantee the existence of the upper semi-continuous curve \(\underline{r}(\cdot)\), inside which lies the vacuum region and outside which lies the non-vacuum region. On the other hand, these uniform estimates also ensure that the solution is strong on certain compact subsets of the fluid region. 

Finally, using standard mass conservation and the finiteness of energy, we derive an accurate upper bound for \(\underline{r}(t)\), showing that the vacuum region of the weak solution we obtain can appear only inside a ball whose radius is strictly less than the radius \(R\) of the domain.

	\section{Global classical solutions away from vacuum \((\frac{N-1}{N}<\alpha<1)\)}
	Throughout this section we always assume that the radially symmetric initial data $(\rho_0,\mathbf u_0)$ satisfies \eqref{C initial data} and that $(\alpha,\gamma)$ satisfies $\frac{N-1}{N}<\alpha\leq1$ and $\gamma>1$.

	\subsection{Local existence of classical solutions}
	The local existence theory for classical solutions below can be established by the arguments similar to those in \cite{Cho-Kim-2006,ya-2015,zpx-2014,zhang-2025}, so we omit the proof.
	\begin{lema}\label{Lem local existence}
		Assume that the radially symmetric initial data $(\rho_0,\mathbf{u}_0)$ satisfies \eqref{C initial data} and that $(\alpha,\gamma)$ satisfy $\alpha>\frac{N-1}{N}$ and $\gamma>1$. Then there exists a small time \(T_0>0\) such that the initial-boundary-value problem \eqref{0}–\eqref{0-2} possesses a unique radially symmetric classical solution $(\rho,\mathbf{u})$ on \(\Omega\times[0,T_0]\) with $\rho>0$, and for every \(0<\tau<T_0\), which satisfies
		\begin{align*}
			\begin{cases}
				\rho\in C\big([0,T_0];H^3(\Omega)\big),~
				\rho_t\in C\big([0,T_0];H^2(\Omega)\big),
				\nonumber\\
				%&\nabla\rho\in C([0,T];L^6\cap D^1\cap D^2),~
				%(\nabla\rho)_t\in C([0,T];H^1),\\
				\mathbf{u}\in C\big([0,T_0];H_0^1(\Omega)\cap H^3(\Omega)\big)\cap L^2\big(0,T_0;H^4(\Omega)\big),~
				\\
				\mathbf{u}_t\in L^\infty\big(0,T_0;H_0^1(\Omega)\big)\cap L^2\big(0,T_0;H^2(\Omega)\big),\nonumber\\
				\mathbf{u}_t\in L^\infty\big(\tau,T_0;H^2(\Omega)\big),\ \mathbf{u}_{tt}\in  L^2\big(0,T_0;L^2(\Omega)\big),\nonumber\\
				\mathbf{u}_{tt}\in  L^\infty\big(\tau,T_0;L^2(\Omega)\big)\cap L^2\big(\tau,T_0;H^1(\Omega)\big).
			\end{cases}
		\end{align*}
	\end{lema}
	
	\subsection{Some useful estimates}
 By Lemma \ref{Lem local existence}, this yields a unique local classical solution to the initial-boundary-value problem \eqref{0}--\eqref{0-2} on $\Omega\times[0,T_0]$. Writing $T^*$ for the maximal existence time of this classical solution, we fix $0<T<T^*$.
	
	We derive several useful estimates in Lagrangian coordinates. Without loss of generality, assume that $\int_0 ^R {\rho {r^{N-1}}dr} =\int_0 ^R {\rho_0 {r^{N-1}}dr=1}$. Define the coordinates transformation
	\begin{equation*}
		y(r,t) = \int_0 ^r {\rho(s,t) {s^{N-1}}ds},\ \tau(r,t)  = t,
	\end{equation*}
	which translates the domain $[0,R]\times[0,T]$ into $[0,1]\times[0,T]$ and satisfies
	\begin{equation*}
		\frac{{\partial y}}{{\partial r}} = \rho {r^{N-1}},\
		\frac{{\partial y}}{{\partial t}} =  - \rho u{r^{N-1}},\
		\frac{{\partial \tau }}{{\partial r}} = 0,\
		\frac{{\partial \tau }}{{\partial t}} = 1,\
		\frac{{\partial r}}{{\partial \tau }} = u.
	\end{equation*}
	Then, \eqref{Equ 2}-\eqref{boundary condition2} is changed to
	\begin{equation}\label{Equ 3}
		\begin{cases}
			\partial_\tau\rho + {\rho ^2}\partial_y({r^{N-1}}u) = 0,\\
			{r^{ - (N-1)}}\partial_\tau u + {\partial_y{\rho ^\gamma }} - \alpha\partial_y{\big({\rho ^{1+\alpha}}{\partial_y({r^{N-1}}u)}\big)} + \frac{{N-1}}{r}{\partial_y\rho ^\alpha u}=0,\\
		\end{cases}
	\end{equation}
	with the initial and boundary conditions given by
	\begin{align}
		\begin{split}\label{C1}
			\left\{
			\begin{array}{l}
			{\left. {(\rho ,u)} \right|_{\tau = 0}} = ({\rho _0},{u_0}),\ y\in[0,1],\\
			u(0 ,\tau)=u(1,\tau) = 0,\ \tau\geq0.
		\end{array}
			\right.
		\end{split}
	\end{align}
	
	The following proposition, taken from \cite{zhang-2025}, shows that the integrability of the velocity field is controlled by the \(r\)-weighted \(L^{p}\) estimates of the density. 
	\begin{prop}\label{Prop C n}
		Assume that  \(1<n<n_N(\alpha)\) with \(n\in\mathbb R\). For any $0<\tau<T$, it holds that
		\begin{align}
			\int_0^1 (u^2+\rho^{\gamma-1})dy+&\int_0^\tau\int_0^1\Big(\rho^{\alpha-1}\frac{u^2}{r^2}+\rho^{1+\alpha}(\partial_y u)^2 r^{2(N-1)}\Big)dyd\tau\leq C;\label{C u^2}\\
			\int_0^1 u^{2n}dy+\int_0^\tau\int_0^1\Big(\rho^{\alpha-1}\frac{u^{2n}}{r^2}
			&+\rho^{1+\alpha} (\partial_y u)^2 u^{2n-2}r^{2(N-1)}\Big)dyd\tau\nonumber\\
			&\leq C+C\int_0^\tau\int_0^1\rho^{2n(\gamma-\alpha)+\alpha-1}r^{2n-2}dyd\tau,\label{C u^2n}
		\end{align}
		where $C$ is a positive constant independent of $T$.
	\end{prop}
	\begin{proof}
		Multiplying $\eqref{Equ 3}_2$ by \(r^{N-1}u\) and integrating the resulting equation over $(0,1)\times(0,\tau)$, we obtain \eqref{C u^2} by using the fact $\alpha>\frac{N-1}{N}$ and the boundary condition \eqref{C1}.
		
		Multiplying $\eqref{Equ 3}_2$ by \(r^{N-1}u^{2n-1}\) and integrating the resulting equation over $(0,1)$ gives 
		\begin{align*}
			\frac{1}{2n}\frac{d}{d\tau}\int_0^1u^{2n}dy&+\int_0^1\alpha\rho^{1+\alpha}\partial_y(r^{N-1} u)\partial_y(r^{N-1}u^{2n-1})dy\nonumber- \int_0^1(N-1)\rho^{\alpha}\partial_y(r^{N-2}u^{2n})dy\\
			&=\int_0^1\rho^{\gamma}\partial_y(r^{N-1}u^{2n-1})dy.
		\end{align*}
		A direct computation shows
		\begin{align}\label{C2}
			\begin{split}
				&\frac{1}{2n}\frac{d}{d\tau}\int_0^1u^{2n}dy+
				\int_0^1\Big(\big(\alpha(N-1)^2-(N-1)(N-2)\big)\rho^{\alpha-1}\frac{u^{2n}}{r^2}\\
				&\quad+(2n-1)\alpha\rho^{1+\alpha} (\partial_y u)^2 u^{2n-2}r^{2(N-1)}+2n(N-1)(\alpha-1)\rho^\alpha u^{2n-1}\partial_yur^{N-2}\Big)dy\\
				&=\int_0^1 \Big((2n-1)\rho^{\gamma}u^{2n-2}\partial_yu r^{N-1}+(N-1)\rho^{\gamma-1}\frac{u^{2n-1}}{r} \Big)dy.
			\end{split}
		\end{align}
		Young’s inequality yields
		\begin{align}\label{C3}
			\begin{split}
				&\quad 2n(N-1)(\alpha-1)\rho^\alpha u^{2n-1}\partial_yur^{N-2}\\
				&\ge-(1-2\varepsilon)(2n-1)\alpha\rho^{1+\alpha} (\partial_y u)^2
                u^{2n-2}r^{2(N-1)}-\frac{\big(2n(N-1)(\alpha-1)\big)^2}{4(1-2\varepsilon)(2n-1)\alpha}\rho^{\alpha-1}\frac{u^{2n}}{r^2}\\
				&\geq -(1-2\varepsilon)\left((2n-1)\alpha\rho^{1+\alpha} (\partial_y u)^2 u^{2n-2}r^{2(N-1)}+\big(\alpha(N-1)^2-(N-1)(N-2)\big)\rho^{\alpha-1}\frac{u^{2n}}{r^2}\right),
			\end{split}
		\end{align}
		where \(\varepsilon\) is taken sufficiently small so that
		\begin{align}\label{C4}
			\frac{\big(2n(N-1)(\alpha-1)\big)^2}{4(1-2\varepsilon)(2n-1)\alpha}<(1-2\varepsilon)\big(\alpha(N-1)^2-(N-1)(N-2)\big).
		\end{align}
		It is direct to verify that an \(\varepsilon\) satisfying \eqref{C4} exists, since \(n<n_N(\alpha)\). Using \eqref{C3}, we obtain from \eqref{C2} that
		\begin{align}\label{C5}
			\begin{split}
				&\frac{1}{2n}\frac{d}{d\tau}\int_0^1u^{2n}dy+
				2\varepsilon\int_0^1\Big(\big(\alpha(N-1)^2-(N-1)(N-2)\big)\rho^{\alpha-1}\frac{u^{2n}}{r^2}\\
				&\quad+(2n-1)\alpha\rho^{1+\alpha} (\partial_y u)^2 u^{2n-2}r^{2(N-1)}\Big)dy\\
				&\leq C\int_0^1 \Big((2n-1)\rho^{\gamma}u^{2n-2}\partial_yu r^{N-1}+(N-1)\rho^{\gamma-1}\frac{u^{2n-1}}{r} \Big)dy.
			\end{split}
		\end{align}
		Applying Young’s inequality once more, we get
		\begin{align}\label{C6}
			\begin{split}
				&\frac{1}{2n}\frac{d}{d\tau}\int_0^1u^{2n}dy+
				\varepsilon\int_0^1\Big(\big(\alpha(N-1)^2-(N-1)(N-2)\big)\rho^{\alpha-1}\frac{u^{2n}}{r^2}\\
				&\quad+(2n-1)\alpha\rho^{1+\alpha} (\partial_y u)^2 u^{2n-2}r^{2(N-1)}\Big)dy\\
				&\leq C\int_0^1\rho^{2n(\gamma-\alpha)+\alpha-1}r^{2n-2}dy.
			\end{split}
		\end{align}
		Integrating the above expression with respect to $\tau$ gives \eqref{C u^2n}.
	\end{proof}
	
	Next, we provide estimates for the derivatives of the density field.
	\begin{prop}\label{Prop C m}
		Assume that \(1<m\in M_{\text{set}}\). For any $0<\tau<T$, it holds that
		\begin{align}
			\int_0^1(u+r^{N-1}\partial_y\rho^\alpha)^2dy&+\int_0^\tau\int_0^1\rho^{\gamma-\alpha}(r^{N-1}\partial_y\rho^\alpha)^{2}dyd\tau\leq C;\label{C nablarho^2}\\
			\int_0^1(u+r^{N-1}\partial_y\rho^\alpha)^{2m}dy&+\int_0^\tau\int_0^1\rho^{\gamma-\alpha}(r^{N-1}\partial_y\rho^\alpha)^{2m}dyd\tau\nonumber\\
			&\leq C+C\int_0^\tau\int_0^1\rho^{\gamma-\alpha}u^{2m}dyd\tau,\label{C nablarho^2m}
		\end{align}
		where $C$ is a positive constant independent of $T$.
	\end{prop}
	\begin{proof}
		From equation \eqref{Equ 3} we derive the following B–D entropy equation:
		\begin{align}\label{2d B-D}
			\partial_\tau(u+r^{N-1}\partial_y\rho^\alpha)+\partial_y\rho^\gamma r^{N-1}=0.
		\end{align}
    	Multiplying \eqref{2d B-D} by \(u+r^{N-1}\partial_y\rho^{\alpha}\) and integrating the resulting equation over \((0,1)\times(0,\tau)\), we use equation $\eqref{Equ 3}_1$ and boundary condition $\eqref{C1}_2$ to obtain
		\begin{align*}
			\int_0^1\frac{1}{2}(u+r^{N-1}\partial_y\rho^\alpha)^2+\frac{\rho^{\gamma-1}}{\gamma-1}dy+\alpha\gamma\int_0^\tau\int_0^1\rho^{\gamma+\alpha-2}(\partial_y\rho)^2r^{2(N-1)}dyd\tau\\
			=\int_0^1\frac{1}{2}(u_0+r^{N-1}\partial_y\rho_0^\alpha)^2+\frac{\rho_0^{\gamma-1}}{\gamma-1}dy,
		\end{align*}
        which shows that \eqref{C nablarho^2} holds.
		Multiplying \eqref{2d B-D} by \(\bigl(u+r^{N-1}\partial_y\rho^{\alpha}\bigr)^{2m-1}\) and integrating the resulting equation over \(y\in(0,1)\) gives
		\begin{align*}
			\frac{1}{2m}\frac{d}{d\tau}\int_0^1 (u+r^{N-1}\partial_y\rho^\alpha)^{2m}dy+\int_0^1\gamma\rho^{\gamma-1}\partial_y\rho(u+r^{N-1}\partial_y\rho^\alpha)^{2m-1}r^{N-1}dy=0.
		\end{align*}
		Using Lemma \ref{Lem ks}, we obtain 
		\begin{align*}
			\gamma\rho^{\gamma-1}\partial_y\rho(u+r^{N-1}\partial_y\rho^\alpha)^{2m-1}r^{N-1}&= r^{N-1}\partial_y\rho^\alpha(u+r^{N-1}\partial_y\rho^\alpha)^{2m-1}\frac{\gamma}{\alpha}\rho^{\gamma-\alpha}\\
			&\geq\varepsilon \rho^{\gamma-\alpha}(r^{N-1}\partial_y\rho^\alpha)^{2m}-C\rho^{\gamma-\alpha}u^{2m},
		\end{align*}
		where $\varepsilon$ is a small generic constant depending only on $m, \gamma, \alpha$. Therefore, we have
		\begin{align*}
			\frac{1}{2m}\frac{d}{d\tau}\int_0^1 (u+r^{N-1}\partial_y\rho^\alpha)^{2m}dy+\varepsilon\int_0^1\rho^{\gamma-\alpha}(r^{N-1}\partial_y\rho^\alpha)^{2m}dy\leq C\int_0^1 \rho^{\gamma-\alpha}u^{2m}dy.
		\end{align*}
		Integrating the inequatity with respect to $\tau$, we obtain \eqref{C nablarho^2m}.
	\end{proof}
	
	Finally, we establish a key $r$-weighted \(L^\infty\) estimate for the density, which will be used to obtain both upper and positive lower bounds on \(\rho\).
	\begin{prop}\label{Prop C 2d r-weight}
		Let \(N=2\). For every \(1\le s<\infty\), there exists a constant \(C(s)>0\) independent of \(T\) such that
		\begin{align}\label{C7}
			\sup_{0\leq t\leq T}\norm{\rho}_{L^s(\Omega)}\leq C(s).
		\end{align}
		Moreover, for every \(0<\xi\ll1\), there exists a constant \(C(\xi)>0\) independent of \(T\) such that
		\begin{align}\label{C8}
			\sup_{0\leq t\leq T}\norm{\rho^{\alpha-\frac{1}{2}}r^\xi}_{L^\infty(0,R)}\leq C(\xi).
		\end{align}
	\end{prop}
	\begin{proof}
		From \eqref{C u^2} and \eqref{C nablarho^2} we know that \(\sup_{0\le t\le T}\|\rho\|_{L^\gamma(\Omega)}\) and \(\sup_{0\le t\le T}\|\nabla \rho^{\alpha-\frac{1}{2}}\|_{L^2(\Omega)}\) are independent of \(T\).  Since \(\alpha-\frac{1}{2}<\gamma\), we obtain a $T$-independent bound on \(\sup_{0\le t\le T}\|\rho^{\alpha-\frac{1}{2}}\|_{L^1(\Omega)}\).  The standard Sobolev embedding then yields a $T$-independent bound on \(\sup_{0\le t\le T}\|\rho^{\alpha-\frac{1}{2}}\|_{L^s(\Omega)}\), which in turn gives a $T$-independent bound on \(\sup_{0\le t\le T}\|\rho\|_{L^s(\Omega)}\) for any finite $1\leq s<\infty$. 
		
		Using the one-dimensional Sobolev embedding and \eqref{C7}, we obtain
		\begin{align*}
			&\norm{\rho^{\alpha-\frac{1}{2}} r^{\xi}}_{L^\infty(0,R)}\leq C\int_0^R \rho^{\alpha-\frac{1}{2}} r^{\xi} dr+C\int_0^R |\partial_r \rho^{\alpha-\frac{1}{2}}| r^{\xi} dr+C\int_0^R \rho^{\alpha-\frac{1}{2}} r^{\xi-1}dr\\
			&\leq C\int_\Omega \rho^{\alpha-\frac{1}{2}} r^{\xi-1}dx+C\int_\Omega |\nabla \rho^{\alpha-\frac{1}{2}}|r^{\xi-1}dx+C\int_\Omega \rho^{\alpha-\frac{1}{2}} r^{\xi-2}dx\\
			&\leq C\left(\int_\Omega|\nabla\rho^{\alpha-\frac{1}{2}}|^2dx\right)^{\frac{1}{2}}\left(\int_\Omega r^{2\xi-2}dx\right)^{\frac{1}{2}}+C\left(\int_\Omega\rho^{(\alpha-\frac{1}{2})\frac{4}{\xi}} dx\right)^{\frac{\xi}{4}}\left(\int_\Omega r^{\frac{4\xi-8}{4-\xi}}dx\right)^{\frac{4-\xi}{4}}\\
			&\leq C(\xi).
		\end{align*}
		This completes the proof of Proposition \ref{Prop C 2d r-weight}.
	\end{proof}
	
	\begin{prop}\label{Prop C 3d r-weight}
		Let \(N=3\). Then there exists a constant \(C>0\) independent of \(T\) such that
		\begin{align}\label{C9}
			\sup_{0\leq t\leq T}\norm{\rho}_{L^{6\alpha-3}(\Omega)}\leq C.
		\end{align}
		Moreover, for every \(0<\xi\ll1\), there exists a constant \(C(\xi)>0\) independent of \(T\) such that
		\begin{align}\label{C10}
			\sup_{0\leq t\leq T}\norm{\rho^{\alpha-\frac{1}{2}}r^{\frac{1}{2}+\xi}}_{L^\infty(0,R)}\leq C(\xi).
		\end{align}
	\end{prop}
	\begin{proof}
		From \eqref{C u^2} and \eqref{C nablarho^2} we know that \(\sup_{0\le t\le T}\|\rho\|_{L^\gamma(\Omega)}\) and \(\sup_{0\le t\le T}\|\nabla \rho^{\alpha-\frac{1}{2}}\|_{L^2(\Omega)}\) are independent of $T$. Since \(\alpha-\frac{1}{2}<\gamma\), we obtain a $T$-independent bound on \(\sup_{0\le t\le T}\|\rho^{\alpha-\frac{1}{2}}\|_{L^1(\Omega)}\). Therefore, the standard Sobolev embedding yields \eqref{C9}.
		
		Using the one-dimensional Sobolev embedding and \eqref{C9}, one has
		\begin{align*}
			&\norm{\rho^{\alpha-\frac{1}{2}} r^{\frac{1}{2}+\xi}}_{L^\infty(0,R)}\leq C\int_0^R \rho^{\alpha-\frac{1}{2}} r^{\frac{1}{2}+\xi} dr+C\int_0^R |\partial_r \rho^{\alpha-\frac{1}{2}}| r^{\frac{1}{2}+\xi} dr+C\int_0^R \rho^{\alpha-\frac{1}{2}} r^{\xi-\frac{1}{2}}dr\\
			&\leq C\int_\Omega \rho^{\alpha-\frac{1}{2}} r^{\xi-\frac{3}{2}}dx+C\int_\Omega |\nabla \rho^{\alpha-\frac{1}{2}}|r^{\xi-\frac{3}{2}}dx+C\int_\Omega \rho^{\alpha-\frac{1}{2}} r^{\xi-\frac{5}{2}}dx\\
			&\leq C\left(\int_\Omega|\nabla\rho^{\alpha-\frac{1}{2}}|^2dx\right)^{\frac{1}{2}}\left(\int_\Omega r^{2\xi-3}dx\right)^{\frac{1}{2}}+C\left(\int_\Omega\rho^{6\alpha-3} dx\right)^{\frac{1}{6}}\left(\int_\Omega r^{\frac{6}{5}\xi-3}dx\right)^{\frac{5}{6}}\\
			&\leq C(\xi).
		\end{align*}
		We have thus completed the proof of Proposition \ref{Prop C 3d r-weight}.
	\end{proof}
	
	\subsection{Upper bound and positive lower bound of the density $(N=2)$}
	We first state a technical lemma.
	\begin{lema}\label{Lem C 2k}
		Let \(N=2\).  For any \((\alpha,\gamma)\) satisfying \eqref{C 2d alpha}, there exists an $\alpha$-dependent constant \(k\in M_{\text{set}}\) with $k>1$ such that  
		\begin{align}\label{C11}
			1-\frac{k\sqrt{2k-1}-2k+1}{k^2-2k+1}<\alpha<1-\frac{1}{2k}.
		\end{align}  
	\end{lema}
	\begin{rmk}
		The first inequality in \eqref{C11} is equivalent to $k<n_2(\alpha)$.
	\end{rmk}
	\begin{proof}[Proof of the Lemma \ref{Lem C 2k}]
		We introduce the auxiliary function
		\begin{align*}
			h(a)=\frac{a\sqrt{2a-1}-2a+1}{(a-1)^2}-\frac{1}{2a}, \quad a\in(1,\infty).
		\end{align*}
		It is direct to verify that \(h\in C^\infty(1,\infty)\) and \(\lim_{a\to 1^+}h(a)=0\).  
		We next claim  
		\begin{align}\label{C15}  
			h(a)>0,\quad a\in(1,\infty),  
		\end{align}  
		an assertion equivalent to
		\begin{align}\label{C12}
			\tilde{h}(a)=2 a^2\sqrt{2a-1}-5a^2+4a-1>0,\quad a\in(1,\infty).
		\end{align}
		Note that
		\begin{align*}
			\tilde{h}'(a)=(10a-4)\left(\frac{a}{\sqrt{2a-1}}-1\right)>0, \quad a\in(1,\infty),
		\end{align*}
		which, combined with \(\tilde{h}(1)=0\), implies \eqref{C12}.
		
		For any \(\alpha\) satisfying \eqref{C 2d alpha}, set \(k_0=\tfrac{1}{2(1-\alpha)}\). Using \eqref{C15}, we get 
		\begin{align*}
			1-\frac{k_0\sqrt{2k_0-1}-2k_0+1}{k_0^2-2k_0+1}<1-\frac{1}{2k_0}=\alpha.
		\end{align*}
		By the continuity of the function and the density of \(M_{\mathrm{set}}\), we choose \(k\in M_{\mathrm{set}}\) with \(k>k_0\) and \(k-k_0\) sufficiently small so that \eqref{C11} holds. This completes the proof of Lemma \ref{Lem C 2k}.
	\end{proof}
	
	We are now in a position to establish upper and lower bounds of the density. To this end, we state the key Proposition \ref{Prop C 2d k}, which supplies an estimate for the velocity field and the derivative of the density.
	\begin{prop}\label{Prop C 2d k}
		Assume that \eqref{C 2d alpha} holds. Let \(k\) be as defined in Lemma \ref{Lem C 2k}. Then there exists a constant \(C(T)>0\) such that
		\begin{align}\label{C 2d k}
			\sup_{0\leq t\leq T}\int_0^R\rho u^{2k}rdr+\int_0^T\int_0^R\Big(\rho^{\alpha}\frac{u^{2k}}{r}
			+\rho^{\alpha} (\partial_ru)^2 u^{2k-2}r\Big)drdt\leq C(T)
		\end{align}
		and
		\begin{align}\label{C 2d l}
			\sup_{0\leq t\leq T}\int_0^R |\partial_r\rho^{\alpha-1+\frac{1}{2k}}|^{2k}rdr+\int_0^T\int_0^R |\partial_r\rho^{\alpha-1+\frac{\gamma-\alpha+1}{2k}}|^{2k}rdrdt\leq C(T).
		\end{align}
	\end{prop}
	
	\begin{proof}
		We first establish \eqref{C 2d k}. 
		Using \eqref{C8}, we obtain from \eqref{C u^2n} that
		\begin{align*}
			\begin{split}
				&\quad\sup_{0\leq t\leq T}\int_0^R\rho u^{2k}rdr+\int_0^T\int_0^R\Big(\rho^{\alpha}\frac{u^{2k}}{r}
				+\rho^{\alpha} (\partial_r u)^2 u^{2k-2}r\Big)drdt\\
				&\leq C+C\int_0^T\int_0^R\rho^{2k(\gamma-\alpha)+\alpha}r^{2k-1}drdt\\
				&\leq C+C\sup_{0\le t\le T}\norm{\rho^{\alpha-\frac{1}{2}}r^\xi}_{L^\infty(0,R)}^{\frac{2k(\gamma-\alpha)+\alpha}{\alpha-1/2}}\int_0^T\int_0^R r^{2k-1-\frac{2k(\gamma-\alpha)+\alpha}{\alpha-1/2}\xi}drdt\\
				&\leq C(T),
			\end{split}
		\end{align*}
		where \(\xi>0\) is chosen sufficiently small so that
		\begin{align*}
			2k-1-\frac{2k(\gamma-\alpha)+\alpha}{\alpha-1/2}\xi>-1.
		\end{align*}
		
		We next prove \eqref{C 2d l}. Invoking \eqref{C8} and \eqref{C 2d k}, one derives from \eqref{C nablarho^2m} that
		\begin{align*}
			\begin{split}
				&\quad\sup_{0\leq t\leq T}\int_0^R |\partial_r\rho^{\alpha-1+\frac{1}{2k}}|^{2k}rdr+\int_0^T\int_0^R |\partial_r\rho^{\alpha-1+\frac{\gamma-\alpha+1}{2k}}|^{2k}rdrdt\\
				&\leq C\sup_{0\leq t\leq T}\int_0^R\rho u^{2k}rdr+ C\sup_{0\leq t\leq T}\int_0^R\rho(u+\rho^{-1}\partial_r\rho^\alpha)^{2k}rdr+\int_0^T\int_0^R|\partial_r\rho^{\alpha-1+\frac{\gamma-\alpha+1}{2k}}|^{2k}rdrdt\nonumber\\
				&\leq C(T)+C\int_0^T\int_0^R\rho^{\gamma-\alpha+1}u^{2k}rdrdt\\
				&\leq C(T)+C\int_0^T\int_0^R\rho^\alpha\frac{u^{2k}}{r}\rho^{\gamma-2\alpha+1}r^2drdt\\
				&\leq C(T)+C\sup_{0\le t\le T}\norm{\rho^{\alpha-\frac{1}{2}}r^\xi}_{L^\infty(0,R)}^{\frac{\gamma-2\alpha+1}{\alpha-1/2}}\int_0^T\int_0^R\rho^\alpha \frac{u^{2k}}{r}drdt\\
				&\leq C(T),
			\end{split}
		\end{align*}
		where we have used 
		\begin{align*} 
			\gamma-2\alpha+1>0,
		\end{align*} 
		and \(\xi\) is taken sufficiently small so that
		\begin{align*}
			\left(\frac{\gamma-2\alpha+1}{\alpha-1/2}\right)\xi<2.
		\end{align*}
		This completes the proof of Proposition \ref{Prop C 2d k}.
	\end{proof}
	
	Finally, we establish the upper and positive lower bounds for the density in two dimensions. Denote that
	\begin{align*}
		\begin{split}
			R_T:=\sup_{0\le t\le T}\norm{\rho(t)}_{L^\infty(\Omega)}+1,\quad V_T:=\sup_{0\le t\le T}\norm{\rho^{-1}(t)}_{L^\infty(\Omega)}+1.
		\end{split}
	\end{align*}
	\begin{prop}\label{Prop C 2d RT}
		Assume that \eqref{C 2d alpha} holds. There exists a constant \(C(T)>0\) such that
		\begin{align}\label{C 2d RT}
			R_T\leq C(T).
		\end{align}
	\end{prop}
	\begin{proof}
		Using \eqref{C7} and \eqref{C 2d l}, we obtain from the one-dimensional Sobolev embedding that
		\begin{align*}
			&\norm{\rho^\alpha}_{L^\infty(0,R)}\leq C\int_0^R \rho^\alpha dr+C\int_0^R |\partial_r\rho^\alpha|dr\\
			&\leq C\int_0^R\rho^{\alpha}r^{\frac{1}{3}}r^{-\frac{1}{3}}dr+C\int_0^R|\partial_r\rho^{\alpha-1+\frac{1}{2k}}|\rho^{1-\frac{1}{2k}}dr\\
			&\leq C\left(\int_0^R\rho^{3\alpha}rdr\right)^{\frac{1}{3}}\left(\int_0^Rr^{-\frac{1}{2}}dr\right)^{\frac{2}{3}}+C\left(\int_0^R|\partial_r\rho^{\alpha-1+\frac{1}{2k}}|^{2k}rdr\right)^{\frac{1}{2k}}\left(\int_0^R \rho r^{-\frac{1}{2k-1}} dr\right)^{\frac{2k-1}{2k}}\\
			&\leq C(T)+C(T)\int_0^R\rho r^{-\frac{1}{2k-1}}dr \\
			&\leq C(T)+C(T)\left(\int_0^R \rho^{\frac{2k}{k-1}}rdr\right)^{\frac{k-1}{2k}}\left(\int_0^Rr^{-\frac{2k^2-k+1}{2k^2+k-1}}dr\right)^{\frac{k+1}{2k}}\\
			&\leq C(T).
		\end{align*}
		This completes the proof of Proposition \ref{Prop C 2d RT}.
	\end{proof}
	
	\begin{prop}\label{Prop C 2d VT}
		Assume that \eqref{C 2d alpha} holds. There exists a constant \(C(T)>0\) such that
		\begin{align}\label{C 2d VT}
			V_T\leq C(T).
		\end{align}
	\end{prop}
	\begin{proof}
		Let \(v(y,\tau)=\frac{1}{\rho(y,\tau)}\). By $\eqref{Equ 3}_1$, we know that \(v\) satisfies the equation \(v_\tau=(r u)_y\). Integrating this equation over \((0,1)\times(0,\tau)\) and using the boundary condition \eqref{C1}, we obtain
		\begin{align}\label{C16}
			\int_0^1 v(y,\tau)dy=\int_0^1 v_0(y)dy\leq C.
		\end{align}
		Choose \(\xi\) sufficiently small so that \(0<\xi<\frac{k-1}{2k}\).  Using \eqref{C 2d l}, \eqref{C16} and the one-dimensional Sobolev embedding, we obtain
		\begin{align}\label{C17}
			\begin{split}
				&v(y,\tau)\leq \int_0^1 v dy+\int_0^1 |\partial_y v|dy\leq C+C\int_0^1 |\partial_y\rho^\alpha|v^{\alpha+1}dy\\
				&\leq C+C\left(\int_0^1|\partial_y\rho^\alpha|^{2k}r^{2k}dy\right)^{\frac{1}{2k}}\left(\int_0^1vr^{-\frac{2}{2\xi+1}}dy\right)^{\xi+\frac{1}{2}}\left(\int_0^1v^{(\alpha+\frac{1}{2}-\xi)\frac{2k}{k-1-2k\xi}}dy\right)^{\frac{k-1-2k\xi}{2k}}\\
				&\leq C+C(T)\left(\int_0^1v^{(\alpha+\frac{1}{2}-\xi)\frac{2k}{k-1-2k\xi}-1+1}dy\right)^{\frac{k-1-2k\xi}{2k}}\\
				&\leq C(T)V_T^{\alpha+\frac{1}{2k}}.
			\end{split}
		\end{align}
		Taking the supremum of \eqref{C17} over \([0,1]\times[0,T]\) and applying Young's inequality together with \eqref{C11}, we obtain \eqref{C 2d VT}.
	\end{proof}
	
	\subsection{Upper bound and positive lower bound of the density $(N=3)$}
	We begin by stating a technical lemma.
	\begin{lema}\label{Lem C 3k}
		Let \(N=3\).  For any \((\alpha,\gamma)\) satisfying \eqref{C 3d alpha}, there exist an $\alpha,\gamma$-dependent constant \(k\in M_{\text{set}}\) with $k>1.91$ such that  
		\begin{align}\label{C21}
			1 - \frac{\sqrt{4k(4k^2 - k - 1) + 1} - 6k + 3}{4k^2 - 8k + 4}<\alpha<1-\frac{1}{2k}
		\end{align}  
		and
		\begin{align}\label{C22}
			1<\gamma<6\alpha-3+\frac{3-5\alpha}{2k}.
		\end{align}
	\end{lema}
		\begin{rmk}
		The first inequality in \eqref{C21} is equivalent to $k<n_3(\alpha)$.
	\end{rmk}
	\begin{proof}[Proof of the Lemma \ref{Lem C 3k}]
		One readily checks 
		\begin{align}\label{C23}
			1 - \frac{\sqrt{4n(4n^2 - n - 1) + 1} - 6n + 3}{4n^2 - 8n + 4}<1-\frac{1}{2n}\Leftarrow n>1.55.
		\end{align}
		Since \(n_3(\alpha)\) is strictly increasing on \((\frac23,1)\), \(\alpha>0.686\) implies \(n_3(\alpha)>1.91\). Let   \(k_0=n_3(\alpha)\). Then \eqref{C23} gives
		\begin{align*}
			1 - \frac{\sqrt{4k_0(4k_0^2 - k_0 - 1) + 1} - 6k_0 + 3}{4k_0^2 - 8k_0 + 4}=\alpha<1-\frac{1}{2k_0}
		\end{align*}
		and \eqref{C 3d alpha} gives
		\begin{align*}
			1<\gamma<6\alpha-3+\frac{3-5\alpha}{2k_0}.
		\end{align*}
		By the continuity of the functions and the density of \(M_{\mathrm{set}}\) in \((1,\infty)\), we can choose \(k\in M_{\mathrm{set}}\) with \(k<k_0\) and \(k_0-k\) sufficiently small so that \eqref{C21} and \eqref{C22} hold.
	\end{proof}
	
	\begin{prop}\label{Prop C 3d k}
		Assume that \eqref{C 3d alpha} holds. Let \(k\) be as defined in Lemma \ref{Lem C 3k}. Then there exists a constant $0\leq\sigma<(3\alpha-2)(1-\frac{1}{k})$ such that
		\begin{align}\label{C24}
			\alpha-\frac{\alpha}{2k}+\sigma<\gamma<3\alpha-1+\frac{\alpha-1}{2k}+\sigma.
		\end{align}
		Furthermore, it holds that
		\begin{align}\label{C 3d k}
			\sup_{0\leq t\leq T}\int_0^R\rho u^{2k}r^2dr+\int_0^T\int_0^R\Big(\rho^{\alpha}u^{2k}
			+\rho^{\alpha} (\partial_ru)^2 u^{2k-2}r^2\Big)drdt\leq C(T)R_T^{2k\sigma}.
		\end{align}
		and
		\begin{align}\label{C 3d l}
			\sup_{0\leq t\leq T}\int_0^R |\partial_r\rho^{\alpha-1+\frac{1}{2k}}|^{2k}r^2dr+\int_0^T\int_0^R |\partial_r\rho^{\alpha-1+\frac{\gamma-\alpha+1}{2k}}|^{2k}r^2drdt\leq C(T)R_T^{2k\sigma}.
		\end{align}
	\end{prop}
	\begin{proof}
		It is readily verified that a constant \(\sigma\) satisfying the required condition exists, thanks to \eqref{C22}. 
		
		Next, we prove \eqref{C 3d k}. We use Proposition \ref{Prop C n}, \eqref{C10} and \eqref{C24} to obtain
		\begin{align*}
			\begin{split}
				&\quad\sup_{0\leq t\leq T}\int_0^R\rho u^{2k}r^2dr+\int_0^T\int_0^R\Big(\rho^{\alpha}u^{2k}
				+\rho^{\alpha} (\partial_ru)^2 u^{2k-2}r^2\Big)drdt\\
				&\leq C+C\int_0^T\int_0^R\rho^{2k(\gamma-\alpha)+\alpha}r^{2k}drdt\\
				&\leq C+C\sup_{0\le t\le T}\norm{\rho^{\alpha-\frac{1}{2}}r^{\frac{1}{2}+\xi}}_{L^\infty(0,R)}^{\frac{2k(\gamma-\alpha)+\alpha-2k\sigma}{\alpha-1/2}}R_T^{2k\sigma}\int_0^T\int_0^R r^{2k-\frac{2k(\gamma-\alpha)+\alpha-2k\sigma}{\alpha-1/2}\left(\frac{1}{2}+\xi\right)}drdt\\
				&\leq C(T)R_T^{2k\sigma},
			\end{split}
		\end{align*}
		where \(\xi>0\) is a sufficiently small constant satisfying
		\begin{align*}
			2k-\frac{2k(\gamma-\alpha)+\alpha-2k\sigma}{\alpha-1/2}\left(\frac{1}{2}+\xi\right)>-1.
		\end{align*}
		Such a $\xi$ exists because \eqref{C24}. We thus obtain \eqref{C 3d k}.
		
		Next we  \eqref{C 3d l}. Invoking \eqref{C10} and \eqref{C 3d k}, one derives from \eqref{C nablarho^2m} that
		\begin{align*}
			\begin{split}
				&\quad\sup_{0\leq t\leq T}\int_0^R |\partial_r\rho^{\alpha-1+\frac{1}{2k}}|^{2k}r^2dr+\int_0^T\int_0^R |\partial_r\rho^{\alpha-1+\frac{\gamma-\alpha+1}{2k}}|^{2k}r^2drdt\\
				&\leq C\sup_{0\leq t\leq T}\int_0^R\rho u^{2k}r^2dr+ C\sup_{0\leq t\leq T}\int_0^R\rho(u+\rho^{-1}\partial_r\rho^\alpha)^{2k}r^2dr+\int_0^T\int_0^R|\partial_r\rho^{\alpha-1+\frac{\gamma-\alpha+1}{2k}}|^{2k}r^2drdt\nonumber\\
				&\leq C(T)R_T^{2k\sigma}+C\int_0^T\int_0^R\rho^{\gamma-\alpha+1}u^{2k}r^2drdt\\
				&\leq C(T)R_T^{2k\sigma}+C\int_0^T\int_0^R\rho^\alpha u^{2k}\rho^{\gamma-2\alpha+1}r^2drdt\\
				&\leq C(T)R_T^{2k\sigma}+C\sup_{0\le t\le T}\norm{\rho^{\alpha-\frac{1}{2}}r^{\frac{1}{2}+\xi}}_{L^\infty(0,R)}^{\frac{\gamma-2\alpha+1}{\alpha-1/2}}\int_0^T\int_0^R\rho^\alpha u^{2k}drdt\\
				&\leq C(T)R_T^{2k\sigma},
			\end{split}
		\end{align*}
		where we have used 
		\begin{align*} 
			\gamma-2\alpha+1>0,
		\end{align*} 
		and \(\xi\) is taken sufficiently small so that
		\begin{align*}
			\left(\frac{\gamma-2\alpha+1}{\alpha-1/2}\right)\left(\frac{1}{2}+\xi\right)<2.
		\end{align*}
		Such a $\xi$ exists because $\gamma<6\alpha-3$. This completes the proof of Proposition \ref{Prop C 3d k}.
	\end{proof}
	
	We now establish the upper and positive lower bounds for the density in three dimensions.
	\begin{prop}\label{Prop C 3d RT}
		Assume that \eqref{C 3d alpha} holds. There exists a constant \(C(T)>0\) such that
		\begin{align}\label{C 3d RT}
			R_T\leq C(T).
		\end{align}
	\end{prop}
	\begin{proof}
		Choose \(\beta\) such that
		\[
		\sigma<\beta<(3\alpha-2)\Bigl(1-\frac1k\Bigr),
		\]
		where \(\sigma\) is given in Proposition \ref{Prop C 3d k}.
		In view of \eqref{C21}, one readily verifies that \(\beta\) satisfies 
		\begin{align}\label{C25}
			\max\Bigl\{0,\alpha-1+\frac1{2k}\Bigr\}
			<\beta<
			\min\Bigl\{2\alpha-1,(3\alpha-2)\Bigl(1-\frac1k\Bigr)\Bigr\}.
		\end{align}
		Therefore, combining the one-dimensional Sobolev embedding with \eqref{C9}, \eqref{C 3d l} and \eqref{C25}, we obtain
		\begin{align*}
			\begin{split}
				&\norm{\rho^\beta}_{L^\infty(0,R)}\leq C\int_0^R\rho^\beta dr+C\int_0^R |\partial_r\rho^\beta| dr\\
				&\leq C\int_0^R (\rho^{6\alpha-3}r^2)^{\frac{\beta}{6\alpha-3}}r^{-\frac{2\beta}{6\alpha-3}}dr+C\int_0^R|\partial_r\rho^{\alpha-1+\frac{1}{2k}}|\rho^{\beta-\alpha+1-\frac{1}{2k}}dr\\
				&\leq C\left(\int_0^R\rho^{6\alpha-3}r^2dr\right)^{\frac{\beta}{6\alpha-3}}\left(\int_0^Rr^{-\frac{2\beta}{6\alpha-\beta-3}}dr\right)^{\frac{6\alpha-\beta-3}{6\alpha-3}}\\
				&\quad+C\left(\int_0^R|\partial_r\rho^{\alpha-1+\frac{1}{2k}}|^{2k}r^2dr\right)^{\frac{1}{2k}}\left(\int_0^R\rho^{\frac{2k}{2k-1}(\beta-\alpha)+1}r^{-\frac{2}{2k-1}}dr\right)^{\frac{2k-1}{2k}}\\
				&\leq C+C(T)R_T^{\sigma}\sup_{0\le t\le T}\norm{\rho^{\alpha-\frac{1}{2}}r^{\frac{1}{2}+\xi}}_{L^\infty}^{\frac{\beta-\alpha+\frac{2k-1}{2k}}{\alpha-1/2}}\left(\int_0^R r^{-\frac{2}{2k-1}-\left(\frac{1}{2}+\xi\right)\frac{\frac{2k}{2k-1}(\beta-\alpha)+1}{\alpha-1/2}}dr\right)^{\frac{2k-1}{2k}}\\
				&\leq C(T)R_T^\sigma,
			\end{split}
		\end{align*}
		where \(\xi\) is taken sufficiently small so that
		\begin{align*}
			-\frac{2}{2k-1}-\left(\frac{1}{2}+\xi\right)\frac{\frac{2k}{2k-1}(\beta-\alpha)+1}{\alpha-1/2}>-1.
		\end{align*}
		Such a \(\xi\) exists by virtue of \eqref{C25}. Applying Young’s inequality, we obtain 
		\begin{align*}
			R_T^{\beta}\leq \frac{1}{2}R_T^{\beta}+C(T),
		\end{align*}
		which establishes \eqref{C 3d RT}.
	\end{proof}
	
	\begin{prop}\label{Prop C 3d VT}
		Assume that \eqref{C 3d alpha} holds. There exists a constant \(C(T)>0\) such that
		\begin{align}\label{C 3d VT}
			V_T\leq C(T).
		\end{align}
	\end{prop}
	\begin{proof}
		Let \(v(y,\tau)=\frac{1}{\rho(y,\tau)}\). By $\eqref{Equ 3}_1$, we know that \(v\) satisfies the equation \(v_\tau=(r^2 u)_y\). Integrating this equation over \((0,1)\times(0,\tau)\) and using the boundary condition \eqref{C1}, we obtain
		\begin{align}\label{C26}
			\int_0^1 v(y,\tau)dy=\int_0^1 v_0(y)dy\leq C.
		\end{align}
		Choose \(\xi\) sufficiently small so that \(0<\xi<\frac{2k-3}{4k}\).  Using \eqref{C 3d l}, \eqref{C26} and the one-dimensional Sobolev embedding, we obtain
		\begin{align}\label{C27}
			\begin{split}
				&v(y,\tau)\leq \int_0^1 v dy+\int_0^1 |\partial_y v|dy\leq C+C\int_0^1 |\partial_y\rho^\alpha|v^{\alpha+1}dy\\
				&\leq C+C\left(\int_0^1|\partial_y\rho^\alpha|^{2k}r^{4k}dy\right)^{\frac{1}{2k}}\left(\int_0^1vr^{-\frac{3}{\xi+1}}dy\right)^{\frac{2+2\xi}{3}}\left(\int_0^1v^{(\alpha+\frac{1}{3}-\frac{2\xi}{3})\frac{6k}{2k-3-4k\xi}}dy\right)^{\frac{2k-3-4k\xi}{6k}}\\
				&\leq C+C(T)\left(\int_0^1v^{(\alpha+\frac{1}{3}-\frac{2\xi}{3})\frac{6k}{2k-3-4k\xi}-1+1}dy\right)^{\frac{2k-3-4k\xi}{6k}}\\
				&\leq C(T)V_T^{\alpha+\frac{1}{2k}}.
			\end{split}
		\end{align}
		Taking the supremum of \eqref{C27} over \([0,1]\times(0,T)\) and applying Young's inequality together with \eqref{C21}, we obtain \eqref{C 3d VT}.
	\end{proof}
	
	\subsection{Uniform upper bound of the density $(N=2,3)$}
	We first show that the \(L^4\)-integrability of the velocity field can be established independently of time.
	
	\begin{prop}\label{Prop 3.11}
		Assume that \eqref{C 2d unialpha} or \eqref{C 3d unialpha} holds. Then there exists a constant \(C>0\) independent of \(T\) such that
		\begin{align}\label{uu1}
			\int_0^1 u^4dy+\int_0^\tau\int_0^1\Big( \rho^{\alpha-1}\frac{ u^4}{r^2}+\rho^{1+\alpha}(\partial_yu)^2u^{2}r^{2(N-1)}\Big)dyd\tau\leq C.
		\end{align}
	\end{prop}
	\begin{proof}
		Multiplying $\eqref{Equ 3}_2$ by \(r^{N-1}u^{3}\) and integrating the resulting equation over $(0,1)$, we obtain from \eqref{C2} that
		\begin{align}\label{uu2}
			\begin{split}
				&\frac{1}{4}\frac{d}{d\tau}\int_0^1u^{4}dy+
				\int_0^1\Big((\alpha(N-1)^2-(N-1)(N-2))\rho^{\alpha-1}\frac{u^{4}}{r^2}\\
				&\quad+3\alpha\rho^{1+\alpha} {(\partial_y u)}^2 u^{2}r^{2(N-1)}+4(N-1)(\alpha-1)\rho^\alpha u^{3}\partial_y ur^{N-2}\Big)dy\\
				&=\int_0^1 \Big(3\rho^{\gamma}u^{2}\partial_y u r^{N-1}+(N-1)\rho^{\gamma-1}\frac{u^{3}}{r} \Big)dy.
			\end{split}
		\end{align}
		Observing that \(2<n_N(\alpha)\), we apply Young's inequality to obtain a universal constant \(\varepsilon>0\) such that
		\begin{align*}
			\begin{split}
				&\frac{d}{d\tau}\int_0^1 u^4dy+2\varepsilon\int_0^1 \Big(\rho^{\alpha-1}\frac{u^4}{r^2}+\rho^{1+\alpha}(\partial_yu)^2u^2r^{2(N-1)}\Big)dy\\
				&\leq  C\int_0^1\rho^\gamma u^2|\partial_yu| r^{N-1}dy+C\int_0^1\rho^{\gamma-1}\frac{|u|^3}{r}dy\\
				&\leq C\int_0^1\left(\rho^{1+\alpha}(\partial_yu)^2u^{2}r^{2(N-1)}\right)^{\frac{1}{2}}\left(\rho^{2\gamma-\alpha-1}u^2\right)^{\frac{1}{2}}dy+C\int_0^1 \left(\rho^{\alpha-1}\frac{u^{4}}{r^2}\right)^{\frac{1}{2}}\left(\rho^{2\gamma-\alpha-1}u^2\right)^{\frac{1}{2}}dy\\
				&\leq \varepsilon \int_0^1 \rho^{1+\alpha}(\partial_yu)^2u^2r^{2(N-1)}dy +\varepsilon\int_0^1 \rho^{\alpha-1}\frac{u^{4}}{r^2}dy+C\int_0^1 \rho^{2\gamma-\alpha-1}u^2dy.
			\end{split}    
		\end{align*}
		
		When \(N=2\), using \eqref{C8}, we obtain, for sufficiently small \(\xi\),
		\begin{align*}
			\begin{split}
			&\quad C\int_0^1 \rho^{2\gamma-\alpha-1}u^2dy=C\int_0^1 \rho^{\alpha-1}\frac{u^2}{r^2}\rho^{2\gamma-2\alpha}r^2dy\\
			&\leq C\norm{\rho^{\alpha-\frac{1}{2}}r^\xi}^{\frac{2(\gamma-\alpha)}{\alpha-1/2}}_{L^\infty(0,R)}\int _0^1\rho^{\alpha-1}\frac{u^2}{r^2}dy\leq C\int _0^1\rho^{\alpha-1}\frac{u^2}{r^2}dy,
			\end{split}
		\end{align*}
		where we used
		\begin{align*}
			\frac{2(\gamma-\alpha)}{\alpha-1/2}\ge0\Leftrightarrow\gamma\ge\alpha.
		\end{align*}
		
		When \(N=3\), we use \eqref{C10} to get that
		\begin{align*}
			\begin{split}
			&\quad C\int_0^1 \rho^{2\gamma-\alpha-1}u^2dy=C\int_0^1 \rho^{\alpha-1}\frac{u^2}{r^2}\rho^{2\gamma-2\alpha}r^2dy\\ 
			&\leq C\norm{\rho^{\alpha-\frac{1}{2}}r^{\frac{1}{2}+\xi}}^{\frac{2(\gamma-\alpha)}{\alpha-1/2}}_{L^\infty(0,R)}\int_0^1 \rho^{\alpha-1}\frac{u^2}{r^2}dy\leq C\int _0^1\rho^{\alpha-1}\frac{u^2}{r^2}dy,
			\end{split}
		\end{align*}
		where we have used $\gamma<3\alpha-1$ to choose a sufficiently small \(\xi\) such that
		\begin{align*}
			0\leq \left(\frac{1}{2}+\xi\right)\frac{2(\gamma-\alpha)}{\alpha-1/2}\leq 2.
		\end{align*}
		
		Therefore, we obtain
		\begin{align*}
			\frac{d}{d\tau}\int_0^1 u^4dy+\varepsilon\int_0^1 \Big(\rho^{\alpha-1}\frac{u^4}{r^2}+\rho^{1+\alpha}(\partial_yu)^2u^2r^{2(N-1)}\Big)dy\leq C\int _0^1\rho^{\alpha-1}\frac{u^2}{r^2}dy.
		\end{align*}
		Integrating the above inequality with respect to $\tau$ and invoking \eqref{C u^2} yields \eqref{uu1}.
	\end{proof}
	
	\begin{prop}\label{Prop C 4}
		Assume that \eqref{C 2d unialpha} or \eqref{C 3d unialpha} holds. Then there exists a constant \(C>0\) independent of $T$ such that
		\begin{align}\label{C 2d 4}
			\sup_{0\leq t\leq T}\int_0^R\rho u^{4}r^{N-1}dr+\int_0^T\int_0^R\Big(\rho^{\alpha}u^{4}r^{N-3}
			+\rho^{\alpha} (\partial_ru)^2 u^{2}r^{N-1}\Big)drdt\leq C
		\end{align}
		and
		\begin{align}\label{C 2d l4}
			\sup_{0\leq t\leq T}\int_0^R \rho^{4\alpha-7}|\partial_r\rho|^4r^{N-1}dr+\int_0^T\int_0^R \rho^{\gamma+3\alpha-7}|\partial_r\rho|^4r^{N-1}drdt\leq C.
		\end{align}
	\end{prop}
	\begin{proof}
		Rewriting \eqref{uu1} in Eulerian coordinates gives \eqref{C 2d 4}. 
		
		We next prove \eqref{C 2d l4}. Noting that \(2\in M_{\mathrm{set}}\), so we can apply Proposition \ref{Prop C m} to obtain
		\begin{align*}
			\begin{split}
				&\quad\sup_{0\leq t\leq T}\int_0^R \rho^{4\alpha-7}|\partial_r\rho|^4r^{N-1}dr+\int_0^T\int_0^R \rho^{\gamma+3\alpha-7}|\partial_r\rho|^4r^{N-1}drdt\\
				&\leq C\sup_{0\leq t\leq T}\int_0^R\rho u^{4}r^{N-1}dr+ C\sup_{0\leq t\leq T}\int_0^R\rho(u+\rho^{-1}\partial_r\rho^\alpha)^{4}r^{N-1}dr+\int_0^T\int_0^R \rho^{\gamma+3\alpha-7}|\partial_r\rho|^4r^{N-1}drdt\nonumber\\
				&\leq C+C\int_0^T\int_0^R\rho^{\gamma-\alpha+1}u^{4}r^{N-1}drdt\\
				&\leq C+C\sup_{0\le t\le T}\norm{\rho^{\gamma-2\alpha+1}r^2}_{L^\infty(0,R)}\int_0^T\int_0^R\rho^\alpha u^4r^{N-3}drdt\\
				&\leq C+C\sup_{0\le t\le T}\norm{\rho^{\gamma-2\alpha+1}r^2}_{L^\infty(0,R)},
			\end{split}
		\end{align*}
		where we used \eqref{C 2d 4} and
		\begin{align*}
			\gamma\ge2\alpha-1.
		\end{align*}
		Therefore, it suffices to obtain \eqref{C 2d l4} by estimating \(\sup_{0\le t\le T}\norm{\rho^{\gamma-2\alpha+1}r^2}_{L^\infty(0,R)}\). 
		
		When \(N=2\), using \eqref{C8}, we obtain,
		\begin{align*}
			\sup_{0\le t\le T}\norm{\rho^{\gamma-2\alpha+1}r^2}_{L^\infty(0,R)}\leq C\sup_{0\le t\le T}\norm{\rho^{\alpha-\frac{1}{2}}r^{\xi}}_{L^\infty(0,R)}^{\frac{\gamma-2\alpha+1}{\alpha-1/2}}\leq C,
		\end{align*}
		where \(\xi\) is taken sufficiently small so that
		\begin{align*}
			\left(\frac{\gamma-2\alpha+1}{\alpha-1/2}\right)\xi\leq2.
		\end{align*}
		
		When \(N=3\), we use \eqref{C10} to obtain
		\begin{align*}
			\sup_{0\le t\le T}\norm{\rho^{\gamma-2\alpha+1}r^2}_{L^\infty(0,R)}\leq C\sup_{0\le t\le T}\norm{\rho^{\alpha-\frac{1}{2}}r^{\frac{1}{2}+\xi}}_{L^\infty(0,R)}^{\frac{\gamma-2\alpha+1}{\alpha-1/2}}\leq C,
		\end{align*}
		where we have used $\gamma<6\alpha-3$ to choose a sufficiently small \(\xi\) such that
		\begin{align*}
			\left(\frac{1}{2}+\xi\right)\frac{\gamma-2\alpha+1}{\alpha-1/2}\leq 2.
		\end{align*}
		
		This completes the proof of Proposition \ref{Prop C 4}.
	\end{proof}
	
	Next, we devote ourselves to obtaining a uniform upper bound for the density.
	\begin{prop}\label{Prop C uniRT}
		Assume that \eqref{C 2d unialpha} or \eqref{C 3d unialpha} holds. Then there exists a constant \(C>0\) independent of $T$ such that
		\begin{align}\label{C uniRT}
			R_T\leq C.
		\end{align}
	\end{prop}
	\begin{proof}
		When $N=2$, using \eqref{C7} and \eqref{C 2d l4}, we obtain from the one-dimensional Sobolev embedding that
		\begin{align*}
			&\norm{\rho^\alpha}_{L^\infty(0,R)}\leq C\int_0^R \rho^\alpha dr+C\int_0^R |\partial_r\rho^\alpha|dr\\
			&\leq C+C\left(\int_0^R\rho^{4\alpha-7}|\partial_r\rho|^4rdr\right)^{\frac{1}{4}}\left(\int_0^R \rho r^{-\frac{1}{3}} dr\right)^{\frac{3}{4}}\leq C+C\int_0^R\rho r^{-\frac{1}{3}}dr \\
			&\leq C+C\left(\int_0^R \rho^{4}rdr\right)^{\frac{1}{4}}\left(\int_0^Rr^{-\frac{7}{9}}dr\right)^{\frac{3}{4}}\leq C.
		\end{align*}   
		
		When \(N=3\), we choose \(\beta\) satisfying
		\begin{align}\label{3d uni beta}
			\max\Bigl\{0,\alpha-\frac3{4}\Bigr\}
			<\beta<\frac{3}{2}\alpha-1.
		\end{align}
		Combining the one-dimensional Sobolev embedding with \eqref{C9}, \eqref{C10} and \eqref{C 2d l4}, we obtain
		\begin{align*}
			\begin{split}
				&\norm{\rho^\beta}_{L^\infty(0,R)}\leq C\int_0^R\rho^\beta dr+C\int_0^R |\partial_r\rho^\beta| dr\\
				&\leq C+C\left(\int_0^R\rho^{4\alpha-7}|\partial_r\rho|^4r^2dr\right)^{\frac{1}{4}}\left(\int_0^R\rho^{\frac{4}{3}(\beta-\alpha)+1}r^{-\frac{2}{3}}dr\right)^{\frac{3}{4}}\\
				&\leq C+C\norm{\rho^{\alpha-\frac{1}{2}}r^{\frac{1}{2}+\xi}}_{L^\infty(0,R)}^{\frac{\frac{4}{3}(\beta-\alpha)+1}{\alpha-1/2}}\int_0^R r^{-\frac{2}{3}-\left(\frac{1}{2}+\xi\right)\frac{\frac{4}{3}(\beta-\alpha)+1}{\alpha-1/2}}dr\\
				&\leq C,
			\end{split}
		\end{align*}
		where we have used \eqref{3d uni beta} to choose a sufficiently small \(\xi>0\) such that
		\begin{align*}
			-\frac{2}{3}-\left(\frac{1}{2}+\xi\right)\frac{\frac{4}{3}(\beta-\alpha)+1}{\alpha-1/2}>-1.
		\end{align*}
		
		This completes the proof of Proposition \ref{Prop C uniRT}.
	\end{proof}
	
	Under the hypotheses of Theorem \ref{Thm2} we have thus established a uniform upper bound for the density on $\Omega\times[0,T]$. In two dimensions, \eqref{C 2d alpha} implies \eqref{C 2d unialpha}, so the positive lower bound follows directly from Proposition \ref{Prop C 2d VT} under assumption \eqref{C 2d unialpha}.  
	In three dimensions, however, \eqref{C 3d alpha} alone does not guarantee \eqref{C 3d unialpha}. Consequently, an additional statement is required to treat the positive lower bound of density under assumption \eqref{C 3d unialpha}, and we provide it next.
	
	\begin{prop}\label{Prop ex 3dVT}
		Assume that \eqref{C 3d unialpha} holds. There exists a constant \(C(T)>0\) such that
		\begin{align}\label{C 3d VT2}
			V_T\leq C(T).
		\end{align}
	\end{prop}
	\begin{proof}
		For brevity, we only outline the proof.  As stated in Lemma \ref{Lem C 3k}, for $\alpha\in(0.689,1)$ we can choose $k\in M_{set}$ satisfying \eqref{C21}.  Applying Proposition~\ref{Prop C n} and \eqref{C uniRT} to this \(k\), we obtain a $T$-dependent bound on \(\sup_{0\le t\le T}\int_0^R\rho u^{2k}r^2\mathrm{d}r\), which in turn guarantees the boundedness of \(\sup_{0\le t\le T}\int_0^R|\partial_r \rho^{\alpha-1+\frac{1}{2k}}|^{2k}r^2\mathrm{d}r\) by an argument similar to the one used for \eqref{C 2d l4}.  Once the estimate for the density derivative is available, repeating the proof of Proposition~\ref{Prop C 3d VT} yields \eqref{C 3d VT2}.
	\end{proof}
	
	\subsection{Higher-order estimates}
	In this subsection we carry out higher-order estimates for the local classical solution $(\rho,\mathbf u)$. Throughout this subsection we always assume $N=2$ or $3$. For brevity, we denote the standard Lebesgue and Sobolev spaces as follows:
	\begin{align*}
		L^p=L^p(\Omega),\quad W^{k,p}=W^{k,p}(\Omega),\quad H^k=W^{k,2}(\Omega).
	\end{align*}

\begin{prop}\label{Prop H1}
		Let $N=2$ with $l>1$, or $N=3$ with $1.5<l<3$, and assume that
	\begin{align}\label{hh2}
		\begin{split}
			\sup_{0\leq t\leq T}\Big(\norm{\mathbf{u}}_{L^{2l}}&+\norm{\rho}_{L^\infty}+\norm{\rho^{-1}}_{L^\infty}+\norm{\nabla\rho}_{L^2}+\norm{\nabla\rho}_{L^{2l}}\Big)\\
			&+\int_0^T\left(\norm{\nabla\rho}_{L^2}^2+\norm{\nabla\rho}_{L^{2l}}^{2l}+\norm{\nabla\mathbf{u}}_{L^2}^2\right)dt\leq M.
		\end{split}
	\end{align}
	Then there exists a constant \(C(M)>0\), independent of time \(T\), such that
	\begin{align}\label{h1}
		\sup_{0\leq t\leq T}\norm{\nabla\mathbf{u}}_{L^2}+\int_0^T\int_\Omega \rho |\dot{\mathbf{u}}|^2dxdt\leq C(M).
	\end{align}
\end{prop}
\begin{proof}
	Note that in the radially symmetric setting, \(\mathbb{D}\mathbf{u}=\nabla \mathbf{u}\). Therefore, we rewrite $\eqref{0}_2$ as
	\begin{align}\label{hh3}
		\rho\dot{\mathbf{u}}+\nabla P=\div\left(\mu(\rho)\nabla\mathbf{u}\right)+\nabla(\lambda(\rho)\div\mathbf{u})
	\end{align}
	Multiplying the above equation by $\dot{\mathbf{u}}$ and integrating the resulting equation over \(\Omega\) by parts gives
	\begin{align}\label{hI0}
		\int_\Omega \rho|\dot{\mathbf{u}}|^2dx=-\int_\Omega\nabla P\cdot\dot{\mathbf{u}}dx-\int_\Omega \mu(\rho)\nabla\mathbf{u}:\nabla\dot{\mathbf{u}}dx-\int_\Omega \lambda(\rho)\div\mathbf{u}\div\dot{\mathbf{u}}dx:=\sum_{i=1}^{3}I_i.
	\end{align}
	Hölder's inequality implies
	\begin{align}\label{hI1}
		\begin{split}
		I_1&=-\int_\Omega \nabla P\cdot\dot{\mathbf{u}}dx\leq C(M)\int_\Omega |\nabla\rho|\sqrt{\rho}|\dot{\mathbf{u}}|dx\\
		&\leq \frac{1}{6}\norm{\sqrt{\rho}\dot{\mathbf{u}}}_{L^2}^2+C(M)\norm{\nabla\rho}_{L^2}^2.
		\end{split}
	\end{align}
	From \eqref{hh3} we know that \(\mathbf{u}\) satisfies the elliptic system:
	\begin{align}\label{hh4}
		\left\{
		\begin{array}{l}
			\Delta\mathbf{u}+(\alpha-1)\nabla\div\mathbf{u}=\frac{1}{\rho^\alpha}(\rho\dot{\mathbf{u}}+\nabla P-\nabla\mu(\rho)\cdot\nabla\mathbf{u}-\nabla\lambda(\rho)\div\mathbf{u}),\\
			\mathbf{u}|_{\partial \Omega}=0.
		\end{array}
		\right.
	\end{align}
	Standard \(L^p\) estimates for elliptic systems and \eqref{hh2} imply
	\begin{align*}
		\norm{\mathbf{u}}_{H^2}&\leq C(M)(\norm{\rho\dot{\mathbf{u}}}_{L^2}+\norm{\nabla\rho}_{L^2}+\norm{|\nabla\rho||\nabla\mathbf{u}|}_{L^2})\\
		&\leq C(M)(\norm{\sqrt{\rho}\dot{\mathbf{u}}}_{L^2}+\norm{\nabla\rho}_{L^2}+\norm{\nabla\rho}_{L^{2l}}\norm{\nabla\mathbf{u}}_{L^{\frac{2l}{l-1}}})\\
		&\leq C(M)(\norm{\sqrt{\rho}\dot{\mathbf{u}}}_{L^2}+\norm{\nabla\rho}_{L^2}+\norm{\nabla\mathbf{u}}_{L^2}^{1-\frac{N}{2l}}\norm{\nabla\mathbf{u}}_{H^1}^{\frac{N}{2l}})\\
		&\leq \frac{1}{2}\norm{\mathbf{u}}_{H^2}+C(M)(\norm{\sqrt{\rho}\dot{\mathbf{u}}}_{L^2}+\norm{\nabla\rho}_{L^2}+\norm{\nabla\mathbf{u}}_{L^2}),
	\end{align*}
	Therefore, we obtain
	\begin{align}\label{hh5}
		\norm{\mathbf{u}}_{H^2}\leq C(M)(\norm{\sqrt{\rho}\dot{\mathbf{u}}}_{L^2}+\norm{\nabla\rho}_{L^2}+\norm{\nabla\mathbf{u}}_{L^2}).
	\end{align}
	Using \eqref{hh2}, \eqref{hh5} and Young’s inequality, we obtain
	\begin{align}\label{hI2}
		\begin{split}
			&I_2=-\int_\Omega \mu(\rho)\nabla\mathbf{u}:\nabla\mathbf{u}_tdx-\int_\Omega\mu(\rho)\nabla\mathbf{u}:\nabla(\mathbf{u}\cdot\nabla\mathbf{u})dx\\
			&=-\frac{1}{2}\frac{d}{dt}\int_\Omega \mu(\rho)|\nabla\mathbf{u}|^2dx+\frac{1}{2}\int_\Omega\partial_t\mu(\rho)|\nabla\mathbf{u}|^2dx-\int_\Omega\mu(\rho)\nabla\mathbf{u}:\nabla(\mathbf{u}\cdot\nabla\mathbf{u})dx\\
			&=-\frac{1}{2}\frac{d}{dt}\int_\Omega \mu(\rho)|\nabla\mathbf{u}|^2dx-\frac{1}{2}\int_\Omega\div(\mu(\rho)\mathbf{u})|\nabla\mathbf{u}|^2dx\\
			&\quad-\frac{\alpha-1}{2}\int_\Omega\mu(\rho)\div\mathbf{u}|\nabla\mathbf{u}|^2dx-\int_\Omega\mu(\rho)\nabla\mathbf{u}:\nabla(\mathbf{u}\cdot\nabla\mathbf{u})dx\\
			&\leq -\frac{1}{2}\frac{d}{dt}\int_\Omega \mu(\rho)|\nabla\mathbf{u}|^2dx+C(M)\int_\Omega|\nabla\rho||\mathbf{u}||\nabla\mathbf{u}|^2dx+C(M)\int_\Omega |\mathbf{u}||\nabla\mathbf{u}||\nabla^2\mathbf{u}|dx\\
			&\leq -\frac{1}{2}\frac{d}{dt}\int_\Omega \mu(\rho)|\nabla\mathbf{u}|^2dx+ C(M)\norm{\nabla\rho}_{L^{2l}}\norm{\mathbf{u}}_{L^{2l}}\norm{\nabla\mathbf{u}}_{L^{\frac{2l}{l-1}}}^2+C(M)\norm{\mathbf{u}}_{L^{2l}}\norm{\nabla^2\mathbf{u}}_{L^2}\norm{\nabla\mathbf{u}}_{L^{\frac{2l}{l-1}}}\\
			&\leq -\frac{1}{2}\frac{d}{dt}\int_\Omega \mu(\rho)|\nabla\mathbf{u}|^2dx+C(M)\norm{\nabla\mathbf{u}}_{L^{\frac{2l}{l-1}}}^2+C(M)\norm{\nabla^2\mathbf{u}}_{L^2}\norm{\nabla\mathbf{u}}_{L^{\frac{2l}{l-1}}}\\
			&\leq -\frac{1}{2}\frac{d}{dt}\int_\Omega \mu(\rho)|\nabla\mathbf{u}|^2dx+C(M)\norm{\nabla\mathbf{u}}_{L^2}^{\frac{2l-N}{l}}\norm{\nabla\mathbf{u}}_{H^1}^{\frac{N}{l}}+C(M)\norm{\nabla\mathbf{u}}_{L^2}^{\frac{2l-N}{2l}}\norm{\nabla \mathbf{u}}_{H^1}^{\frac{2l+N}{2l}}\\
			&\leq -\frac{1}{2}\frac{d}{dt}\int_\Omega \mu(\rho)|\nabla\mathbf{u}|^2dx+\frac{1}{6}\norm{\sqrt{\rho}\dot{\mathbf{u}}}_{L^2}^2+C(M)\norm{\nabla\rho}_{L^2}^2+C(M)\norm{\nabla \mathbf{u}}_{L^2}^2.
		\end{split}
	\end{align}
	Similarly, estimating \(I_3\), we obtain
	\begin{align}\label{hI3}
		\begin{split}
			&I_3=-\int_\Omega \lambda(\rho)\div\mathbf{u}\div\mathbf{u}_tdx-\int_\Omega\lambda(\rho)\div\mathbf{u}\div(\mathbf{u}\cdot\nabla\mathbf{u})dx\\
			&=-\frac{1}{2}\frac{d}{dt}\int_\Omega \lambda(\rho)(\div\mathbf{u})^2dx+\frac{1}{2}\int_\Omega\partial_t\lambda(\rho)(\div\mathbf{u})^2dx-\int_\Omega\lambda(\rho)\div\mathbf{u}\div(\mathbf{u}\cdot\nabla\mathbf{u})dx\\
			&=-\frac{1}{2}\frac{d}{dt}\int_\Omega \lambda(\rho)(\div\mathbf{u})^2dx-\frac{1}{2}\int_\Omega \div(\lambda(\rho) \mathbf{u})(\div\mathbf{u})^2dx\\
			&\quad-\frac{(\alpha-1)}{2}\int_\Omega\lambda(\rho)(\div\mathbf{u})^3dx-\int_\Omega\lambda(\rho)\div\mathbf{u}\div(\mathbf{u}\cdot\nabla\mathbf{u})dx\\
			&\leq -\frac{1}{2}\frac{d}{dt}\int_\Omega \lambda(\rho)(\div\mathbf{u})^2dx+C(M)\int_\Omega|\nabla\rho||\mathbf{u}||\nabla\mathbf{u}|^2dx+C(M)\int_\Omega |\mathbf{u}||\nabla\mathbf{u}||\nabla^2\mathbf{u}|dx\\
			&\leq -\frac{1}{2}\frac{d}{dt}\int_\Omega \lambda(\rho)(\div\mathbf{u})^2dx+\frac{1}{6}\norm{\sqrt{\rho}\dot{\mathbf{u}}}_{L^2}^2+C(M)\norm{\nabla\rho}_{L^2}^2+C(M)\norm{\nabla \mathbf{u}}_{L^2}^2.
		\end{split}
	\end{align}
	Inserting \eqref{hI1}, \eqref{hI2} and \eqref{hI3} into \eqref{hI0} gives
	\begin{align}\label{hh6}
		\begin{split}
			\frac{d}{dt}\int_\Omega( \mu(\rho)|\nabla\mathbf{u}|^2+\lambda(\rho)(\div\mathbf{u})^2)dx +\int_\Omega\rho|\dot{\mathbf{u}}|^2dx\leq C(M)(\norm{\nabla\rho}_{L^2}^2+\norm{\nabla\mathbf{u}}_{L^2}^2).
		\end{split}
	\end{align}
Integrating the above equality over $(0,t)$ and using \eqref{hh2}, we obtain
\begin{align}\label{hh6.5}
	\int_\Omega( \mu(\rho)|\nabla\mathbf{u}|^2+\lambda(\rho)(\div\mathbf{u})^2)dx+\int_0^t\int_\Omega \rho|\dot{\mathbf{u}}|^2dxdt\leq C(M).
\end{align}	
Note that $\frac{N-1}{N}<\alpha<1$ gives 
\begin{align*}
	\begin{split}
		&\quad\int_\Omega( \mu(\rho)|\nabla\mathbf{u}|^2+\lambda(\rho)(\div\mathbf{u})^2)dx\\
		&=\int_\Omega \Big(\rho^\alpha|\nabla\mathbf{u}|^2+(\alpha-1)\rho^\alpha(\div\mathbf{u})^2\Big)dx\\
		&\ge \int_\Omega (1+(\alpha-1)N)\rho^\alpha |\nabla\mathbf{u}|^2dx\\
		&\ge \frac{1+(\alpha-1)N}{M^\alpha}\int_\Omega |\nabla\mathbf{u}|^2dx,
	\end{split}
\end{align*}
which, combined with \eqref{hh6.5}, shows \eqref{h1}. This completes the proof of Proposition \ref{Prop H1}.
\end{proof}

	\begin{prop}\label{Prop H2}
	Assume that \eqref{hh2} holds. Then there exists a constant \(C(M)>0\), independent of $T$, such that
	\begin{align}\label{h2}
		\sup_{0\leq t\leq T}\left(\norm{\sqrt{\rho}\dot{\mathbf{u}}}_{L^2}+\norm{\mathbf{u}}_{H^2}\right)+\int_0^T\int_\Omega|\nabla\dot{\mathbf{u}}|^2dxdt\leq C(M).
	\end{align}
\end{prop}
\begin{proof}
	We use the method of Hoff in \cite{hoff-1995} to prove \eqref{h2}.  Operating $\dot{u}^j[\partial/\partial t+\div(\mathbf{u}\cdot)]$ to \(\eqref{hh3}^j\), summing over \(j\), integrating the resulting equation over \(\Omega\) and using integration by parts, we obtain
	\begin{align}\label{hh7}
		\begin{split}
			&\left(\dfrac{1}{2}\int\rho|\dot{\mathbf{u}}|^2dx\right)_t\\
			&=-\int_\Omega\dot{u}^j[\partial_j P_t+\div(\mathbf{u}\partial_j P)]dx+\int_\Omega\dot{u}_j[\partial_{it}(\mu(\rho)\partial_i u_j)+\div(\mathbf{u}\partial_i(\mu(\rho)\partial_i u_j))]dx\\
			&\quad+\int_\Omega \dot{u}^j[\partial_{jt}(\lambda(\rho)\div \mathbf{u})+\div (\mathbf{u}\partial_j(\lambda(\rho)\div \mathbf{u}))]dx:=\sum_{i=1}^3N_i.
		\end{split}
	\end{align}
	Using equation $\eqref{0}_1$, we obtain by integration by parts that
	\begin{align}\label{hn1}
		\begin{split}
			N_1&=-\int _\Omega\dot{u}^j[\partial_j P_t+\div(\partial_j P\mathbf{u})]dx\\
			&=\int_\Omega\left[ -P'\rho\div \mathbf{u}\partial_j\dot{u}^j+\partial_k(\partial_j\dot{u}^ju^k)P-P\partial_j(\partial_k\dot{u}^ju^k)  \right]dx\\
			&\leq C(M)\norm{\nabla \mathbf{u}}_{L^2}\norm{\nabla \dot{\mathbf{u}}}_{L^2}\\
			&\leq\varepsilon\norm{\nabla \dot{\mathbf{u}}}_{L^2}^2+C(\varepsilon,M)\norm{\nabla \mathbf{u}}_{L^2}^2,
		\end{split}
	\end{align}
	where \(\varepsilon>0\) is a sufficiently small positive constant to be determined. From \eqref{hh2} and equation $\eqref{0}_1$, we obtain by integration by parts that
	\begin{align}\label{hh12}
		\begin{split}
			&N_2=\int_\Omega\dot{u}_j[\partial_{it}(\mu(\rho)\partial_i u_j)+\div(\mathbf{u}\partial_i(\mu(\rho)\partial_i u_j))]dx\\
			&=-\int_\Omega\partial_i \dot{u}_j\partial_t(\mu(\rho)\partial_iu_j)dx-\int_\Omega \partial_k \dot{u}_ju_k\partial_i(\mu(\rho)\partial_iu_j)dx\\
			&=-\int_\Omega \mu(\rho)|\nabla\dot{\mathbf{u}}|^2dx+\int_\Omega\mu(\rho)\partial_i\dot{u}_j\partial_i(u_k\partial_ku_j)dx-\int_\Omega\partial_i\dot{u}_j\partial_t\mu(\rho)\partial_iu_jdx\\
			&\quad-\int_\Omega\partial_k\dot{u}_ju_k\partial_i\mu(\rho)\partial_iu_jdx-\int_\Omega \partial_k\dot{u}_ju_k\mu(\rho)\partial_{ii}u_jdx\\
			&\leq -\int_\Omega \mu(\rho)|\nabla\dot{\mathbf{u}}|^2dx+C(M)\int_\Omega\left(|\nabla\dot{\mathbf{u}}||\nabla\mathbf{u}|^2+ |\nabla\dot{\mathbf{u}}||\mathbf{u}||\nabla^2\mathbf{u}|+|\nabla\dot{\mathbf{u}}||\mathbf{u}||\nabla\rho||\nabla\mathbf{u}|\right)dx\\
			&\leq -\int_\Omega \mu(\rho)|\nabla\dot{\mathbf{u}}|^2dx+\varepsilon\int_\Omega|\nabla\dot{\mathbf{u}}|^2dx+C(\varepsilon,M)\int_\Omega\left(|\mathbf{u}|^2|\nabla^2 \mathbf{u}|^2+|\nabla\mathbf{u}|^4+|\nabla\rho|^2|\mathbf{u}|^2|\nabla\mathbf{u}|^2\right)dx.
		\end{split}
	\end{align}
	Using standard \(L^p\) estimates for elliptic systems, we obtain from $\eqref{hh4}$ and \eqref{hh2} that
	\begin{align*}
		\norm{\mathbf{u}}_{W^{2,{2l}}}&\leq C(M)(\norm{\rho\dot{\mathbf{u}}}_{L^{2l}}+\norm{\nabla\rho}_{L^{2l}}+\norm{|\nabla\rho||\nabla\mathbf{u}|}_{L^{2l}})\\
		&\leq C(M)(\norm{\sqrt{\rho}\dot{\mathbf{u}}}_{L^{2l}}+\norm{\nabla\rho}_{L^{2l}}+\norm{\nabla\rho}_{L^{2l}}\norm{\nabla\mathbf{u}}_{L^{\infty}})\\
		&\leq C(M)(\norm{\sqrt{\rho}\dot{\mathbf{u}}}_{L^{2l}}+\norm{\nabla\rho}_{L^{2l}}+\norm{\nabla\mathbf{u}}_{L^2}^{\frac{2l-N}{Nl+2l-N}}\norm{\nabla\mathbf{u}}_{W^{1,{2l}}}^{\frac{Nl}{Nl+2l-N}})\\
		&\leq \frac{1}{2}\norm{\mathbf{u}}_{W^{2,{2l}}}+C(M)(\norm{\sqrt{\rho}\dot{\mathbf{u}}}_{L^{2l}}+\norm{\nabla\rho}_{L^{2l}}+\norm{\nabla\mathbf{u}}_{L^2})
	\end{align*}
	which implies that
	\begin{align}\label{hh8}
		\norm{\mathbf{u}}_{W^{2,{2l}}}\leq C(M)(\norm{\sqrt{\rho}\dot{\mathbf{u}}}_{L^{2l}}+\norm{\nabla\rho}_{L^{2l}}+\norm{\nabla\mathbf{u}}_{L^2}).
	\end{align}
	Observing that \(\mathbf{u}|_{\partial \Omega}=0\) implies \(\|\mathbf{u}\|_{L^{\frac{2l}{l-1}}}\leq C\|\nabla \mathbf{u}\|_{L^2}\leq C(M)\), we thus obtain from \eqref{hh8} that
	\begin{align}\label{hh9}
		\begin{split}
			C(\varepsilon,M)\int_\Omega|\mathbf{u}|^2|\nabla^2\mathbf{u}|^2dx&\leq C(\varepsilon,M) \norm{\mathbf{u}}_{L^{\frac{2l}{l-1}}}^2 \norm{\nabla^2\mathbf{u}}_{L^{2l}}^2 \leq C(\varepsilon,M)\norm{\mathbf{u}}_{L^{\frac{2l}{l-1}}}^2\norm{\mathbf{u}}_{W^{2,2l}}^2\\
			&\leq C(\varepsilon,M)\norm{\mathbf{u}}_{L^{\frac{2l}{l-1}}}^2(\norm{\sqrt{\rho}\dot{\mathbf{u}}}_{L^{2l}}^2+\norm{\nabla\rho}_{L^{2l}}^2+\norm{\nabla\mathbf{u}}_{L^2}^2)\\
			&\leq C(\varepsilon,M)(\norm{\sqrt{\rho}\dot{\mathbf{u}}}_{L^2}^{\frac{N+2l-Nl}{l}}\norm{\nabla\dot{\mathbf{u}}}_{L^2}^{\frac{Nl-N}{l}}+\norm{\mathbf{u}}_{L^{\frac{2l}{l-1}}}^2\norm{\nabla\rho}_{L^{2l}}^2+\norm{\nabla\mathbf{u}}_{L^2}^2)\\
			&\leq \varepsilon\norm{\nabla\dot{\mathbf{u}}}_{L^2}^2+C(\varepsilon,M)(\norm{\sqrt{\rho}\dot{\mathbf{u}}}_{L^2}^2+\norm{\nabla\rho}_{L^{2l}}^{2l}+\norm{\mathbf{u}}_{L^{\frac{2l}{l-1}}}^{\frac{2l}{l-1}}+\norm{\nabla\mathbf{u}}_{L^2}^2)\\
			&\leq \varepsilon\norm{\nabla\dot{\mathbf{u}}}_{L^2}^2+C(\varepsilon, M)(\norm{\sqrt{\rho}\dot{\mathbf{u}}}_{L^2}^2+\norm{\nabla\rho}_{L^{2l}}^{2l}+\norm{\nabla\mathbf{u}}_{L^2}^2),
		\end{split}
	\end{align}
	where the last inequality follows from $\frac{2l}{l-1}>2$ and \eqref{h1}. We use \eqref{hh5}, \eqref{h1} and \eqref{hh2} to obtain
	\begin{align}\label{hh10}
		\begin{split}
			C(\varepsilon,M)\int_\Omega|\nabla\mathbf{u}|^4dx&=C(\varepsilon,M) \norm{\nabla\mathbf{u}}_{L^4}^4\leq C(\varepsilon,M)\norm{\nabla\mathbf{u}}_{H^1}^4\\
			&\leq C(\varepsilon,M)(\norm{\sqrt{\rho}\dot{\mathbf{u}}}_{L^2}^4+\norm{\nabla\rho}_{L^2}^4+\norm{\nabla\mathbf{u}}_{L^2}^4)\\
			&\leq C(\varepsilon,M)(\norm{\sqrt{\rho}\dot{\mathbf{u}}}_{L^2}^4+\norm{\nabla\rho}_{L^2}^2+\norm{\nabla\mathbf{u}}_{L^2}^2).
		\end{split}
	\end{align}
	Arguing as in the estimate for \eqref{hh9}, one has
	\begin{align}\label{hh11}
		\begin{split}
			C(\varepsilon,M)\int_\Omega|\nabla\rho|^2|\mathbf{u}|^2|\nabla\mathbf{u}|^2dx&\leq C(\varepsilon,M)\norm{\nabla\rho}_{L^{2l}}^2\norm{\mathbf{u}}_{L^{\frac{2l}{l-1}}}^2\norm{\nabla\mathbf{u}}_{L^{\infty}}^2\\
			&\leq C(\varepsilon,M)\norm{\mathbf{u}}_{L^{\frac{2l}{l-1}}}^2\norm{\nabla\mathbf{u}}_{W^{1,{2l}}}^2 \\
			&\leq \varepsilon\norm{\nabla\dot{\mathbf{u}}}_{L^2}^2+C(\varepsilon,M)(\norm{\sqrt{\rho}\dot{\mathbf{u}}}_{L^2}^2+\norm{\nabla\rho}_{L^{2l}}^{2l}+\norm{\nabla\mathbf{u}}_{L^2}^2).
		\end{split}
	\end{align}
	Inserting \eqref{hh9}, \eqref{hh10} and \eqref{hh11} into \eqref{hh12} shows that
	\begin{align}\label{hn2}
		\begin{split}
			N_2\leq -\int_\Omega \mu(\rho)|\nabla\dot{\mathbf{u}}|^2dx&+3\varepsilon \int_\Omega|\nabla\dot{\mathbf{u}}|^2dx+C(\varepsilon,M)\norm{\sqrt{\rho}\dot{\mathbf{u}}}_{L^2}^4\\
			&+C(\varepsilon,M)(\norm{\sqrt{\rho}\dot{\mathbf{u}}}_{L^2}^2+\norm{\nabla\rho}_{L^2}^2+\norm{\nabla\rho}_{L^{2l}}^{2l}+\norm{\nabla\mathbf{u}}_{L^{2}}^2).
		\end{split}
	\end{align}
	Similarly, estimating \(N_3\), we have
	\begin{align}\label{hh13}
		\begin{split}
			&N_3=\int_\Omega \dot{u}^j[\partial_{jt}(\lambda(\rho)\div \mathbf{u})+\div (\mathbf{u}\partial_j(\lambda(\rho)\div \mathbf{u}))]dx\\
			&=-\int_\Omega\div\dot{\mathbf{u}}\partial_t(\lambda(\rho)\div\mathbf{u})dx-\int_\Omega \partial_i\dot{u}_ju_i\partial_j(\lambda(\rho)\partial_k u_k)dx\\
			&=-\int_\Omega\lambda(\rho)(\div\dot{\mathbf{u}})^2dx+\int_\Omega \lambda(\rho)\div\dot{\mathbf{u}}\div(\mathbf{u}\cdot\nabla\mathbf{u})dx-\int_\Omega\partial_t\lambda(\rho)\div\dot{\mathbf{u}}\div\mathbf{u}dx\\
			&\quad-\int_\Omega \partial_i\dot{u}_ju_i\partial_j(\lambda(\rho)\partial_k u_k)dx\\
			&\leq -\int_\Omega\lambda(\rho)(\div\dot{\mathbf{u}})^2dx+C(M)\int_\Omega\left(|\nabla\dot{\mathbf{u}}||\nabla\mathbf{u}|^2+ |\nabla\dot{\mathbf{u}}||\mathbf{u}||\nabla^2\mathbf{u}|+|\nabla\dot{\mathbf{u}}||\mathbf{u}||\nabla\rho||\nabla\mathbf{u}|\right)dx\\
			&\leq -\int_\Omega\lambda(\rho)(\div\dot{\mathbf{u}})^2dx+\varepsilon\int_\Omega|\nabla\dot{\mathbf{u}}|^2dx+C(\varepsilon,M)\int_\Omega\left(|\mathbf{u}|^2|\nabla^2 \mathbf{u}|^2+|\nabla\mathbf{u}|^4+|\nabla\rho|^2|\mathbf{u}|^2|\nabla\mathbf{u}|^2\right)dx.
		\end{split}
	\end{align}
	Inserting \eqref{hh9}, \eqref{hh10} and \eqref{hh11} into \eqref{hh13} gives
	\begin{align}\label{hn3}
		\begin{split}
			N_3\leq -\int_\Omega\lambda(\rho)(\div\dot{\mathbf{u}})^2dx&+3\varepsilon \int_\Omega|\nabla\dot{\mathbf{u}}|^2dx+C(\varepsilon,M)\norm{\sqrt{\rho}\dot{\mathbf{u}}}_{L^2}^4\\
			&+C(\varepsilon,M)(\norm{\sqrt{\rho}\dot{\mathbf{u}}}_{L^2}^2+\norm{\nabla\rho}_{L^2}^2+\norm{\nabla\rho}_{L^{2l}}^{2l}+\norm{\nabla\mathbf{u}}_{L^2}^2).
		\end{split}
	\end{align}
	Using \eqref{hn1}, \eqref{hn2} and \eqref{hn3}, we obtain from \eqref{hh7} that 
	\begin{align*}
		\begin{split}
			&\quad\frac{1}{2}\frac{d}{dt}\int_\Omega\rho|\dot{\mathbf{u}}|^2dx+\int_\Omega( \mu(\rho)|\nabla\dot{\mathbf{u}}|^2+\lambda(\rho)(\div\dot{\mathbf{u}})^2)dx-7\varepsilon\norm{\nabla\dot{\mathbf{u}}}_{L^2}^2\\
			&\leq  C(\varepsilon,M)(\norm{\sqrt{\rho}\dot{\mathbf{u}}}_{L^2}^2+\norm{\nabla\rho}_{L^2}^2+\norm{\nabla\rho}_{L^{2l}}^{2l}+\norm{\nabla\mathbf{u}}_{L^2}^2)+C(\varepsilon,M)\norm{\sqrt{\rho}\dot{\mathbf{u}}}_{L^2}^4.
		\end{split}
	\end{align*}
	Therefore, one has
	\begin{align*}
		\begin{split}
			&\quad\frac{1}{2}\frac{d}{dt}\int_\Omega\rho|\dot{\mathbf{u}}|^2dx+\left(\frac{1+(\alpha-1)N}{M^\alpha}-7\varepsilon\right)\int_\Omega |\nabla\dot{\mathbf{u}}|^2dx\\
			&\leq C(\varepsilon,M)(\norm{\sqrt{\rho}\dot{\mathbf{u}}}_{L^2}^2+\norm{\nabla\rho}_{L^2}^2+\norm{\nabla\rho}_{L^{2l}}^{2l}+\norm{\nabla\mathbf{u}}_{L^2}^2)+C(\varepsilon,M)\norm{\sqrt{\rho}\dot{\mathbf{u}}}_{L^2}^4.
		\end{split}
	\end{align*}
	Choose \(\varepsilon>0\) sufficiently small so that $7\varepsilon<\frac{1+(\alpha-1)N}{2M^\alpha}$. Then we get
		\begin{align}\label{hh14}
		\begin{split}
			&\quad\frac{1}{2}\frac{d}{dt}\int_\Omega\rho|\dot{\mathbf{u}}|^2dx+\frac{1+(\alpha-1)N}{2M^\alpha}\int_\Omega |\nabla\dot{\mathbf{u}}|^2dx\\
			&\leq C(M)(\norm{\sqrt{\rho}\dot{\mathbf{u}}}_{L^2}^2+\norm{\nabla\rho}_{L^2}^2+\norm{\nabla\rho}_{L^{2l}}^{2l}+\norm{\nabla\mathbf{u}}_{L^2}^2)+C(M)\norm{\sqrt{\rho}\dot{\mathbf{u}}}_{L^2}^4.
		\end{split}
	\end{align}
Using Gronwall’s inequality together with \eqref{h1}, \eqref{hh2} and \eqref{hh5}, we obtain \eqref{h2}.
\end{proof}

\begin{prop}\label{Prop H3}
	Assume that \eqref{hh2} holds.  Then there exists a constant \(C(M,T)>0\) depending on \(M\) and \(T\) such that
	\begin{align}\label{h3}
		\sup_{0\leq t\leq T}\norm{\rho}_{H^2}+\int_0^T\norm{\mathbf{u}}_{H^3}^2dt\leq C(M,T).
	\end{align}
\end{prop}
\begin{proof}
	Operating \(\partial_i\partial_j\) to equation $\eqref{0}_1$, we obtain
	\begin{align*}
		\begin{split}
			\partial_t(\partial_{ij}\rho)&+\partial_{ij} u_k\partial_k \rho+\partial_i u_k\partial_{jk}\rho+\partial_j u_k\partial_{ik}\rho+u_k \partial_{ijk}\rho\\
			&+\partial_{ij}\rho\partial_ku_k+\partial_i\rho\partial_{jk}u_k+\partial_j\rho\partial_{ik}u_k+\rho\partial_{ijk} u_k=0.
		\end{split}
	\end{align*}
	Multiplying the above equation by \(\partial_{ij}\rho\), summing over \(i,j\), and integrating the resulting equation over \(\Omega\), we use integration by parts to obtain
	\begin{align}\label{hh15}
		\begin{split}
			&\quad\frac{d}{dt}\int_\Omega |\nabla^2\rho|^2dx\leq C\left(\int_\Omega|\nabla\rho||\nabla^2\mathbf{u}||\nabla^2\rho|dx+\int_\Omega|\nabla\mathbf{u}||\nabla^2\rho|^2dx+\int_\Omega\rho|\nabla^3\mathbf{u}||\nabla^2\rho|dx\right)\\
			&\leq C(M)(1+\norm{\nabla\mathbf{u}}_{L^\infty})\int_\Omega|\nabla^2\rho|^2dx+C(M)\int_\Omega |\nabla\rho|^2|\nabla^2\mathbf{u}|^2dx+C(M)\int_\Omega|\nabla^3\mathbf{u}|^2dx\\
			&\leq C(M)(1+\norm{\nabla\mathbf{u}}_{L^\infty})\int_\Omega|\nabla^2\rho|^2dx+C(M)\norm{\nabla\rho}_{L^{2l}}^2\norm{\nabla^2\mathbf{u}}_{L^{\frac{2l}{l-1}}}^2+C(M)\int_\Omega|\nabla^3\mathbf{u}|^2dx\\
			&\leq C(M)(1+\norm{\nabla\mathbf{u}}_{L^\infty})\int_\Omega|\nabla^2\rho|^2dx+C(M)+C(M)\int_\Omega|\nabla^3\mathbf{u}|^2dx,
		\end{split}
	\end{align}
	where the last inequality follows from \eqref{hh2} and \eqref{h2}. We obtain from \eqref{hh8} that
	\begin{align}\label{hh16}
		\begin{split}
			\norm{\nabla\mathbf{u}}_{L^\infty}&\leq C\norm{\nabla\mathbf{u}}_{W^{1,{2l}}}\leq C(M)(\norm{\sqrt{\rho}\dot{\mathbf{u}}}_{L^{2l}}+\norm{\nabla\rho}_{L^{2l}}+\norm{\nabla\mathbf{u}}_{L^2})\\
			&\leq C(M)+C(M)\norm{\nabla\dot{\mathbf{u}}}_{L^2}.
		\end{split}
	\end{align}
	Using the elliptic system \eqref{hh4} once more, we obtain from \eqref{hh2}, \eqref{h2} and \eqref{hh16} that
	\begin{align*}
		\begin{split}
			\norm{\mathbf{u}}_{H^3}^2&\leq C\norm{\frac{1}{\rho^\alpha}(\rho\dot{\mathbf{u}}+\nabla P-\nabla\mu(\rho)\cdot\nabla\mathbf{u}-\nabla\lambda(\rho)\div\mathbf{u})}_{H^1}^2\\
			&\leq C(M)\left(1+\norm{|\nabla\rho| |\dot{\mathbf{u}}|}_{L^2}^2+\norm{\nabla\dot{\mathbf{u}}}_{L^2}^2+\norm{\nabla^2\rho}_{L^2}^2+\norm{\nabla\rho}_{L^4}^2+\norm{|\nabla\rho|^2|\nabla\mathbf{u}|}_{L^2}^2\right.\\
			&\left.\quad+\norm{|\nabla^2\rho||\nabla\mathbf{u}|}_{L^2}^2+\norm{|\nabla\rho||\nabla^2\mathbf{u}|}_{L^2}^2\right)\\
			&\leq C(M)\left(1+\norm{\nabla\dot{\mathbf{u}}}_{L^2}^2\right)\left(1+\norm{\nabla^2\rho}_{L^2}^2\right)+C(M)\norm{\nabla^2 \mathbf{u}}_{L^{\frac{2l}{l-1}}}^2\\
			&\leq C(M)\left(1+\norm{\nabla\dot{\mathbf{u}}}_{L^2}^2\right)\left(1+\norm{\nabla^2\rho}_{L^2}^2\right)+C(M)\norm{\nabla^2\mathbf{u}}_{L^2}^{\frac{2l-N}{l}}\norm{\nabla^2\mathbf{u}}_{H^1}^{\frac{N}{l}}\\
			&\leq \frac{1}{2}\norm{\mathbf{u}}_{H^3}^2+C(M)\left(1+\norm{\nabla\dot{\mathbf{u}}}_{L^2}^2\right)\left(1+\norm{\nabla^2\rho}_{L^2}^2\right), 
		\end{split}
	\end{align*}
	which implies 
	\begin{align}\label{hh17}
		\norm{\mathbf{u}}_{H^3}^2\leq C(M)\left(1+\norm{\nabla\dot{\mathbf{u}}}_{L^2}^2\right)\left(1+\norm{\nabla^2\rho}_{L^2}^2\right).
	\end{align}
	Using \eqref{hh16} and \eqref{hh17}, we obtain from \eqref{hh15} that
	\begin{align*}
		\frac{d}{dt}\int_\Omega |\nabla^2\rho|^2dx\leq C(M)\left(1+\norm{\nabla\dot{\mathbf{u}}}_{L^2}^2\right)\left(1+\norm{\nabla^2\rho}_{L^2}^2\right)
	\end{align*}
	This, combined with Gronwall's inequality, \eqref{hh17} and \eqref{h2}, gives \eqref{h3}.
\end{proof}

\begin{prop}\label{Prop H4}
	Assume that \eqref{hh2} holds. Then there exists a constant \(C(M,T)>0\) depending on \(M\) and \(T\) such that
	\begin{align}\label{h4}
		\sup_{0\leq t\leq T}(\norm{\nabla\dot{\mathbf{u}}}_{L^2}+\norm{\nabla\mathbf{u}_t}_{L^2}+\norm{\mathbf{u}}_{H^3})+\int_0^T(\norm{\mathbf{u}_{tt}}_{L^2}^2+\norm{\mathbf{u}_t}_{H^2}^2)dt\leq C(M,T).
	\end{align}
\end{prop}
\begin{proof}
	Differentiating \eqref{hh3} with respect to $t$ gives
	\begin{align}\label{hh18}
		-\Delta \mathbf{u}_t-(\alpha-1)\nabla\div\mathbf{u}_t=-\left[\frac{1}{\rho^\alpha}\left(\rho\dot{\mathbf{u}}+\nabla P-\nabla\mu(\rho)\cdot\nabla\mathbf{u}-\nabla\lambda(\rho)\div\mathbf{u}\right)\right]_t.
	\end{align}
	Multiplying by \(\mathbf{u}_{tt}\) and integrating the resulting equation over \(\Omega\) by parts gives
	\begin{align}\label{hh20}
		\begin{split}
			&\quad\frac{1}{2}\frac{d}{dt}\int_\Omega\Big(|\nabla\mathbf{u}_t|^2+(\alpha-1)(\div\mathbf{u}_t)^2\Big)dx+\int_\Omega\rho^{1-\alpha}|\mathbf{u}_{tt}|^2dx\\
			&\leq C(M)\int_\Omega|\mathbf{u}_{tt}|\left( |\rho_t||\mathbf{u}_t|+|\rho_t||\mathbf{u}||\nabla\mathbf{u}|+|\mathbf{u}_t||\nabla\mathbf{u}|+|\mathbf{u}||\nabla\mathbf{u}_t|+|\rho_t||\nabla\rho|+|\nabla\rho_t|\right.\\
			&\quad\left.+|\rho_t||\nabla\rho||\nabla \mathbf{u}|+|\nabla\rho_t||\nabla\mathbf{u}|+|\nabla\rho||\nabla \mathbf{u}_t|\right)dx\\
			&\leq \varepsilon\int_\Omega |\mathbf{u}_{tt}|^2dx+C(\varepsilon,M)\int_\Omega\left( |\rho_t|^2|\mathbf{u}_t|^2+|\rho_t|^2|\mathbf{u}|^2|\nabla\mathbf{u}|^2+|\mathbf{u}_t|^2|\nabla\mathbf{u}|^2+|\mathbf{u}|^2|\nabla\mathbf{u}_t|^2\right.\\
			&\quad\left.+|\rho_t|^2|\nabla\rho|^2+|\nabla\rho_t|^2+|\rho_t|^2|\nabla\rho|^2|\nabla \mathbf{u}|^2+|\nabla\rho_t|^2|\nabla\mathbf{u}|^2+|\nabla\rho|^2|\nabla \mathbf{u}_t|^2\right)dx\\
			&\leq \varepsilon\int_\Omega |\mathbf{u}_{tt}|^2dx+C(\varepsilon,M)\left(\norm{\rho_t}_{L^6}^2\norm{\mathbf{u}_t}_{L^6}^2+\norm{\rho_t}_{L^6}^2\norm{\mathbf{u}}_{L^\infty}^2\norm{\nabla\mathbf{u}}_{L^6}^2+\norm{\mathbf{u}_t}_{L^6}^2\norm{\nabla\mathbf{u}}_{L^6}^2\right.\\
			&\quad +\norm{\mathbf{u}}_{L^\infty}^2\norm{\nabla\mathbf{u}_t}_{L^2}^2+\norm{\rho_t}_{L^6}^2\norm{\nabla\rho}_{L^6}^2+\norm{\nabla\rho_t}_{L^2}^2+\norm{\rho_t}_{L^6}^2\norm{\nabla\rho}_{L^6}^2\norm{\nabla\mathbf{u}}_{L^6}^2\\
			&\quad+\left.\norm{\nabla\rho_t}_{L^2}^2\norm{\nabla\mathbf{u}}_{L^\infty}^2+\norm{\nabla\rho}_{L^6}^2\norm{\nabla\mathbf{u}_t}_{L^3}^2\right),
		\end{split}
	\end{align}
	where \(\varepsilon>0\) is a sufficiently small constant to be determined. Note that
	\begin{align}\label{hh21}
		\begin{split}
			\norm{\nabla \mathbf{u}_t}_{L^2}^2\leq C\norm{\nabla \dot{\mathbf{u}}}_{L^2}^2+C\norm{\nabla(\mathbf{u}\cdot\nabla\mathbf{u})}_{L^2}^2\leq C\norm{\nabla\dot{\mathbf{u}}}_{L^2}^2+C(M);\\
			\norm{\nabla \dot{\mathbf{u}}}_{L^2}^2\leq C \norm{\nabla \mathbf{u}_t}_{L^2}^2+C\norm{\nabla(\mathbf{u}\cdot\nabla\mathbf{u})}_{L^2}^2\leq C\norm{\nabla\mathbf{u}_t}_{L^2}^2+C(M).
		\end{split}
	\end{align}
	Using \eqref{hh18} and the boundary condition \(\mathbf{u}_t|_{\partial\Omega}=0\), we obtain from the \(L^p\) estimate for the elliptic system that
	\begin{align*}
		\norm{\mathbf{u}_t}_{H^2}^2&\leq C\norm{\left[\frac{1}{\rho^\alpha}\left(\rho\dot{\mathbf{u}}+\nabla P-\nabla\mu\cdot\nabla\mathbf{u}-\nabla\lambda\div\mathbf{u}\right)\right]_t}_{L^2}^2\\
		&\leq C(M)\left(\norm{\rho_t\mathbf{u}_t}_{L^2}^2+\norm{\mathbf{u}_{tt}}_{L^2}^2+\norm{\rho_t|\mathbf{u}||\nabla\mathbf{u}|}_{L^2}^2+\norm{|\mathbf{u}_t||\nabla\mathbf{u}|}_{L^2}^2+\norm{|\mathbf{u}||\nabla\mathbf{u}_t|}_{L^2}^2+\right.\\
		&\quad\left.\norm{\rho_t|\nabla\rho|}_{L^2}^2+\norm{\nabla\rho_t}_{L^2}^2+\norm{\rho_t|\nabla\rho||\nabla\mathbf{u}|}_{L^2}^2+\norm{|\nabla\rho_t||\nabla\mathbf{u}|}_{L^2}^2+\norm{|\nabla\rho||\nabla\mathbf{u}_t|}_{L^2}^2\right)\\
		&\leq C(M)(\norm{\rho_t}_{L^6}^2\norm{\mathbf{u}_t}_{L^6}^2+\norm{\mathbf{u}_{tt}}_{L^2}^2+\norm{\rho_t}_{L^6}^2\norm{\mathbf{u}}_{L^\infty}^2\norm{\nabla\mathbf{u}}_{L^6}^2+\norm{\mathbf{u}_t}_{L^6}^2\norm{\nabla\mathbf{u}}_{L^6}^2\\
		&\quad+ \norm{\mathbf{u}}_{L^\infty}^2\norm{\nabla\mathbf{u}_t}_{L^2}^2+\norm{\rho_t}_{L^6}^2\norm{\nabla\rho}_{L^6}^2+\norm{\nabla\rho_t}_{L^2}^2+\norm{\rho_t}_{L^6}^2\norm{\nabla\rho}_{L^6}^2\norm{\nabla\mathbf{u}}_{L^6}^2\\
		&\quad+\norm{\nabla\rho_t}_{L^2}^2\norm{\nabla\mathbf{u}}_{L^\infty}^2+\norm{\nabla\rho}_{L^6}^2\norm{\nabla\mathbf{u}_t}_{L^3}^2)\\
		&\leq C(M,T)(1+\norm{\nabla\dot{\mathbf{u}}}_{L^2}^2+\norm{\mathbf{u}_{tt}}_{L^2}^2+\norm{\nabla\mathbf{u}_t}_{L^3}^2)\\
		&\leq C(M,T)(1+\norm{\nabla\dot{\mathbf{u}}}_{L^2}^2+\norm{\mathbf{u}_{tt}}_{L^2}^2)+C(M,T)\norm{\nabla \mathbf{u}_t}_{L^2}^{\frac{6-N}{3}}\norm{\nabla \mathbf{u}_t}_{H^1}^{\frac{N}{3}}\\
		&\leq \frac{1}{2}\norm{\nabla\mathbf{u}_t}_{H^1}^2+C(M,T)(1+\norm{\nabla\dot{\mathbf{u}}}_{L^2}^2+\norm{\mathbf{u}_{tt}}_{L^2}^2), 
	\end{align*}
	which shows that
	\begin{align}\label{hh19}
		\norm{\mathbf{u}_t}_{H^2}^2\leq C(M,T)(1+\norm{\nabla\dot{\mathbf{u}}}_{L^2}^2+\norm{\mathbf{u}_{tt}}_{L^2}^2).
	\end{align}
	Therefore, we obtain from \eqref{hh20}--\eqref{hh19} that
	\begin{align*}
		&\quad\frac{1}{2}\frac{d}{dt}\int_\Omega\Big(|\nabla\mathbf{u}_t|^2+(\alpha-1)(\div\mathbf{u}_t)^2\Big)dx+\int_\Omega\rho^{1-\alpha}|\mathbf{u}_{tt}|^2dx\\
		&\leq \varepsilon\int_\Omega |\mathbf{u}_{tt}|^2dx+C(\varepsilon,M,T)(1+\norm{\nabla\dot{\mathbf{u}}}_{L^2}^2+\norm{\nabla\mathbf{u}_t}_{L^3}^2)\\
		&\leq \varepsilon\int_\Omega |\mathbf{u}_{tt}|^2dx+C(\varepsilon,M,T)(1+\norm{\nabla\dot{\mathbf{u}}}_{L^2}^2+\norm{\nabla \mathbf{u}_t}_{L^2}^{\frac{6-N}{3}}\norm{\nabla \mathbf{u}_t}_{H^1}^{\frac{N}{3}})\\
		&\leq 2\varepsilon\int_\Omega |\mathbf{u}_{tt}|^2dx+C(\varepsilon,M,T)+C(\varepsilon,M,T)\norm{\nabla\dot{\mathbf{u}}}_{L^2}^2.
	\end{align*}
	Choosing \(\varepsilon>0\) sufficiently small so that \(2\varepsilon<\frac{M^{\alpha-1}}{2}\) holds, we obtain
	\begin{align*}
		\frac{1}{2}\frac{d}{dt}\int_\Omega\Big(|\nabla\mathbf{u}_t|^2+(\alpha-1)(\div\mathbf{u}_t)^2\Big)dx+\frac{M^{\alpha-1}}{2}\int_\Omega |\mathbf{u}_{tt}|^2dx\leq C(M,T)+C(M,T)\norm{\nabla\dot{\mathbf{u}}}_{L^2}^2.
	\end{align*}
	Integrating with respect to time, we obtain \eqref{h4} from \eqref{h2}, \eqref{hh17}, \eqref{hh21} and \eqref{hh19}.
\end{proof}

\begin{prop}\label{Prop H5}
	Assume that \eqref{hh2} holds. Then there exists a constant \(C(M,T)>0\) depending on \(M\) and \(T\) such that
	\begin{align}\label{h5}
		\sup_{0\leq t\leq T}\norm{\rho}_{H^3}+\int_0^T\norm{\mathbf{u}}_{H^4}^2dt\leq C(M,T).
	\end{align}
\end{prop}
\begin{proof}
	Applying \(\partial_i\partial_j\partial_k\) to equation $\eqref{0}_1$, multiplying the resulting equation by \(\partial_{ijk}\rho\), summing over \(i,j,k\), and integrating the resulting equation over \(\Omega\), we use integration by parts, \eqref{h3}, \eqref{h4} and \eqref{hh16} to obtain
	\begin{align}\label{hh23}
		\begin{split}
			&\quad\frac{1}{2}\frac{d}{dt}\int_\Omega |\nabla^3\rho|^2dx\leq C\int_\Omega |\nabla^3\rho|\left(\rho|\nabla^4\mathbf{u}|+|\nabla\rho||\nabla^3\mathbf{u}|+|\nabla^2\rho||\nabla^2\mathbf{u}|+|\nabla^3\rho||\nabla\mathbf{u}|\right)dx\\
			&\leq C(M)(\norm{\nabla^3\rho}_{L^2}^2+\norm{\nabla^4\mathbf{u}}_{L^2}^2+\norm{\nabla\rho}_{L^\infty}^2\norm{\nabla^3\mathbf{u}}_{L^2}^2+\norm{\nabla^2\rho}_{L^6}^2\norm{\nabla^2\mathbf{u}}_{L^6}^2+\norm{\nabla^3\rho}_{L^2}^2\norm{\nabla\mathbf{u}}_{L^\infty}^2)\\
			&\leq C(M,T)+C(M,T)\norm{\nabla^3\rho}_{L^2}^2+C(M,T)\norm{\nabla^4\mathbf{u}}_{L^2}^2.
		\end{split}
	\end{align}
	From the elliptic system \eqref{hh4}, \eqref{h3} and \eqref{h4}, we obtain
	\begin{align}\label{hh22}
		\begin{split}
			&\norm{\mathbf{u}}_{H^4}^2\leq C \norm{\frac{1}{\rho^\alpha}(\rho \mathbf{u}_t+\rho\mathbf{u}\cdot\nabla\mathbf{u}+\nabla P-\nabla\mu\cdot\nabla\mathbf{u}-\nabla\lambda\div\mathbf{u})}_{H^2}^2\\
			&\leq C(M,T)+ C(M)(\norm{|\nabla^2\rho||\mathbf{u}_t|}_{L^2}^2+\norm{|\nabla\rho||\nabla\mathbf{u}_t|}_{L^2}^2+\norm{\nabla^2\mathbf{u}_t}_{L^2}^2+\norm{|\nabla\rho|^2|\mathbf{u}_t|}_{L^2}^2)\\
			&\quad+C(M)(\norm{|\nabla^2\rho||\mathbf{u}||\nabla\mathbf{u}|}_{L^2}^2+\norm{|\nabla\rho||\nabla\mathbf{u}|^2}_{L^2}^2+\norm{|\mathbf{u}||\nabla^3\mathbf{u}|}_{L^2}^2+\norm{|\nabla\rho|^2|\mathbf{u}||\nabla\mathbf{u}|}_{L^2}^2\\
			&\quad+\norm{|\nabla\rho||\mathbf{u}||\nabla^2 \mathbf{u}|}_{L^2}^2+\norm{|\nabla^2\mathbf{u}||\nabla \mathbf{u}|}_{L^2}^2)+C(M)(\norm{\nabla^3\rho}_{L^2}^2+\norm{|\nabla^2\rho||\nabla\rho|}_{L^2}^2+\norm{|\nabla\rho|^3}_{L^2}^2)\\
			&\quad +C(M)(\norm{|\nabla^3\mathbf{u}||\nabla\rho|}_{L^2}^2+\norm{|\nabla^2\mathbf{u}||\nabla^2\rho|}_{L^2}^2+\norm{|\nabla^2\mathbf{u}||\nabla\rho|^2}_{L^2}^2+\norm{|\nabla\rho||\nabla^2\rho||\nabla\mathbf{u}|}_{L^2}^2\\
			&\quad+\norm{|\nabla^3\rho||\nabla\mathbf{u}|}_{L^2}^2+\norm{|\nabla\mathbf{u}||\nabla\rho|^3}_{L^2}^2)\\
			&\leq C(M,T)(1+\norm{\nabla^3\rho}_{L^2}^2+\norm{\mathbf{u}_t}_{H^2}^2)\\
			&\leq C(M,T)(1+\norm{\nabla^3\rho}_{L^2}^2+\norm{\mathbf{u}_{tt}}_{L^2}^2),
		\end{split}
	\end{align}
	where the last inequality follows from \eqref{hh19}. Inserting \eqref{hh22} into \eqref{hh23}, we get
	\begin{align}\label{hh25}
		\frac{1}{2}\frac{d}{dt}\int_\Omega |\nabla^3\rho|^2dx\leq C(M,T)(1+\norm{\nabla^3\rho}_{L^2}^2+\norm{\mathbf{u}_{tt}}_{L^2}^2),
	\end{align}
	By Gronwall's inequality, \eqref{h4} and \eqref{hh22} imply \eqref{h5}.
\end{proof}

\subsection{The proof of Theorem \ref{Thm1}}
	\begin{prop}\label{Prop H6}
	Under the assumptions of Theorem \ref{Thm1}, there exists a constant \(C(T)>0\) such that
	\begin{align}\label{hh24}
		\begin{split}
			&\sup_{0\leq t\leq T}\left(\norm{\rho}_{H^3}+\norm{\rho_t}_{H^2}+\norm{\rho_{tt}}_{L^2}\right)+\int_0^T\norm{\rho_{tt}}_{H^1}^2dt\leq C(T);\\
			&\sup_{0\leq t\leq T}\left(\norm{\mathbf{u}}_{H^3}+\norm{\mathbf{u}_t}_{H^1}\right)+\int_0^T(\norm{\mathbf{u}}_{H^4}^2+\norm{\mathbf{u}_t}_{H^2}^2+\norm{\mathbf{u}_{tt}}_{L^2}^2)dt\leq C(T).
		\end{split}
	\end{align}
\end{prop}
\begin{proof}
	By  Proposition  \ref{Prop C 2d RT}, \ref{Prop C 3d RT}, \ref{Prop C 2d VT} and \ref{Prop C 3d VT}, we obtain that there exists a constant \(C(T)>0\) such that 
	\begin{align}\label{hhh1}
	\sup_{0\leq t\leq T}\norm{\rho}_{L^\infty}+\sup_{0\leq t\leq T}\norm{\rho^{-1}}_{L^\infty}\leq C(T).
	\end{align}
	
	For \(N=2\), set \(l = k\), where \(k\) is as defined in \eqref{C11}. Then, by Proposition \ref{Prop C 2d k}, we obtain 
	\begin{align}\label{hhh2}
		\sup_{0\le t\le T}\norm{\rho^{\frac{1}{2l}}\mathbf{u}}_{L^{2l}}+\sup_{0\le t\le T}\norm{\nabla\rho^{\alpha-1+\frac{1}{2l}}}_{L^{2l}}+\int_0^T\norm{\nabla\rho^{\alpha-1+\frac{\gamma-\alpha+1}{2l}}}_{L^{2l}}^{2l}dt\leq C(T).
	\end{align}
	Therefore, \eqref{hhh1}, \eqref{C u^2}, \eqref{C nablarho^2}, \eqref{hhh2} together imply
		\begin{align}\label{hhh3}
			\begin{split}
				\sup_{0\leq t\leq T}\Big(\norm{\mathbf{u}}_{L^{2l}}&+\norm{\rho}_{L^\infty}+\norm{\rho^{-1}}_{L^\infty}+\norm{\nabla\rho}_{L^2}+\norm{\nabla\rho}_{L^{2l}}\Big)\\
				&+\int_0^T\left(\norm{\nabla\rho}_{L^2}^2+\norm{\nabla\rho}_{L^{2l}}^{2l}+\norm{\nabla\mathbf{u}}_{L^2}^2\right)dt\leq C(T).
			\end{split}
		\end{align}	
		
	For $N=3$, choose any \(l\) satisfying \(1.5<l<\min\{3,k\}\), where \(k\) is as defined in \eqref{C21}. Therefore, \eqref{hhh1}, \eqref{C u^2}, \eqref{C nablarho^2}, \eqref{C 3d k}, \eqref{C 3d l} together imply \eqref{hhh3}.

	 Therefore, by Proposition \ref{Prop H1}-Proposition \ref{Prop H5}, we obtain \eqref{hh24} except for the estimates of \(\rho_{t}\) and \(\rho_{tt}\). Indeed, from equation $\eqref{0}_1$, one has $\sup_{0\leq t\leq T}\norm{\rho_t}_{H^2}\leq C(T).$ Differentiating equation $\eqref{0}_1$ with respect to $t$, we obtain 
	\begin{align*}
		\rho_{tt}+\rho_t\div\mathbf{u}+\rho\div\mathbf{u}_t+\nabla\rho_t\cdot\mathbf{u}+\nabla\rho\cdot\mathbf{u}_t=0,
	\end{align*}
	which shows that
	\begin{align*}
		\sup_{0\leq t\leq T}\norm{\rho_{tt}}_{L^2}+\int_0^T\norm{\rho_{tt}}_{H^1}^2dt\leq C(T).
	\end{align*}
	We have completed the proof.
\end{proof}

The pair \((\rho,\mathbf{u})\) in Proposition \ref{Prop H6} can be shown to belong to \(C([0,T];H^3)\). Indeed, \(\rho\in L^\infty(0,T;H^3)\) together with \(\rho_t\in L^\infty(0,T;H^2)\) implies \(\rho\in C([0,T];H^2)\).  The \(L^2\)-continuity of the third-order derivatives of \(\rho\) follows from standard arguments on weak convergence and norm convergence.  Moreover, since \(\mathbf{u}\in L^2(0,T;H^4)\) and \(\mathbf{u}_t\in L^2(0,T;H^2)\), standard interpolation gives \(\mathbf{u}\in C([0,T];H^3)\).

We now prove that \(T^*=\infty\).  If not, then \(T^*<\infty\), and the constant \(C(T)\) in Proposition \ref{Prop H6} can be replaced by \(C(T^*)\), where \(C(T^*)\) depends on \(T^*\) and not on \(T<T^*\).  Therefore, we can define
\begin{align*}
	(\rho(T^*),\mathbf{u}(T^*))=\lim\limits_{t\to T^*}(\rho(t),\mathbf{u}(t))\in H^3.
\end{align*}
The pair \((\rho(T^*),\mathbf{u}(T^*))\) can serve as new initial data, and at this time \(\rho(T^*)\) has a positive lower bound.  By the local existence Lemma \ref{Lem local existence}, there exists a small time \(t_0>0\) such that the new initial-boundary-value problem has a solution on \([T^*,T^*+t_0]\).  This contradicts the maximality of \(T^*\). We have completed the proof of Theorem \ref{Thm1}.

\subsection{The proof of Theorem \ref{Thm2}}
Note that global existence of classical solutions has been established by Theorem \ref{Thm1}  for \(N=2\) with \(\alpha>0.54, \gamma>1\) and for \(N=3\) with $(\alpha,\gamma)\in\Big\{(\alpha,\gamma)|0.689<\alpha<1,1<\gamma<3\alpha-1\Big\}\cap\Big\{(\alpha,\gamma)|0.686<\alpha<1,1<\gamma<6\alpha-3+\frac{3-5\alpha}{2n_3(\alpha)}\Big\}$.  We now sketch the proof for the remaining case \(N=3\) with $(\alpha,\gamma)\in\Big\{(\alpha,\gamma)|0.689<\alpha<1,1<\gamma<3\alpha-1\Big\}-\Big\{(\alpha,\gamma)|0.686<\alpha<1,1<\gamma<6\alpha-3+\frac{3-5\alpha}{2n_3(\alpha)}\Big\}.$ In this case, Propositions \ref{Prop C uniRT} and \ref{Prop ex 3dVT} yield positive upper and lower bounds for the density, while Proposition \ref{Prop C 4} provides the \(L^{4}\) bound on the velocity field and the derivative of the density.  Consequently, combining this with \eqref{C u^2} and \eqref{C nablarho^2}, we know that \eqref{hh2} holds for \(l=2\) with \(M\) replaced by \(C(T)\).  Iterating Propositions \ref{Prop H1}-\ref{Prop H5} then gives the higher-order estimates, and repeating the arguments of the previous section, we obtain global existence of solutions.

Now, we assume that \((\rho,\mathbf{u})\) is the global classical solution constructed in Theorem \ref{Thm2}. We obtain from \eqref{C u^2}, \eqref{C nablarho^2}, \eqref{C 2d l4} and \eqref{C uniRT} that, 
\begin{align}\label{aa11}
	\begin{split}
		\sup_{0\leq t<\infty}&\left(\norm{\rho}_{L^\infty}+\norm{\nabla\rho^{\alpha-\frac{1}{2}}}_{L^2}+\norm{\rho^{\alpha-\frac{7}{4}}\nabla\rho}_{L^{4}}\right)\\
		&+\int_0^\infty\norm{\nabla\rho^{\frac{\gamma+\alpha-1}{2}}}_{L^{2}}^2+\norm{\rho^{\frac{\gamma+3\alpha-7}{4}}\nabla\rho}_{L^{4}}^{4}+\norm{\rho^{\frac{\alpha}{2}}\nabla \mathbf{u}}_{L^2}^2dt\leq C.
	\end{split}
\end{align}

We now show that the lower bound of the density is in fact independent of time.
	\begin{prop}\label{Prop rho epsilon}
	There exist constants \(\varepsilon>0\) depending only on the initial data such that
	\begin{align}\label{aa1}
		\rho\ge \varepsilon, \quad \forall (x,t)\in \Omega\times[0,\infty).
	\end{align}
	Furthermore, it holds that 
	\begin{align}\label{aa10}
		\lim\limits_{t\to\infty}\norm{\rho(t)-\overline{\rho_0}}_{C(\overline{\Omega})}=0.
	\end{align}
\end{prop}
\begin{proof}
	We choose \(b>1\) sufficiently large so that
	\begin{align}\label{aa2}
		\begin{split}
			\sup_{0\leq t<\infty}&\left(\norm{\rho}_{L^\infty}+\norm{\nabla\rho^{b-1}}_{L^2}+\norm{\nabla\rho^b}_{L^{4}}\right)\\
			&+\int_0^\infty\norm{\nabla\rho^{b-1}}_{L^2}^2+\norm{\nabla\rho^{b}}_{L^{4}}^{4}dt\leq C.
		\end{split}
	\end{align}
	By the embedding inequality we have
	\begin{align}\label{aa3}
		\norm{\rho^b-\overline{\rho^b}}_{C(\overline{\Omega})}\leq C\norm{\rho^b-\overline{\rho^b}}_{L^4}^{\frac{4-N}{4}}\norm{\nabla\rho^b}_{L^{4}}^{\frac{N}{4}}\leq C\norm{\rho^b-\overline{\rho^b}}_{L^{4}}^{\frac{4-N}{4}}.
	\end{align}
	Next, we claim that
	\begin{align}\label{aa4}
		g(t)=\norm{(\rho^b-\overline{\rho^b})(t)}_{L^4}^4\rightarrow 0, \text{ as }t\rightarrow\infty.
	\end{align}
	To prove \eqref{aa4}, we first show that \(\int_{0}^{\infty}|g'(t)|dt\le C\). A direct computation gives
	\begin{align*}
		g'(t)&=4b\left\langle (\rho^b-\overline{\rho^b})^3\rho^{b-1},\rho_t\right\rangle-4(\overline{\rho^b})_t\int_\Omega(\rho^b-\overline{\rho^b})^3dx\\
		&=-4b\left\langle \nabla\left((\rho^b-\overline{\rho^b})^3\rho^{b-1}\right),\sqrt{\rho}\sqrt{\rho}\mathbf{u}\right\rangle-4(\overline{\rho^b})_t\int_\Omega(\rho^b-\overline{\rho^b})^3dx:=I_1+I_2.
	\end{align*}
	Using \eqref{aa2}, we obtain
	\begin{align}\label{aa5}
		\begin{split}
			&\int_0^\infty|I_1|dt\leq C\int_0^\infty\int_\Omega\Big((\rho^b-\overline{\rho^b})^2|\nabla\rho^b|\rho^{b-\frac{1}{2}}\sqrt{\rho}|\mathbf{u}|+\left|\rho^b-\overline{\rho^b}\right|^3|\nabla\rho^{b-1}|\sqrt{\rho}\sqrt{\rho}|\mathbf{u}|\Big)dxdt\\
			&\leq C\int_0^\infty\left(\norm{\sqrt{\rho}\mathbf{u}}_{L^2}\norm{\nabla\rho^b}_{L^2}\norm{\rho^b-\overline{\rho^b}}_{L^\infty}^2+\norm{\sqrt{\rho}\mathbf{u}}_{L^2}\norm{\nabla\rho^{b-1}}_{L^2}\norm{\rho^b-\overline{\rho^b}}_{L^\infty}^2\right)dt\\
			&\leq C\int_0^\infty\left(\norm{\nabla\rho^{b-1}}_{L^2}^2+\norm{\nabla \rho^b}_{L^{4}}^{4}\right)dt\leq C.
		\end{split}
	\end{align}
	Note that
	\begin{align*}
		\sup_{0\leq t<\infty}\left|\frac{d}{dt}\overline{\rho^b}\right|&=b\sup_{0\leq t<\infty}\left|\langle \rho^{b-1},\rho_t\rangle\right|=b\sup_{0\leq t<\infty}\left|\int_\Omega \nabla\rho^{b-1}\sqrt{\rho}\cdot\sqrt{\rho}\mathbf{u}dx\right|\\
		&\leq C\norm{\sqrt{\rho}\mathbf{u}}_{L^\infty L^2}\norm{\nabla\rho^{b-1}}_{L^\infty L^2}\leq C.
	\end{align*}
	Therefore,
	\begin{align}\label{aa6}
		\int_0^\infty|I_2|dt\leq C\int_0^\infty\norm{\rho^b-\overline{\rho^b}}^2_{L^2}dt\leq C\int_0^\infty \norm{\nabla\rho^b}_{L^2}^2dt\leq C.
	\end{align}
	Finally, we obtain
	\begin{align}\label{aa7}
		\int_0^\infty g(t)dt\leq C\norm{\rho^b-\overline{\rho^b}}_{L^\infty L^\infty}^2\int_0^\infty\|\rho^b-\overline{\rho^b}\|_{L^2}^2dt\leq C\int_0^\infty\norm{\nabla \rho^b}_{L^2}^2dt\leq C.
	\end{align}
	The estimates \eqref{aa5}–\eqref{aa7} imply \eqref{aa4}. Therefore, we obtain from \eqref{aa3} that
	\begin{align}\label{aa8}
		\norm{\rho^b(t)-\overline{\rho^b}(t)}_{C(\overline{\Omega})}\to 0,\text{ as }t\rightarrow\infty.
	\end{align}
	Then, by Jensen's inequality we obtain
	\begin{align*}
		\overline{\rho^b}\ge \overline{\rho}^b=\overline{\rho_0}^b>0.
	\end{align*}
	Together with \eqref{aa8}, this shows that there exist \(\varepsilon_0\) and \(T^*\) such that
	\begin{align}\label{aa9}
		\rho(x,t)\ge\varepsilon_0,\quad \forall (x,t)\in\Omega\times[T^*,\infty).
	\end{align}
	We note that \(\rho\) has a positive lower bound on \(\Omega\times[0,T^*]\), which together with \eqref{aa9} proves \eqref{aa1}. 
	
	Below we prove \eqref{aa10}. We first prove that \(\overline{\rho^b}(t)\) converges to some positive constant as \(t\to\infty\). Indeed, from \eqref{aa1}, we have
	\begin{align}
		\begin{split}
			\left|\frac{d}{dt}\overline{\rho^b}(t)\right|&=\left|\int_\Omega\partial_t \rho^b dx\right|\leq C\left|\int_\Omega(\rho^b-\overline{\rho^b})\div\mathbf{u}dx\right|\\
			&\leq C\int_\Omega \Big(|\nabla\mathbf{u}|^2+|\rho^b-\overline{\rho^b}|^2\Big)dx\\
			&\leq \frac{C}{\varepsilon^\alpha}\int_\Omega\rho^\alpha|\nabla\mathbf{u}|^2dx+C\int_\Omega|\nabla\rho^b|^2dx.
		\end{split}
	\end{align}
	Integrating with respect to $t$ and using \eqref{aa11} and \eqref{aa2} gives
	\begin{align*}
		\int_0^\infty \left|\frac{d}{dt}\overline{\rho^b}(t)\right|dt<\infty.
	\end{align*}
	This shows that \(\overline{\rho^b}(t)\) converges to some positive constant \(a_0\), which together with \eqref{aa8} implies
	\begin{align}
		\norm{\rho^b(t)-a_0}_{C(\overline{\Omega})}\rightarrow 0,\text{ as }t\rightarrow\infty.
	\end{align}
	Using \(b>1\), we obtain
	\begin{align}
		\norm{\rho(t)-a_0^{\frac{1}{b}}}_{C(\overline{\Omega})}\rightarrow 0,\text{ as }t\rightarrow\infty.
	\end{align}
	It is easily checked, by applying the dominated convergence theorem, that \(a_0^{1/b}=\overline{\rho_0}\).  Thus, we have completed the proof.
\end{proof}

Therefore, combining \eqref{C 2d 4}, \eqref{aa11}, and \eqref{aa1}, we obtain
	\begin{align}\label{aa13}
	\begin{split}
		\sup_{0\leq t<\infty}\Big(\norm{\mathbf{u}}_{L^4}&+\norm{\rho}_{L^\infty}+\norm{\rho^{-1}}_{L^\infty}+\norm{\nabla\rho}_{L^2}+\norm{\nabla\rho}_{L^4}\Big)\\
		&+\int_0^\infty\left(\norm{\nabla\rho}_{L^2}^2+\norm{\nabla\rho}_{L^{4}}^{4}+\norm{\nabla\mathbf{u}}_{L^2}^2\right)dt\leq C.
	\end{split}
\end{align}

	\begin{prop}
	It holds that
	\begin{align}\label{hh26}
		\lim\limits_{t\to\infty}\norm{\nabla\mathbf{u}(t)}_{L^2}=0.
	\end{align}
\end{prop}
\begin{proof}
	We set \(g(t)=\int_\Omega(\mu(\rho)|\nabla\mathbf{u}|^2+\lambda(\rho)(\div\mathbf{u})^2)dx\). Using \eqref{aa13}, one has
	\begin{align}\label{aa12}
		\int_0^\infty g(t)dt\leq C\int_0^\infty\int_\Omega |\nabla \mathbf{u}|^2dxdt\leq C.
	\end{align}
	We obtain from \eqref{hh6} that
	\begin{align*}
		g'(t)\leq C (\norm{\nabla\rho(t)}_{L^2}^2+\norm{\nabla\mathbf{u}(t)}_{L^2}^2).
	\end{align*}
	Letting \(t\) be sufficiently large and taking \(s\in[t-1,t]\), we have
	\begin{align*}
		g(t)=g(s)+\int_s^t g'(\tau)d\tau\leq g(s)+C\int_{t-1}^t (\norm{\nabla\rho}_{L^2}^2+\norm{\nabla\mathbf{u}}_{L^2}^2)d\tau.
	\end{align*}
	Integrating with respect to \(s\) over \([t-1,t]\) gives
	\begin{align}\label{hh28}
		g(t)\leq \int_{t-1}^t g(\tau)d\tau+C\int_{t-1}^t (\norm{\nabla\rho}_{L^2}^2+\norm{\nabla\mathbf{u}}_{L^2}^2)d\tau.
	\end{align}
	Therefore, we know from \eqref{aa13} and \eqref{aa12} that
	\begin{align*}
		g(t)\to 0, \text{ as }t\to\infty,
	\end{align*}
	which, combined with \eqref{C 2d unialpha}, \eqref{C 3d unialpha} and \eqref{aa1}, implies \eqref{hh26}.
\end{proof}

\begin{prop}\label{Prop lb}
	It holds that
	\begin{align}\label{hh27}
		\lim\limits_{t\to\infty}(\norm{\nabla\rho(t)}_{L^2}+\norm{\nabla^2\mathbf{u}(t)}_{L^2})=0.
	\end{align}
\end{prop}
\begin{proof}
	We first prove that
	\begin{align}\label{hh29}
		\int_\Omega\rho|\dot{\mathbf{u}}|^2dx(t)\to 0, \text{ as }  t\to\infty. 
	\end{align}
	Similarly to \eqref{hh28}, we obtain from \eqref{hh14} that, for sufficiently large \(t\),
	\begin{align}
		\begin{split}
			\int_\Omega\rho|\dot{\mathbf{u}}|^2dx(t)\leq& \int_{t-1}^{t}\int_\Omega\rho|\dot{\mathbf{u}}|^2dxd\tau\\
			&+C\int_{t-1}^t(\norm{\sqrt{\rho}\dot{\mathbf{u}}}_{L^2}^2+\norm{\nabla\rho}_{L^2}^2+\norm{\nabla\rho}_{L^4}^4+\norm{\nabla\mathbf{u}}_{L^2}^2+\norm{\sqrt{\rho}\dot{\mathbf{u}}}_{L^2}^4)d\tau.
		\end{split}
	\end{align}
	Letting \(t\to\infty\), we get \eqref{hh29} from \eqref{h1}, \eqref{h2} and \eqref{aa13}. Next, we prove
	\begin{align}\label{hh30}
		\int_\Omega|\nabla\rho|^2dx(t)\to 0, \text{ as }  t\to\infty. 
	\end{align}
	Using \eqref{hh5}, one has
	\begin{align*}
		\frac{d}{dt}\int_\Omega|\nabla\rho|^2dx&\leq C\int_\Omega(|\nabla\mathbf{u}||\nabla\rho|^2+\rho|\nabla^2\mathbf{u}||\nabla\rho|)dx\\
		&\leq C\int_\Omega(|\nabla\mathbf{u}|^2+|\nabla\rho|^4+|\nabla^2\mathbf{u}|^2+|\nabla\rho|^2)dx\\
		&\leq C\int_\Omega (|\nabla \mathbf{u}|^2+|\nabla\rho|^4+|\nabla\rho|^2+\rho|\dot{\mathbf{u}}|^2)dx.
	\end{align*}
	Therefore, similarly to \eqref{hh28}, we get
	\begin{align}
		\begin{split}
			\int_\Omega|\nabla\rho|^2dx(t)\leq \int_{t-1}^t\int_\Omega|\nabla\rho|^2dxd\tau+C\int_{t-1}^t\int_\Omega (|\nabla \mathbf{u}|^2+|\nabla\rho|^4+|\nabla\rho|^2+\rho|\dot{\mathbf{u}}|^2)dxd\tau.
		\end{split}
	\end{align}
	Letting \(t\to\infty\), we get \eqref{hh30} from \eqref{h1} and \eqref{aa13}.  Finally, from \eqref{hh5} one obtains
	\begin{align}
		\norm{\nabla^2\mathbf{u}(t)}_{L^2}\leq C(\norm{\sqrt{\rho}\dot{\mathbf{u}}}_{L^2}+\norm{\nabla\rho(t)}_{L^2}+\norm{\nabla\mathbf{u}(t)}_{L^2})\to 0,\text{ as }t\to \infty.
	\end{align}
	The proof is complete.
\end{proof}

\section{Global classical solutions away from vacuum including  viscous shallow water equations \((\alpha=1)\)}
Throughout this section, we always assume \(\alpha = 1\). Suppose that \((\rho, \mathbf{u})\) is the unique local classical solution to the initial‑boundary‑value problem \eqref{0}–\eqref{0-2} with initial data satisfying \eqref{C initial data}, defined on \(\Omega \times [0, T]\) for some $T>0$. The existence of such a solution is guaranteed by Lemma \ref{Lem local existence}.

\subsection{Upper bound for the density}
Noting that \(n_2(1) = +\infty\), we can obtain the following proposition by following exactly the same arguments as in Propositions \ref{Prop 3.11}–\ref{Prop C uniRT}. For brevity, we omit the proof here.
\begin{prop}
    Assume that \eqref{SV-2gamma} holds. Then there exists a constant \(C>0\), independent of \(T\), such that
		\begin{align}\label{SV 2d k}
			\sup_{0\leq t\leq T}\int_0^R\rho u^{4}rdr+\int_0^T\int_0^R\Big(\frac{\rho u^{4}}{r}
			+\rho (\partial_ru)^2 u^{2}r\Big)drdt\leq C
		\end{align}
		as well as
		\begin{align}\label{SV 2d l}
			\sup_{0\leq t\leq T}\int_0^R |\partial_r\rho^{\frac{1}{4}}|^{4}rdr+\int_0^T\int_0^R |\partial_r\rho^{\frac{\gamma}{4}}|^{4}rdrdt\leq C
		\end{align}
		and finally,
		\begin{align}\label{SV 2d RT}
			\sup_{0\leq t\leq T}\norm{\rho}_{L^\infty(0,R)}\leq C.
		\end{align}
\end{prop}

\subsection{Positive lower bound for the density}
In this subsection we establish a positive lower bound for the density, which is the most crucial part in proving the global existence of solutions for the case \(\alpha = 1\).

\begin{lema}
  Assume that \eqref{SV-2gamma} holds. Then there exists a constant \(C(T)>0\), such that
\begin{align}\label{SV-u 2,inf}
    \norm{\sqrt{\rho}\mathbf{u}}_{L^2(0,T;L^\infty(\Omega))}\leq C(T).
\end{align}
\end{lema}
\begin{proof}
Applying Lemma \ref{Lem 2D L inf}, we obtain    
\begin{align*}
			\norm{\sqrt{\rho}\mathbf{u}}_{L^\infty(\Omega)}&\leq C\norm{\sqrt{\rho}\nabla\mathbf{u}}_{L^2(\Omega)}+C\norm{\nabla\sqrt{\rho}\mathbf{u}}_{L^2(\Omega)}\\
			&\leq C\norm{\sqrt{\rho}\nabla\mathbf{u}}_{L^2(\Omega)}+C\norm{\nabla\rho^{\frac{1}{4}}}_{L^4(\Omega)}\norm{\rho^{\frac{1}{4}}\mathbf{u}}_{L^4(\Omega)},
\end{align*}
which, together with \eqref{C u^2}, \eqref{SV 2d k}, and \eqref{SV 2d l}, yields \eqref{SV-u 2,inf}.
\end{proof}

From now until the end of this subsection, we return to calculations in Lagrangian coordinates. Define the effective velocity  
\begin{align}\label{def w}
w = u + r\partial_y \rho, \quad \text{on } [0,1]\times[0,T].
\end{align}
The following proposition provides an \(L^\infty\) estimate for the effective velocity.
\begin{prop}
    Assume that \eqref{SV-2gamma} holds. Then there exists a constant \(C(T)>0\), such that
    \begin{align}\label{w inf}
        \sup_{0\leq \tau\leq T}\norm{w}_{L^\infty((0,1))}\leq C(T).
    \end{align}
\end{prop}
\begin{proof}
    Let $n\in\mathbb{R}$ be sufficiently large. Multiplying both sides of the B-D entropy equation \(\partial_\tau w + r\partial_y\rho^{\gamma} = 0\) by \(w^{2n-1}\) and integrating the resulting equation over \([0,1]\) yields
    	\begin{align*}
		\frac{1}{2n}\frac{d}{d\tau}\int_0^1 w^{2n}dy+\int_0^1\gamma\rho^{\gamma-1} r\partial_y \rho w^{2n-1}dy=0,
	\end{align*}
which, combined with the definition of the effective velocity \eqref{def w}, gives
	\begin{align*}
			&\quad\frac{1}{2n}\frac{d}{d\tau}\int_0^1 w^{2n}dy+\int_0^1\gamma\rho^{\gamma-1}w^{2n}dy\\
            &=\int_0^1\gamma\rho^{\gamma-1}uw^{2n-1}dy\\
			&\leq \gamma\norm{w}_{L^{2n}((0,1))}^{2n-1}\norm{\sqrt{\rho}u}_{L^{2n}((0,1))}\norm{\rho^{\gamma-\frac{3}{2}}}_{L^\infty((0,1))}\\
			&\leq C\left(\int_0^1 w^{2n}dy+1\right)\norm{\sqrt{\rho}u}_{L^\infty((0,1))},
	\end{align*}
where in the last inequality we have used \eqref{SV-2gamma} and \eqref{SV 2d RT}. Therefore, applying Gronwall’s inequality we obtain for any $\tau\in (0,T)$,
	\begin{align*}
			\int_0^1 w^{2n}(\tau)dy\leq \exp\left\{2nC\int_0^T\norm{\sqrt{\rho}u}_{L^\infty((0,1))}d\tau\right\}\left(\int_0^1 w_0^{2n}dy+2nC\int_0^T\norm{\sqrt{\rho}u}_{L^\infty((0,1))}d\tau\right).
	\end{align*}
    Using \eqref{SV-u 2,inf}, we obtain
    \begin{align*}
        \norm{w(\tau)}_{L^{2n}((0,1))}\leq C(T)\norm{w_0}_{L^{2n}((0,1))}+C(T)\leq C(T)\norm{w_0}_{L^\infty((0,1))}+C(T).
    \end{align*}
    Letting \(n \to \infty\), we obtain \eqref{w inf}.
\end{proof}

Next, we use the \(L^\infty\) estimate of the effective velocity to derive an \(L^\infty\) estimate for the velocity. 
\begin{prop}
     Assume that \eqref{SV-2gamma} holds. Then there exists a constant \(C(T)>0\), such that
    \begin{align}\label{u inf}
        \sup_{0\leq \tau\leq T}\norm{u}_{L^\infty((0,1))}\leq C(T).
    \end{align}
\end{prop}
\begin{proof}
    Let $n\in\mathbb{R}$ be sufficiently large. Multiplying $\eqref{Equ 3}_2$ by $ru^{2n-1}$, integrating the resulting equation over \((0,1)\), and using integration by parts together with the boundary condition \eqref{C1} and the fact that \(\alpha=1\), we obtain 
   	\begin{align*}
		&\quad\frac{1}{2n}\frac{d}{d\tau}\int_0^1 u^{2n}dy+\int_0^1\Big(\frac{u^{2n}}{r^2}+(2n-1)\rho^2(\partial_y u)^2 u^{2n-2}r^2\Big)dy\\
        &\leq \int_0^1|\partial_y\rho^\gamma| r |u|^{2n-1}dy\\
		&\leq \int_0^1\gamma\rho^{\gamma-1}(|w|+|u|)|u|^{2n-1}dy\\
		&\leq C\int_0^1u^{2n}dy+C\int_0^1w^{2n}dy\\
		&\leq  C\int_0^1u^{2n}dy+CC(T)^{2n},
	\end{align*}
where we have used \eqref{SV 2d RT}, \eqref{w inf}, and \eqref{def w}. Applying Gronwall's inequality, we obtain for any $\tau\in (0,T)$
\begin{align*}
		\int_0^1 u^{2n}(\tau)dy\leq \exp\left\{2nCT\right\}\left(\int_0^1 u_0^{2n}dy+2nTCC(T)^{2n}\right),
\end{align*}
which implies that
    \begin{align*}
        \norm{u(\tau)}_{L^{2n}((0,1))}\leq C(T)\norm{u_0}_{L^{2n}((0,1))}+C(T)\leq C(T)\norm{u_0}_{L^\infty((0,1))}+C(T).
    \end{align*}
Letting \(n \to \infty\), we obtain \eqref{u inf}.
\end{proof}

Finally, we prove that the density admits a positive lower bound, which is essential to ensure the global existence of the solution.
\begin{prop}\label{Prop S-V VT}
     Assume that \eqref{SV-2gamma} holds. Then there exists a constant \(C(T)>0\), such that
    \begin{align}\label{SV 2d VT}
        V_T\leq C(T).
    \end{align}
\end{prop}
\begin{proof}
    Using \eqref{w inf} and \eqref{u inf}, we directly obtain
    \begin{align}\label{rrho_y inf}
        \sup_{0\leq \tau\leq T}\norm{r\partial_y \rho}_{L^\infty((0,1))}\leq C(T).
    \end{align}
    Let \(v = \rho^{-1}\). Then \(v\) satisfies the equation \(v_\tau = (ru)_y\). Integrating this equation over \((0,1) \times (0,\tau)\) and using \eqref{C1}, we obtain
    \begin{align}\label{SV-v-norm}
			\int_0^1 v(y,\tau)dy=\int_0^1 v_0(y)dy\leq C.
    \end{align}
    Let \(\beta \in (0,1)\) be a parameter to be determined later. For any \((y,\tau) \in (0,1)\times(0, T)\), using the one‑dimensional Sobolev embedding, Hölder's inequality, \eqref{rrho_y inf}, and \eqref{SV-v-norm}, we have
    \begin{align*}
			v^\beta(y,\tau)&\leq \int_0^1 v^\beta dy+\beta\int_0^1 v^{\beta-1}|\partial_y v|dy\\
			&\leq \left(\int_0^1 vdy\right)^{\frac{1}{\beta}}+\beta\int_0^1|\partial_y\rho| v^{\beta+1}dy\\
			&\leq C(\beta)+\beta\norm{r\partial_y\rho}_{L^\infty((0,1))}\left(\int_0^1 vr^{-\frac{3}{2}}dy\right)^{\frac{2}{3}}\left(\int_0^1 v^{3\beta+1}dy\right)^{\frac{1}{3}}\\
			&\leq C(\beta)+C(T)\beta V_T^\beta.
	\end{align*}
    Taking the supremum over \((y,\tau)\) of the above inequality and setting \(\beta = \min\{\frac{1}{2C(T)}, \frac{1}{2}\}\), we obtain \eqref{SV 2d VT}. This completes the proof of Proposition \ref{Prop S-V VT}.
\end{proof}

\subsection{The proof of Theorem \ref{Thm SV 2d}}
Combining \eqref{SV 2d k}, \eqref{SV 2d RT}, \eqref{SV 2d VT}, \eqref{C nablarho^2}, \eqref{SV 2d l}, and \eqref{C u^2}, we obtain the following bound:
\begin{align}
		\begin{split}
			\sup_{0\leq t\leq T}\Big(\norm{\mathbf{u}}_{L^{4}(\Omega)}&+\norm{\rho}_{L^\infty(\Omega)}+\norm{\rho^{-1}}_{L^\infty(\Omega)}+\norm{\nabla\rho}_{L^2(\Omega)}+\norm{\nabla\rho}_{L^{4}(\Omega)}\Big)\\
			&+\int_0^T\left(\norm{\nabla\rho}_{L^2(\Omega)}^2+\norm{\nabla\rho}_{L^{4}(\Omega)}^{4}+\norm{\nabla\mathbf{u}}_{L^2(\Omega)}^2\right)dt\leq C(T).
		\end{split}
	\end{align}
    Using Propositions \ref{Prop H1}-\ref{Prop H5} we obtain estimates for higher‑order derivatives. Hence, by a continuity argument, we have proved that the local classical solution \((\rho, \mathbf{u})\) exists globally in time.

    Regarding the large‑time behavior of the solution, we recall the  uniform estimates \eqref{SV 2d RT}, \eqref{C nablarho^2}, \eqref{SV 2d l}, and \eqref{C u^2} already obtained:
    \begin{align*}
	\begin{split}
		\sup_{0\leq t<\infty}&\left(\norm{\rho}_{L^\infty(\Omega)}+\norm{\nabla\rho^{\frac{1}{2}}}_{L^2(\Omega)}+\norm{\nabla\rho^{\frac{1}{4}}}_{L^{4}(\Omega)}\right)\\
		&+\int_0^\infty\norm{\nabla\rho^{\frac{\gamma}{2}}}_{L^{2}(\Omega)}^2+\norm{\nabla\rho^{\frac{\gamma}{4}}}_{L^{4}(\Omega)}^{4}+\norm{\rho^{\frac{1}{2}}\nabla \mathbf{u}}_{L^2(\Omega)}^2dt\leq C.
	\end{split}
\end{align*}
This, together with Propositions \ref{Prop rho epsilon}-\ref{Prop lb}, implies that the density possesses a uniform positive lower bound, and consequently, the large‑time behavior of the solution follows. 

We have completed the proof of Theorem \ref{Thm SV 2d}.

\section{Global weak solutions allowing vacuum \((\alpha\ge 1)\)}\label{Sec4}
Throughout this section we always assume \(\gamma>1\) and \(\alpha\ge 1\) if \(N=2\) or \(1\le\alpha<11.7\) if \(N=3\); the radially symmetric initial data \((\rho_0,\mathbf m_0)\) are taken to satisfy \eqref{2d weak initial} with \(p,q\) fulfilling
\begin{align*}
	 \frac{N}{2}<q<p<n_N(\alpha), \quad p\in\mathbb{R},\quad q\in M_{set}.
\end{align*}

\subsection{Construction of the approximate system I}
We employ the approximate system from \cite{Guo-2008}: 
\begin{align}\label{App  Stm Euler}
	\left\{
	\begin{array}{l}
		\partial_t\rho_\eta+\div(\rho_\eta\mathbf{u}_\eta)=0,\\
		\partial_t(\rho_\eta \mathbf{u}_\eta)+\div(\rho_\eta\mathbf{u}_\eta\otimes\mathbf{u}_\eta)+\nabla(\rho_\eta^\gamma)=\div(\tilde{\mu}(\rho_\eta)\nabla \mathbf{u}_\eta)+\nabla(\tilde{\lambda}(\rho_\eta)\div\mathbf{u}_\eta),\\
		(\rho_{\eta},\mathbf{u}_{\eta})|_{t=0}=(\rho_{\eta,0},\mathbf{u}_{\eta,0}),\\
		\mathbf{u}_{\eta}|_{\partial\Omega_\eta}=0
	\end{array}
	\right.
\end{align}
where $t\ge0$ and $x\in \Omega_\eta=\Omega-B_\eta$. The viscosity coefficients satisfy
\begin{align}
	\tilde{\mu}(\rho_\eta)=\rho_\eta^\alpha+\eta\rho^\delta_\eta, \quad
	\tilde{\lambda}(\rho_\eta)=(\alpha-1)\rho_\eta^\alpha+\eta(\delta-1)\rho_\eta^\delta,
\end{align}
where \(\delta\) is chosen as follows:{\small
	\begin{align}
		&1-\frac{p\sqrt{2p-1}-2p+1}{p^2-2p+1}<\delta<1-\frac{1}{2p},\  N=2;\label{2d delta}\\
		&1-\frac{\sqrt{4p(4p^2-p-1)+1}-6p+3}{4p^2-8p+4}<\delta<1-\frac{1}{2p},\  N=3 \text{ and } p\geq  1.55;\label{3d delta p large}\\
		&1-\frac{\sqrt{4n(4n^2-n-1)+1}-6n+3}{4n^2-8n+4}|_{n=1.55}<\delta=0.677<1-\frac{1}{2n}|_{n=1.55},\  N=3 \text{ and } p< 1.55.\label{3d delta p small}
\end{align}}
\begin{rmk} 
	The artificial viscosity term is added to obtain an \(\eta\)-dependent positive lower bound for \(\rho_\eta\), which is crucial for the global existence of the approximate solution. It is readily verified that the \(\delta\) appearing in \eqref{2d delta} and \eqref{3d delta p large} is well-defined.
\end{rmk}
\begin{rmk}
	For \(N=2\), or for \(N=3\) and \(p\ge 1.55\), the definition of \(\delta\) immediately gives \(p<n_N(\delta)\).  
	When \(N=3\) and \(p<1.55\), we likewise have \(p<1.55<n_3(\delta)\).
\end{rmk}

We now specify the initial data for the approximate system.  Without loss of generality, we assume the total initial mass $\int_\Omega\rho_0 dx=1$.  First, extend \(\rho_0\) continuously to \(\mathbb R^N\) and set
\begin{align}\label{weak rho ini}
	\rho_{\eta,0}=C_\eta\left(j_{\varepsilon(\eta)}*\rho_0^{\alpha-1+\frac{1}{2q}}+\eta^{\frac{\alpha-1+\frac{1}{2q}}{\alpha-\delta}}\right)^{\frac{1}{\alpha-1+\frac{1}{2q}}}, \text{ in }\Omega,
\end{align}
where the constant \(C_\eta\) is chosen so that \(\int_{\Omega_\eta} \rho_{\eta,0}\,dx=1\), and \(j_{\varepsilon(\eta)}\) is the standard mollifier.  We then define
\begin{align}\label{weak u ini}
	\mathbf{u}_{\eta,0}=\left(\frac{1}{\rho_{\eta,0}}\left(\frac{|\mathbf{m}_0|^{2p}}{\rho_0^{2p-1}}\alpha_\eta\right)*j_{\varepsilon(\eta)}\right)^{\frac{1}{2p}}\frac{\mathbf{x}}{r}, \text{ in }\Omega,
\end{align}
where \(\alpha_\eta(\cdot)\) is a smooth cut-off function such that \(\alpha_\eta(x)=0\) whenever \(0\le|x|\le 2\eta\) or \(|x|\ge R-\eta\), and \(\alpha(x)=1\) when \(3\eta\le|x|\le R-2\eta\). One verifies that the initial data of the approximate system possess the following properties:
\begin{align}\label{2d weak initial data}
	\begin{split}
		&\rho_{\eta,0}\rightarrow \rho_0 \text{ in }L^\gamma(\Omega),\quad \nabla\rho_{\eta,0}^{\alpha-\frac{1}{2}}\rightarrow\nabla\rho_0^{\alpha-\frac{1}{2}} \text{ in }L^2(\Omega),\quad \rho_{\eta,0}\ge C\eta^{\frac{1}{\alpha-\delta}},\\
		&\nabla\rho_{\eta,0}^{\alpha-1+\frac{1}{2q}}\rightarrow  \nabla\rho_{0}^{\alpha-1+\frac{1}{2q}}\text{ in }L^{2q}(\Omega), \quad |\mathbf{m}_{\eta,0}|^{2p}/\rho_{\eta,0}^{2p-1}\rightarrow |\mathbf{m}_0|^{2p}/\rho_0^{2p-1} \text{ in }L^1(\Omega),\\
		&|\mathbf{m}_{\eta,0}|^{2}/\rho_{\eta,0}\rightarrow |\mathbf{m}_0|^{2}/\rho_0 \text{ in }L^1(\Omega), \quad \mathbf{u}_{\eta,0}|_{\partial\Omega_\eta}=0,
	\end{split}
\end{align}
where \(C>0\) is a generic constant and $\mathbf{m}_{\eta,0}=\rho_{\eta,0}\mathbf{u}_{\eta,0}$. Restricting the obtained pair \((\rho_{\eta,0},\mathbf{u}_{\eta,0})\) to \(\Omega_\eta\), we still denote it by \((\rho_{\eta,0},\mathbf{u}_{\eta,0})\).

 In spherical coordinates, the approximate system \eqref{App  Stm Euler} takes the following form
\begin{align}\label{App  Stm Ball}
	\left\{
	\begin{array}{l}
		(\rho_\eta)_t+(\rho_\eta u_\eta)_r+\frac{N-1}{r}\rho_\eta u_\eta=0,\\
		\rho_\eta (u_\eta)_t+\rho_\eta u_\eta (u_\eta)_r+(\rho_\eta^\gamma)_r-\left(\frac{\alpha\rho_\eta^\alpha+\eta\delta\rho_\eta^\delta}{r^{N-1}}(r^{N-1}u_\eta)_r\right)_r+\frac{N-1}{r}(\rho_\eta^\alpha+\eta \rho_\eta^\delta)_r u_\eta=0,\\
		(\rho_\eta, u_\eta)|_{t=0}=(\rho_{\eta,0}, u_{\eta,0}),\\
		u_\eta(\eta,t)=u_\eta(R,t)=0,
	\end{array}
	\right.
\end{align}
where $t\ge 0$ and $r\in(\eta, R)$.

  Define \(y(r,t)=\int_\eta^r \rho_\eta(s,t)s^{N-1}ds\) and \(\tau(r,t)=t\).  Then, in Lagrangian coordinates, the approximate system \eqref{App  Stm Euler} becomes
\begin{align}\label{app system lag}
	\left\{
	\begin{array}{l}
		(\rho_\eta)_\tau+\rho^2_\eta(r^{N-1}u_\eta )_y=0,\\
		r^{-(N-1)}(u_\eta)_\tau+(\rho_\eta^\gamma)_y-((\alpha\rho_\eta^{\alpha+1}+\eta\delta\rho^{1+\delta}_\eta)(r^{N-1}u_\eta)_y)_y+\frac{N-1}{r}(\rho_\eta^\alpha+\eta\rho_\eta^\delta)_yu_\eta=0,\\
		(\rho_\eta, u_\eta)|_{\tau=0}=(\rho_{\eta,0}, u_{\eta,0}),\\
		u_\eta(0,\tau)=u_\eta(1,\tau)=0,
	\end{array}
	\right.
\end{align}
where $\tau\ge0$ and $y\in(0,1)$. 

\subsection{Local solvability of the approximate system}
We study the global solvability of system \eqref{app system lag} in Lagrangian coordinates.  Fix \(\eta>0\).  It is easily seen that \(\rho_0^\eta\in C^{1+\beta}[0,1]\) and \(u_0^\eta\in C^{2+\beta}[0,1]\) for any \(0<\beta<1\).  As remarked in \cite{Guo-2008}, the system is now essentially one-dimensional and contains no singularities.  For such problems a standard argument yields a unique local classical solution \((\rho_\eta, u_\eta)\) with \(\rho_\eta,\partial_y\rho_\eta,\partial_{\tau y}\rho_\eta,u_\eta,\partial_yu_\eta,\partial_\tau u_\eta,\partial_{yy}u_\eta\in C^{\beta,\frac{\beta}{2}}([0,1]\times[0,T_0])\) for some \(T_0>0\). 
\subsection{Some useful estimates}
Let \(T^{*}\) be the maximal existence time of the local classical solution \((\rho_{\eta}, u_{\eta})\) to the approximate system \eqref{app system lag}, and fix \(0<T<T^{*}\). 
	\begin{prop}\label{Prop WI n}
	There exists a constant \(C>0\) independent of \(\tau\) and \(\eta\) such that
	\begin{align}
		\int_0^1 (u_\eta^2+\rho_\eta^{\gamma-1})dy+\int_0^\tau\int_0^1\Big(\rho_\eta^{\alpha-1}\frac{u_\eta^{2}}{r^2}&+\rho_\eta^{1+\alpha} (\partial_y{u_\eta})^2r^{2(N-1)}\Big)dyd\tau \nonumber\\
		+\eta\int_0^\tau\int_0^1&\Big(\rho_\eta^{\delta-1}\frac{u_\eta^{2}}{r^2}+\rho_\eta^{1+\delta} (\partial_y{u_\eta})^2r^{2(N-1)}\Big)dyd\tau\leq C.\label{u^2 weak}
	\end{align}
	For $1<n\leq p$, there exists a constant \(C>0\) independent of \(\tau\) and \(\eta\) such that
	\begin{align}
		\int_0^1 u_\eta^{2n}dy+\int_0^\tau\int_0^1\Big(\rho_\eta^{\alpha-1}\frac{u_\eta^{2n}}{r^2}&+\rho_\eta^{1+\alpha} {(\partial_y{u_\eta})}^2 u_\eta^{2n-2}r^{2(N-1)}\Big)dyd\tau\nonumber\\
		&\leq C+C\int_0^\tau\int_0^1\rho_\eta^{2n(\gamma-\alpha)+\alpha-1}r^{2(n-1)}dyd\tau.\label{u^2n weak}
	\end{align}
	For $p<n<\min\{n_N(\alpha),n_N(\delta)\}$, there exists a constant \(C(\eta)>0\) independent of \(\tau\) but depending on \(\eta\) such that
	\begin{align}
		\int_0^1 u_\eta^{2n}dy+\int_0^\tau\int_0^1\Big(\rho_\eta^{\alpha-1}\frac{u_\eta^{2n}}{r^2}&+\rho^{1+\alpha} {(\partial_y{u_\eta})}^2 u_\eta^{2n-2}r^{2(N-1)}\Big)dyd\tau\nonumber\\
		&\leq C(\eta)+C\int_0^\tau\int_0^1\rho_\eta^{2n(\gamma-\alpha)+\alpha-1}r^{2(n-1)}dyd\tau.\label{u^2n weak eta}
	\end{align}
\end{prop}
\begin{proof}
	Multiplying $\eqref{app system lag}_2$ by \(r^{N-1}u_\eta\) and integrating the resulting equation over $(0,1)\times(0,\tau)$, we obtain \eqref{u^2 weak} by using the fact $1<\min\{n_N(\alpha), n_N(\delta)\}$ and the boundary condition $\eqref{app system lag}_4$.
	
	Fix $1<n<\min\{n_N(\alpha),n_N(\delta)\}$. Multiplying $\eqref{app system lag}_2$ by \(r^{N-1}u_\eta^{2n-1}\) and integrating the resulting equation over $(0,1)$ gives 
	\begin{align*}
		&\frac{1}{2n}\frac{d}{d\tau}\int_0^1u_\eta^{2n}dy+\int_0^1\alpha\rho_\eta^{1+\alpha}(r^{N-1} u_\eta)_y(r^{N-1}u_\eta^{2n-1})_y-(N-1) \rho_\eta^{\alpha}(r^{N-2}u_\eta^{2n})_ydy\\
		&+\eta\int_0^1\Big(\delta\rho_\eta^{1+\delta}(r^{N-1} u_\eta)_y(r^{N-1}u_\eta^{2n-1})_y-(N-1)\rho_\eta^{\delta}(r^{N-2}u_\eta^{2n})_y\Big)dy=\int_0^1\rho_\eta^{\gamma}(r^{N-1}u_\eta^{2n-1})_ydy.
	\end{align*}
	A direct computation shows
	\begin{align}\label{W1}
	\begin{split}
		&\frac{1}{2n}\frac{d}{d\tau}\int_0^1u_\eta^{2n}dy+
		\int_0^1\Big((\alpha(N-1)^2-(N-1)(N-2))\rho_\eta^{\alpha-1}\frac{u_\eta^{2n}}{r^2}\\
		&\quad+(2n-1)\alpha\rho_\eta^{1+\alpha}(\partial_yu_\eta)^2 u_\eta^{2n-2}r^{2(N-1)}+2n(N-1)(\alpha-1)\rho_\eta^\alpha u_\eta^{2n-1}\partial_yu_\eta r^{N-2}\Big)dy\\
		&\quad+\eta\int_0^1\Big((\delta(N-1)^2-(N-1)(N-2))\rho_\eta^{\delta-1}\frac{u_\eta^{2n}}{r^2}\\
		&\quad+(2n-1)\delta\rho_\eta^{1+\delta}(\partial_yu_\eta)^2 u_\eta^{2n-2}r^{2(N-1)}+2n(N-1)(\delta-1)\rho_\eta^\delta u_\eta^{2n-1}\partial_yu_\eta r^{N-2}\Big)dy\\
		&=\int_0^1 \Big((2n-1)\rho_\eta^{\gamma}u_\eta^{2n-2}\partial_y u_\eta r^{N-1}+(N-1)\rho_\eta^{\gamma-1}\frac{u_\eta^{2n-1}}{r} \Big)dy.
	\end{split}
\end{align}
Since \(1<n<\min\{n_{N}(\alpha),n_N(\delta)\}\), we can choose \(\varepsilon>0\) sufficiently small so that
\begin{align*}
	\frac{(2n(N-1)(\alpha-1))^2}{4(1-2\varepsilon)(2n-1)\alpha}<(1-2\varepsilon)(\alpha(N-1)^2-(N-1)(N-2))
\end{align*}
and 
\begin{align*}
	\frac{(2n(N-1)(\delta-1))^2}{4(1-2\varepsilon)(2n-1)\delta}<(1-2\varepsilon)(\delta(N-1)^2-(N-1)(N-2)).
\end{align*}
Therefore, Young's inequality yields
	\begin{align}\label{W1.1}
	\begin{split}
		&\quad 2n(N-1)(\alpha-1)\rho_\eta^\alpha u_\eta^{2n-1}\partial_yu_\eta r^{N-2}\\
		&\ge-(1-2\varepsilon)(2n-1)\alpha\rho_\eta^{1+\alpha} (\partial_y{u_\eta})^2 u_\eta^{2n-2}r^{2(N-1)}-\frac{(2n(N-1)(\alpha-1))^2}{4(1-2\varepsilon)(2n-1)\alpha}\rho_\eta^{\alpha-1}\frac{u_\eta^{2n}}{r^2}\\
		&\geq -(1-2\varepsilon)(2n-1)\alpha\rho_\eta^{1+\alpha} (\partial_y u_\eta)^2 u_\eta^{2n-2}r^{2(N-1)}\\
		&\quad-(1-2\varepsilon)(\alpha(N-1)^2-(N-1)(N-2))\rho_\eta^{\alpha-1}\frac{u_\eta^{2n}}{r^2}
	\end{split}
\end{align}
and 
	\begin{align}\label{W1.2}
	\begin{split}
		&\quad 2n(N-1)(\delta-1)\rho_\eta^\delta u_\eta^{2n-1}\partial_yu_\eta r^{N-2}\\
		&\ge-(1-2\varepsilon)(2n-1)\delta\rho_\eta^{1+\delta} (\partial_y{u_\eta})^2 u_\eta^{2n-2}r^{2(N-1)}-\frac{(2n(N-1)(\delta-1))^2}{4(1-2\varepsilon)(2n-1)\delta}\rho_\eta^{\delta-1}\frac{u_\eta^{2n}}{r^2}\\
		&\geq -(1-2\varepsilon)(2n-1)\delta\rho_\eta^{1+\delta} (\partial_y u_\eta)^2 u_\eta^{2n-2}r^{2(N-1)}\\
		&\quad-(1-2\varepsilon)(\delta(N-1)^2-(N-1)(N-2))\rho_\eta^{\delta-1}\frac{u_\eta^{2n}}{r^2}.
	\end{split}
\end{align}
Using \eqref{W1.1} and \eqref{W1.2}, we obtain from \eqref{W1} that
	\begin{align}\label{W2}
	\begin{split}
		&\frac{1}{2n}\frac{d}{d\tau}\int_0^1u_\eta^{2n}dy+
		2\varepsilon\int_0^1\Big((\alpha(N-1)^2-(N-1)(N-2))\rho_\eta^{\alpha-1}\frac{u_\eta^{2n}}{r^2}\\
		&\quad+(2n-1)\alpha\rho_\eta^{1+\alpha} (\partial_y u_\eta)^2 u_\eta^{2n-2}r^{2(N-1)}\Big)dy\\
		&\leq\int_0^1 \Big((2n-1)\rho_\eta^{\gamma}u_\eta^{2n-2}\partial_y u_\eta r^{N-1}+(N-1)\rho_\eta^{\gamma-1}\frac{u_\eta^{2n-1}}{r} \Big)dy.
	\end{split}
\end{align}
Applying Young’s inequality once more, we get
\begin{align}\label{W3}
	\begin{split}
		&\frac{1}{2n}\frac{d}{d\tau}\int_0^1u_\eta^{2n}dy+
		\varepsilon\int_0^1\Big((\alpha(N-1)^2-(N-1)(N-2))\rho_\eta^{\alpha-1}\frac{u_\eta^{2n}}{r^2}\\
		&\quad+(2n-1)\alpha\rho_\eta^{1+\alpha} (\partial_y u_\eta)^2 u_\eta^{2n-2}r^{2(N-1)}\Big)dy\\
		&\leq C\int_0^1\rho_\eta^{2n(\gamma-\alpha)+\alpha-1}r^{2n-2}dy.
	\end{split}
\end{align}
Integrate the above expression over \((0,\tau)\). If \(1<n\le p\), we obtain from \eqref{2d weak initial data} that
\begin{align}\label{W4}
	\begin{split}
		&\frac{1}{2n}\int_0^1u_\eta^{2n}dy+
		\varepsilon\int_0^\tau\int_0^1\Big((\alpha(N-1)^2-(N-1)(N-2))\rho_\eta^{\alpha-1}\frac{u_\eta^{2n}}{r^2}\\
		&\quad+(2n-1)\alpha\rho_\eta^{1+\alpha} (\partial_y u_\eta)^2 u_\eta^{2n-2}r^{2(N-1)}\Big)dyd\tau\\
		&\leq\frac{1}{2n}\int_0^1u_{\eta,0}^{2n}dy+C\int_0^\tau \int_0^1\rho_\eta^{2n(\gamma-\alpha)+\alpha-1}r^{2n-2}dyd\tau\\
		&\leq C\int_0^1u_{\eta,0}^{2p}dy+C\int_0^1u_{\eta,0}^{2}dy+C\int_0^\tau \int_0^1\rho_\eta^{2n(\gamma-\alpha)+\alpha-1}r^{2n-2}dyd\tau\\
		&\leq C+C\int_0^\tau \int_0^1\rho_\eta^{2n(\gamma-\alpha)+\alpha-1}r^{2n-2}dyd\tau.
	\end{split}
\end{align}
If \(p<n<\min\{n_N(\alpha),n_N(\delta)\}\), the definition of \(u_{\eta,0}\) implies \eqref{u^2n weak eta}. This completes the proof of Proposition \ref{Prop WI n}.
\end{proof}

	\begin{prop}\label{Prop W m}
	Assume that $1<m\leq q$ and $m\in M_{set}$. It holds that
	\begin{align}
		\int_0^1(u_\eta+r^{N-1}(\rho_\eta^\alpha)_y+\eta r^{N-1}(\rho^\delta_\eta)_y)^2dy&+\int_0^\tau\int_0^1\rho_\eta^{\gamma-\alpha}(r^{N-1}(\rho_\eta^\alpha)_y)^{2}dyd\tau\leq C;\label{W nablarho^2}\\
		\int_0^1(u_\eta+r^{N-1}(\rho_\eta^\alpha)_y+\eta r^{N-1}(\rho^\delta_\eta)_y)^{2m}dy&+\int_0^\tau\int_0^1\rho_\eta^{\gamma-\alpha}(r^{N-1}(\rho_\eta^\alpha)_y)^{2m}dyd\tau\nonumber\\
		&\leq C+C\int_0^\tau\int_0^1\rho_\eta^{\gamma-\alpha}u_\eta^{2m}dyd\tau,\label{W nablarho^2m}
	\end{align}
	where $C$ is a positive constant independent of $\tau$ and $\eta$.
\end{prop}

	\begin{proof}
	From equation \eqref{app system lag} we derive the following B–D entropy equation:
	\begin{align}\label{2d B-D weak}
		(u_\eta+r^{N-1}(\rho_\eta^\alpha)_y+\eta r^{N-1}(\rho_\eta^\delta)_y)_\tau+(\rho_\eta^\gamma)_yr^{N-1}=0.
	\end{align}
	Multiplying \eqref{2d B-D weak} by (\(u_\eta+r^{N-1}(\rho_\eta^\alpha)_y+\eta r^{N-1}(\rho_\eta^\delta)_y\)) and integrating the resulting equation over $(0,1)\times(0,\tau)$ gives
	\begin{align*}
		\int_0^1\frac{1}{2}(u_\eta+r^{N-1}(\rho_\eta^\alpha)_y+\eta r^{N-1}(\rho_\eta^\delta)_y)^2+\frac{\rho_\eta^{\gamma-1}}{\gamma-1}dy+\alpha\gamma\int_0^\tau\int_0^1\rho_\eta^{\gamma+\alpha-2}(\partial_y\rho_\eta)^2r^{2(N-1)}dyd\tau\\
		+\eta\delta\gamma\int_0^\tau\int_0^1\rho_\eta^{\gamma+\delta-2}(\partial_y \rho_\eta)^2r^{2(N-1)}dyd\tau=\int_0^1\frac{1}{2}(u_{\eta,0}+r^{N-1}(\rho_{\eta,0}^\alpha)_y+\eta r^{N-1}(\rho_{\eta,0}^\delta)_y)^2+\frac{\rho_{\eta,0}^{\gamma-1}}{\gamma-1}dy,
	\end{align*}
	which together with \eqref{2d weak initial data} yields \eqref{W nablarho^2}.
	
	Multiplying \eqref{2d B-D weak} by \(\bigl(u_\eta+r^{N-1}(\rho_\eta^\alpha)_y+\eta r^{N-1}(\rho_\eta^\delta)_y\bigr)^{2m-1}\) and integrating the resulting equation over \(y\in(0,1)\) gives
	\begin{align*}
		\frac{1}{2m}\frac{d}{d\tau}\int_0^1 &(u_\eta+r^{N-1}(\rho_\eta^\alpha)_y+\eta r^{N-1}(\rho_\eta^\delta)_y)^{2m}dy\\
		&+\gamma\int_0^1\rho_\eta^{\gamma-1}\partial_y\rho_\eta(u_\eta+r^{N-1}(\rho_\eta^\alpha)_y+\eta r^{N-1}(\rho_\eta^\delta)_y)^{2m-1}r^{N-1}dy=0.
	\end{align*}
	Using Lemma \ref{Lem ks}, we obtain 
	\begin{align*}
		&\quad\rho_\eta^{\gamma-1}\partial_y\rho_\eta(u_\eta+r^{N-1}(\rho_\eta^\alpha)_y+\eta r^{N-1}(\rho_\eta^\delta)_y)^{2m-1}r^{N-1}\\
		&= (\alpha \rho_\eta^{\alpha-1}+\eta\delta \rho_\eta^{\delta-1})r^{N-1}\partial_y\rho_\eta(u_\eta+(\alpha \rho_\eta^{\alpha-1}+\eta\delta \rho_\eta^{\delta-1})r^{N-1}\partial_y\rho_\eta)^{2m-1}\frac{\rho_\eta^{\gamma-1}}{\alpha\rho_\eta^{\alpha-1}+\eta\delta\rho_\eta^{\delta-1}}\\
		&\geq\varepsilon((\alpha \rho_\eta^{\alpha-1}+\eta\delta \rho_\eta^{\delta-1})r^{N-1}\partial_y\rho_\eta)^{2m}\frac{\rho_\eta^{\gamma-1}}{\alpha\rho_\eta^{\alpha-1}+\eta\delta\rho_\eta^{\delta-1}}-\frac{\rho_\eta^{\gamma-1}}{\alpha\rho_\eta^{\alpha-1}+\eta\delta\rho_\eta^{\delta-1}}u_\eta^{2m}\\
		&=\varepsilon(\alpha \rho_\eta^{\alpha-1}+\eta\delta \rho_\eta^{\delta-1})^{2m-1}(r^{N-1}\partial_y\rho_\eta)^{2m}\rho_\eta^{\gamma-1}-\frac{\rho_\eta^{\gamma-1}}{\alpha\rho_\eta^{\alpha-1}+\eta\delta\rho_\eta^{\delta-1}}u_\eta^{2m}\\
		&\ge\frac{\varepsilon}{\alpha}(r^{N-1}(\rho_\eta^\alpha)_y)^{2m}\rho_\eta^{\gamma-\alpha}-\frac{1}{\alpha}\rho_\eta^{\gamma-\alpha}u_\eta^{2m},
	\end{align*}
	where \(\varepsilon>0\) is a sufficiently small generic constant. Therefore, we have
	\begin{align*}
		\frac{1}{2m}\frac{d}{d\tau}\int_0^1 (u_\eta+r^{N-1}(\rho_\eta^\alpha)_y&+\eta r^{N-1}(\rho_\eta^\delta)_y)^{2m}dy+\frac{\varepsilon}{\alpha}\int_0^1\rho_\eta^{\gamma-\alpha}(r^{N-1}(\rho_\eta^\alpha)_y)^{2m}dy\\
		&\leq C\int_0^1 \rho_\eta^{\gamma-\alpha}u_\eta^{2m}dy.
	\end{align*}
	Integrating the inequatity with respect to time and using \eqref{2d weak initial data} we obtain that
	{\small
	\begin{align*}
		&\quad\frac{1}{2m}\int_0^1 (u_\eta+r^{N-1}(\rho_\eta^\alpha)_y+\eta r^{N-1}(\rho_\eta^\delta)_y)^{2m}dy+\frac{\varepsilon}{\alpha}\int_0^\tau\int_0^1\rho_\eta^{\gamma-\alpha}(r^{N-1}(\rho_\eta^\alpha)_y)^{2m}dyd\tau\\
		&\leq \frac{1}{2m}\int_0^1 (u_{\eta,0}+r^{N-1}(\rho_{\eta,0}^\alpha)_y+\eta r^{N-1}(\rho_{\eta,0}^\delta)_y)^{2m}dy+C\int_0^\tau\int_0^1\rho_\eta^{\gamma-\alpha}u_\eta^{2m}dyd\tau\\
		&\leq C\left(\int_{\Omega_\eta} \rho_{\eta,0}|\mathbf{u}_{\eta,0}|^{2m}dx+\norm{\nabla \rho_{\eta,0}^{\alpha-1+\frac{1}{2m}}}_{L^{2m}(\Omega_\eta)}^{2m}+\norm{\eta\rho_{\eta,0}^{\delta-2+\frac{1}{2m}}\nabla \rho_{\eta,0}}_{L^{2m}(\Omega_\eta)}^{2m}\right)+C\int_0^\tau\int_0^1\rho_\eta^{\gamma-\alpha}u_\eta^{2m}dyd\tau\\
		&\leq C\left(1+\norm{\nabla \rho_{\eta,0}^{\alpha-1+\frac{1}{2m}}}_{L^{2m}(\Omega_\eta)}^{2m}+\norm{\nabla \rho_{\eta,0}^{\alpha-1+\frac{1}{2m}}}_{L^{2m}(\Omega_\eta)}^{2m}\norm{\rho^{-1}_{\eta,0}}_{L^\infty(\Omega_\eta)}^{2m(\alpha-\delta)}\eta^{2m}\right)+C\int_0^\tau\int_0^1\rho_\eta^{\gamma-\alpha}u_\eta^{2m}dyd\tau\\
		&\leq C\left(1+\norm{\nabla\rho_{\eta,0}^{\alpha-1+\frac{1}{2q}}}_{L^{2m}(\Omega_\eta)}^{2m}\norm{\rho_{\eta,0}}_{L^\infty(\Omega_\eta)}^{1-\frac{m}{q}}\right)+C\int_0^\tau\int_0^1\rho_\eta^{\gamma-\alpha}u_\eta^{2m}dyd\tau\\
		&\leq C+C\int_0^\tau\int_0^1\rho_\eta^{\gamma-\alpha}u_\eta^{2m}dyd\tau.
	\end{align*}
}
	Therefore, we complete the proof.
\end{proof}

We extend \((\rho_\eta, \mathbf{u}_\eta)\) continuously from \(\Omega_\eta\times[0,T]\) to \(\Omega\times[0,T]\). Specifically,
\begin{align}\label{vv5}
	\tilde{\rho}_\eta(r,t)=\left\{
	\begin{array}{ll}
		\rho_\eta(r,t),& \text{ if } \eta\leq r\leq R,\\
		\rho_\eta(\eta,t),&\text{ if }  0\leq r\leq \eta;\\
	\end{array}
	\right.\quad
	\tilde{u}_\eta(r,t)=\left\{
	\begin{array}{ll}
		u_\eta(r,t),& \text{ if } \eta\leq r\leq R,\\
		0,&\text{ if }  0\leq r\leq \eta.\\
	\end{array}
	\right.
\end{align}
For brevity we omit the tilde notation.  From now on, \((\rho_\eta,\mathbf{u}_\eta)\) denotes the functions defined on \(\Omega\times[0,T]\).

	\begin{prop}\label{Prop WI 2d r}
	Let $N=2$. Assume that 
	\begin{align*}
		\gamma>\alpha-\frac{1}{2}.
	\end{align*}
	 For any \(1\le s<\infty\), there exists a constant \(C(s)>0\), independent of \(\eta\) and \(T\), such that
	\begin{align}\label{2d weak rho Lp}
		\sup_{0\le t\le T}\|\rho_\eta\|_{L^s(\Omega)}\le C(s).
	\end{align}
	Moreover, for any \(0<\xi\ll 1\), there exists a constant \(C(\xi)>0\), independent of \(\eta\) and \(T\), such that
	\begin{align}\label{r weight 2d}
		\sup_{0\le t\le T}\|\rho_\eta^{\alpha-\frac12}r^\xi\|_{L^\infty(\eta,R)}\le C(\xi).
	\end{align}
\end{prop}
\begin{proof}
	We choose \(2>\zeta>\frac{2\alpha-1}{\gamma}\). Using \eqref{u^2 weak} and \eqref{W nablarho^2}, the one-dimensional Sobolev embedding yields
	\begin{align}
		\begin{split}
			&\norm{\rho_\eta^{\alpha-\frac{1}{2}}r^\zeta}_{L^\infty(\eta,R)}\leq C\int_\eta^R \rho_\eta^{\alpha-\frac{1}{2}} r^\zeta dr+C\int_\eta^R \rho_\eta^{\alpha-\frac{1}{2}}r^{\zeta-1}dr+C\int_\eta^R |\partial_r \rho_\eta^{\alpha-\frac{1}{2}}| r^\zeta dr\\
			&\leq C\int_\eta^R (\rho_\eta^\gamma r)^{\frac{2\alpha-1}{2\gamma}} r^{\zeta-\frac{2\alpha-1}{2\gamma}}dr+C\int_\eta^R (\rho_\eta^\gamma r)^{\frac{2\alpha-1}{2\gamma}} r^{\zeta-1-\frac{2\alpha-1}{2\gamma}}dr+C\int_\eta^R(|\partial_r \rho_\eta^{\alpha-\frac{1}{2}}|^2 r)^{\frac{1}{2}}r^{\zeta-\frac{1}{2}}dr\\
			&\leq C+C\int_\eta^R r^{\frac{2\gamma}{2\gamma-2\alpha+1}\zeta+\frac{-2\alpha+1}{2\gamma-2\alpha+1}}dr+C\int_\eta^R r^{\frac{2\gamma}{2\gamma-2\alpha+1}\zeta+\frac{-2\gamma-2\alpha+1}{2\gamma-2\alpha+1}}dr+C\int_\eta^R r^{2\zeta-1}dr\leq C,
		\end{split}
	\end{align}
	where we have used
	\begin{align*}
		\frac{2\gamma}{2\gamma-2\alpha+1}\zeta+\frac{-2\gamma-2\alpha+1}{2\gamma-2\alpha+1}>-1\Leftrightarrow \zeta>\frac{2\alpha-1}{\gamma}.
	\end{align*}
	This implies
	\begin{align}
		\rho_\eta^{\alpha-\frac{1}{2}}(\eta,t)\leq C \eta^{-\zeta}.
	\end{align}
	Therefore, by the definition of \(\rho_\eta\) and \eqref{u^2 weak}, we obtain
	\begin{align}
		\begin{split}
			\norm{\rho_\eta^{\alpha-\frac{1}{2}}}_{L^1(\Omega)}&=\norm{\rho_\eta^{\alpha-\frac{1}{2}}}_{L^1(\Omega_\eta)}+\norm{\rho_\eta^{\alpha-\frac{1}{2}}}_{L^1(\Omega-\Omega_\eta)}\leq C+C\eta^{-\zeta+2}\leq C.
		\end{split}
	\end{align}
	Noting also that \(\|\nabla\rho_\eta^{\alpha-\frac12}\|_{L^2(\Omega)}=\|\nabla\rho_\eta^{\alpha-\frac12}\|_{L^2(\Omega_\eta)}\le C\), we apply the standard Sobolev embedding on \(\Omega\) to obtain a bound on \(\sup_{0\le t\le T}\|\rho_\eta^{\alpha-\frac12}\|_{L^s(\Omega)}\) for any finite \(s\), from which \eqref{2d weak rho Lp} follows. 
	
	We next prove \eqref{r weight 2d}.  For sufficiently small \(\xi>0\), using \eqref{W nablarho^2}, \eqref{2d weak rho Lp} and Young's inequality, we obtain
	\begin{align}
		\begin{split}
			&\norm{\rho_\eta^{\alpha-\frac{1}{2}}r^\xi}_{L^\infty(\eta,R)}\leq C\int_\eta^R \rho_\eta^{\alpha-\frac{1}{2}} r^\xi dr+C\int_\eta^R \rho_\eta^{\alpha-\frac{1}{2}}r^{\xi-1}dr+C\int_\eta^R |\partial_r \rho_\eta^{\alpha-\frac{1}{2}}| r^\xi dr\\
			&\leq C\int_\eta^R (\rho_\eta^{(\alpha-\frac{1}{2})\frac{3}{\xi}}r)^{\frac{\xi}{3}} r^{\frac{2\xi}{3}}dr+ C\int_\eta^R (\rho_\eta^{(\alpha-\frac{1}{2})\frac{3}{\xi}}r)^{\frac{\xi}{3}} r^{\frac{2\xi-3}{3}}dr+C\int_\eta^R(|\partial_r \rho_\eta^{\alpha-\frac{1}{2}}|^2 r)^{\frac{1}{2}}r^{\xi-\frac{1}{2}}dr\\
			&\leq C(\xi)+C\int_\eta^R r^{\frac{2\xi}{3-\xi}}dr+C\int_\eta^R r^{\frac{2\xi-3}{3-\xi}}dr+C\int_\eta^R r^{2\xi-1}dr\leq C(\xi). 
		\end{split}
	\end{align}
	
	This completes the proof of Proposition \ref{Prop WI 2d r}.
\end{proof}

	\begin{prop}\label{Prop WI 3d r}
	Let $N=3$. Assume that 
	\begin{align*}
		\gamma>\alpha-\frac{1}{2}.
	\end{align*}
	Then there exists a constant \(C>0\), independent of \(\eta\) and \(T\), such that
	\begin{align}\label{3d weak rho Lp}
		\sup_{0\leq t\leq T}\norm{\rho_\eta}_{L^{6\alpha-3}(\Omega)}\leq C.
	\end{align}
	For any \(0<\xi\ll 1\), there exists a constant \(C(\xi)>0\), independent of \(\eta\) and \(T\), such that
	\begin{align}\label{r weight 3d}
		\sup_{0\le t\le T}\|\rho_\eta^{\alpha-\frac12}r^{\frac{1}{2}+\xi}\|_{L^\infty(\eta,R)}\le C(\xi).
	\end{align}
\end{prop}
	\begin{proof}
	We choose $3>\zeta>\max\{\frac{6\alpha-3}{2\gamma},\frac{1}{2}\}$. Using \eqref{u^2 weak} and \eqref{W nablarho^2}, the one-dimensional Sobolev embedding yields
	\begin{align}
		\begin{split}
			&\norm{\rho_\eta^{\alpha-\frac{1}{2}}r^\zeta}_{L^\infty(\eta,R)}\leq C\int_\eta^R \rho_\eta^{\alpha-\frac{1}{2}} r^\zeta dr+C\int_\eta^R \rho_\eta^{\alpha-\frac{1}{2}}r^{\zeta-1}dr+C\int_\eta^R |\partial_r \rho_\eta^{\alpha-\frac{1}{2}}| r^\zeta dr\\
			&\leq C\int_\eta^R (\rho_\eta^\gamma r^2)^{\frac{2\alpha-1}{2\gamma}} r^{\zeta-\frac{2\alpha-1}{\gamma}}dr+C\int_\eta^R (\rho_\eta^\gamma r^2)^{\frac{2\alpha-1}{2\gamma}} r^{\zeta-1-\frac{2\alpha-1}{\gamma}}dr+C\int_\eta^R(|\partial_r \rho_\eta^{\alpha-\frac{1}{2}}|^2 r^2)^{\frac{1}{2}}r^{\zeta-1}dr\\
			&\leq C+C\int_\eta^R r^{\frac{2\gamma}{2\gamma-2\alpha+1}\zeta+\frac{-4\alpha+2}{2\gamma-2\alpha+1}}dr+C\int_\eta^R r^{\frac{2\gamma}{2\gamma-2\alpha+1}\zeta+\frac{-2\gamma-4\alpha+2}{2\gamma-2\alpha+1}}dr+C\int_\eta^R r^{2\zeta-2}dr\leq C,
		\end{split}
	\end{align}
	where we have used
	\begin{align*}
		\frac{2\gamma}{2\gamma-2\alpha+1}\zeta+\frac{-2\gamma-4\alpha+2}{2\gamma-2\alpha+1}>-1\Leftrightarrow \zeta>\frac{6\alpha-3}{2\gamma}.
	\end{align*}
	This implies
	\begin{align}
		\rho_\eta^{\alpha-\frac{1}{2}}(\eta,t)\leq C \eta^{-\zeta}.
	\end{align}
	Therefore, by the definition of \(\rho_\eta\) and \eqref{u^2 weak}, we obtain
	\begin{align}
		\begin{split}
			\norm{\rho_\eta^{\alpha-\frac{1}{2}}}_{L^1(\Omega)}&=\norm{\rho_\eta^{\alpha-\frac{1}{2}}}_{L^1(\Omega_\eta)}+\norm{\rho_\eta^{\alpha-\frac{1}{2}}}_{L^1(\Omega-\Omega_\eta)}\leq C+C\eta^{-\zeta+3}\leq C.
		\end{split}
	\end{align}
	Note also that \(\|\nabla\rho_\eta^{\alpha-\frac12}\|_{L^2(\Omega)}=\|\nabla\rho_\eta^{\alpha-\frac12}\|_{L^2(\Omega_\eta)}\le C\). Applying the standard Sobolev embedding on \(\Omega\), we obtain a bound on \(\sup_{0\le t\le T}\|\rho_\eta^{\alpha-\frac12}\|_{L^6(\Omega)}\), from which \eqref{3d weak rho Lp} follows. 
	
	We next prove \eqref{r weight 3d}. For sufficiently small \(\xi>0\), using \eqref{W nablarho^2}, \eqref{3d weak rho Lp} and Young's inequality, we obtain
	\begin{align*}
		\begin{split}
			&\norm{\rho_\eta^{\alpha-\frac{1}{2}}r^{\frac{1}{2}+\xi}}_{L^\infty(\eta,R)}\leq C\int_\eta^R \rho_\eta^{\alpha-\frac{1}{2}} r^{\frac{1}{2}+\xi} dr+C\int_\eta^R \rho_\eta^{\alpha-\frac{1}{2}}r^{\xi-\frac{1}{2}}dr+C\int_\eta^R |\partial_r \rho_\eta^{\alpha-\frac{1}{2}}| r^{\frac{1}{2}+\xi} dr\\
			&\leq C\int_\eta^R (\rho_\eta^{6\alpha-3}r^2)^{\frac{1}{6}} r^{\frac{1}{6}+\xi}dr+ C\int_\eta^R (\rho_\eta^{6\alpha-3}r^2)^{\frac{1}{6}} r^{\xi-\frac{5}{6}}dr+C\int_\eta^R(|\partial_r \rho_\eta^{\alpha-\frac{1}{2}}|^2 r^2)^{\frac{1}{2}}r^{\xi-\frac{1}{2}}dr\\
			&\leq C+C\int_\eta^R r^{\frac{6}{5}\xi+\frac{1}{5}}dr+C\int_\eta^R r^{\frac{6}{5}\xi-1}dr+C\int_\eta^R r^{2\xi-1}dr\leq C(\xi). 
		\end{split}
	\end{align*}
	
	This completes the proof of Proposition \ref{Prop WI 3d r}.
\end{proof}

\subsection{Upper bound and positive lower bound of $\rho_\eta$ $(N=2)$}
We first derive a Mellet–Vasseur-type estimate together with bounds on the derivative of the density.
\begin{prop}\label{Prop W 2d k}
	Under the assumptions of Theorem \ref{Thm3}, there exists a constant \(C(T)>0\) independent of \(\eta\) such that
	\begin{align}\label{W 2d k}
		\sup_{0\leq t\leq T}\int_\eta^R\rho_\eta u_\eta^{2p}rdr+\int_0^T\int_\eta^R\Big(\rho_\eta^{\alpha}\frac{u_\eta^{2p}}{r}
		+\rho_\eta^{\alpha} (\partial_ru_\eta)^2 u_\eta^{2p-2}r\Big)drdt\leq C(T)
	\end{align}
	and
	\begin{align}\label{W 2d l}
		\sup_{0\leq t\leq T}\int_\eta^R |\partial_r\rho_\eta^{\alpha-1+\frac{1}{2q}}|^{2q}rdr+\int_0^T\int_\eta^R |\partial_r\rho_\eta^{\alpha-1+\frac{\gamma-\alpha+1}{2q}}|^{2q}rdrdt\leq C(T).
	\end{align}
\end{prop}
\begin{proof}
	We first prove \eqref{W 2d k}. Using \eqref{r weight 2d}, we obtain from \eqref{u^2n weak} that
		\begin{align*}
		\begin{split}
			&\sup_{0\leq t\leq T}\int_\eta^R\rho_\eta u_\eta^{2p}rdr+\int_0^T\int_\eta^R\Big(\rho_\eta^{\alpha}\frac{u_\eta^{2p}}{r}
			+\rho_\eta^{\alpha} (\partial_r u_\eta)^2 u_\eta^{2p-2}r\Big)drdt\\
			&\leq C+C\int_0^T\int_\eta^R\rho_\eta^{2p(\gamma-\alpha)+\alpha}r^{2p-1}drdt\\
			&\leq C+C\sup_{0\le t\le T}\norm{\rho_\eta^{\alpha-\frac{1}{2}}r^\xi}_{L^\infty(\eta,R)}^{\frac{2p(\gamma-\alpha)+\alpha}{\alpha-1/2}}\int_0^T\int_\eta^R r^{2p-1-\frac{2p(\gamma-\alpha)+\alpha}{\alpha-1/2}\xi}drdt\\
			&\leq C(T),
		\end{split}
	\end{align*}
	where we have used
	\begin{align*}
		2p(\gamma-\alpha)+\alpha\ge 0\Longleftrightarrow \gamma\ge\left(1-\frac{1}{2p}\right)\alpha
	\end{align*}
	and \(\xi>0\) is chosen sufficiently small so that
	\begin{align*}
		2p-1-\frac{2p(\gamma-\alpha)+\alpha}{\alpha-1/2}\xi>-1.
	\end{align*}
	
	Next, we prove \eqref{W 2d l}. Since \(q\in M_{\text{set}}\), we apply \eqref{W nablarho^2m} to obtain
		\begin{align*}
		\begin{split}
			&\quad\sup_{0\leq t\leq T}\int_\eta^R |\partial_r\rho_\eta^{\alpha-1+\frac{1}{2q}}|^{2q}rdr+\int_0^T\int_\eta^R |\partial_r\rho_\eta^{\alpha-1+\frac{\gamma-\alpha+1}{2q}}|^{2q}rdrdt\\
			&\leq\sup_{0\leq t\leq T}\int_\eta^R\rho_\eta u_\eta^{2q}rdr+ \sup_{0\leq t\leq T}\int_\eta^R\rho_\eta(u_\eta+\rho_\eta^{-1}\partial_r\rho_\eta^\alpha+\eta\rho^{-1}_\eta \partial_r \rho_\eta^\delta)^{2q}rdr\\
			&\quad+\int_0^T\int_\eta^R|\partial_r\rho_\eta^{\alpha-1+\frac{\gamma-\alpha+1}{2q}}|^{2q}rdrdt\nonumber\\
			&\leq C\sup_{0\le t\le T}\int_\eta^R\rho_\eta(u_\eta^{2p}+u_\eta^2)rdr+C\int_0^T\int_\eta^R\rho_\eta^{\gamma-\alpha+1}u_\eta^{2q}rdrdt\\
			&\leq C(T)+C\left(\int_0^T\int_\eta^R\rho_\eta u_\eta^{2p}rdrdt\right)^{\frac{q}{p}}\left(\int_0^T\int_\eta^R\rho_\eta^{(\gamma-\alpha)\frac{p}{p-q}+1}rdrdt\right)^{\frac{p-q}{p}}\\
			&\leq C(T),
		\end{split}
	\end{align*}
	where the last inequality follows from \eqref{W 2d k},  \eqref{2d weak rho Lp} and the fact 
	\begin{align*}
		(\gamma-\alpha)\frac{p}{p-q}+1\ge 0\Longleftrightarrow \gamma\ge \alpha-1+\frac{q}{p}.
	\end{align*}
	
	This completes the proof of Proposition \ref{Prop W 2d k}.
\end{proof}

Finally, we establish the upper and positive lower bounds for the density in two dimensions. Denote that
\begin{align*}
	\begin{split}
		R_{\eta, T}:=\sup_{0\le t\le T}\norm{\rho_\eta(t)}_{L^\infty(\Omega)}+1,\quad V_{\eta,T}:=\sup_{0\le t\le T}\norm{\rho_\eta^{-1}(t)}_{L^\infty(\Omega)}+1.
	\end{split}
\end{align*}
\begin{prop}\label{Prop W 2d RT}
	Under the assumptions of Theorem \ref{Thm3}, there exists a constant \(C(T)>0\) independent of \(\eta\) such that
	\begin{align}\label{W 2d RT}
		R_{\eta,T}\leq C(T).
	\end{align}
\end{prop}
\begin{proof}
	Using  \eqref{2d weak rho Lp} and \eqref{W 2d l}, we obtain from the one-dimensional Sobolev embedding that
	\begin{align*}
		&\norm{\rho_\eta^\alpha}_{L^\infty(\eta,R)}\leq C\int_\eta^R \rho_\eta^\alpha dr+C\int_\eta^R |\partial_r\rho_\eta^\alpha|dr\\
		&\leq C\int_0^R\rho_\eta^{\alpha}r^{\frac{1}{3}}r^{-\frac{1}{3}}dr+C\int_\eta^R|\partial_r\rho_\eta^{\alpha-1+\frac{1}{2q}}|\rho_\eta^{1-\frac{1}{2q}}dr\\
		&\leq C\left(\int_\eta^R\rho_\eta^{3\alpha}rdr\right)^{\frac{1}{3}}\left(\int_\eta^Rr^{-\frac{1}{2}}dr\right)^{\frac{2}{3}}+C\left(\int_\eta^R|\partial_r\rho_\eta^{\alpha-1+\frac{1}{2q}}|^{2q}rdr\right)^{\frac{1}{2q}}\left(\int_\eta^R \rho_\eta r^{-\frac{1}{2q-1}} dr\right)^{\frac{2q-1}{2q}}\\
		&\leq C(T)+C(T)\int_\eta^R\rho_\eta r^{-\frac{1}{2q-1}}dr \\
		&\leq C(T)+C(T)\left(\int_\eta^R \rho_\eta^{\frac{2q}{q-1}}rdr\right)^{\frac{q-1}{2q}}\left(\int_\eta^Rr^{-\frac{2q^2-q+1}{2q^2+q-1}}dr\right)^{\frac{q+1}{2q}}\\
		&\leq C(T).
	\end{align*}
	This completes the proof of Proposition \ref{Prop W 2d RT}.
\end{proof}

We then use the method in \cite{Guo-2008} to obtain a positive lower bound for \(\rho_\eta\) that depends on \(\eta\).
\begin{prop}
	Under the assumptions of Theorem \ref{Thm3}, there exists a constant \(C(\eta,T)>0\) such that
	\begin{align}\label{2d weak VT}
		V_{\eta, T}\le C(\eta,T).
	\end{align}
\end{prop}
\begin{proof}
	We claim that
	\begin{align}\label{xx4}
		\sup_{0\leq \tau\leq T}\int_0^1|\partial_y \rho_\eta^\delta|^{2p}dy\leq C(\eta, T).
	\end{align}
	Indeed, fixing \(y\in[0,1]\) and integrating $\eqref{2d B-D weak}$ with respect to $\tau$ from \(0\) to \(\tau\) gives
	\begin{align}\label{xx3}
		[r(\rho^\alpha_\eta)_y+\eta r(\rho_\eta^\delta)_y](y,\tau)=[u_{\eta,0}+r(\rho^\alpha_{\eta,0})_y+\eta r(\rho_{\eta,0}^\delta)_y] (y)-u_\eta(y,\tau)-\int_0^\tau \partial_y(\rho_\eta^\gamma) r(y,s)ds.
	\end{align}
	Multiplying the above equation by \([r(\rho^\alpha_\eta)_y+\eta r(\rho_\eta^\delta)_y]^{2p-1}\) and integrating over \([0,1]\), we obtain from \eqref{W 2d k} and Young’s inequality that
	\begin{align}
		\begin{split}
			&\int_0^1 [r(\rho^\alpha_\eta)_y+\eta r(\rho_\eta^\delta)_y]^{2p}dy\\
			&\leq C\left(\int_0^1 [r(\rho^\alpha_\eta)_y+\eta r(\rho_\eta^\delta)_y]^{2p}dy\right)^{\frac{2p-1}{2p}}\Big(\norm{u_{\eta,0}+r(\rho^\alpha_{\eta,0})_y+\eta r(\rho_{\eta,0}^\delta)_y}_{L^{2p}([0,1])}\\
			&\quad+\norm{u_\eta}_{L^{2p}([0,1])}+\Big(\int_0^\tau\norm{\partial_y(\rho_\eta^\gamma) r}_{L^{2p}([0,1])}^{2p}ds\Big)^{\frac{1}{2p}}\Big)\\
			&\leq \frac{1}{2}\int_0^1 [r(\rho^\alpha_\eta)_y+\eta r(\rho_\eta^\delta)_y]^{2p}dy+C(\eta,T)+C\int_0^\tau\norm{\partial_y(\rho_\eta^\gamma) r}_{L^{2p}([0,1])}^{2p}ds.
		\end{split}
	\end{align}
	Therefore, we obtain from $\delta<1$, $r\ge \eta$ and  $\eqref{W 2d RT}$ that
	\begin{align}
		\begin{split}
			\int_0^1[\eta r(\rho_\eta^\delta)_y]^{2p}dy&\leq C(\eta,T)+C\int_0^\tau\norm{\partial_y(\rho_\eta^\gamma) r}_{L^{2p}([0,1])}^{2p}ds\\
			&\leq C(\eta,T)+C(\eta)\int_0^\tau\norm{\rho_\eta^{\gamma-\delta}}_{L^\infty([0,1])}^{2p}\norm{\partial_y(\rho_\eta^\delta)}_{L^{2p}([0,1])}^{2p}ds\\
			&\leq C(\eta,T)+C(\eta,T)\int_0^\tau\norm{\partial_y(\rho_\eta^\delta)}_{L^{2p}([0,1])}^{2p}ds,
		\end{split}
	\end{align}
which implies that
	\begin{align*}
		\int_0^1|\partial_y(\rho_\eta^\delta)|^{2p}dy\leq C(\eta,T)+C(\eta,T)\int_0^\tau\norm{\partial_y(\rho_\eta^\delta)}_{L^{2p}([0,1])}^{2p}ds.
	\end{align*}
Using Gronwall’s inequality, we obtain \eqref{xx4}.
	
	We next establish the lower bound for the density.  Set \(v_\eta(y,\tau)=\frac{1}{\rho_\eta(y,\tau)}\). The continuity equation $\eqref{app system lag}_1$ gives \((v_\eta)_\tau=(r u_\eta)_y\), so we have
	\begin{align}\label{xx4.5}
		\int_0^1 v_\eta(y,\tau)dy=\int_0^1 v_\eta(y,0)dy\leq C(\eta).
	\end{align}
	Then, by the one-dimensional Sobolev embedding, $\eqref{xx4}$ and \eqref{xx4.5}, one has
	\begin{align}\label{xx8}
		\begin{split}
			v_\eta(y,\tau)&\leq \int_0^1 v_\eta(y,\tau)dy+\int_0^1|\partial_y v_\eta(y,\tau)|dy\\
			&\leq C(\eta)+C\int_0^1 v_\eta^{1+\delta}|\partial_y \rho_\eta^\delta|dy\\
			&\leq C(\eta)+C\left(\int_0^1|\partial_y \rho_\eta^\delta|^{2p}dy\right)^{\frac{1}{2p}}\left(\int_0^1 v_\eta^{(1+\delta)\frac{2p}{2p-1}}dy\right)^{\frac{2p-1}{2p}}\\
			&\leq C(\eta)+C(\eta,T)\left(\int_0^1 v_\eta^{(1+\delta)\frac{2p}{2p-1}-1+1}dy\right)^{\frac{2p-1}{2p}}\\
			&\leq C(\eta)+C(\eta,T)(V_{\eta,T})^{\delta+\frac{1}{2p}}.
		\end{split}
	\end{align}
Taking the supremum of the above expression over \((y,\tau)\in[0,1]\times[0,T]\) and applying Young’s inequality together with \eqref{2d delta}, we obtain \eqref{2d weak VT}.
\end{proof}

\subsection{Upper bound and positive lower bound of $\rho_\eta$ $(N=3)$}
\begin{prop}\label{Prop W 3d k}
	Under the assumptions of Theorem \ref{Thm4},  there exists a constant $0\leq\sigma<\min\{2\alpha-1,(3\alpha-2)(1-\frac{1}{q})\}$ such that
	\begin{align}\label{sigma gamma}
		\alpha-\frac{\alpha}{2p}+\sigma<\gamma<3\alpha-1+\frac{\alpha-1}{2p}+\sigma.
	\end{align}
	Furthermore, there exists a constant \(C(T)>0\) independent of \(\eta\) such that
	\begin{align}\label{W 3d k}
		\sup_{0\leq t\leq T}\int_\eta^R\rho_\eta u_\eta^{2p}r^2dr+\int_0^T\int_\eta^R\Big(\rho_\eta^{\alpha}u_\eta^{2p}
		+\rho_\eta^{\alpha} (\partial_ru_\eta)^2 u_\eta^{2p-2}r^2\Big)drdt\leq C(T)R_{\eta,T}^{2p\sigma}
	\end{align}
	and
	\begin{align}\label{W 3d l}
		\sup_{0\leq t\leq T}\int_\eta^R |\partial_r\rho_\eta^{\alpha-1+\frac{1}{2q}}|^{2q}r^2dr+\int_0^T\int_\eta^R |\partial_r\rho_\eta^{\alpha-1+\frac{\gamma-\alpha+1}{2q}}|^{2q}r^2drdt\leq C(T)R_{\eta,T}^{2q\sigma}.
	\end{align}
\end{prop}
\begin{proof}
	It is readily verified that a constant \(\sigma\) satisfying the required condition exists, thanks to \eqref{WI 3d gamma}. 	
	
	Next, we establish \eqref{W 3d k}. We use \eqref{u^2n weak}, \eqref{r weight 3d} and \eqref{sigma gamma} to obtain
	\begin{align*}
		\begin{split}
			&\sup_{0\leq t\leq T}\int_\eta^R\rho_\eta u_\eta^{2p}r^2dr+\int_0^T\int_\eta^R\Big(\rho_\eta^{\alpha}u_\eta^{2p}
			+\rho_\eta^{\alpha} (\partial_ru_\eta)^2 u_\eta^{2p-2}r^2\Big)drdt\\
			&\leq C+C\int_0^T\int_\eta^R\rho_\eta^{2p(\gamma-\alpha)+\alpha}r^{2p}drdt\\
			&\leq C+C\sup_{0\le t\le T}\norm{\rho_\eta^{\alpha-\frac{1}{2}}r^{\frac{1}{2}+\xi}}_{L^\infty(\eta,R)}^{\frac{2p(\gamma-\alpha)+\alpha-2p\sigma}{\alpha-1/2}}R_{\eta,T}^{2p\sigma}\int_0^T\int_\eta^R r^{2p-\frac{2p(\gamma-\alpha)+\alpha-2p\sigma}{\alpha-1/2}\left(\frac{1}{2}+\xi\right)}drdt\\
			&\leq C(T)R_{\eta,T}^{2p\sigma},
		\end{split}
	\end{align*}
	where we have used
	\begin{align*}
		2p(\gamma-\alpha)+\alpha-2p\sigma> 0\Longleftrightarrow \gamma>\alpha-\frac{\alpha}{2p}+\sigma
	\end{align*}
	 and \(\xi>0\) is a sufficiently small constant satisfying
	\begin{align*}
		2p-\frac{2p(\gamma-\alpha)+\alpha-2p\sigma}{\alpha-1/2}\left(\frac{1}{2}+\xi\right)>-1.
	\end{align*}
	Such a $\xi$ exists because \eqref{sigma gamma}. We thus obtain \eqref{W 3d k}.
	
	Finally, we proof \eqref{W 3d l}. Invoking \eqref{r weight 3d}, \eqref{u^2 weak} and \eqref{W 3d k}, one derives from \eqref{W nablarho^2m} that
		\begin{align*}
		\begin{split}
			&\quad\sup_{0\leq t\leq T}\int_\eta^R |\partial_r\rho_\eta^{\alpha-1+\frac{1}{2q}}|^{2q}r^2dr+\int_0^T\int_\eta^R |\partial_r\rho_\eta^{\alpha-1+\frac{\gamma-\alpha+1}{2q}}|^{2q}r^2drdt\\
			&\leq\sup_{0\leq t\leq T}\int_\eta^R\rho_\eta u_\eta^{2q}r^2dr+ \sup_{0\leq t\leq T}\int_\eta^R\rho_\eta(u_\eta+\rho_\eta^{-1}\partial_r\rho_\eta^\alpha+\eta\rho^{-1}_\eta \partial_r \rho_\eta^\delta)^{2q}r^2dr\\
			&\quad+\int_0^T\int_\eta^R|\partial_r\rho_\eta^{\alpha-1+\frac{\gamma-\alpha+1}{2q}}|^{2q}r^2drdt\nonumber\\
			&\leq \sup_{0\leq t\leq T}\int_\eta^R\rho_\eta u_\eta^{2q}r^2dr+C\int_0^T\int_\eta^R\rho_\eta^{\gamma-\alpha+1}u_\eta^{2q}r^2drdt\\
			&\leq \sup_{0\leq t\leq T}\left(\int_\eta^R\rho_\eta u_\eta^{2p}r^2dr\right)^{\frac{q}{p}}\left(\int_\eta^R\rho_\eta r^2dr\right)^{\frac{p-q}{p}}\\
			&\quad+C\sup_{0\le t\le T}\norm{\rho_\eta^{\gamma-2\alpha+1}r^2}_{L^\infty(\eta,R)}\left(\int_0^T\int_\eta^R\rho_\eta^\alpha u_\eta^{2p}drdt\right)^{\frac{q-1}{p-1}}\left(\int_0^T\int_\eta^R\rho_\eta^\alpha u_\eta^{2}drdt\right)^{\frac{p-q}{p-1}}\\
			&\leq C(T)R_{\eta,T}^{2q\sigma}+C(T)\sup_{0\le t\le T}\norm{\rho_\eta^{\gamma-2\alpha+1}r^2}_{L^\infty(\eta,R)}R_{\eta,T}^{2p\sigma\frac{q-1}{p-1}}\\
			&\leq C(T)R_{\eta,T}^{2q\sigma}+C(T)\sup_{0\le t\le T}\norm{\rho^{\alpha-\frac{1}{2}}r^{\frac{1}{2}+\xi}}_{L^\infty(0,R)}^{\frac{\gamma-2\alpha+1}{\alpha-1/2}}R_{\eta,T}^{2q\sigma}\leq C(T)R_{\eta,T}^{2q\sigma},
		\end{split}
	\end{align*}
	where we have used 
	\begin{align*} 
		\gamma-2\alpha+1\ge0,
	\end{align*} 
	and \(\xi\) is taken sufficiently small so that
	\begin{align*}
		\left(\frac{\gamma-2\alpha+1}{\alpha-1/2}\right)\left(\frac{1}{2}+\xi\right)<2.
	\end{align*}
	Such a $\xi$ exists because $\gamma<6\alpha-3$. This completes the proof of Proposition \ref{Prop W 3d k}.
\end{proof}

Next, we derive an \(\eta\)-independent upper bound and an \(\eta\)-dependent positive lower bound for \(\rho_\eta\) in three dimensions.
\begin{prop}\label{Prop W 3d RT}
	Under the assumptions of Theorem \ref{Thm4}, there exists a constant \(C(T)>0\) independent of \(\eta\) such that
	\begin{align}\label{W 3d RT}
		R_{\eta,T}\leq C(T).
	\end{align}
\end{prop}
\begin{proof}
Choose \(\beta\) such that
\begin{align}\label{w beta}
	\max\Bigl\{\sigma,\alpha-1+\frac1{2q}\Bigr\}
	<\beta<
	\min\Bigl\{2\alpha-1,(3\alpha-2)\Bigl(1-\frac1q\Bigr)\Bigr\},
\end{align}
where \(\sigma\) is given in Proposition \ref{Prop W 3d k}. Therefore, combining the one-dimensional Sobolev embedding with \eqref{3d weak rho Lp}, \eqref{r weight 3d}, \eqref{W 3d l} and \eqref{w beta}, we obtain
\begin{align*}
	\begin{split}
		&\norm{\rho_\eta^\beta}_{L^\infty(\eta,R)}\leq C\int_\eta^R\rho_\eta^\beta dr+C\int_\eta^R |\partial_r\rho_\eta^\beta| dr\\
		&\leq C\int_\eta^R (\rho_\eta^{6\alpha-3}r^2)^{\frac{\beta}{6\alpha-3}}r^{-\frac{2\beta}{6\alpha-3}}dr+C\int_\eta^R|\partial_r\rho_\eta^{\alpha-1+\frac{1}{2q}}|\rho_\eta^{\beta-\alpha+1-\frac{1}{2q}}dr\\
		&\leq C\left(\int_\eta^R\rho_\eta^{6\alpha-3}r^2dr\right)^{\frac{\beta}{6\alpha-3}}\left(\int_\eta^Rr^{-\frac{2\beta}{6\alpha-\beta-3}}dr\right)^{\frac{6\alpha-\beta-3}{6\alpha-3}}\\
		&\quad+C\left(\int_\eta^R|\partial_r\rho_\eta^{\alpha-1+\frac{1}{2q}}|^{2q}r^2dr\right)^{\frac{1}{2q}}\left(\int_\eta^R\rho_\eta^{\frac{2q}{2q-1}(\beta-\alpha)+1}r^{-\frac{2}{2q-1}}dr\right)^{\frac{2q-1}{2q}}\\
		&\leq C+C(T)R_{\eta,T}^{\sigma}\sup_{0\le t\le T}\norm{\rho_\eta^{\alpha-\frac{1}{2}}r^{\frac{1}{2}+\xi}}_{L^\infty(\eta,R)}^{\frac{\beta-\alpha+\frac{2q-1}{2q}}{\alpha-1/2}}\left(\int_0^R r^{-\frac{2}{2q-1}-\left(\frac{1}{2}+\xi\right)\frac{\frac{2q}{2q-1}(\beta-\alpha)+1}{\alpha-1/2}}dr\right)^{\frac{2q-1}{2q}}\\
		&\leq C(T)R_{\eta,T}^\sigma,
	\end{split}
\end{align*}
where \(\xi\) is taken sufficiently small so that
\begin{align*}
	-\frac{2}{2q-1}-\left(\frac{1}{2}+\xi\right)\frac{\frac{2q}{2q-1}(\beta-\alpha)+1}{\alpha-1/2}>-1.
\end{align*}
Such a \(\xi\) exists by virtue of \eqref{w beta}. Applying Young’s inequality, we obtain 
\begin{align*}
	R_{\eta,T}^{\beta}\leq \frac{1}{2}R_{\eta,T}^{\beta}+C(T),
\end{align*}
which implies $\eqref{W 3d RT}$.

This completes the proof of Proposition \ref{Prop W 3d RT}.
\end{proof}

The approach to obtaining the $\eta$-dependent positive lower bound on the density is identical to the two-dimensional case, because away from the sphere center the system is essentially one-dimensional.  Hence we only outline the proof.
\begin{prop}\label{Prop W 3d VT}
	Under the assumptions of Theorem \ref{Thm4}, there exists a constant \(C(\eta,T)>0\) such that
	\begin{align}\label{3d weak VT} 
		V_{\eta,T}\le C(\eta,T).
	\end{align}
\end{prop}
\begin{proof}
	We consider two cases.  In the first case \(p\ge 1.55\), we have obtain from \eqref{W 3d k} and \eqref{W 3d RT} that $\sup_{0\leq \tau\leq T}\int_0^1 u^{2p}_\eta dy\leq C(T)$, which yields $\sup_{0\leq \tau\leq T}\int_0^1|\partial_y\rho_\eta^\delta|^{2p}dy\leq C(\eta,T)$.  We then obtain $V_{\eta,T}\leq C(\eta,T)$  provided 
	\begin{align}
		\delta<1-\frac{1}{2p},
	\end{align} 
	which is guaranteed by the definition of \(\delta\) in \eqref{3d delta p large}.
	
	The second case is \(p<1.55\). Using \eqref{W 3d RT} and the simple fact that \(1.55<\min\{n_3(\alpha),n_3(\delta)\}\), we obtain from \eqref{u^2n weak eta} that
	$\sup_{0\leq \tau\leq T}\int_0^1 |u_\eta|^{3.1}dy\leq C(\eta,T)$, which yields 
	$\sup_{0\leq \tau\leq T}\int_0^1|\partial_y \rho_\eta^\delta|^{3.1}dy\leq C(\eta, T)$. Therefore, we obtain $V_{\eta,T}\leq C(\eta,T)$  provided that
	\begin{align}
		\delta<1-\frac{1}{3.1},
	\end{align} 
	which is guaranteed by the definition of \(\delta\) in \eqref{3d delta p small}. 
	
	This completes the proof of Proposition \ref{Prop W 3d VT}.
\end{proof}

\subsection{Uniform upper bound of $\rho_\eta$ $(N=2,3)$}
In this subsection we establish a uniform upper bound for the density. We first show that the \(L^4\)-integrability of the velocity field can be established independently of $T$ and $\eta$.

\begin{prop}
	 Under the assumptions of Theorem \ref{Thm5}, there exists a constant \(C>0\) independent of \(T\) and $\eta$ such that
	\begin{align}\label{vv1}
	\int_0^1 u_\eta^4dy+\int_0^\tau\int_0^1 \Big(\rho_\eta^{\alpha-1}\frac{ u_\eta^4}{r^2}+\rho^{1+\alpha}(\partial_yu_\eta)^2u_\eta^{2}r^{2(N-1)}\Big)dyd\tau\leq C.
\end{align}
\end{prop}
	\begin{proof}
	Multiplying $\eqref{app system lag}_2$ by \(r^{N-1}u_\eta^3\) and integrating the resulting equation over \((0,1)\) gives
	\begin{align*}
		&\frac{1}{4}\frac{d}{d\tau}\int_0^1 u_\eta^4dy+\int_0^1(\alpha(N-1)^2-(N-1)(N-2)) \rho_\eta^{\alpha-1}\frac{u_\eta^4}{r^2}+3\alpha\rho_\eta^{1+\alpha}(\partial_yu_\eta)^2u_\eta^2r^{2(N-1)}dy\\
		&\quad+\eta\left(\int_0^1(\delta(N-1)^2-(N-1)(N-2)) \rho_\eta^{\delta-1}\frac{u_\eta^4}{r^2}+3\delta\rho_\eta^{1+\delta}(\partial_yu_\eta)^2u_\eta^2r^{2(N-1)}dy\right)\\
		&= \int_0^1\Big(3 \rho_\eta^\gamma u_\eta^2\partial_yu_\eta r^{N-1}+(N-1)\rho_\eta^{\gamma-1}\frac{u_\eta^3}{r}\Big)dy+\int_0^14(N-1)(1-\alpha)\rho_\eta^\alpha u_\eta^3 \partial_yu_\eta r^{N-2}dy\\
		&\quad+\eta\int_0^14(N-1)(1-\delta)\rho_\eta^\delta u_\eta^3 \partial_yu_\eta r^{N-2}dy.
	\end{align*}
	From \eqref{2d van}, \eqref{3d van},  \eqref{2d delta} and \eqref{3d delta p large} we
	see that $\min\{n_N(\alpha), n_N(\delta)\}>2$. Therefore, by Young’s inequality, we obtain a sufficiently small generic constant \(\varepsilon>0\) such that
	\begin{align}
		\begin{split}
			&\quad\frac{1}{4}\frac{d}{d\tau}\int_0^1 u_\eta^4dy+2\varepsilon\int_0^1 \rho_\eta^{\alpha-1}\frac{u_\eta^4}{r^2}+\rho_\eta^{1+\alpha}(\partial_yu_\eta)^2u_\eta^2r^{2(N-1)}dy\\
			&\leq  C\int_0^1\rho_\eta^\gamma u_\eta^2|\partial_yu_\eta| r^{N-1}dy+C\int_0^1\rho_\eta^{\gamma-1}\frac{|u_\eta|^3}{r}dy\\
			&\leq C\int_0^1\left(\rho_\eta^{1+\alpha}(\partial_yu_\eta)^2u_\eta^{2}r^{2(N-1)}\right)^{\frac{1}{2}}\left(\rho_\eta^{2\gamma-\alpha-1}u_\eta^2\right)^{\frac{1}{2}}dy+C\int_0^1 \left(\rho_\eta^{\alpha-1}\frac{u_\eta^{4}}{r^2}\right)^{\frac{1}{2}}\left(\rho_\eta^{2\gamma-\alpha-1}u_\eta^2\right)^{\frac{1}{2}}dy\\
			&\leq \varepsilon \int_0^1 \rho_\eta^{1+\alpha}(\partial_yu_\eta)^2u_\eta^2r^{2(N-1)}dy +\varepsilon\int_0^1 \rho_\eta^{\alpha-1}\frac{u_\eta^{4}}{r^2}dy+C\int_0^1 \rho_\eta^{2\gamma-\alpha-1}u_\eta^2dy.
		\end{split}
	\end{align}
	
	When \(N=2\), using \eqref{r weight 2d}, we obtain, for sufficiently small \(\xi\),
	\begin{align*}
		\begin{split}
			&\quad C\int_0^1 \rho_\eta^{2\gamma-\alpha-1}u_\eta^2dy=C\int_0^1 \rho_\eta^{\alpha-1}\frac{u_\eta^2}{r^2}\rho_\eta^{2\gamma-2\alpha}r^2dy\\
			&\leq C\norm{\rho_\eta^{\alpha-\frac{1}{2}}r^\xi}^{\frac{2(\gamma-\alpha)}{\alpha-1/2}}_{L^\infty(\eta,R)}\int _0^1\rho_\eta^{\alpha-1}\frac{u_\eta^2}{r^2}dy\leq C\int _0^1\rho_\eta^{\alpha-1}\frac{u_\eta^2}{r^2}dy,
		\end{split}
	\end{align*}
	where we used
	\begin{align*}
		\frac{2(\gamma-\alpha)}{\alpha-1/2}\ge0\Longleftrightarrow\gamma\ge\alpha.
	\end{align*}
	
	When \(N=3\), we use \eqref{r weight 3d} to get that
	\begin{align*}
		\begin{split}
			&\quad C\int_0^1 \rho_\eta^{2\gamma-\alpha-1}u_\eta^2dy=C\int_0^1 \rho_\eta^{\alpha-1}\frac{u_\eta^2}{r^2}\rho_\eta^{2\gamma-2\alpha}r^2dy\\ 
			&\leq C\norm{\rho_\eta^{\alpha-\frac{1}{2}}r^{\frac{1}{2}+\xi}}^{\frac{2(\gamma-\alpha)}{\alpha-1/2}}_{L^\infty(\eta,R)}\int_0^1 \rho_\eta^{\alpha-1}\frac{u_\eta^2}{r^2}dy\leq C\int _0^1\rho_\eta^{\alpha-1}\frac{u_\eta^2}{r^2}dy,
		\end{split}
	\end{align*}
	where we have used $\alpha\leq \gamma<3\alpha-1$ to choose a sufficiently small \(\xi\) such that
	\begin{align*}
		0\leq \left(\frac{1}{2}+\xi\right)\frac{2(\gamma-\alpha)}{\alpha-1/2}\leq 2.
	\end{align*}
	
	Therefore, one has
	\begin{align*}
		\frac{1}{4}\frac{d}{d\tau}\int_0^1 u_\eta^4dy+\varepsilon\int_0^1\Big( \rho_\eta^{\alpha-1}\frac{u_\eta^4}{r^2}+\rho_\eta^{1+\alpha}(\partial_yu_\eta)^2u_\eta^2r^{2(N-1)}\Big)dy\leq C\int _0^1\rho_\eta^{\alpha-1}\frac{u_\eta^2}{r^2}dy.
	\end{align*}
	Integrating with respect to $\tau$ and using \eqref{u^2 weak} and \eqref{2d weak initial data} complete the proof.
\end{proof}

	\begin{prop}\label{Prop W 4}
	Under the assumptions of Theorem \ref{Thm5}, there exists a constant \(C>0\) independent of $T$ and $\eta$ such that
	\begin{align}\label{W 2d 4}
		\sup_{0\leq t\leq T}\int_\eta^R\rho_\eta u_\eta^{4}r^{N-1}dr+\int_0^T\int_\eta^R\Big(\rho_\eta^{\alpha}u_\eta^{4}r^{N-3}
		+\rho_\eta^{\alpha} (\partial_ru_\eta)^2 u_\eta^{2}r^{N-1}\Big)drdt\leq C
	\end{align}
	and
	\begin{align}\label{W 2d l4}
		\sup_{0\leq t\leq T}\int_\eta^R |\partial_r\rho_\eta^{\alpha-1+\frac{1}{2q}}|^{2q}r^{N-1}dr+\int_0^T\int_\eta^R |\partial_r\rho_\eta^{\alpha-1+\frac{\gamma-\alpha+1}{2q}}|^{2q}r^{N-1}drdt\leq C.
	\end{align}
\end{prop}

	\begin{proof}
	Rewriting \eqref{vv1} in Eulerian coordinates gives \eqref{W 2d 4}. 
	
	We next prove \eqref{W 2d l4}. Noting that \(q\in M_{\mathrm{set}}\), so we can apply Proposition \ref{Prop W m} to obtain
	\begin{align*}
		\begin{split}
			&	\sup_{0\leq t\leq T}\int_\eta^R |\partial_r\rho_\eta^{\alpha-1+\frac{1}{2q}}|^{2q}r^{N-1}dr+\int_0^T\int_\eta^R |\partial_r\rho_\eta^{\alpha-1+\frac{\gamma-\alpha+1}{2q}}|^{2q}r^{N-1}drdt\\
			&\leq\sup_{0\leq t\leq T}\int_\eta^R\rho_\eta u_\eta^{2q}r^{N-1}dr+ \sup_{0\leq t\leq T}\int_\eta^R\rho_\eta(u_\eta+\rho_\eta^{-1}\partial_r\rho_\eta^\alpha+\eta\rho^{-1}_\eta \partial_r \rho_\eta^\delta)^{2q}r^{N-1}dr\\
			&\quad+\int_0^T\int_\eta^R|\partial_r\rho_\eta^{\alpha-1+\frac{\gamma-\alpha+1}{2q}}|^{2q}r^{N-1}drdt\nonumber\\
			&\leq \sup_{0\leq t\leq T}\int_\eta^R\rho_\eta u_\eta^{2q}r^{N-1}dr+C\int_0^T\int_\eta^R\rho_\eta^{\gamma-\alpha+1}u_\eta^{2q}r^{N-1}drdt\\
			&\leq \sup_{0\leq t\leq T}\left(\int_\eta^R\rho_\eta u_\eta^{4}r^{N-1}dr\right)^{\frac{q}{2}}\left(\int_\eta^R\rho_\eta r^{N-1}dr\right)^{\frac{2-q}{2}}\\
			&\quad+C\sup_{0\le t\le T}\norm{\rho_\eta^{\gamma-2\alpha+1}r^2}_{L^\infty(\eta,R)}\left(\int_0^T\int_\eta^R\rho_\eta^\alpha u_\eta^{4}r^{N-3}drdt\right)^{q-1}\left(\int_0^T\int_\eta^R\rho_\eta^\alpha u_\eta^{2}r^{N-3}drdt\right)^{2-q}\\
			&\leq C+C\sup_{0\le t\le T}\norm{\rho_\eta^{\gamma-2\alpha+1}r^2}_{L^\infty(\eta,R)},
		\end{split}
	\end{align*}
	where we used \eqref{u^2 weak}, \eqref{W 2d 4} and
	\begin{align*}
		\gamma\ge2\alpha-1.
	\end{align*}
	Therefore, it suffices to obtain \eqref{W 2d l4} by estimating \(\sup_{0\le t\le T}\norm{\rho_\eta^{\gamma-2\alpha+1}r^2}_{L^\infty(\eta,R)}\). 
	
	When \(N=2\), using \eqref{r weight 2d}, we obtain,
	\begin{align*}
		\sup_{0\le t\le T}\norm{\rho_\eta^{\gamma-2\alpha+1}r^2}_{L^\infty(\eta,R)}\leq C\sup_{0\le t\le T}\norm{\rho_\eta^{\alpha-\frac{1}{2}}r^{\xi}}_{L^\infty(\eta,R)}^{\frac{\gamma-2\alpha+1}{\alpha-1/2}}\leq C,
	\end{align*}
	where \(\xi\) is taken sufficiently small so that
	\begin{align*}
		\left(\frac{\gamma-2\alpha+1}{\alpha-1/2}\right)\xi\leq2.
	\end{align*}
	
	When \(N=3\), we use \eqref{r weight 3d} to obtain
	\begin{align*}
		\sup_{0\le t\le T}\norm{\rho_\eta^{\gamma-2\alpha+1}r^2}_{L^\infty(\eta,R)}\leq C\sup_{0\le t\le T}\norm{\rho_\eta^{\alpha-\frac{1}{2}}r^{\frac{1}{2}+\xi}}_{L^\infty(\eta,R)}^{\frac{\gamma-2\alpha+1}{\alpha-1/2}}\leq C,
	\end{align*}
	where we have used $\gamma<6\alpha-3$ to choose a sufficiently small \(\xi\) such that
	\begin{align*}
		\left(\frac{1}{2}+\xi\right)\frac{\gamma-2\alpha+1}{\alpha-1/2}\leq 2.
	\end{align*}
	
	This completes the proof of Proposition \ref{Prop W 4}.
\end{proof}

Finally, we establish a uniform upper bound for the density.
	\begin{prop}
	Under the assumptions of Theorem \ref{Thm5}, there exist a constant $C>0$ independent of $T$ and $\eta$ such that
	\begin{align}\label{vv RT}
		R_{\eta,T}\leq C.
	\end{align}
\end{prop}
	\begin{proof}
	We obtain from \eqref{2d weak rho Lp} and \eqref{3d weak rho Lp} that \begin{align}\label{vv6}
		\sup_{0\le t\le T}\norm{\rho_\eta^{\alpha-1+\frac{1}{2q}}}_{L^1(\Omega)}\leq C,
	\end{align}
	where \(C\) is independent of \(T\) and \(\eta\).  Using the standard Sobolev embedding, we obtain from \eqref{W 2d l4} and \eqref{vv6} that 
	\begin{align*}
		R_{\eta, T}\leq C.
	\end{align*}
	We have completed the proof.
\end{proof}

\subsection{Global solvability of the approximate system}
	Under the hypotheses of Theorem \ref{Thm3} or Theorem \ref{Thm4}, for any $0<T<T^*$, from \eqref{W 2d RT}, \eqref{2d weak VT}, \eqref{W 3d RT} and \eqref{3d weak VT} we conclude that $\rho_\eta(y,\tau)$ is bounded from above and below on $[0,1]\times[0,T]$. From \eqref{W 2d k} and \eqref{W 3d k} we know that $u_\eta$ is bounded in $L^\infty(0,T;L^{2p}([0,1]))$, and from \eqref{W 2d l} and \eqref{W 3d l} we know that $\partial_y \rho_\eta$ is bounded in $L^\infty(0,T;L^{2q}([0,1]))$. Furthermore, differentiating equation \eqref{app system lag} and applying standard energy methods yields higher-order estimates for $(\rho_\eta,u_\eta)$.  We then invoke Schauder theory for linear parabolic equations to conclude that the $C^{\beta,\beta/2}([0,1]\times[0,T])$-norms of $\rho_\eta,\partial_y\rho_\eta,\partial_{\tau y}\rho_\eta,u_\eta, \partial_yu_\eta,\partial_\tau u_\eta$ and $\partial_{yy}u_\eta$ are bounded. Therefore, we can continue the local solution globally in time and obtain the unique global solution \((\rho_\eta,u_\eta)\) of the initial-boundary-value problem \eqref{app system lag} satisfying, for any $T>0$, 
\begin{align*}
	\rho_\eta,\partial_y\rho_\eta,\partial_{\tau y}\rho_\eta,u_\eta,\partial_y u_\eta,\partial_\tau u_\eta,\partial_{yy}u_\eta\in C^{\beta,\beta/2}([0,1]\times[0,T])
\end{align*}
for some $0<\beta<1$, and $\rho_\eta>0$ on $[0,1]\times[0,T]$. This can be done in a similar way as in \cite{jiang-2001}.

\subsection{Taking the limit and stability of the approximate solution sequence}
With the uniform estimates for the approximate solutions in hand, we can study the stability of the approximate solution sequence.  Indeed, the stability of approximate solution sequences in two and three dimensions for periodic domains and the whole space has already been investigated in \cite{Mellet-2007}.  The stability of the weak sequence for the approximate system \eqref{App  Stm Euler} has also been proved in \cite{Guo-2008}.  Our $\eta$-independent estimates for the approximate solutions are stronger than theirs, so the stability result for the approximate solution sequence follows easily. The proof in this subsection follows \cite{Mellet-2007} and \cite{Guo-2008}. In this subsection we do not distinguish between two and three dimensions except where specifically stated.

Under the hypotheses of Theorems \ref{Thm3} and \ref{Thm4}, for any $T>0$, we obtained solutions \((\rho_\eta,u_\eta)\) on $\Omega_\eta\times[0,T]$ and extended them continuously to \(\Omega\times[0,T]\).  Consider a sequence \(\eta_j\to 0\) as \(j\to\infty\).  Write \(\rho_j=\rho_{\eta_j}\), \(\mathbf{u}_j=\mathbf{u}_{\eta_j}\), and set \(\Omega_{\frac1n}=\Omega-B_{\frac1n}(0)\) for \(n=1,2,\dots\). Let us recall the $\eta$-independent estimates \eqref{u^2 weak}, \eqref{W nablarho^2}, \eqref{W 2d k}-\eqref{W 2d RT} and \eqref{W 3d k}-\eqref{W 3d RT} we have already obtained for the approximate solutions:
\begin{align}
	\sup_{0\le t\le T}\Big(\norm{\rho_j}_{L^\infty(\Omega)}&+\norm{\nabla\rho_j^{\alpha-\frac12}}_{L^2(\Omega)}+\norm{\nabla\rho_j^{\alpha-1+\frac{1}{2q}}}_{L^{2q}(\Omega)}\Big)\nonumber\\
	&+\int_0^T\Big(\norm{\nabla\rho_j^{\frac{\gamma+\alpha-1}{2}}}_{L^2(\Omega)}^2+\norm{\nabla\rho_j^{\alpha-1+\frac{\gamma-\alpha+1}{2q}}}_{L^{2q}(\Omega)}^{2q}\Big)dt\leq C(T),\label{ttrho}\\
	\sup_{0\le t\le T}\int_{\Omega}\Big(\rho_j|\mathbf{u}_j|^2&+\rho_j|\mathbf{u}_j|^{2p}\Big)dx+\int_0^T\int_\Omega\Big(\rho_j^{\alpha}|\nabla\mathbf{u}_j|^2+\eta\rho_j^{\delta}|\nabla\mathbf{u}_j|^2\Big)dxdt\leq C(T).\label{ttu}
\end{align}

\begin{prop}\label{Prop W rho RT}
	There exists a subsequence of \(\rho_j\), still denoted by itself, such that for any \(n\in\mathbb{N}^+\) and any \(1\le s<\infty\) we have
	\begin{align}\label{tt9}
		\begin{split}
			\rho_j \rightarrow \rho, \text{ in } C(\overline{\Omega}_{\frac{1}{n}}\times[0,T]),\quad \rho_j \rightarrow \rho \text{ in }L^s(\Omega\times(0,T)),
		\end{split}
	\end{align}
	where \(\rho\) is a radially symmetric function and \(\rho\in L^\infty(\Omega\times(0,T))\cap C((\overline{\Omega}\setminus\{0\})_{\text{loc}}\times[0,T])\).
\end{prop}
\begin{proof}
	Fix \(n\) and choose \(b\ge\alpha\). When \(j\) is sufficiently large, we have
	\begin{align*}
		\partial_t \rho_j^b+\mathbf{u}_j\cdot\nabla\rho_j^b+b\rho_j^b\div \mathbf{u}_j=0,\text{ in }\Omega_{\frac{1}{n}}.
	\end{align*}
	Therefore, one has
	\begin{align*}
		\partial_t\rho_j^b=-\frac{2b}{2b-1}\sqrt{\rho_j}\mathbf{u}_j\cdot\nabla\rho_j^{b-\frac{1}{2}}-b\rho_j^{b-\frac{\alpha}{2}}\rho_j^{\frac{\alpha}{2}}\div\mathbf{u}_j ,
	\end{align*}
	which implies that 
	\begin{align}\label{tt9.1}
		\partial_t\rho_j^b \text{ is bounded in }L^2(0,T;L^1(\Omega_{\frac{1}{n}})).
	\end{align}
	Moreover, since
	\begin{align*}
		|\nabla\rho_j^b|\leq C\rho_j^{b-\alpha+1-\frac{1}{2q}}|\nabla\rho_j^{\alpha-1+\frac{1}{2q}}|,
	\end{align*}
	we have 
		\begin{align}\label{tt9.2}
		|\nabla\rho_j^b| \text{ is bounded in }L^\infty(0,T;L^{2q}(\Omega)).
	\end{align}
The Aubin–Lions lemma, together with \eqref{tt9.1} and \eqref{tt9.2}, shows that, up to a subsequence still denoted by itself,
	\begin{align*}
		\rho_j^b\rightarrow \rho^b, \text{ in }C(\overline{\Omega}_{\frac{1}{n}}\times[0,T]).
	\end{align*}
For \(n=1,2,\dots\), a standard diagonal argument allows us to extract a subsequence of \(\rho_j\), still denoted by itself, such that for any \(n\in\mathbb N^+\) it holds that
	\begin{align*}
		\rho^b_j\rightarrow \rho^b, \text{ in }C(\overline{\Omega}_{\frac{1}{n}}\times[0,T]).
	\end{align*}
	This shows that \(\rho\in C((\overline{\Omega}\setminus\{0\})_{\text{loc}}\times[0,T])\) and 
	\begin{align*}
		\rho_j \rightarrow \rho, \text{ in } C(\overline{\Omega}_{\frac{1}{n}}\times[0,T]),\forall n\in \mathbb{N}^+.
	\end{align*}  
Since the \(\rho_j\) are uniformly bounded and converge to \(\rho\) almost everywhere, \(\rho\) is also in \(L^\infty(\Omega\times(0,T))\) and its bound does not exceed the uniform bound of the \(\rho_j\).  Moreover, each \(\rho_j\) is radially symmetric, so the limit \(\rho\) is radially symmetric as well.  Finally, it follows from the dominated convergence theorem that, for every \(1\le s<\infty\),
	\begin{align*}
		\rho_j \rightarrow \rho \text{ in }L^s(\Omega\times(0,T)).
	\end{align*}
	We have completed the proof.
\end{proof}

\begin{prop}
	There exists a radially symmetric function \(\mathbf{u}\) such that, up to a subsequence,
	\begin{align}\label{tt strong}
		\sqrt{ \rho_j }\mathbf{u}_j\rightarrow \sqrt{\rho} \mathbf{u}, \text{ in }L^2(\Omega\times(0,T)).
	\end{align}
\end{prop}
\begin{proof}
	We take \(b\) sufficiently large. Since  
	\begin{align*}
		\nabla(\rho_j^b\mathbf{u}_j)=\frac{2b}{2b-1}\sqrt{\rho_j }\mathbf{u}_j\cdot\nabla\rho_j^{b-\frac{1}{2}}+\rho_j^{b-\frac{\alpha}{2}}\rho_j^{\frac{\alpha}{2}}\nabla\mathbf{u}_j,
	\end{align*}
	it follows that 
	\begin{align}\label{tt6}
		\nabla(\rho_j^b\mathbf{u}_j) \text{ is bounded in }L^2(0,T;L^1(\Omega)).
	\end{align}
	 Fix \(n\in\mathbb{N}^+\). We claim that for sufficiently large \(j\),
\begin{align}\label{tt6.5}
	\partial_t(\rho_j^b \mathbf{u}_j)\text{ is bounded in }L^2(0,T;W^{-1,s'}(\Omega_{\frac{1}{n}})),
\end{align}
where $s=\max\{p', 2N\}$. Indeed, a direct computation shows that
	\begin{align}\label{tt0}
		\partial_t(\rho_j^b \mathbf{u}_j)=\partial_t(\rho_j\mathbf{u}_j)\rho_j^{b-1}+\rho_j\mathbf{u}_j\partial_t(\rho_j^{b-1}), \text{ in }\Omega_{\frac{1}{n}}.
	\end{align}
	We first consider the second term on the right-hand side of \eqref{tt0}. For every \(\bm{\psi}\in L^2(0,T;W_0^{1,s}(\Omega_{\frac{1}{n}}))\), one has
	\begin{align*}
		\begin{split}
			&|\langle\rho_j\mathbf{u}_j\partial_t(\rho_j^{b-1}),\bm{\psi}\rangle|=\frac{b-1}{b}|\langle\partial_t(\rho_j^b) \mathbf{u}_j,\bm{\psi}\rangle|\\
			&\leq \frac{b-1}{b}\left(|\langle\div(\rho_j^b \mathbf{u}_j),\mathbf{u}_j\cdot\bm{\psi}\rangle|+|\langle(b-1)\rho_j^b\div\mathbf{u}_j,\mathbf{u}_j\cdot\bm{\psi}\rangle|\right)\\
			&=\frac{b-1}{b}\left(|\langle\rho_j^b \mathbf{u}_j,\nabla(\mathbf{u}_j\cdot\bm{\psi})\rangle|+|\langle(b-1)\rho_j^b\div\mathbf{u}_j,\mathbf{u}_j\cdot\bm{\psi}\rangle|\right)\\
			&\leq C(T)\norm{\sqrt{\rho_j}\mathbf{u}_j}_{L^\infty(0,T; L^2(\Omega_{\frac{1}{n}}))}\norm{\rho_j^{\frac{\alpha}{2}}\nabla\mathbf{u}_j}_{L^2(0,T;L^2(\Omega_{\frac{1}{n}}))}\norm{\bm{\psi}}_{L^2(0,T;L^\infty(\Omega_{\frac{1}{n}}))}\\
			&\quad+C(T)\norm{\rho_j|\mathbf{u}_j|^2}_{L^\infty(0,T; L^p(\Omega_{\frac{1}{n}}))}\norm{\nabla\bm{\psi}}_{L^2(0,T;L^{s}(\Omega_{\frac{1}{n}}))}\\
			&\leq C(T)\norm{\bm{\psi}}_{L^2(0,T;W^{1,s}_0(\Omega_{\frac{1}{n}}))}.
		\end{split}
	\end{align*}
	This shows that
	\begin{align}\label{tt1}
		\rho_j\mathbf{u}_j\partial_t(\rho_j^{b-1}) \text{ is bounded in }L^2(0,T;W^{-1,s'}(\Omega_{\frac{1}{n}})).
	\end{align}
 Next, we consider the second term on the right-hand side of \eqref{tt0}. A direct calculation shows that
	\begin{align}\label{tt1.5}
		\begin{split}
			\partial_t(\rho_j\mathbf{u}_j)\rho_j^{b-1}=&\Big(-\div(\rho_j\mathbf{u}_j\otimes\mathbf{u}_j)-\nabla(\rho_j^\gamma)+\div((\rho^\alpha_j+\eta_j\rho^\delta_j)\nabla\mathbf{u}_j)\\
			&+\nabla(((\alpha-1)\rho_j^\alpha+\eta_j(\delta-1)\rho_j^\delta)\div\mathbf{u}_j)\Big)\times\rho_j^{b-1}, \text{ in }\Omega_{\frac{1}{n}}.
		\end{split}
	\end{align}
	For every \(\bm{\psi}\in L^2(0,T;W_0^{1,s}(\Omega_{\frac{1}{n}}))\), we have
	\begin{align*}
		\begin{split}
			&|\langle\div(\rho_j\mathbf{u}_j\otimes\mathbf{u}_j)\rho_j^{b-1},\bm{\psi}\rangle|=|\langle\rho_j\mathbf{u}_j\otimes\mathbf{u}_j,\nabla(\rho_j^{b-1}\bm{\psi})\rangle|\\
			&\leq\frac{b-1}{b}|\langle\mathbf{u}_j\otimes\mathbf{u}_j,\nabla(\rho_j^b)\otimes\bm{\psi}\rangle|+|\langle\mathbf{u}_j\otimes\mathbf{u}_j,\rho_j^b\nabla\bm{\psi}\rangle|\\
			&\leq \frac{b-1}{b}|\langle\div((\mathbf{u}_j\otimes\mathbf{u}_j)\cdot\bm{\psi}),\rho_j^b\rangle|+|\langle\mathbf{u}_j\otimes\mathbf{u}_j,\rho_j^b\nabla\bm{\psi}\rangle|\\
			&\leq C(T)\norm{\sqrt{\rho_j}\mathbf{u}_j}_{L^\infty(0,T; L^2(\Omega_{\frac{1}{n}}))}\norm{\rho_j^{\frac{\alpha}{2}}\nabla\mathbf{u}_j}_{L^2(0,T;L^2(\Omega_{\frac{1}{n}}))}\norm{\bm{\psi}}_{L^2(0,T;L^\infty(\Omega_{\frac{1}{n}}))}\\
			&\quad+C(T)\norm{\rho_j|\mathbf{u}_j|^2}_{L^\infty(0,T; L^p(\Omega_{\frac{1}{n}}))}\norm{\nabla\bm{\psi}}_{L^2(0,T;L^s(\Omega_{\frac{1}{n}}))}\\
			&\leq C(T)\norm{\bm{\psi}}_{L^2(0,T;W^{1,s}_0(\Omega_{\frac{1}{n}}))}.
		\end{split}
	\end{align*}
This implies that 
\begin{align}\label{tt2}
-\div(\rho_j\mathbf{u}_j\otimes\mathbf{u}_j)\rho_j^{b-1} \text{ is bounded in }L^2(0,T;W^{-1,s'}(\Omega_{\frac{1}{n}})).
\end{align}
It is easily verified that 
\begin{align}\label{tt2.5}
	\nabla(\rho_j^\gamma)\rho_j^{b-1}\text{ is bounded in }L^\infty (0,T;L^{2}(\Omega)).
\end{align}
For every \(\bm{\psi}\in L^2(0,T;W_0^{1,s}(\Omega_{\frac{1}{n}}))\), it holds that
	\begin{align*}
		\begin{split}
		&|\langle\div((\rho^\alpha_j+\eta_j\rho^\delta_j)\nabla\mathbf{u}_j)\rho_j^{b-1},\bm{\psi}\rangle|=|\langle ((\rho^\alpha_j+\eta_j\rho^\delta_j)\nabla\mathbf{u}_j),\nabla(\rho_j^{b-1}\bm{\psi})\rangle|\\
			&\leq C(T)\norm{\nabla\rho_j^{\alpha-\frac{1}{2}}}_{L^\infty(0,T; L^{2}(\Omega_\frac{1}{n}))}\norm{\rho_j^{\frac{\alpha}{2}}\nabla\mathbf{u}_j}_{L^2(0,T;L^2(\Omega_{\frac{1}{n}}))}\norm{\bm{\psi}}_{L^2(0,T;L^\infty(\Omega_{\frac{1}{n}}))}\\
			&\quad+C(T)\norm{\rho_j^{\frac{\alpha}{2}}\nabla\mathbf{u}_j}_{L^2(0,T;L^2(\Omega_{\frac{1}{n}}))}\norm{\nabla\bm{\psi}}_{L^2(0,T;L^s(\Omega_{\frac{1}{n}}))}\\
			&\leq C(T)\norm{\bm{\psi}}_{L^2(0,T;W^{1,s}_0(\Omega_{\frac{1}{n}}))}.
		\end{split}
	\end{align*}
This implies that 
\begin{align}\label{tt2.6}
	\div((\rho^\alpha_j+\eta_j\rho^\delta_j)\nabla\mathbf{u}_j)\rho_j^{b-1} \text{ is bounded in } L^2(0,T;W^{-1,s'}(\Omega_{\frac{1}{n}})).
\end{align}
Similarly,
\begin{align}\label{tt2.7}
	\nabla(((\alpha-1)\rho_j^\alpha+\eta_j(\delta-1)\rho_j^\delta)\div\mathbf{u}_j)\rho_j^{b-1}\text{ is bounded in } L^2(0,T;W^{-1,s'}(\Omega_{\frac{1}{n}})).
\end{align}
Therefore, substituting \eqref{tt2}--\eqref{tt2.7} into \eqref{tt1.5} gives
\begin{align}
	\partial_t(\rho_j \mathbf{u}_j)\rho_j^{b-1} \text{ is bounded in }L^2(0,T;W^{-1,s'}(\Omega_{\frac{1}{n}})),
\end{align}
which together with \eqref{tt1}, shows that \eqref{tt6.5}. 

Therefore, by the Aubin–Lions lemma together with \eqref{tt6} and \eqref{tt6.5}, there exists a subsequence of \(\rho_j^b\mathbf{u}_j\), still denoted by itself, such that for any \(1\le s<\tfrac32\), 
	\begin{align*}
		\rho_j^b\mathbf{u}_j\to\overline{\rho^b\mathbf{u}}\quad\text{in }L^2(0,T;L^s(\Omega_{\tfrac1n})).
	\end{align*}
For \(n=1,2,\dots\) we apply the standard diagonal argument to extract a subsequence \(\rho_j^b\mathbf{u}_j\), still denoted by itself, such that for any \(n\in\mathbb N^+\) and $1\leq s<\frac{3}{2}$
	\begin{align}\label{tt5}
	\rho_j^b\mathbf{u}_j\to\overline{\rho^b\mathbf{u}}\quad\text{in }L^2(0,T;L^s(\Omega_{\tfrac1n})).
\end{align}

	We define
	\begin{align}
		\mathbf{u}=\left\{
		\begin{array}{ll}
			\frac{\overline{\rho^b\mathbf{u}}}{\rho^b},& \text{ in }\{\rho>0\};\\
			0,&\text{ in }\{\rho=0\}.
		\end{array}
		\right.
	\end{align}
	We obtain from \eqref{tt5} that \(\rho_j^b\mathbf{u}_j\to\overline{\rho^b\mathbf{u}}\) a.e. in \(\Omega\times(0,T)\), which implies \(\mathbf{u}_j\to\mathbf{u}\) a.e. in \(\{\rho>0\}\). Noting that each \(\mathbf{u}_j\) is radially symmetric, we conclude that \(\mathbf{u}\) is also radially symmetric. 
	
	Finally, we prove \eqref{tt strong}. By Fatou's lemma, one has
	\begin{align}\label{tt7}
		\int_0^T\int_\Omega \rho|\mathbf{u}|^{2p}dxdt=\iint_{\{\rho>0\}}\rho|\mathbf{u}|^{2p}dxdt\leq \liminf_{j}\int_0^T\int_\Omega\rho_j|\mathbf{u}_j|^{2p}dxdt\leq C(T).
	\end{align}
	Given a sufficiently large \(M\), we obtain
	\begin{align*}
		\int_0^T\int_\Omega|\sqrt{\rho_j} \mathbf{u}_j-\sqrt{\rho}\mathbf{u}|^2dxdt&\leq C\int_0^T\int_\Omega |\sqrt{\rho_j}\mathbf{u}_j\chi_{\{|\mathbf{u}_j|>M\}}|^2dxdt\\
		&\quad+ C\int_0^T\int_\Omega |\sqrt{\rho_j}\mathbf{u}_j\chi_{\{|\mathbf{u}_j|\leq M\}}-\sqrt{\rho}\mathbf{u}\chi_{\{|\mathbf{u}|\leq M\}}|^2dxdt\\
		&\quad+C\int_0^T\int_\Omega|\sqrt{\rho}\mathbf{u}\chi_{\{|\mathbf{u}|>M\}}|^2dxdt:=\sum_{i=1}^3I_i.
	\end{align*}
	From \eqref{ttu} and \eqref{tt7} we obtain
	\begin{align*}
		I_1+I_3\leq \frac{C}{M^{2p-2}}\int_0^T\int_\Omega\rho_j |\mathbf{u}_j|^{2p}dxdt+\frac{C}{M^{2p-2}}\int_0^T\int_\Omega\rho |\mathbf{u}|^{2p}dxdt\leq \frac{C(T)}{M^{2p-2}}.
	\end{align*}
	Therefore, for any given \(\varepsilon>0\), we can choose \(M\) sufficiently large such that \(I_{1}+I_{3}\le 2\varepsilon\), and then, by the dominated convergence theorem, pick \(j\) sufficiently large so that \(I_{2}\le\varepsilon\). Therefore, we finish the proof.
\end{proof}

We now show that the radially symmetric pair \((\rho,\mathbf{u})\) is a weak solution in the sense of Definition \ref{Def weak sol}.
\begin{prop}
	The pair \((\rho,\mathbf{u})\) possesses the regularity stated in \eqref{weak regularity}.
\end{prop}
\begin{proof}
	We obtain from \eqref{ttrho} that
	\begin{align}
		&\nabla\rho_j^{\alpha-\frac{1}{2}}\rightharpoonup \nabla\rho^{\alpha-\frac{1}{2}},\text{ in }L^\infty(0,T;L^2(\Omega)),\label{tt11}\\
		&\nabla\rho_j^{\alpha-1+\frac{1}{2q}}\rightharpoonup \nabla\rho^{\alpha-1+\frac{1}{2q}},\text{ in } L^\infty(0,T,L^{2q}(\Omega)).\label{tt 10}
	\end{align}
	From \eqref{ttu} and \eqref{tt strong} we obtain
	\begin{align*}
		\sqrt{\rho_j}\mathbf{u}_j\rightharpoonup\sqrt{\rho}\mathbf{u},\text{ in }L^\infty(0,T,L^{2}(\Omega)).
	\end{align*}
	Proposition \ref{Prop W rho RT} implies $\rho\in L^\infty(\Omega\times(0,T))$.
	Verifying that \(\rho^\alpha\nabla\mathbf{u}\in L^2(0,T;W^{-1,1}_{\text{loc}}(\Omega))\) is straightforward from its definition \eqref{diff term}. Therefore, we have completed the proof.
\end{proof}

\begin{prop}
	The weak form \eqref{mass equ weak} of the continuity equation holds.
\end{prop}
\begin{proof}
	We know that equation 
	\begin{align}
		\partial_t \rho_j+\div(\rho_j \mathbf{u}_j)=0
	\end{align}
	holds on \(\Omega_j\).  For arbitrary $t_2\ge t_1\ge 0$ and $\zeta\in C^1(\overline{\Omega}\times[t_1,t_2])$, multiplying the equation by \(\zeta\) and integrating by parts gives
	\begin{align}
		\int_{\Omega_j}\rho_j\zeta dx|_{t_1}^{t_2}=\int_0^T\int_{\Omega_j}\rho_j\partial_t\zeta+\sqrt{\rho_j}\sqrt{\rho_j}\mathbf{u}_j\cdot\nabla\zeta dxdt.
	\end{align}
	Letting \(j\to\infty\) and using \eqref{tt9} and \eqref{tt strong}, we obtain \eqref{mass equ weak}.
\end{proof}

\begin{prop}
	The weak form \eqref{rnm mass equ} of the renormalized continuity equation holds.
\end{prop}
\begin{proof}
	Let \(b\ge \alpha\). On \(\Omega_j\), $(\rho_j, \mathbf{u}_j)$ satisfies
	\begin{align}
		\partial_t \rho_j^b+\div(\rho_j^b\mathbf{u}_j)+(b-1)\rho_j^b\div \mathbf{u}_j=0.
	\end{align}
	For arbitrary $t_2\ge t_1\ge 0$ and $\phi\in C^1(\overline{\Omega}\times[t_1,t_2])$, multiplying the equation by \(\phi\) and integrating by parts gives
	\begin{align*}
		\begin{split}
			\int_{\Omega_j} \rho_j^b\phi dx\left|_{t_1}^{t_2}\right.=\int_{t_1}^{t_2}\int_{\Omega_j}\left( \rho_j^b\partial_t \phi +\frac{2b(b-1)}{2b-1}\sqrt{\rho_j}\mathbf{u}_j\cdot\nabla \rho_j^{b-\frac{1}{2}}\phi+b\rho_j^{b-\frac{1}{2}}\sqrt{\rho_j}\mathbf{u}_j\cdot\nabla\phi\right) dxdt.
		\end{split}
	\end{align*}
	Letting \(j\to\infty\) and using \eqref{tt9} and \eqref{tt strong}, we obtain \eqref{rnm mass equ}.
\end{proof}

\begin{prop}
	It holds that
	\begin{align*}
		\rho\in  C(\overline{\Omega}\times[0,T]).
	\end{align*}
\end{prop}
\begin{proof}
	Equation $\eqref{rnm mass equ}$ shows that for any \(\phi\in C^1(\overline\Omega)\) we have for any $b\ge \alpha$,
	\begin{align}
		\langle \partial_t \rho^b, \phi\rangle=\int_{\Omega}\frac{2b(b-1)}{2b-1}\sqrt{\rho}\mathbf{u}\cdot\nabla \rho^{b-\frac{1}{2}}\phi+b\rho^{b-\frac{1}{2}}\sqrt{\rho}\mathbf{u}\cdot\nabla\phi dx, \text{ in }\mathcal{D}'(0,T).
	\end{align}
	Therefore, we obtain
	\begin{align*}
		\sup_{0\leq t\leq T}|\langle \partial_t \rho^b,\phi\rangle|&\leq C(b)\norm{\sqrt{\rho}\mathbf{u}}_{L^\infty L^2(\Omega)}\norm{\nabla \rho^{b-\frac{1}{2}}}_{L^\infty L^{2q}(\Omega)}\norm{\phi}_{H^1(\Omega)}\\
		&\quad+C(b)\norm{\rho^{b-\frac{1}{2}}}_{L^\infty L^\infty(\Omega)}\norm{\sqrt{\rho}\mathbf{u}}_{L^\infty L^2(\Omega)}\norm{\phi}_{H^1(\Omega)}\leq C(b,T)\norm{\phi}_{H^1(\Omega)},
	\end{align*}
	which shows that
	\begin{align*}
		\partial_t \rho^b \text{ in } L^\infty(0,T;(H^{1}(\Omega))^*).
	\end{align*}
	Combined with the spatial-derivative estimate \(\rho^b\in L^\infty(0,T; W^{1,2q}(\Omega))\), the Aubin–Lions lemma yields \(\rho^b\in C(\overline\Omega\times[0,T])\), which implies $\rho\in  C(\overline{\Omega}\times[0,T])$.
\end{proof}

\begin{prop}
	The weak form \eqref{mom equ weak} of the momentum equation holds.
\end{prop}
\begin{proof}
	The approximate solutions satisfy the momentum equation: for any \(\bm{\psi}\in C^2(\overline\Omega\times[0,T])\) such that \(\bm{\psi}(x,t)=0\) on \(\partial\Omega\) and \(\bm{\psi}(x,T)=0\), we have
	\begin{align}\label{tt12}
		\begin{split}
			0&=\int_{\Omega_j}\rho_{j,0}\mathbf{u}_{j,0}\cdot\bm{\psi}(x,0) dx+\int_0^T\int_{\Omega_j}\sqrt{\rho_j}\sqrt{\rho_j}\mathbf{u}_j\cdot\partial_t\bm{\psi}+\sqrt{\rho_j}\mathbf{u}_j\otimes\sqrt{\rho_j}\mathbf{u}_j:\nabla\bm{\psi}+\rho_j^\gamma\div\bm{\psi} dxdt\\
			&\quad-(\langle\rho_j^\alpha\nabla\mathbf{u}_j,\nabla\bm{\psi}\rangle+(\alpha-1)\langle\rho_j^\alpha\div\mathbf{u}_j,\div\bm{\psi}\rangle)-\eta_j(\langle\rho_j^\delta\nabla\mathbf{u}_j,\nabla\bm{\psi}\rangle+(\delta-1)\langle\rho_j^\delta\div\mathbf{u}_j,\div\bm{\psi}\rangle)\\
			&\quad+\int_0^T\int_{\partial\Omega_j}-\bm{\psi}\cdot\sqrt{\rho_j}\mathbf{u}_j\otimes\sqrt{\rho_j}\mathbf{u}_j\cdot \mathbf{n}-\rho_j^\gamma\bm{\psi}\cdot\mathbf{n}dSdt\\
			&\quad+\int_0^T\int_{\partial \Omega_j}(\rho_j^\alpha+\eta_j\rho_j^\delta)\bm{\psi}\cdot\nabla\mathbf{u}_j\cdot\mathbf{n}+((\alpha-1)\rho_j^\alpha+\eta_j(\delta-1)\rho_j^\delta)\div\mathbf{u}_j\bm{\psi}\cdot\mathbf{n}dSdt=\sum_{k=1}^{6}K^j_k.
		\end{split}
	\end{align}
	Using \eqref{2d weak initial data}, \eqref{weak rho ini} and \eqref{weak u ini}, we obtain
	\begin{align}\label{K1}
		K_1^j\to \int_{\Omega}\mathbf{m}(0)\cdot\bm{\psi}(x,0)dx, \text{ as }j\rightarrow\infty.
	\end{align}
	From  \eqref{tt9} and \eqref{tt strong} we obtain
	\begin{align}\label{K2}
		K_2^j\to \int_0^T\int_{\Omega}\sqrt{\rho}\sqrt{\rho}\mathbf{u}\cdot\partial_t\bm{\psi}+\sqrt{\rho}\mathbf{u}\otimes\sqrt{\rho}\mathbf{u}:\nabla\bm{\psi}+\rho^\gamma\div\bm{\psi} dxdt, \text{ as } j\to \infty. 
	\end{align}
	The diffusion terms of the approximate system can be rewritten as
	\begin{align*}
		\begin{split}
			\langle \rho_j^\alpha \nabla\mathbf{u}_j,\nabla\bm{\psi}\rangle&=-\int_0^T\int_{\Omega_j} \rho_j^{\alpha-\frac{1}{2}}\sqrt{\rho_j}\mathbf{u}_j\cdot\Delta\bm{\psi}dxdt\\
			&\quad	-\frac{2\alpha}{2\alpha-1}\int_0^T\int_{\Omega_j}\nabla \rho_j^{\alpha-\frac{1}{2}}\cdot\nabla\bm{\psi}\cdot\sqrt{\rho_j}\mathbf{u}_j dxdt;\\
			\langle \rho_j^\alpha \div\mathbf{u}_j,\div\bm{\psi}\rangle&=-\int_0^T\int_{\Omega_j} \rho_j^{\alpha-\frac{1}{2}}\sqrt{\rho_j}\mathbf{u}_j\cdot\nabla\div\bm{\psi} dxdt\\
			&\quad-\frac{2\alpha}{2\alpha-1}\int_0^T\int_\Omega\nabla \rho_j^{\alpha-\frac{1}{2}}\cdot\sqrt{\rho_j}\mathbf{u}_j \div\bm{\psi} dxdt.
		\end{split}
	\end{align*}
	Therefore, using \eqref{tt9}, \eqref{tt strong} and \eqref{tt11}, we obtain
	\begin{align}\label{K3}
		K_3^j\to -\langle \rho^\alpha \nabla \mathbf{u},\nabla\bm{\psi}\rangle-(\alpha-1)\langle \rho^\alpha\div\mathbf{u},\div\bm{\psi}\rangle, \text{ as }j\to \infty,
	\end{align}
	where the diffusion terms are defined as in \eqref{diff term}. For the diffusion terms arising from the artificial viscosity, we get from \eqref{ttu} that
	\begin{align}\label{K4}
		\begin{split}
			|K_4^j|\leq C\sqrt{\eta_j}\sqrt{\eta_j}\norm{\rho_\eta^{\frac{\delta}{2}}\nabla \mathbf{u}_j}_{L^2(0,T;L^2(\Omega_j))}\norm{\nabla\bm{\psi}}_{L^2(0,T;L^2(\Omega_j))}\leq C\sqrt{\eta_j}\to 0, \text{ as }j\to \infty. 
		\end{split}
	\end{align}
	Finally, we only need to handle the boundary terms.  Noting that \(\mathbf{u}_j|_{\partial \Omega_j}=0\), we have
	\begin{align}\label{K5}
		|K_5^j|\leq \left|\int_0^T\int_{\partial\Omega_j}\rho_j^\gamma\bm{\psi}\cdot\mathbf{n}dSdt\right|\leq C(T)\eta_j^{N-1}\rightarrow 0, \text{ as }j\to \infty.
	\end{align}
	Noting that
	\begin{align}
		\partial_t \rho_j^\alpha +\mathbf{u}_j\cdot \nabla\rho_j^\alpha +\alpha\rho_j^\alpha\div\mathbf{u}_j=0,
	\end{align}
	we therefore have
	\begin{align}\label{tt10}
		\begin{split}
			&\left|\int_0^T\int_{\partial \Omega_j}(\alpha-1)\rho_j^\alpha\div\mathbf{u}_j\bm{\psi}\cdot\mathbf{n}dSdt\right|=\left|\frac{\alpha-1}{\alpha}\int_0^T\int_{\partial \Omega_j}\partial_t\rho_j^\alpha \bm{\psi}\cdot\mathbf{n}dSdt\right|\\
			&=\left|\frac{\alpha-1}{\alpha}\left(\int_{\partial \Omega_j}\rho_j^\alpha \bm{\psi}\cdot\mathbf{n}dS(T)-\int_{\partial \Omega_j}\rho_j^\alpha \bm{\psi}\cdot\mathbf{n}dS(0)-\int_0^T\int_{\partial \Omega_j}\rho_j^\alpha \partial_t\bm{\psi}\cdot\mathbf{n}dSdt\right)\right|\\
			&\leq C(T)\eta_j^{N-1}\rightarrow 0.
		\end{split}
	\end{align}
	Similarly,
	\begin{align}
		\left|\int_0^T\int_{\partial \Omega_j}\eta_j(\delta-1)\rho_j^\delta\div\mathbf{u}_j\bm{\psi}\cdot\mathbf{n}dSdt\right|\rightarrow 0.
	\end{align}
	Noting the following simple fact:
	\begin{align*}
		\bm{\psi}\cdot\nabla\mathbf{u}_j\cdot\mathbf{n}|_{\partial\Omega_j}=\bm{\psi}\cdot\nabla(\mathbf{u}_j\cdot\mathbf{n})|_{\partial\Omega_j}=\partial_r u_j \bm{\psi}\cdot \mathbf{n}|_{\partial\Omega_j}=\div\mathbf{u}_j\bm{\psi}\cdot\mathbf{n}|_{\partial\Omega_j}
	\end{align*}
	This implies that, similarly to \eqref{tt10}, we have
	\begin{align*}
		\left|	\int_0^T\int_{\partial \Omega_j}(\rho_j^\alpha+\eta_j\rho_j^\delta)\bm{\psi}\cdot\nabla\mathbf{u}_j\cdot\mathbf{n}dSdt\right|\to0, \text{ as }j\rightarrow\infty.
	\end{align*}
	Therefore, we obtain
	\begin{align}\label{K6}
		|K_6^j|\rightarrow 0,\text{ as }j\to\infty.
	\end{align}
	Letting \(j\to\infty\) and using \eqref{K1}-\eqref{K5} and \eqref{K6}, we obtain \eqref{mom equ weak} from \eqref{tt12}. We have completed the proof.
\end{proof}

\subsection{Finite-time vacuum vanishing}
In this subsection we investigate the finite-time vanishing of the vacuum region. Assuming the conditions of Theorem \ref{Thm5} hold, let \((\rho,\mathbf u)\) be the weak solution on \(\Omega\times[0,\infty)\) obtained in Theorem \ref{Thm4}. From \eqref{u^2 weak}, \eqref{vv RT}, \eqref{W nablarho^2} and \eqref{W 2d l4} we obtain 
	\begin{align}
	\begin{split}
		\sup_{0\leq t<\infty}&\left(\norm{\sqrt{\rho}\mathbf{u}}_{L^2(\Omega)}+\norm{\rho}_{L^\infty(\Omega)}+\norm{\nabla\rho^{\alpha-\frac{1}{2}}}_{L^2(\Omega)}+\norm{\nabla\rho^{\alpha-1+\frac{1}{2q}}}_{L^{2q}(\Omega)}\right)\\
		&+\int_0^\infty\Big(\norm{\nabla\rho^{\frac{\gamma+\alpha-1}{2}}}_{L^{2}(\Omega)}^2+\norm{\nabla\rho^{\alpha-1+\frac{\gamma-\alpha+1}{2q}}}_{L^{2q}(\Omega)}^{2q}\Big)dt\leq C.
	\end{split}
\end{align}

	\begin{prop}\label{Prop rho epsilon weak}
	Under the assumptions of Theorem \ref{Thm5}, there exist constants \(T_0>0\) and \(\rho^->0\), such that
	\begin{align}\label{rho varepsilon}
		\rho\ge \rho^-, \quad \forall (x,t)\in \Omega\times[T_0,\infty).
	\end{align}
\end{prop}
\begin{proof}
	We choose \(b>1\) sufficiently large so that
	\begin{align}\label{zz4}
		\begin{split}
			\sup_{0\leq t\leq \infty}&\left(\norm{\rho}_{L^\infty(\Omega)}+\norm{\nabla\rho^{b-1}}_{L^2(\Omega)}+\norm{\nabla\rho^b}_{L^{2q}(\Omega)}\right)\\
			&+\int_0^\infty\norm{\nabla\rho^{b-1}}_{L^2(\Omega)}^2+\norm{\nabla\rho^{b}}_{L^{2q}(\Omega)}^{2q}dt\leq C.
		\end{split}
	\end{align}
	By the embedding inequality we have
	\begin{align}\label{zz3}
		\norm{\rho^b-\overline{\rho^b}}_{C(\overline{\Omega})}\leq C\norm{\rho^b-\overline{\rho^b}}_{L^4(\Omega)}^{\frac{4q-2N}{4q-2N+Nq}}\norm{\nabla\rho^b}_{L^{2q}(\Omega)}^{\frac{Nq}{4q-2N+Nq}}\leq C\norm{\rho^b-\overline{\rho^b}}_{L^{4}(\Omega)}^{\frac{4q-2N}{4q-2N+Nq}}\rightarrow0,
	\end{align}
	if 
	\begin{align}\label{zz2}
		g(t)=\norm{\rho^b-\overline{\rho^b}}_{L^4(\Omega)}^4\rightarrow 0, \text{ as }t\rightarrow\infty.
	\end{align}
	Then, by Jensen's inequality we obtain
	\begin{align*}
		\overline{\rho^b}\ge \overline{\rho}^b=\overline{\rho_0}^b>0.
	\end{align*}
	Together with \eqref{zz3}, this shows that there exist \(\rho^-\) and \(T_0\) for which \eqref{rho varepsilon} holds. Thus it remains only to verify that \eqref{zz2} holds.
	
	We first show that \(\int_{0}^{\infty}|g'(t)|dt\le C\). A direct computation gives
	\begin{align*}
		g'(t)&=4b\left\langle (\rho^b-\overline{\rho^b})^3\rho^{b-1},\rho_t\right\rangle-4(\overline{\rho^b})_t\int_\Omega(\rho^b-\overline{\rho^b})^3dx\\
		&=-4b\left\langle \nabla\left((\rho^b-\overline{\rho^b})^3\rho^{b-1}\right),\sqrt{\rho}\sqrt{\rho}\mathbf{u}\right\rangle-4(\overline{\rho^b})_t\int_\Omega(\rho^b-\overline{\rho^b})^3dx:=I_1+I_2.
	\end{align*}
	Using \eqref{zz4}, we obtain
	\begin{align}\label{zz5}
		\begin{split}
			&\int_0^\infty|I_1|dt\leq C\int_0^\infty\int_\Omega|\rho^b-\overline{\rho^b}|^2|\nabla\rho^b|\rho^{b-\frac{1}{2}}\sqrt{\rho}|\mathbf{u}|+\left|\rho^b-\overline{\rho^b}\right|^3|\nabla\rho^{b-1}|\sqrt{\rho}\sqrt{\rho}|\mathbf{u}|dxdt\\
			&\leq C\int_0^\infty\left(\norm{\sqrt{\rho}\mathbf{u}}_{L^2(\Omega)}\norm{\nabla\rho^b}_{L^2(\Omega)}\norm{\rho^b-\overline{\rho^b}}_{L^\infty(\Omega)}^q+\norm{\sqrt{\rho}\mathbf{u}}_{L^2(\Omega)}\norm{\nabla\rho^{b-1}}_{L^2(\Omega)}\norm{\rho^b-\overline{\rho^b}}_{L^\infty(\Omega)}^q\right)dt\\
			&\leq C\int_0^\infty\left(\norm{\nabla\rho^{b-1}}_{L^2(\Omega)}^2+\norm{\nabla \rho^b}_{L^{2q}(\Omega)}^{2q}\right)dt\leq C.
		\end{split}
	\end{align}
	Note that
	\begin{align*}
		\sup_{0\leq t<\infty}\left|\frac{d}{dt}\overline{\rho^b}\right|&=b\sup_{0\leq t<\infty}\left|\langle \rho^{b-1},\rho_t\rangle\right|=b\sup_{0\leq t<\infty}\left|\int_\Omega \nabla\rho^{b-1}\sqrt{\rho}\cdot\sqrt{\rho}\mathbf{u}dx\right|\\
		&\leq C\sup_{0\leq t<\infty}\norm{\sqrt{\rho}\mathbf{u}}_{ L^2(\Omega)}\sup_{0\leq t<\infty}\norm{\nabla\rho^{b-1}}_{L^2(\Omega)}\leq C.
	\end{align*}
	Therefore,
	\begin{align}\label{zz6}
		\int_0^\infty|I_2|dt\leq C\int_0^\infty\norm{\rho^b-\overline{\rho^b}}^2_{L^2(\Omega)}dt\leq C\int_0^\infty \norm{\nabla\rho^b}_{L^2(\Omega)}^2dt\leq C.
	\end{align}
	Finally, we obtain
	\begin{align}\label{zz7}
		\int_0^\infty g(t)dt\leq C\sup_{0\leq t<\infty}\norm{\rho^b-\overline{\rho^b}}_{ L^\infty(\Omega)}^2\int_0^\infty\|\rho^b-\overline{\rho^b}\|_{L^2(\Omega)}^2dt\leq C\int_0^\infty\norm{\nabla \rho^b}_{L^2(\Omega)}^2dt\leq C.
	\end{align}
	The estimates \eqref{zz5}–\eqref{zz7} imply \eqref{zz2}, and the proof of Proposition \ref{Prop rho epsilon weak} is complete.
\end{proof}

\section{Global weak solutions with possible vacuum at the origin $(\alpha=1)$}
\subsection{Construction of the approximate system II}
	We begin by introducing the approximate system for \eqref{0}--\eqref{0-2} with sufficiently small \(\iota>0\):
	\begin{align}\label{App  Stm Euler'}
		\left\{
		\begin{array}{l}
			\partial_t\rho_\iota+\div(\rho_\iota \mathbf{u}_\iota)=0,\\
			\partial_t(\rho_\iota \mathbf{u}_\iota)+\div(\rho_\iota\mathbf{u}_\iota\otimes\mathbf{u}_\iota)+\nabla(\rho_\iota^\gamma)=\div(\rho_\iota\nabla \mathbf{u}_\iota),\\
			(\rho_{\iota},\mathbf{u}_{\iota})|_{t=0}=(\rho_{\iota,0},\mathbf{u}_{\iota,0}),\\
			\mathbf{u}_{\iota}|_{\partial\Omega_\iota}=0,
		\end{array}
		\right.
	\end{align}
	where $t\ge0$ and $x\in \Omega_\iota=\Omega-B_\iota$. Next, we specify the initial data \((\rho_{\iota,0},\mathbf{u}_{\iota,0})\) of the approximate system.  Without loss of generality, assume \(\int_\Omega\rho_0\,dx=1\).  Extend \(\rho_{\iota,0}\) constantly to \(\mathbb{R}^N\) and define
	\begin{align*}
		\rho_{\iota,0}=C_\iota\rho_0*j_{\varepsilon(\iota)}, \text{ in }\Omega,
	\end{align*}
	where \(C_\iota\) is a constant chosen so that \(\int_{\Omega_\iota}\rho_{\iota,0}\,dx=1\) and \(j_{\varepsilon(\iota)}\) is the standard mollifier. We then define
	\begin{align*}
		\mathbf{u}_{\iota,0}=(\mathbf{u}_0\alpha_\iota)*j_{\varepsilon(\iota)}, \text{ in }\Omega,
	\end{align*}
	where \(\alpha_\iota(\cdot)\) is a smooth cut-off function such that \(\alpha_\iota(x)=0\) whenever \(0\le|x|\le 2\iota\) or \(|x|\ge R-\iota\), and \(\alpha(x)=1\) when \(3\iota\le|x|\le R-2\iota\), and $|\nabla\alpha_\iota|\leq \frac{C}{\iota}$. One verifies that the initial data of the approximate system possess the following properties:
\begin{align}\label{2d weak initial data'}
	\begin{split}
		&\rho_{\iota,0}\rightarrow \rho_0 \text{ in }W^{1,{2q}}(\Omega),\quad \nabla\rho_{\iota,0}^{\frac{1}{2}}\rightarrow\nabla\rho_0^{\frac{1}{2}} \text{ in }L^2(\Omega),\quad \rho_{\iota,0}\ge C(\underline{\rho_0})>0,\\
		&\nabla\rho_{\iota,0}^{\frac{1}{2q}}\rightarrow  \nabla\rho_{0}^{\frac{1}{2q}}\text{ in }L^{2q}(\Omega), \quad \rho_{\iota,0}|\mathbf{u}_{\iota,0}|^{2p}\rightarrow \rho_{0}|\mathbf{u}_{0}|^{2p} \text{ in }L^1(\Omega),\\
		&\rho_{\iota,0}|\mathbf{u}_{\iota,0}|^{2}\rightarrow \rho_{0}|\mathbf{u}_{0}|^{2} \text{ in }L^1(\Omega), \quad \mathbf{u}_{\iota,0}|_{\partial\Omega_\iota}=0,
	\end{split}
\end{align}
where \(C(\underline{\rho_0})\) is a constant depending only on \(\underline{\rho_0}\). Moreover, for any \(n\in\mathbb N^+\), when \(\iota\) is sufficiently small, we have
\begin{align}
	&\norm{\nabla\rho_{\iota,0}}_{L^\infty(\Omega\backslash B_{1/n})}\leq C,\label{ini nabla rho}\\
	&\norm{\mathbf{u}_{\iota,0}}_{H^1(\Omega\backslash B_{1/n})}+\norm{\mathbf{u}_{\iota,0}}_{L^\infty(\Omega\backslash B_{1/n})}\leq C,\label{ini nabla u}
\end{align}
where $C$ is independent of $\iota$ due to \eqref{1'}.

In spherical coordinates, the approximate system \eqref{App  Stm Euler'} takes the following form
\begin{align}\label{App  Stm Ball'}
	\left\{
	\begin{array}{l}
		(\rho_\iota)_t+(\rho_\iota u_\iota)_r+\frac{N-1}{r}\rho_\iota u_\iota=0,\\
		\rho_\iota (u_\iota)_t+\rho_\iota u_\iota (u_\iota)_r+(\rho_\iota^\gamma)_r-\left(\frac{\rho_\iota}{r^{N-1}}(r^{N-1}u_\iota)_r\right)_r+\frac{N-1}{r}(\rho_\iota)_r u_\iota=0,\\
		(\rho_\iota, u_\iota)|_{t=0}=(\rho_{\iota,0}, u_{\iota,0}),\\
		u_\iota(\iota,t)=u_\iota(R,t)=0,
	\end{array}
	\right.
\end{align}
where $t\ge 0$ and $r\in(\iota, R)$.

 Define \(y(r,t)=\int_\iota^r \rho_\iota(s,t)s^{N-1}ds\) and \(\tau(r,t)=t\).  Then, in Lagrangian coordinates, the approximate system \eqref{App  Stm Euler'} becomes
\begin{align}\label{app system lag'}
	\left\{
	\begin{array}{l}
		(\rho_\iota)_\tau+\rho^2_\iota(r^{N-1}u_\iota )_y=0,\\
		r^{-(N-1)}(u_\iota)_\tau+(\rho_\iota^\gamma)_y-(\rho_\iota^2(r^{N-1}u_\iota)_y)_y+\frac{N-1}{r}(\rho_\iota)_yu_\iota=0,\\
		(\rho_\iota, u_\iota)|_{\tau=0}=(\rho_{\iota,0}, u_{\iota,0}),\\
		u_\iota(0,\tau)=u_\iota(1,\tau)=0,
	\end{array}
	\right.
\end{align}
where $\tau\ge0$ and $y\in(0,1)$. 

\subsection{Global existence of weak solutions}
As in Section \ref{Sec4}, we investigate the global solvability of the approximate system \eqref{app system lag'} in Lagrangian coordinates. Let \((\rho_\iota,u_\iota)\) be a uniqule local classical solution of system \eqref{app system lag'} defined on \([0,1]\times[0,T]\) for some fixed \(T>0\). To show that the approximate solution exists globally in time, it suffices to obtain an upper bound and a positive lower bound for \(\rho_\iota\).

We first establish the standard energy estimate and the B-D entropy estimate. Indeed, it suffices to note that \(n_N(1)=+\infty\) and to use the properties of the initial data \((\rho_{0,\iota},\mathbf{u}_{0,\iota})\) stated in \eqref{2d weak initial data'}.  Repeating the proofs of Propositions \ref{Prop WI n} and \ref{Prop W m}, we obtain the following Proposition.
\begin{prop}
	Assume that \eqref{wii alpha} holds. Then there exists a constant \(C>0\), independent of \(\iota\) and \(\tau\), such that
\begin{align}
	\int_0^1 (u_\iota^2+\rho_\iota^{\gamma-1})dy+\int_0^\tau\int_0^1\Big(\frac{u_\iota^{2}}{r^2}+\rho_\iota^{2} (\partial_y{u_\iota})^2r^{2(N-1)}\Big)dyd\tau \leq C,\label{u^2 weak'}
\end{align}
	and 
\begin{align}
		\int_0^1(r^{N-1}(\rho_\iota)_y)^2dy+\int_0^\tau\int_0^1\rho_\iota^{\gamma-1}(r^{N-1}(\rho_\iota)_y)^{2}dyd\tau\leq C.\label{W nablarho^2'}
\end{align}
\end{prop}

Next, using the facts that \( n_N(1) = +\infty \), \( q \in M_{\mathrm{set}} \), and that \eqref{2d weak initial data'} holds, and following the proofs of Propositions \ref{Prop W 2d k}–\ref{Prop W 3d RT}, we obtain higher integrability of the velocity field and of the density gradient, together with an upper bound on the density.
\begin{prop}\label{Prop wii 2p}
	Under the assumptions of Theorem \ref{T_2}, there exists a constant \(C(T)>0\), independent of \(\iota\), such that
\begin{align}\label{W 2d k'}
	\sup_{0\leq \tau\leq T}\int_0^1u_\iota^{2p}dy+\int_0^T\int_0^1\Big(\frac{u_\iota^{2p}}{r^2}
	+\rho_\iota^2(\partial_yu_\iota)^2 u_\iota^{2p-2}r^{2(N-1)}\Big)dyd\tau\leq C(T)
\end{align}
as well as
\begin{align}\label{W 2d l'}
	\sup_{0\leq \tau\leq T}\int_0^1 (r^{N-1}(\rho_\iota)_y)^{2q}dy+\int_0^T\int_0^1\rho_\iota^{\gamma-1}(r^{N-1}(\rho_\iota)_y)^{2q}dyd\tau\leq C(T)
\end{align}
and finally,
\begin{align}\label{W 2d RT'}
	\sup_{0\leq \tau\leq T}\norm{\rho_\iota(\tau)}_{L^\infty(0,1)}\leq C(T).
\end{align}
\end{prop}

Finally, we derive a positive lower bound for \(\rho_\iota\) that depends on \(\iota\).  It should be noted that Proposition \ref{Prop wii lower} below can be obtained by a slight modification of the proof of Proposition \ref{Prop wii VT}, or the reader may refer to \cite{guo-2012}.  To avoid repetition, we omit the proof of Proposition \ref{Prop wii lower}.
\begin{prop}\label{Prop wii lower}
	Under the assumptions of Theorem \ref{T_2}, there exists a constant \(C(\iota,T)>0\) such that
	\begin{align}\label{W 2d VT'}
		\sup_{0\leq \tau\leq T}\norm{\rho^{-1}_\iota(\tau)}_{L^\infty(0,1)}\leq C(\iota,T).
	\end{align}
\end{prop}

With \eqref{W 2d RT'} and \eqref{W 2d VT'} at hand, we can argue in a standard way that $(\rho_\iota,u_\iota)$ is a global classical solution on $[0,1]\times[0,\infty)$; see the previous section or \cite{Guo-2008} for details.  Using the $\iota$-independent estimates \eqref{W 2d k'}--\eqref{W 2d RT'} and standard compactness arguments, we establish the global existence of the weak solution stated in Theorem \ref{T_2}; see Section \ref{Sec4} or \cite{Guo-2008} for details.

\subsection{Some \(\iota\)-independent estimates}
In this section we follow the ideas of \cite{Jensen-2004} to derive several \(\iota\)-independent estimates, which play a key role in characterizing the vacuum states and the regularity of the solution in the flow region. To this end, we fix a sufficiently small constant $h > 0$. Then, for any $\iota > 0$, we define the particle path $r_h^{\iota}(t)$ by
\begin{equation}\label{353}
	h=\int_\iota^{r_h^{\iota} (t)}\rho_{\iota}(r,t)r^{N-1}dr,
\end{equation}
where $\rho_{\iota}(r,t)$ is the global classical solution to the approximate system \eqref{App  Stm Ball'} defined on $[\iota,R]\times[0,\infty)$. From \eqref{u^2 weak'}, one can easily check that there exists a constant $C(h) > 0$, independent of $\iota$ and $t$, satisfying
\begin{align}\label{wii 3}
	r_h^{\iota}(t) \ge (C(h))^{-1}.
\end{align}

Next, the following Proposition \ref{Prop wii VT} shows that $\rho_\iota$ admits a $\iota$-independent positive lower bound to the right of the particle path determined by $h$.
\begin{prop}\label{Prop wii VT}
	Under the assumptions of Theorem \ref{T_2}, there exists a constant $C(h/2, T)>0$, independent of $\iota$, such that for all $(y, \tau) \in [h, 1] \times [0, T]$,
	\begin{equation}\label{92}
		\rho_{\iota}(y, \tau)\geq (C(h/2, T))^{-1}.
	\end{equation}
\end{prop}
\begin{proof}
	Fix $\beta>2$. We define the following cut-off function $\phi_h(y)\in C([0,1])$:
	\begin{align}\label{wii 4}
		\phi_h(y)=\left\{
		\begin{array}{ll}
			0,& y\in[0,\frac{h}{2}],\\
			(y-\frac{h}{2})^\beta,& y\in [\frac{h}{2},\frac{3h}{4}],\\
			\text{linear connection},& y\in [\frac{3h}{4},h],\\
			1,&[h,1].
		\end{array}\right.
	\end{align}
	Next, we set
	\begin{align}
		v_\iota(y,\tau)=\frac{1}{({\rho_\iota} r_\iota^{N-1})(y,\tau)},\nonumber
	\end{align}
	then $\eqref{app system lag'}_1$ gives
	\begin{align}\label{74}
		\partial_\tau v_\iota^{\beta} = \beta v_\iota^{\beta -1} \partial_y(r_\iota^{N-1} {u_\iota}){r_\iota^{1-N}} - (N-1) \beta v_\iota^{\beta} {u_\iota} r_\iota^{-1}.
	\end{align}
	Multiplying the above equation by the cut-off function \(\phi_h(y)\), integrating the resulting equation over \([\frac{h}{2},1]\times[0,\tau]\), and using integration by parts together with boundary condition $\eqref{app system lag'}_4$, we obtain
	\begin{align}\label{wii 1}
		\begin{split}
			\int_\frac{h}{2}^1 v_\iota^\beta\phi_h dy(\tau)&=\int_\frac{h}{2}^1 v_\iota^\beta\phi_h dy(0)+\int_0^\tau \int_\frac{h}{2}^1 (\beta v^{\beta-1}_\iota \partial_y(r^{N-1}_\iota u_\iota) r^{1-N}_\iota-(N-1)\beta v_\iota^\beta u_\iota r^{-1}_\iota)\phi_h dy\\
			&=\int_\frac{h}{2}^1 v_\iota^\beta\phi_h dy(0)+\beta(\beta-1)\int_0^\tau\int_\frac{h}{2}^1v_\iota^\beta \partial_y\rho_\iota r_\iota^{N-1} u_\iota \phi_hdyds\\
			&\quad-\beta\int_0^\tau\int_\frac{h}{2}^1 v_\iota^{\beta-1}u_\iota\partial_y\phi_hdyds+(N-1)(\beta^2-\beta)\int_0^\tau\int_\frac{h}{2}^1 v_\iota^\beta u_\iota r_\iota^{-1}\phi_h dyds.
		\end{split}
	\end{align}
	From $\eqref{app system lag'}_1$ and $\eqref{app system lag'}_2$ we derive the B-D entropy equation  
	\begin{align*}
		(u_\iota+r^{N-1}\partial_y\rho_\iota)_\tau+r^{N-1}\partial_y (\rho_\iota^\gamma)=0,
	\end{align*}
	which, upon integration over $(0,s)$, gives
	\begin{align}\label{wii 2}
	r_\iota^{N-1}\partial_y\rho_\iota(y,s)=r_\iota^{N-1}\partial_y \rho_{\iota}(y,0)+u_\iota(y,0)-u_\iota(y,\tau)-\int_0^s \partial_y(\rho_\iota^\gamma) r_\iota^{N-1}(y,\xi)d\xi.
	\end{align}
	Substituting \eqref{wii 2} into \eqref{wii 1} yields
	\begin{align}
		\begin{split}
			&\quad\int_\frac{h}{2}^1 v_\iota^\beta\phi_h dy(\tau)\\
			&=\int_\frac{h}{2}^1 v_\iota^\beta\phi_h dy(0)+\beta(\beta-1)\int_0^\tau\int_\frac{h}{2}^1v_\iota^\beta \Big(r_\iota^{N-1}\partial_y \rho_{\iota}(y,0)+u_{0,\iota}(y)-u_\iota(y,s)\Big) u_\iota \phi_hdyds\\
			&\quad-\beta(\beta-1)\int_0^\tau\int_\frac{h}{2}^1v_\iota^\beta\left(\int_0^s\partial_y(\rho_\iota^\gamma)r_\iota^{N-1}(y,\xi) d\xi\right)u_\iota\phi_hdyds-\beta\int_0^\tau\int_\frac{h}{2}^1 v_\iota^{\beta-1}u_\iota\partial_y\phi_hdyds\\
			&\quad+(N-1)(\beta^2-\beta)\int_0^\tau\int_\frac{h}{2}^1 v_\iota^\beta u_\iota r_\iota^{-1}\phi_h dyds,
		\end{split}
	\end{align}
	which, together with Young's inequality, implies
	\begin{align}\label{wii 6}
		\begin{split}
			&\quad\int_\frac{h}{2}^1 v_\iota^\beta\phi_h dy(\tau)+\frac{\beta(\beta-1)}{2}\int_0^\tau\int_\frac{h}{2}^1 v_\iota^\beta u_\iota^2\phi_h dyds\\
			&\leq \int_\frac{h}{2}^1 v_\iota^\beta\phi_h dy(0)+C\int_0^\tau\int_\frac{h}{2}^1 v_\iota^\beta |r_\iota^{N-1}\partial_y \rho_{\iota}(y,0)+u_{0,\iota}(y)|^2\phi_hdyds\\
			&\quad +C\int_0^\tau\int_\frac{h}{2}^1 v_\iota^\beta\left(\int_0^s\partial_y(\rho_\iota^\gamma)r_\iota^{N-1}(y,\xi) d\xi\right)^2\phi_h dyds+C\int_0^\tau\int_\frac{h}{2}^1 v_\iota^{\beta-2}|\partial_y \phi_h|^2\phi_h^{-1}dyds\\
			&\quad +C\int_0^\tau\int_\frac{h}{2}^1 v_\iota^\beta r_\iota^{-2}\phi_h dyds\\
			&=\sum_{i=1}^5 W_i.
		\end{split}
	\end{align}
	Using \eqref{2d weak initial data'} and \eqref{wii 3}, we obtain
	\begin{align}\label{E W_1}
		|W_1|\leq C\norm{\rho_{\iota}^{-\beta}r_\iota^{\beta(1-N)}(0)}_{L^\infty((h/2,1))}\leq C(h/2). 
	\end{align}
	Similarly, from \eqref{ini nabla rho}, \eqref{ini nabla u} and \eqref{wii 3}, one gets
	\begin{align}\label{E W_2}
		\begin{split}
		|W_2|&\leq C\left(\norm{r_\iota^{N-1}\partial_y\rho_{\iota}(0)}_{L^\infty((h/2,1))}^2+\norm{u_{\iota,0}}_{L^\infty((h/2,1))}^2\right)\int_0^\tau\int_\frac{h}{2}^1 v_\iota^\beta \phi_h dyds\\
		&\leq C(h/2)\int_0^\tau\int_\frac{h}{2}^1 v_\iota^\beta \phi_h dyds.
		\end{split}
	\end{align}
	We now estimate the complex term $W_3$. We claim that for all $(y,s)\in [h/2,1]\times[0,T]$,
	\begin{align}\label{wii 5}
		\int_0^s|\partial_y(\rho_\iota^\gamma) r_\iota^{N-1}(y,\xi)|^pd\xi\leq C(h/2,T).
	\end{align}
	In fact, it follows from \eqref{wii 2}, \eqref{W 2d RT'}, \eqref{ini nabla rho}, \eqref{ini nabla u} and \eqref{wii 3} that
	\begin{align}
		\begin{split}
			&\int_0^s |\partial_y(\rho_\iota^\gamma) r_\iota^{N-1}(y,\xi)|^p
			d\xi=\gamma^p\int_0^s \rho_\iota^{p(\gamma-1)}|\partial_y\rho_\iota r_\iota^{N-1}|^p d\xi\\
			&=\gamma^p\int_0^s\rho_\iota^{p(\gamma-1)}| r_{\iota,0}^{N-1}\partial_y \rho_{\iota,0}+u_{0,\iota}-u_\iota-\int_0^\xi\partial_y(\rho_\iota^\gamma) r_\iota^{N-1}d\zeta|^pd\xi\\
			&\leq C\left(h/2,T\right)+C\int_0^s \norm{\rho_\iota^{p(\gamma-1)}u_\iota^p}_{L^\infty ((h/2, 1))}d\xi+C(T)\int_0^s\left(\int_0^\xi|\partial_y(\rho_\iota^\gamma) r_\iota^{N-1}(y,\zeta)|^pd\zeta\right)d\xi\\
			&\leq C\left(h/2,T\right)+C\int_0^s\int_{\frac{h}{2}}^1 \Big(\rho_\iota^{p(\gamma-1)}|u_\iota|^p+|\partial_y (\rho_\iota^{p(\gamma-1)})||u_\iota|^p+\rho_\iota^{p(\gamma-1)}|u_\iota|^{p-1}|\partial_y u_\iota|\Big)dyd\xi\\
			&\quad+C(T)\int_0^s\left(\int_0^\xi|\partial_y(\rho_\iota^\gamma) r_\iota^{N-1}(y,\zeta)|^pd\zeta\right)d\xi\\
			&\leq C\left(h/2,T\right)+C\int_0^s\int_{\frac{h}{2}}^1\Big(1+|u_\iota|^{2p}+|\partial_y(\rho_\iota^{p(\gamma-1)})|^2\Big)dyd\xi\\
			&\quad +C\int_0^s\int_{\frac{h}{2}}^1\Big(\rho_\iota^2|\partial_y u_\iota|^2u_\iota^{2p-2} r_\iota^{2(N-1)}+\rho_\iota^{2p(\gamma-1)-2}r_\iota^{-2(N-1)}\Big)dyd\xi\\
			&\quad+C(T)\int_0^s\left(\int_0^\xi|\partial_y(\rho_\iota^\gamma) r_\iota^{N-1}(y,\zeta)|^pd\zeta\right)d\xi\\
			&\leq C\left(h/2,T\right)+C(T)\int_0^s\left(\int_0^\xi|(\rho_\iota^\gamma)_y r_\iota^{N-1}(y,\zeta)|^pd\zeta\right)d\xi,
		\end{split}
	\end{align}
	where in the last inequality we have used \eqref{W 2d k'}, \eqref{W nablarho^2'}, \eqref{wii gamma} and \eqref{W 2d RT'}. Therefore, an application of Gronwall's inequality yields \eqref{wii 5}. Using \eqref{wii 5}, we obtain
	\begin{align}\label{E W_3}
		|W_3|\leq C(h/2,T)\int_0^\tau\int_\frac{h}{2}^1 v_\iota^\beta \phi_h dyds.
	\end{align}
	From the definition \eqref{wii 4} of the cut-off function, we obtain
\begin{align}\label{E W_4}
	\begin{split}
			|W_4|&\leq C\int_0^\tau\int_\frac{h}{2}^1 v_\iota^\beta \phi_h dyds+C\int_0^\tau\int_\frac{h}{2}^1 |\partial_y \phi_h|^\beta\phi_h^{1-\beta}dyds\\
		&\leq C\int_0^\tau\int_\frac{h}{2}^1 v_\iota^\beta \phi_h dyds+C(h/2,T).
	\end{split}
\end{align}
By \eqref{wii 3}, we get
	\begin{align}\label{E W_5}
		\begin{split}
			|W_5|&\leq C(h/2)\int_0^\tau\int_\frac{h}{2}^1 v_\iota^\beta \phi_h dyds.
		\end{split}
	\end{align}
Substituting \eqref{E W_1}--\eqref{E W_5} into \eqref{wii 6} gives
	\begin{align*}
		\begin{split}
			\int_\frac{h}{2}^1 v_\iota^\beta\phi_h dy(\tau)\leq C(h/2,T)+C(h/2,T)\int_0^\tau\int_\frac{h}{2}^1 v_\iota^\beta \phi_h dyds.
		\end{split}
	\end{align*}
Using Young's inequality, one has
\begin{align*}
	\sup_{0\leq \tau\leq T}\int_\frac{h}{2}^1 v_\iota^\beta\phi_h dy(\tau)\leq C(h/2,T),
\end{align*}
	which implies that
	\begin{align}\label{wii 7}
		\sup_{0\leq \tau\leq T}\int_h^1 v_\iota^\beta dy(\tau)\leq C(h/2,T).
	\end{align}
	Denote $V_{\iota,T,h}=\sup_{0\leq \tau\leq T}\norm{\rho_\iota^{-1}(\tau)}_{L^\infty((h,1))}+1$. We then obtain from the one-dimensional Sobolev embedding, \eqref{wii 3}, \eqref{W nablarho^2'} and \eqref{wii 7} that
\begin{align}
	\begin{split}
		\norm{v_\iota^\beta(\tau)}_{L^\infty((h,1))}&\leq \int_h^1 v_\iota^\beta dy+\int_h^1|\partial_y(v_\iota^\beta)|dy\\
		&\leq \int_h^1 v_\iota^\beta dy+C(h/2)\left(\int_h^1v_\iota^{\beta+1}|\partial_y \rho_\iota|dy+\int_h^1 v_\iota^{\beta+1}dy\right)\\
		&\leq C(h/2,T)V_{\iota,T,h}+C(h/2)\left(\int_h^1|\partial_y \rho_\iota|^2dy\right)^{\frac{1}{2}}\left(\int_h^1 v_\iota^{2\beta+2}dy\right)^{\frac{1}{2}}\\
		&\leq C(h/2,T)V_{\iota,T,h}^{\frac{\beta}{2}+1}.
	\end{split}
\end{align}
Taking the supremum over \(\tau\in[0,T]\) and using Young's inequality together with the fact that \(\beta>2\), we obtain \eqref{92}.

This completes the proof of Proposition \ref{Prop wii VT}.
\end{proof}

Then we show that on the right side of the particle path determined by \(h\), the velocity field satisfies a higher-order estimate whose bound is independent of \(\iota\).
\begin{prop}\label{Prop wii nabla u}
	Under the assumptions of Theorem \ref{T_2}, there exists a constant $C(h/4, T)>0$, independent of $\iota$, such that 
		\begin{align}\label{wii 13}
		\sup_{0\leq \tau\leq T}\int_{h}^1 (\partial_y u_\iota)^2dy+\int_0^T\int_{h}^1((\partial_\tau u_\iota)^2+(\partial_{yy}u_\iota)^2) dyd\tau\leq C(h/4,T).
	\end{align}
\end{prop}
\begin{proof}
	It follows from $\eqref{app system lag'}_2$ that
	\begin{align}\label{wii 9}
		(u_\iota)_\tau+(\rho_\iota^\gamma)_y r_\iota^{N-1}-(\rho_\iota^2 (r_\iota^{N-1}u_\iota)_y)_y r_\iota^{N-1} +(N-1)r_{\iota}^{N-2}(\rho_\iota)_y u_\iota=0.
	\end{align}
	Multiplying the above equation by \((u_\iota)_\tau \phi_h\), integrating over \([\frac{h}{2},1]\), and using the boundary condition $\eqref{app system lag'}$, we obtain
	\begin{align}\label{v2}
		\begin{split}
			\int_\frac{h}{2}^1(\partial_\tau u_\iota)^2\phi_h dy &+\int_\frac{h}{2}^1 \partial_y(\rho^\gamma_\iota)r_\iota^{N-1}\partial_\tau u_\iota \phi_hdy\\
			&+\int_\frac{h}{2}^1\rho_\iota^2(r_\iota^{N-1}u_\iota)_y(  r_\iota^{N-1}\partial_\tau u_\iota\phi_h)_ydy+(N-1)\int_\frac{h}{2}^1 r_{\iota}^{N-2}\partial_y \rho_\iota u_\iota \partial_\tau u_\iota \phi_h dy=0.
		\end{split}
	\end{align}
		Then a direct calculation gives
	\begin{align}\label{v1}
		\begin{split}
			&\quad\int_\frac{h}{2}^1 \rho_\iota^2(r_\iota^{N-1}u_\iota)_y( r_\iota^{N-1}\partial_\tau u_\iota\phi_h)_ydy\\
			&=\int_\frac{h}{2}^1\left(\rho_\iota^2 r_\iota^{N-1}\partial_y u_\iota+\frac{(N-1)\rho_\iota u_\iota}{ r_\iota}\right)( r_\iota^{N-1}\partial_\tau u_\iota\phi_h)_ydy\\
			&=(N-1)\int_\frac{h}{2}^1 r_\iota^{N-2}\rho_\iota \partial_y u_\iota \partial_\tau u_\iota \phi_hdy+\int_\frac{h}{2}^1 \partial_y u_\iota \partial_{y\tau}u_\iota\rho_\iota^2 r_\iota^{2(N-1)}\phi_h dy\\
			&\quad +\int_\frac{h}{2}^1r_\iota^{2(N-1)}\partial_y u_\iota \partial_\tau u_\iota \rho_\iota^2\partial_y\phi_hdy-(N-1)\int_\frac{h}{2}^1\left(\frac{\rho_\iota u_\iota}{ r_\iota}\right)_y r_\iota^{N-1}\partial_\tau u_\iota\phi_hdy\\
			&=\frac{1}{2}\frac{d}{d\tau}\int_\frac{h}{2}^1 (\partial_y u_\iota)^2\rho_\iota^2 r_\iota^{2(N-1)}\phi_h dy-\frac{1}{2}\int_\frac{h}{2}^1(\partial_y u_\iota)^2\partial_\tau(\rho_\iota^2 r_\iota^{2(N-1)})\phi_h dy\\
			&\quad +(N-1)\int_\frac{h}{2}^1  r_\iota^{N-2}\rho_\iota\partial_y u_\iota \partial_\tau u_\iota \phi_hdy+\int_\frac{h}{2}^1 r_\iota^{2(N-1)}\partial_y u_\iota \partial_\tau u_\iota \rho_\iota^2\partial_y\phi_hdy\\
            &\quad -(N-1)\int_\frac{h}{2}^1\left(\frac{\rho_\iota u_\iota}{ r_\iota}\right)_y r_\iota^{N-1}\partial_\tau u_\iota\phi_hdy.
		\end{split}
	\end{align}
		Substituting (\ref{v1}) into (\ref{v2}), we have
	\begin{align}\label{v9}
		\begin{split}
			&\quad\frac{1}{2}\frac{d}{d\tau}\int_\frac{h}{2}^1 (\partial_y u_\iota)^2\rho_\iota^2 r_\iota^{2(N-1)}\phi_h dy+\int_\frac{h}{2}^1(\partial_\tau u_\iota)^2\phi_h dy \\
			&=-\int_\frac{h}{2}^1 \partial_y(\rho^\gamma_\iota)r_\iota^{N-1}\partial_\tau u_\iota \phi_hdy+\frac{1}{2}\int_\frac{h}{2}^1(\partial_y u_\iota)^2\partial_\tau(\rho_\iota^2 r_\iota^{2(N-1)})\phi_h dy\\
			&\quad-(N-1)\int_\frac{h}{2}^1 r_\iota^{N-2}\rho_\iota\partial_y u_\iota \partial_\tau u_\iota \phi_hdy-\int_\frac{h}{2}^1 r_\iota^{2(N-1)}\partial_y u_\iota \partial_\tau u_\iota \rho_\iota^2\partial_y\phi_hdy\\
            &\quad+(N-1)\int_\frac{h}{2}^1\left(\frac{\rho_\iota u_\iota}{ r_\iota}\right)_y r_\iota^{N-1}\partial_\tau u_\iota\phi_hdy-(N-1)\int_\frac{h}{2}^1 r_\iota^{N-2}\partial_y \rho_\iota u_\iota \partial_\tau u_\iota \phi_\iota dy\\
			&=\sum_{i=1}^{6}Z_i.
		\end{split}
	\end{align}
	Next, we estimate \(Z_1\)--\(Z_6\).  It follows from \eqref{wii 3}, \eqref{W 2d RT'}, and \eqref{92} that for any $(y,\tau)\in [h/2,1]\times[0,T]$,
	\begin{align}\label{wii 8}
		\begin{split}
		&C(h/2)\leq r_\iota(y,\tau)\leq R,\\
		&(C(h/4,T))^{-1}\leq \rho_\iota(y,\tau)\leq C(T).
		\end{split}
	\end{align}
	Therefore, Young's inequality together with \eqref{u^2 weak'}, \eqref{W nablarho^2'}, \eqref{wii 4} and \eqref{wii 8} implies that
	\begin{align}\label{v7}
		\begin{split}
			&\quad|Z_1|+|Z_3|+|Z_4|+|Z_5|+|Z_6|\\
			&\leq \frac{1}{4}\int_\frac{h}{2}^1(\partial_\tau u_\iota)^2\phi_h dy+C(h/4,T)\int_\frac{h}{2}^1\Big(|\partial_y
			\rho_\iota|^2+|\partial_y u_\iota|^2+|\partial_y\rho_\iota|^2|u_\iota|^2+|u_\iota|^2\Big)\phi_hdy\\
			&\quad+C(h/4,T)\int_\frac{h}{2}^1 |\partial_y u_\iota|^2|\partial_y\phi_h|^2\phi_h^{-1}dy\\
			&\leq \frac{1}{4}\int_\frac{h}{2}^1(\partial_\tau u_\iota)^2\phi_h dy+C(h/4,T)\left(1+\int_{\frac{h}{2}}^1|\partial_yu_\iota|^2dy\right),
		\end{split}
	\end{align}
	where in the last inequality we have used the fact that
	\begin{align}\label{wii 11}
		\norm{u_\iota}_{L^\infty((h/2,1))}^2\leq C\int_{\frac{h}{2}}^1|u_\iota|^2dy+C\int_{\frac{h}{2}}^1|\partial_yu_\iota|^2dy.
	\end{align}
	Finally, we estimate the complex term $Z_2$. A direct calculation, together with \eqref{wii 8}, shows that
	\begin{align}\label{wii 10}
		|Z_2|\leq C(h/4,T)\left(\int_{\frac{h}{2}}^1|\partial_y u_\iota|^3\phi_h dy+\int_{\frac{h}{2}}^1|\partial_y u_\iota|^2|u_\iota|\phi_h dy\right).
	\end{align}
	From \eqref{wii 9} we obtain  
	\begin{align}
		\left(\rho_\iota^2\left(r_\iota^{N-1}\partial_y u_\iota+\frac{(N-1)u_\iota}{\rho_\iota r_\iota}\right)\right)_yr_\iota^{N-1}=(u_\iota)_\tau +r^{N-1}_\iota(\rho_\iota^\gamma)_y+(N-1)r_\iota^{N-2}(\rho_\iota)_yu_\iota,\nonumber
	\end{align}
	which implies that
		\begin{align}
		\rho_\iota^2 r_\iota^{2(N-1)}\partial_{yy}u_\iota&=-\left[(\rho_\iota^2 r_\iota^{N-1})_y\partial_y u_\iota+(N-1)\left(\frac{u_\iota\rho_\iota}{r_\iota}\right)_y\right] r_\iota^{N-1}\\
        &\quad+(u_\iota)_\tau+r_\iota^{N-1}(\rho^\gamma_\iota)_y+(N-1)r_\iota^{N-2}(\rho_\iota)_yu_\iota.\nonumber
	\end{align}
	Thus, it deduces from \eqref{wii 8} that, for any $(y,\tau)\in [h/2,1]\times[0,T]$,
	\begin{align}
		|\partial_{yy} u_\iota(y,\tau)|\leq C(h/4,T)(|\partial_y \rho_\iota|+1)(|\partial_ yu_\iota|+|u_\iota|+1)+C(h/4,T)|\partial_\tau u_\iota|,\nonumber
	\end{align}
	and using \eqref{u^2 weak'}, \eqref{W nablarho^2'}, \eqref{wii 4}, and \eqref{wii 8}, we obtain
		\begin{align}\label{v5}
		\begin{split}
			&\quad\int_\frac{h}{2}^1 |\partial_{yy} u_\iota|^2\phi_h dy\\
			&\leq C(h/4,T)\int_\frac{h}{2}^1 |\partial_\tau u_\iota|^2\phi_h dy+C(h/4,T)\int_\frac{h}{2}^1(|\partial_y \rho_\iota|^2+1)(|\partial_y u_\iota|^2+|u_\iota|^2+1)\phi_h dy\\
			&\leq C(h/4,T)\int_\frac{h}{2}^1 |\partial_\tau u_\iota|^2\phi_h dy+C(h/4,T)\left(1+\int_{\frac{h}{2}}^1|\partial_y u_\iota|^2dy+\int_\frac{h}{2}^1|\partial_y \rho_\iota|^2|\partial_y u_\iota|^2\phi_hdy\right)\\
			&\leq  \frac{1}{2}\int_\frac{h}{2}^1 |\partial_{yy} u_\iota|^2\phi_h dy+C(h/4,T)\int_\frac{h}{2}^1 |\partial_\tau u_\iota|^2\phi_h dy+C(h/4,T)\left(1+\int_{\frac{h}{2}}^1|\partial_y u_\iota|^2dy\right),
		\end{split}
	\end{align}
	where in the last inequality we have used \eqref{W 2d l'}, \eqref{wii 8}, and \eqref{wii 4} to obtain
		\begin{align}
		\begin{split}
			&C(h/4,T)\int_\frac{h}{2}^1|\partial_y \rho_\iota|^2|\partial_y u_\iota|^2\phi_hdy\\
			&\leq C(h/4,T)\left(\int_{\frac{h}{2}}^1|\partial_y \rho_\iota|^{2q}dy\right)^{\frac{1}{q}}\left(\int_\frac{h}{2}^1|\partial_y u_\iota|^{\frac{2q}{q-1}}\phi_h^{\frac{q}{q-1}} dy\right)^{\frac{q-1}{q}}\\
			&\leq C(h/4,T)\norm{(\partial_y u_\iota)^2 \phi_h}_{L^\infty}^{\frac{1}{q}}\left(\int_\frac{h}{2}^1|\partial_y u_\iota|^2\phi_hdy\right)^{\frac{q-1}{q}}\\
			&\leq C(h/4,T)\int_\frac{h}{2}^1\Big((\partial_y u_\iota)^2 \phi_h+|\partial_{yy}u_\iota||\partial_y u_\iota|\phi_h+(\partial_yu_\iota)^2|\partial_y\phi_h|\Big)dy+C(h/4,T)\int_{h/2}^1|\partial_y u_\iota|^2dy\\
			&\leq \frac{1}{2}\int_\frac{h}{2}^1 |\partial_{yy} u_\iota|^2\phi_h dy+C(h/4,T)\int_{h/2}^1|\partial_y u_\iota|^2dy.\nonumber
		\end{split}
	\end{align}
	We obtain from \eqref{v5} that
	\begin{align}\label{wii 12}
		\int_\frac{h}{2}^1 |\partial_{yy} u_\iota|^2\phi_h dy\leq C(h/4,T)\int_\frac{h}{2}^1 |\partial_\tau u_\iota|^2\phi_h dy+C(h/4,T)\left(1+\int_{h/2}^1|\partial_y u_\iota|^2dy\right).
	\end{align}
	Now, based on \eqref{wii 12}, \eqref{wii 11}, \eqref{wii 4}, and \eqref{wii 10}, it is sufficient to estimate $Z_2$,
		\begin{align}\label{v6}
		\begin{split}
			|Z_2|&\leq C(h/4,T)\norm{\partial_y u_\iota \phi_h^{\frac{1}{2}}}_{L^\infty((h/2,1))}\int_\frac{h}{2}^1|\partial_y u_\iota|^2\phi_h^{\frac{1}{2}}dy+C(h/2,T)\norm{u_\iota}_{L^\infty((h/2,1))}\int_\frac{h}{2}^1|\partial_y u_\iota|^2\phi_h dy\\
			&\leq \varepsilon\norm{(\partial_y u_\iota)^2\phi_h}_{L^\infty((h/2,1))}+C(h/4,T)\left(1+\int_{h/2}^1|\partial_y u_\iota|^2dy\right)\left(1+\int_\frac{h}{2}^1|\partial_y u_\iota|^2\phi_h dy\right)\\
			&\leq \varepsilon\int_\frac{h}{2}^1 |\partial_{yy} u_\iota|^2\phi_h dy+C(h/4,T)\left(1+\int_{h/2}^1|\partial_y u_\iota|^2dy\right)\left(1+\int_\frac{h}{2}^1|\partial_y u_\iota|^2\phi_h dy\right)\\
			&\leq \frac{1}{4}\int_\frac{h}{2}^1 |\partial_\tau u_\iota|^2\phi_h dy+C(h/4,T)\left(1+\int_{h/2}^1|\partial_y u_\iota|^2dy\right)\left(1+\int_\frac{h}{2}^1|\partial_y u_\iota|^2\phi_h dy\right),
			\end{split}
	\end{align}
	where \(\varepsilon\) is a sufficiently small positive constant depending on \(h/4\) and \(T\). Substituting \eqref{v7}, \eqref{v6} and \eqref{wii 12} into \eqref{v9} gives
	\begin{align*}
		\begin{split}
			&\quad\frac{d}{d\tau}\int_\frac{h}{2}^1 (\partial_y u_\iota)^2\rho_\iota^2 r_\iota^{2(N-1)}\phi_h dy+\int_\frac{h}{2}^1|\partial_\tau u_\iota|^2\phi_h dy+\int_\frac{h}{2}^1 |\partial_{yy} u_\iota|^2\phi_h dy\\
			&\leq C(h/4,T)\left(1+\int_{h/2}^1|\partial_y u_\iota|^2dy\right)\left(1+\int_\frac{h}{2}^1|\partial_y u_\iota|^2\phi_h dy\right),
		\end{split}
	\end{align*}
	which, together with Gronwall's inequality, \eqref{ini nabla u}, \eqref{wii 8}, and \eqref{u^2 weak'}, yields \eqref{wii 13}.
\end{proof}

Finally, we establish an \(\iota\)-independent bound for the density gradient on the right-hand side of the particle path determined by \(h\).
\begin{prop}\label{Prop nabla rho}
	Under the assumptions of Theorem \ref{T_2}, there exists a constant $C(h/4, T)>0$, independent of $\iota$, such that 
	\begin{align}\label{wii 14}
		\sup_{0\leq \tau\leq T}\norm{u_\iota}_{L^\infty((h,1))}+\sup_{0\leq \tau\leq T}\norm{\partial_y\rho_\iota}_{L^\infty((h,1))}\leq C(h/4,T).
	\end{align}
\end{prop}
\begin{proof}
	Using the one-dimensional Sobolev embedding, together with \eqref{u^2 weak'} and \eqref{wii 13}, we obtain
	\begin{align}\label{wii 15}
		\begin{split}
			\sup_{0\leq \tau\leq T}\norm{u_\iota}_{L^\infty((h,1))}\leq C\sup_{0\leq \tau\leq T}\left(\int_{h}^1 u_\iota^2dy+\int_{h}^1|\partial_y u_\iota|^2dy\right)\leq C(h/4,T).
		\end{split}
	\end{align}
	From \eqref{wii 2}, \eqref{wii 3}, \eqref{ini nabla rho}, \eqref{ini nabla u}, \eqref{wii 15}, and \eqref{wii 5}, it follows that
	\begin{align}
		\begin{split}
			&\quad\sup_{0\leq \tau\leq T}\norm{\partial_y\rho_\iota}_{L^\infty((h,1))}\\
			&\leq C(h)\sup_{0\leq \tau\leq T}\norm{\partial_y\rho_\iota r_\iota^{N-1}(y,0)+u_\iota(y,0)-u_\iota(y,\tau)+\int_0^\tau \partial_y(\rho_\iota^\gamma)r_\iota^{N-1}(y,s)ds}_{L^\infty((h,1))}\\
			&\leq C(h/4,T).
		\end{split}
	\end{align}
This completes the proof of Proposition \ref{Prop nabla rho}.
\end{proof}

Now, combining the estimates of Proposition \ref{Prop wii 2p}, \ref{Prop wii VT}--\ref{Prop nabla rho}, the we change all the estimates back to the Eulerian coordinates. That is, it holds that
\begin{align}
	&r_h^\iota(t)\ge (C(h))^{-1}, \ t\in[0,T]\label{429};\\
	&\ 0<(C(h/2,T))^{-1}\leq\rho_\iota(r,t)\leq C(T),\ (r,t)\in\big[r_h^{\iota} (t),R\big]\times[0,T];\label{430}\\
	&\sup_{t\in[0,T]}\int_\iota^R\rho_\iota u_\iota^{2p} r^{N-1} dr+\int_{0}^{T} \int_\iota^R \big(\rho_\iota u_\iota^{2p}r^{N-3}+{\rho_\iota u_\iota^{2p-2}(\partial_{r} u_\iota)^2}{r^{N-1}}\big)drdt\leq C(T);\label{432}\\
	&\sup_{t\in[0,T]}\int_\iota^R|\partial_r \rho_\iota^{\frac{1}{2q}}|^{2q}r^{N-1}
	dr+\int_0^T\int_\iota^R|\partial_r\rho_\iota^{\frac{\gamma}{2q}}|^{2q}r^{N-1}drdt\leq C(T);\label{431}\\
	&\sup_{t\in[0,T]}\int_{r_h^{\iota} (t)}^R\Big((\partial_{r} u_\iota)^2+(\partial_t \rho_\iota)^2 \Big)dr+\int_0^{T} \int_{r_h^\iota
		(t)}^R{ \Big((\partial_t u_\iota)^2+(\partial_{r r}u_\iota)^2 \Big)}drdt\leq C(h/4,T);\label{433}\\
	&\sup_{t\in[0,T]}\norm{u_\iota}_{ L^\infty((r_h^{\iota} (t),R))}+\sup_{t\in[0,T]}\norm{\partial_{r}\rho_\iota}_{L^\infty((r_h^{\iota} (t),R))}\leq C(h/4,T),\label{434}
\end{align}
where we have used
\begin{align}
	\begin{split}
		\sup_{t\in[0,T]}\int_{r^\iota_h(t)}^R (\partial_{r}u_\iota)^2dr&\leq C(h/4,T)\sup_{\tau\in[0,T]}\int_h^1(\partial_y u_\iota)^2dy\leq C(h/4,T);\\
		\sup_{t\in[0,T]}\int_{r^\iota_h(t)}^R (\partial_{t}\rho_\iota)^2dr&\leq C(h/4,T)\sup_{0\leq t\leq T}\int_{r_h^{\iota} (t)}^R\Big(|\partial_r\rho_\iota|^2|u_\iota|^2+|\partial_r u_\iota|^2+|u_\iota|^2\Big)dr\leq C(h/4,T);\\
		\int_0^T\int_{r^\iota_h(t)}^R (\partial_tu_\iota)^2drdt&\leq C(h/4,T) \int_0^T\int_h^1\Big((\partial_\tau u_\iota)^2+u_\iota^2(\partial_y u_\iota)^2\Big)dyd\tau\leq C(h/4,T);\\
		\int_0^T\int_{r^\iota_h(t)}^R (\partial_{r r}u_\iota)^2drdt&\leq C(h/4,T)\int_0^T\int_h^1\Big((\partial_{yy}u_\iota)^2+(\partial_y u_\iota)^2(\partial_y \rho_\iota)^2+(\partial_yu_\iota)^2\Big)dyd\tau\\
        &\leq C(h/4,T).
	\end{split}
\end{align}

\subsection{The proof of Theorem \ref{T_2}}
The following Proposition implies that the particle paths $r_h^\iota (t) $ are convergent.
	\begin{prop}\label{P_1}
	Under the assumptions of Theorem \ref{T_2}, let $(\rho_\iota,u_\iota)$ and $r_h^\iota (t)$ be as described above. Then there exists a subsequence $\iota_i$ such that $r_h^{\iota_i} (t)$ converges uniformly for $(h,t)$ in compact subset of $(0,1]\times[0,\infty)$,
	\begin{equation}\label{345}
		\lim_{\iota_i\rightarrow 0^+}r_h^{\iota_i} (t)= r_h(t);
	\end{equation}
	the limit $r_h(t)$ is $H\ddot{o}lder$ continuous on these compact sets, and if
	\begin{equation}\label{340}
		\underline{r}(t)=\lim_{h\rightarrow 0^+} r_h(t),
	\end{equation}
	then $\underline{r}(t)$ is a upper semi-continuous curve and $\lim_{t\rightarrow 0} {\underline{r}(t)}=0.$
\end{prop}
\begin{proof}
	Fix \(h_0\in(0,1)\) and \(T>0\). For any $h_0\leq h_1<h_2\leq 1$ and $0\leq t_1< t_2\leq T$, it follows from \eqref{429}, \eqref{430}, and \eqref{434} that
	\begin{align*}
		\begin{split}
			|r^{\iota_i}_{h_1}(t_1)-r_{h_2}^{\iota_i}(t_2)|&\leq |r^{\iota_i}_{h_1}(t_1)-r_{h_2}^{\iota_i}(t_1)|+|r^{\iota_i}_{h_2}(t_1)-r_{h_2}^{\iota_i}(t_2)|\\
			&\leq \int_{h_1}^{h_2}|\partial_h r_h^{\iota_i} (t_1)|dh+\int_{t_1}^{t_2}|\partial_t r_{h_2}^{\iota_i}(t)|dt\\
			&\leq \int_{h_1}^{h_2}(\rho_{\iota_i}(r_h^{\iota_i}(t_1),t_1)(r_h^{\iota_i}(t_1))^{N-1})^{-1}dh+\int_{t_1}^{t_2}|u_{\iota_i}(r_{h_2}^{\iota_i}(t),t)|dt\\
			&\leq C(h_0/4,T)(|h_2-h_1|+|t_2-t_1|),
		\end{split}
	\end{align*}
	where the third inequality follows from the definition \eqref{353} of \(r_h^{\iota_i}(t)\). Therefore, by the Ascoli–Arzelà theorem, \(r_h^{\iota_i}(t)\) converges uniformly on compact sets, and the limit function \(r_h(t)\) is Hölder continuous on compact sets. 
	
	Next, the monotonicity of \(r_h(t)\) with respect to \(h\in[0,1]\) ensures that \(\underline{r}(t)\) is well-defined. Fix \(t>0\). For any \(s>0\) and sufficiently small \(h>0\), we have
		\begin{align}
		\begin{split}
		{\underline{r}(t)}&={\underline{r}(s)}+(r_h(s)-{\underline{r}(s)})+(r_h(t)-r_h(s))+({\underline{r}(t)}-r_h(t))\\
		&\ge{\underline{r}(s)}+(r_h(t)-r_h(s))+({\underline{r}(t)}-r_h(t)),
		\end{split}
	\end{align}
	which, together with the definition of \(\underline{r}(t)\) and the continuity of \(r_h(\cdot)\), implies that
	\begin{align}
		\underline{r}(t)\ge \limsup_{s\to t} \underline{r}(s).
	\end{align}
	This shows that \(\underline{r}(\cdot)\) is a upper semi-continuous curve.
	
	Finally, for sufficiently small \(h>0\) we have
	\begin{align}
		\begin{split}
			\underline{r}(t)&=\underline{r}(t)-r_h(t)+(r_h(t)-r_h^{\iota_i}(t))+r_h^{\iota_i}(t)\\
			&\leq (r_h(t)-r_h^{\iota_i}(t))+r_h^{\iota_i}(t),
		\end{split}
	\end{align}
	which, together with the uniform convergence of \(r_h^{\iota_i}(\cdot)\) on \([0,T]\), implies that
	\begin{align}
		\lim\limits_{t\to0}\underline{r}(t)=0.
	\end{align}
\end{proof}

Extend \((\rho_{\iota_i},u_{\iota_i})\) constantly from \([\iota_i,R]\times[0,T]\) to \([0,R]\times[0,T]\) and still denote the extended functions by \((\rho_{\iota_i},u_{\iota_i})\).  The following proposition shows that the region to the left of \(\underline{r}(\cdot)\) is actually the vacuum region.
\begin{prop}\label{P_3}
Let the hypotheses and notations of Proposition \ref{P_1} be in force. Then there exists a further subsequence, still denoted by $\iota_i$, such that for any $n\in\mathbb{N}^+$,
\begin{equation}\label{327'}
	\rho_{\iota_i}(r,t) r^{N-1} \to \rho(r,t) r^{N-1} \quad \text{as } \iota_i \to 0^+,
\end{equation}
uniformly in $C([1/n,R] \times [0,T])$. Furthermore, for any $t \in [0,T]$, it holds that
\begin{align}
	\rho(r,t) r^{N-1} = 0 \text{ on } [0, \underline{r}(t)], \label{wii 20}\\
	\rho(r,t) > 0  \text{ on } (\underline{r}(t), R]. \label{wii 21}
\end{align}
\end{prop}

\begin{proof}
	Using \eqref{431}, we can prove \eqref{327'} by the same method as in Proposition \ref{Prop W rho RT}. 
	
	It follows from \eqref{430} and \eqref{353} that
		\begin{align}
		&\int_{0}^{\underline{r}(t)}{{\rho_{\iota_i}(r,t)}r^{N-1}}dr\nonumber\\
		&=\Big(\int_0^{\iota_i}+\int_{\iota_i}^{r_h^{\iota_i}(t)}+\int_{r_h^{\iota_i}(t)}^{r_h (t)}+\int_{r_h (t)}^{\underline{r}(t)}\Big){{\rho_{\iota_i}(r,t)}r^{N-1}}dr\nonumber\\
		&\leq C(T){\iota_i}^N+\int_{\iota_i}^{r_h^{\iota_i}(t)}{{\rho_{\iota_i}(r,t)}r^{N-1}}dr\nonumber\\
		&\ \ + C\left(\int_{\iota_i}^R \rho_{\iota_i}^\gamma(r,t) r^{N-1}dr\right)^{\frac{1}{\gamma}}\Big[\Big({r_h (t)}^N-{r_h^{\iota_i}(t)}^N\Big)^{\frac{\gamma-1}{\gamma}}+\Big({\underline{r}(t)}^N-{r_h (t)}^N\Big)^{\frac{\gamma-1}{\gamma}}\Big]\nonumber\\
		&\leq C{\iota_i}^N+h+C\Big({r_h (t)}^N-{r_h^{\iota_i}(t)}^N\Big)^{\frac{\gamma-1}{\gamma}}+C\Big({\underline{r}(t)}^N-{r_h (t)}^N\Big)^{\frac{\gamma-1}{\gamma}}.\nonumber
	\end{align}
	Letting \(\iota_i\to 0\) and then \(h\to 0\) gives
	\begin{align*}
		\int_{0}^{\underline{r}(t)}{{\rho(r,t)}r^{N-1}}dr=0,
	\end{align*}
	 which shows \eqref{wii 20}.
	 
	 Finally, for any \(r>\underline{r}(t)\), there exists \(0<h\) such that \(\underline{r}(t)<r_h(t)<r\).  By \eqref{430} and \eqref{345} we have \(\rho_{\iota_i}(r,t)\ge (C(h/2,T))^{-1}\) for all sufficiently small \(\iota_i\), which together with \eqref{327'} implies $\rho(r,t)\ge (C(h/2,T))^{-1}$.
	 
	 This completes the proof of Proposition \ref{P_3}.
\end{proof}

Define
\begin{equation}\label{336}
	\mathcal{F}\triangleq\big\{(r,t)\mid\underline{r}(t)<r\leq R,\ 0\leq t<\infty\big\}.
\end{equation}
The following Proposition \ref{P_6} provides a precise characterization of the flow region.
\begin{prop}\label{P_6}
	For any $(r,t)\in \mathcal{F}$, there exist two small positive constants $\tilde{r}$ and $\tilde{t}$ such that
	\begin{align}\label{wii 16}
		\begin{split}
			[r-\tilde{r},R]\times[(t-\tilde{t})_+,t+\tilde{t}]\subset(\underline{r}(\cdot)+\tilde{r},R]\times[(t-\tilde{t})_+,t+\tilde{t}]\subset \mathcal{F}.
		\end{split}
	\end{align}
	Hence, $\mathcal{F}\cap \{t>0\}\cap\{r<R\}$ is open.
\end{prop}
\begin{proof}
	Set $\tilde{r}=\frac{1}{3}(r-\underline{r}(t))$. Since $\underline{r}(\cdot)$ is an upper semicontinuous curve, we have
	\begin{align}
		\underline{r}(t)\ge \inf_{\delta>0}\sup_{s\in B_\delta(t)}\underline{r}(s),
	\end{align}
	 which shows that there exists a sufficiently small positive constant $\tilde{t}$ such that
	 \begin{align}
	 	\underline{r}(t)+\tilde{r}>\sup_{s\in B_{\tilde{t}}(t)}\underline{r}(s).
	 \end{align}
	 That is, for all $s\in[(t-\tilde{t})_+,t+\tilde{t}]$, we obtain
	 \begin{align}
	 	\underline{r}(s)+\tilde{r}<\underline{r}(t)+2\tilde{r}\leq  r-\tilde{r},
	 \end{align}
	 which implies \eqref{wii 16}. We have thus proved Proposition \ref{P_6}.
\end{proof}

\begin{prop}\label{P_7}
	Under the hypotheses of Proposition \ref{P_6}, there exist a sufficiently small \(h>0\) and a constant $T>0$ such that
	\begin{align}\label{wii 17}
		[r-\tilde{r},R]\times[(t-\tilde{t})_+,t+\tilde{t}]\subset(r_h(\cdot),R]\times[0,T].
	\end{align}
	Furthermore, for any \(h'<h\), there exists a sufficiently small \(\iota_{i_0}\) such that for all \(0<\iota_i<\iota_{i_0}\) we have
	\begin{align}\label{wii 18}
		r_{h'}^{\iota_i}(s)<r_h(s),\quad \forall s\in[0,T].
	\end{align}
\end{prop}
\begin{proof}
	We define 
	\begin{align}
		g_h(s)=(r_h(s)-(r-2\tilde{r}))_+,\quad s\in [(t-\tilde{t})_+,t+\tilde{t}].
	\end{align}
	Since \(g_h(\cdot)\) is monotonically decreasing in $h$, Proposition \ref{P_6} implies that
	\begin{align*}
		\lim_{h\rightarrow0}g_h(s)=\lim_{h\rightarrow0}(r_h(s)-(r-2\tilde{r}))_{+}=0,
	\end{align*}
	which, together with the continuity of \(g_h(\cdot)\) and Dini's theorem, implies that
	\begin{align}
			g_h(s)\rightrightarrows0,\ s\in[(t-\tilde{t})_+,t+\tilde{t}].
	\end{align}
	Thus, there exists small suitably $h>0$ such that, for all $s\in[(t-\tilde{t})_+,t+\tilde{t}]$,
	\begin{align}
		g_h(s)< \tilde{r},
	\end{align}
	which implies that
	\begin{align}
		r_h(s)< r-\tilde{r}.
	\end{align}
	This establishes \eqref{wii 17}.  Next, we prove \eqref{wii 18}. Hölder's inequality shows that for any \(\iota_i>0\) and $s\in[0,T]$ we have
	\begin{align}
		\begin{split}
			h-h'=&\int_{r_{h'}^{\iota_i}(s)}^{r_h^{\iota_i}(s)}\rho_{\iota_i}(r,s)r^{N-1}dr\\
			&\leq C\left(\int_{r_{h'}^{\iota_i}(s)}^{r_h^{\iota_i}(s)}\rho_{\iota_i}^\gamma(r,s)r^{N-1}dr\right)^{\frac{1}{\gamma}}\left((r_{h}^{\iota_i}(s))^N-(r_{h'}^{\iota_i}(s))^N\right)^{\frac{\gamma-1}{\gamma}}\\
			&\leq C\left((r_{h}^{\iota_i}(s))^N-(r_{h'}^{\iota_i}(s))^N\right)^{\frac{\gamma-1}{\gamma}},
		\end{split}
	\end{align}
	which is equivalent to
	\begin{align}
		(r_{h}^{\iota_i}(s))^N-(r_{h'}^{\iota_i}(s))^N\ge C(h-h')^{\frac{\gamma}{\gamma-1}}.
	\end{align}
	Letting \(\iota_i\to 0\), we obtain
	\begin{align}\label{wii 19}
		r_{h}(s)-r_{h'}(s)\ge (C(h',h))^{-1}>0.
	\end{align}
	Therefore, using Proposition \ref{P_1}, we obtain \eqref{wii 18}. We have proved Proposition \ref{P_7}.
\end{proof}

\begin{prop}\label{P_2}
	Let the hypotheses and notations of Propositions \ref{P_1}-\ref{P_7} be in force. Then there is a further subsequence, still denoted by $\iota_i$, such that
	\begin{equation}\label{327}
		{u_{\iota_i}}(r,t)\rightarrow u(r,t),\ {\rho_{\iota_i}}(r,t)\rightarrow \rho(r,t),\ as\ \iota_i\rightarrow 0^+,
	\end{equation}
	uniformly on $[r-\tilde{r},R]\times[(t-\tilde{t})_+,t+\tilde{t}]$, and the limiting function $(\rho,u)$ is $H\ddot{o}lder$ continuous on these sets.
\end{prop}
\begin{proof}
	This proof can be found in \cite[Proposition 4.3]{wang-2018}, we sketch it here for completeness.

	Using Proposition \ref{P_7}, we obtain small constants \(h>0\) and \(T>0\) such that
	\begin{align*}
		[r-\tilde{r},R]\times[(t-\tilde{t})_+,t+\tilde{t}]\subset (r_h(\cdot),R]\times[0,T].
	\end{align*}
	Fix \(h'<h\). When \(\iota_i\) is sufficiently small, Proposition \ref{P_7} and \eqref{434} give
	\begin{align*}
		|u_{\iota_i}(r,t)|\leq C(h'/4,T),\quad \forall(r,t)\in (r_h(\cdot),R]\times[0, T].
	\end{align*}
	
	Next, we show that the approximate solutions have a uniform Hölder norm on each rectangular region $[r-\tilde{r},R]\times[(t-\tilde{t})_+,t+\tilde{t}]$. For any $t\in[(t-\tilde{t})_+,t+\tilde{t}]$ and $r_1,r_2\in [r-\tilde{r},R]$, it follows from \eqref{433} that
	\begin{align}
		|{u_{\iota_i}}(r_2,t)-{u_{\iota_i}}(r_1,t)|&\leq\Big|\int_{r_1}^{r_2} \partial_r{u_{\iota_i}}(t)dr\Big|\nonumber\\
		&\leq C\Big(\int_{r_1}^{r_2} {({\partial_{r} u_{\iota_i}})}^2dr(t)\Big)^{\frac12}(r_2-r_1)^{\frac12}\nonumber\\
		&\leq C(h'/4,T)(r_2-r_1)^{\frac12}\nonumber.
	\end{align}
	For any $r\in[r-\tilde{r},R]$ and $t_1,t_2\in [(t-\tilde{t})_+,t+\tilde{t}]$, let $k=\sqrt{t_2-t_1}$. Fix $h^{''}\in(h',h)$. By \eqref{wii 19}, there exists a constant \(\varepsilon(h,h'')>0\), depending only on \(h\) and \(h''\), such that
	\begin{align*}
		r_{h}(s)-r_{h^{''}}(s)\ge\varepsilon(h,h''),\quad \forall s\in [0,T].
	\end{align*}
	We split the discussion into two cases.  When \(k<\varepsilon(h,h'')\), we have
	\begin{align*}
		r-k\ge r-\tilde{r}-k\ge r_h(s)-k\ge r_{h^{''}}(s),\quad \forall s\in [(t-\tilde{t})_+,t+\tilde{t}].
	\end{align*}
	Therefore, \eqref{wii 18} and \eqref{433} imply that
	\begin{align}
		&\quad\big|{u_{\iota_i}}(r,t_2)-{u_{\iota_i}}(r,t_1)\big|\\
		&=k^{-1}\int_{r-k}^{r}\big|{u_{\iota_i}}(r,t_2)-
		{u_{\iota_i}}(r,t_1)\big|d\xi\nonumber\\
		&\leq k^{-1}\int_{r-k}^{r}\big|{u_{\iota_i}}(\xi,t_2)-{u_{\iota_i}}(\xi,t_1)\big|d\xi
		+k^{-1}\sum_{i=1}^2{\int_{r-k}^{r}\big|{u_{\iota_i}}(r,t_i)-{u_{\iota_i}}(\xi,t_i)\big|d\xi}\nonumber\\
		&\leq \sqrt{k}\Big(\int_{r-k}^{r}\int_{t_1}^{t_2}\big({\partial_t u_{\iota_i}}(r,t)\big)^2dtdr\Big)^{\frac12}+k^{-1}\sum_{i=1}^2\int_{r-k}^{r}\int_\xi^r|\partial_r u_{\iota_i}(\zeta, t_i)|d\zeta d\xi\nonumber\\
		&\leq\sqrt{k}\Big(\int_{r-k}^{r}\int_{t_1}^{t_2}\big({\partial_t u_{\iota_i}}(r,t)\big)^2dtdr\Big)^{\frac12}+k^{-1}\sum_{i=1}^2\int_{r-k}^{r}\int_\xi^r|\partial_r u_{\iota_i}(\zeta, t_i)|d\zeta d\xi\nonumber\\
		&\leq C(h'/4,T)(t_2-t_1)^{\frac{1}{4}}.
	\end{align}
	When \(k\ge\varepsilon(h'',h)\), it follows from \eqref{434} that
	\begin{align}
		\begin{split}
			|{u_{\iota_i}}(r,t_2)-{u_{\iota_i}}(r,t_1)\big|&\leq 2\sup_{t_1\leq t\leq t_2}\norm{u_{\iota_i}}_{L^\infty((r_h(t),R])}\\
			&\leq 2\sup_{t_1\leq t\leq t_2}\norm{u_{\iota_i}}_{L^\infty((r^{\iota_i}_{h'}(t),R])}\\
			&\leq C(h'/4,T)(\varepsilon(h^{''},h'))^{-\frac{1}{2}}(t_2-t_1)^{\frac{1}{4}}\\
			&\leq C(h'/4,T)(t_2-t_1)^{\frac{1}{4}}.
		\end{split}
	\end{align}
	Thus we have shown that on each rectangle \(u_{\iota_i}\) possesses a uniform H\"{o}lder norm.  we conclude that $\{u_{\iota_i}(r,t)\}$ is uniformly bounded and H\"{o}lder continuous, jointly in $r,t$, on $[r-\tilde{r},R]\times[(t-\tilde{t})_+,t+\tilde{t}]$. Therefore, by the Ascoli–Arzelà theorem, \(u_{\iota_i}\) converges uniformly on the set, and the limit function \(u\) is Hölder continuous on the set. 
	
	Using \eqref{431} and \eqref{433}, the same conclusion can be shown to hold for \(\rho\). This completes the proof of Proposition \ref{P_2}.
\end{proof}

\begin{prop}\label{Prop regu}
	Let the hypotheses and notations of Proposition \ref{P_1}-\ref{P_7} be in force. For any $(r,t)\in\mathcal{F}$, it holds that
	\begin{align}
		\begin{split}
			&\rho\in L^\infty((t-\tilde{t})_+,t+\tilde{t}; W^{1,\infty}([r-\tilde{r},R]));\\
			& \partial_t\rho\in L^\infty((t-\tilde{t})_+,t+\tilde{t}; L^2([r-\tilde{r},R]));\\
			&u\in L^\infty((t-\tilde{t})_+,t+\tilde{t}; H^1([r-\tilde{r},R]))\cap L^2((t-\tilde{t})_+,t+\tilde{t}; H^2([r-\tilde{r},R]));\\
			&\partial_t u\in L^2((t-\tilde{t})_+,t+\tilde{t}; L^2([r-\tilde{r},R]).
		\end{split}
	\end{align}
\end{prop}
\begin{proof}
	Let \(h\) and \(h'\) be as in Proposition \ref{P_7}. Then for sufficiently small \(\iota_i\) we have
	\begin{align}
		[r-\tilde{r},R]\times[(t-\tilde{t})_+,t+\tilde{t}]\subset(r_{h'}^{\iota_i}(\cdot),R]\times[0,T].
	\end{align}
	We thus deduce from \eqref{433} that
	\begin{align*}
		\sup_{t\in[(t-\tilde{t})_+,t+\tilde{t}]}\int_{r-\tilde{r}}^R\Big((\partial_{r} u_{\iota_i})^2+(\partial_t \rho_{\iota_i})^2 \Big)dr+\int_{(t-\tilde{t})_+}^{t+\tilde{t}} \int_{r-\tilde{r}}^R{ \Big((\partial_t u_{\iota_i})^2+(\partial_{r r}u_{\iota_i})^2 \Big)}drdt\leq C(h'/4,T),
	\end{align*}
	and form \eqref{434} that
	\begin{align*}
		\sup_{t\in[(t-\tilde{t})_+,t+\tilde{t}]}\norm{\partial_{r}\rho_{\iota_i}}_{L^\infty([r-\tilde{r},R])}\leq C(h'/4,T).
	\end{align*}
	This shows that
	\begin{align*}
		\begin{split}
			&\rho_{\iota_i}\rightharpoonup\rho,\text{ weak-$*$ in } L^\infty((t-\tilde{t})_+,t+\tilde{t}; W^{1,\infty}([r-\tilde{r},R]));\\
			&\partial_t\rho_{\iota_i}\rightharpoonup\partial_t\rho,\text{ weak-$*$ in }L^\infty((t-\tilde{t})_+,t+\tilde{t}; L^2([r-\tilde{r},R]));\\
			&u_{\iota_i}\rightharpoonup u,\text{ weak-$*$ in }L^\infty((t-\tilde{t})_+,t+\tilde{t}; H^1([r-\tilde{r},R]));\\
			&\partial_t u_{\iota_i}\rightharpoonup\partial_t u,\text{ weak in }L^2((t-\tilde{t})_+,t+\tilde{t}; L^2([r-\tilde{r},R])).
		\end{split}
	\end{align*}
	This completes the proof of Proposition \ref{Prop regu}.
    \end{proof}

    \begin{prop}\label{P_4}
	Let the hypotheses and notations of Proposition \ref{P_1} be in force. It holds that
	\begin{align}
	\underline{r}(t)\leq \left(R^N-\frac{M_0^{\frac{\gamma}{\gamma-1}}}{\omega_N((\gamma-1)E_0)^{\frac{1}{\gamma-1}}}\right)^{\frac{1}{N}},
\end{align}
where \(\omega_N\) denotes the measure of the unit ball in \(\mathbb R^N\).
	\end{prop}
\begin{proof}
	From $\eqref{App  Stm Euler'}_1$ and \eqref{u^2 weak'} it follows that for any \(t\in[0,T]\),
	\begin{align*}
		\begin{split}
			&\int_{\Omega_{\iota_i}}\rho_{\iota_i}(t,x)dx=\int_{\Omega_{\iota_i}}\rho_{\iota_i,0}(x)dx=M_0;\\
			&\frac{1}{\gamma-1}\int_{\Omega_{\iota_i}}\rho_{\iota_i}^\gamma(t,x)dx\leq \frac{1}{2}\int_{\Omega_{\iota_i}}\rho_{\iota_i,0}u_{\iota_i,0}^2dx+ \frac{1}{\gamma-1}\int_{\Omega_{\iota_i}}\rho_{\iota_i,0}^\gamma(x)dx.
		\end{split}
	\end{align*}
	Letting \(\iota_i\to 0\) and using \eqref{430}, \eqref{327'}, and \eqref{2d weak initial data'}, we obtain
	\begin{align*}
		\begin{split}
			&\int_{\Omega}\rho(t,x)dx=\int_{\Omega}\rho_0(x)dx=M_0;\\
			&\frac{1}{\gamma-1}\int_{\Omega}\rho^\gamma(t,x)dx\leq \frac{1}{2}\int_{\Omega}\rho_{0}u_{0}^2dx+ \frac{1}{\gamma-1}\int_{\Omega}\rho_{0}^\gamma(x)dx=E_0.
		\end{split}
	\end{align*}
	Therefore, it follows from Proposition \ref{P_3} that
	\begin{align}
		\begin{split}
			M_0&=\int_{\Omega}\rho(t,x)dx=\int_{\{\rho(t)>0\}}\rho(t,x)dx\\
			&\leq \left(\int_\Omega\rho^\gamma(t,x)dx\right)^{\frac{1}{\gamma}}|\{\rho(t)>0\}|^{\frac{\gamma-1}{\gamma}}\\
			&=((\gamma-1)E_0)^{\frac{1}{\gamma}}\left(\omega_N(R^{N}-(\underline{r}(t))^N)\right)^{\frac{\gamma-1}{\gamma}},
		\end{split}
	\end{align}
	which implies that
	\begin{align}
		\underline{r}(t)\leq \left(R^N-\frac{M_0^{\frac{\gamma}{\gamma-1}}}{\omega_N((\gamma-1)E_0)^{\frac{1}{\gamma-1}}}\right)^{\frac{1}{N}}.
	\end{align}
	This completes the proof of Proposition \ref{P_4}.
\end{proof}

\section{Appendix}
In this section we collect several lemmas required in the proof.
\begin{lema}\label{Lem alpha2}
	The function \(\alpha_{2,-}(x)\) introduced in Definition \ref{Def alpha 2}  is strictly increasing on \((1,\infty)\) with
	\begin{align}\label{Appendix 1}
		\lim\limits_{x\to 1^{+}}\alpha_{2,-}(x)=\frac{1}{2},\quad \lim\limits_{x\to \infty}\alpha_{2,-}(x)=1;
	\end{align}
	while \(\alpha_{2,+}(x)\) is strictly decreasing on \((1,\infty)\) with
	\begin{align}\label{Appendix 2}
		\lim\limits_{x\to 1^{+}}\alpha_{2,+}(x)=+\infty,\quad \lim\limits_{x\to \infty}\alpha_{2,+}(x)=1.
	\end{align}
\end{lema}
\begin{proof}
	By L'Hospital rule, \eqref{Appendix 1} is easily verified.  
	To show that \(\alpha_{2,-}(\cdot)\) is strictly increasing on \((1,\infty)\), we compute directly:
	\begin{align*}
		\alpha_{2,-}'(x)=\frac{1}{(x-1)^3\sqrt{2x-1}}\left(x^2-2x\sqrt{2x-1}+2x-1\right),\quad x>1.
	\end{align*}
	Define the auxiliary function \(f_1(x)=x^2-2x\sqrt{2x-1}+2x-1\). Then a direct computation gives
	\begin{align*}
		f_1'(x)=2\left(1-\frac{1}{\sqrt{2x-1}}\right)\left(x-\sqrt{2x-1}\right)>0,\quad x>1,
	\end{align*}
	which together with \(f_1(1)=0\) shows that \(f_1(x)>0\) for every \(x\in(1,\infty)\). Hence $\alpha_{2,-}'(x)>0$ for all $x\in(1,\infty)$, so $\alpha_{2,-}(\cdot)$ is strictly increasing.  The corresponding statement for $\alpha_{2,+}(\cdot)$ is obtained in the same way and its proof is omitted.
\end{proof}

\begin{lema}\label{Lem alpha3}
	The function \(\alpha_{3,-}(x)\) introduced in Definition \ref{Def alpha 3}  is strictly increasing on \((1,\infty)\) with
\begin{align}\label{Appendix 3}
	\lim\limits_{x\to 1^{+}}\alpha_{3,-}(x)=\frac{2}{3},\quad \lim\limits_{x\to \infty}\alpha_{3,-}(x)=1;
\end{align}
	while \(\alpha_{3,+}(x)\) is strictly decreasing on \((1,\infty)\) with
	\begin{align}\label{Appendix 4}
	\lim\limits_{x\to 1^{+}}\alpha_{3,+}(x)=+\infty,\quad \lim\limits_{x\to \infty}\alpha_{3,+}(x)=1.
	\end{align}
\end{lema}
\begin{proof}
	We only prove the properties of \(\alpha_{3,+}(\cdot)\). Those of \(\alpha_{3,-}(\cdot)\) are analogous. L'Hospital rule immediately verifies \eqref{Appendix 4}. A direct computation gives
		\begin{align*}
		\alpha_{3,+}'(x)=\frac{-4x^3-10x^2+5x-3x\sqrt{16x^3-4x^2-4x+1}}{2(x-1)^3\sqrt{16x^3-4x^2-4x+1}}<0, \quad x>1,
	\end{align*}
	which shows that \(\alpha_{3,+}(x)\) is strictly decreasing on \((1,\infty)\). The proof is complete.
\end{proof}
	Below we prove a technical elementary inequality.
\begin{lema}\label{Lem ks}
	Let \(k\in\mathbb N\) and \(s\in\mathbb N^+\). Then there exists \(\varepsilon>0\) such that
	\begin{align*}
		a(a+b)^{1+\frac{2s}{2k+1}}\ge \varepsilon|a|^{2+\frac{2s}{2k+1}}-|b|^{2+\frac{2s}{2k+1}},\quad \forall a,b\in \mathbb{R}.
	\end{align*}
\end{lema}
\begin{proof}
	Assume without loss of generality that \(a\neq 0\), the inequality is equivalent to
	\begin{align*}
		\left(1+\frac{b}{a}\right)^{1+\frac{2s}{2k+1}}+\left|\frac{b}{a}\right|^{2+\frac{2s}{2k+1}}\ge \varepsilon.
	\end{align*}
	Consider the function $f(t)=\left(1+t\right)^{1+\frac{2s}{2k+1}}+t^{2+\frac{2s}{2k+1}}$, as \(t\to\infty\), \(f(t)\to+\infty\). Hence \(f\) possesses a positive lower bound if and only if it has no zeros. We claim that \(f\) never vanishes. Otherwise, there would exist a \(t_{0}\) such that
	\begin{align*}
		f(t_0)=0\Longleftrightarrow t_0^{\frac{4k+2s+2}{2k+2s+1}}+t_0+1=0.
	\end{align*}
	Define \(g(t)=t^{\frac{d}{c}}+t+1\), where \(d>c\), \(d\) is even, and \(c\) is odd. A simple computation gives
	\begin{align*}
		g'(t)=\frac{d}{c}t^{\frac{d-c}{c}}+1.
	\end{align*}
	Thus \(g\) first decreases and then increases. Its minimum is attained at $t_1=\left(-\frac{c}{d}\right)^{\frac{c}{d-c}}$, and we have
	\begin{align*}
		g(t_1)=\left(\frac{c}{d}\right)^{\frac{d}{d-c}}+\left(-\frac{c}{d}\right)^{\frac{c}{d-c}}+1>0.
	\end{align*}
	This contradicts the assumption that \(f\) has a zero.
\end{proof}

\begin{lema}\label{Lem 2D L inf}
    Assuming \(N=2\) and \(\mathbf{v}(x) = v(r)\frac{x}{r}\) satisfies the boundary conditions \(v(0) = v(R) = 0\), there exists a constant \(C > 0\) such that
\begin{align*}
    \norm{\mathbf{v}}_{L^\infty(\Omega)}\leq C\norm{\nabla \mathbf{v}}_{L^2(\Omega)}.
\end{align*}
\end{lema}
\begin{proof}
    For any \(r \in (0, R)\), a direct computation shows that
    \begin{align*}
        v(r)^2&=2\int_0^r v\partial_r vdr\leq 2\int_0^R |\partial_r v|\left|\frac{v}{r}\right|rdr\\
        &\leq \int_0^R\left(|\partial_r v|^2+\left|\frac{v}{r}\right|^2\right)rdr\leq C\norm{\nabla\mathbf{v}}_{L^2(\Omega)}.
    \end{align*}
    This completes the proof.
\end{proof}

\section*{Acknowledgments}
	X Huang is partially supported by Chinese Academy of Sciences Project for Young Scientists in Basic Research (Grant No. YSBR-031), National Natural Science Foundation of China (Grant Nos. 12494542, 11688101). X Zhang is supported by National Natural Science Foundation of Yulin
University (Grant No. 2025GK11), the Young Elite Scientists Sponsorship Program by Yulin Association for Science and Technology (Grant No. 20250712).
	
	\vspace{1cm}
	\noindent\textbf{Data availability statement.} Data sharing is not applicable to this article.
	
	\vspace{0.3cm}
	\noindent\textbf{Conflict of interest.} The authors declare that they have no conflict of interest.


\begin{thebibliography}{99}
	\bibitem{ad-2003}
	Adams R A, Fournier J F. Sobolev spaces. Pure Appl. Math. (Amst.), 140
	Elsevier/Academic Press, Amsterdam, xiv+305 pp(2003)
    \bibitem{ab-2020}
    Alazard T, Bresch D. Functional inequalities and strong Lyapunov functionals for free surface flows in fluid dynamics. arXiv:2004.03440 (2020)
	\bibitem{bd-2003}
	Bresch D, Desjardins B, Lin C K. On some compressible fluid models: Korteweg, Lubrication, and Shallow water systems. Comm. Partial Differential Equations 28, 843-868(2003)
	\bibitem{bd-2004}
	Bresch D, Desjardins B. Existence of global weak solutions for a 2D viscous shallow water equations and convergence to the quasi-geostrophic model. Comm. Math. Phys. 238, 211-223(2003)
	\bibitem{bd-2006}
	Bresch D, Desjardins B. On the construction of approximate solutions for the 2D viscous shallow water model and for compressible Navier-Stokes models. J. Math. Pures Appl. 86, 362-368(2006)
	\bibitem{va-2022}
	Bresch D, Vasseur A F, Yu C. Global existence of entropy-weak solutions to the compressible Navier-Stokes equations with non-linear density dependent viscosities.
	J. Eur. Math. Soc. 24(5), 1791-1837(2022)
    \bibitem{Haspot-Burtea}
Burtea C, Haspot B. New effective pressure and existence of global strong solution for compressible Navier-Stokes equations with general viscosity coefficient in one dimension.
Nonlinearity 33 (2020), no. 5, 2077–2105.
    \bibitem{Cao-Li-Zhu-1}
Cao Y, Li H, Zhu S G. Global regular solutions for one-dimensional degenerate compressible Navier-Stokes equations with large data and far field vacuum.
SIAM J. Math. Anal. 54(4), 4658–4694(2022)
\bibitem{Cao-Li-Zhu-2}
Cao Y, Li H, Zhu S G. Global spherically symmetric solutions to degenerate compressible Navier-Stokes equations with large data and far field vacuum.
Calc. Var. Partial Differential Equations 63(9), Paper No. 230, 46 pp(2024)
\bibitem{Cho-Choe-Kim-2004}
Cho Y, Choe H J, Kim H. Unique solvability of the initial boundary value problems
for compressible viscous fluids. J. Math. Pures Appl. 83(2): 243-275(2004)
\bibitem{Cho-Kim-2006}
Cho Y, Kim H. On classical solutions of the compressible Navier-Stokes equations
with nonnegative initial densities. Manuscripta Math. 120(1): 91-129(2006)
\bibitem{fan-2022}
	Fan X Y, Li J X, Li J. Global existence of strong and weak solutions to 2D compressible Navier-Stokes system in bounded domains with large data and vacuum.
	Arch. Ration. Mech. Anal. 245(1), 239-278(2002)
\bibitem{FNP-2001}
Feireisl E, Novotn\'{y} A, Petzeltov\'{a} H. On the existence of globally defined weak solutions
to the Navier–Stokes equations. J. Math. Fluid Mech. 3(4), 358–392(2001)
\bibitem{Guo-2008}
	Guo Z H, Jiu Q S, Xin Z P. Spherically symmetric isentropic compressible flows with density-dependent viscosity coefficients. SIAM J. Math. Anal. 39(5), 1402-1427(2008)
	\bibitem{guo-2012}
	Guo Z H, Li H L, Xin Z P. Lagrange structure and dynamics for solutions to the spherically symmetric compressible Navier-Stokes equations.
	Comm. Math. Phys., 309(2), 371-412(2012)
    \bibitem{song-2019}
	Guo Z H, Song W J. Global well-posedness and large-time behavior of classical solutions to the 3D Navier-Stokes system with changed viscosities. J. Math. Phys. 60(3), 29 pp(2019)
	\bibitem{wang-2018}
	Guo Z H, Wang M, Wang Y. Global solution to 3D spherically symmetric compressible Navier-Stokes equations with large data.
	Nonlinear Anal. Real World Appl. 40, 260-289(2018)
	\bibitem{Guo-2025}
	Guo Z H, Xu L, Zhang X Y. Global strong solution of compressible flow with spherically symmetric data and density-dependent viscosities. Adv. Nonlinear Anal. 14(1), 25pp(2025)
     \bibitem{Haspot-2018}
    Haspot B. Existence of global strong solution for the compressible Navier–Stokes equations with degenerate viscosity coefficients in 1D. Math. Nachr. 291(14-15), 2188-2203(2018)   
\bibitem{Hoff-1987}
 Hoff D. Global existence for 1D, compressible, isentropic Navier-Stokes equations with
large initial data. Trans. Amer. Math. Soc. 303(1),
169-1819(1987)
	\bibitem{hoff-1998}
	Hoff D. Global solutions of the equations of one-dimensional, compressible flow with large data and forces, and with differing end states.
	Z. Angew. Math. Phys. 49(5), 774-785(1998)
	\bibitem{hoff-1995}
	Hoff D. Global solutions of the Navier-Stokes equations for multidimensional compressible flow with discontinuous initial data. J. Differential Equations 120(1), 215-254(1995)
    \bibitem{Hoff-1992}
	Hoff D. Spherically symmetric solutions of the Navier-Stokes equations for compressible, isothermal flow with large, discontinuous initial data.
	Indiana Univ. Math. J. 41(4), 1225-1302(1992)
	\bibitem{Jensen-2004}
	Hoff D, Jenssen H K. Symmetric nonbarotropic flows with large data and forces.
	Arch. Ration. Mech. Anal. 173(3), 297-343(2004)
	\bibitem{hoff-2001}
	Hoff D, Smoller J. Non-formation of vacuum states for compressible Navier-Stokes equations. Comm. Math. Phys., 216(2), 255-276(2001)
\bibitem{huang-1205}
	Huang X D, Li J. Existence and blowup behavior of global strong solutions to the two-dimensional barotrpic compressible Navier-Stokes system with vacuum and large initial data.
	J. Math. Pures Appl. 106(1), 123-154(2016)
	\bibitem{huang-1207}
	Huang X D, Li J. Global well-posedness of classical solutions to the Cauchy problem of two-dimensional barotropic compressible Navier-Stokes system with vacuum and large initial data. SIAM J. Math. Anal. 54(3), 3192-3214(2022)
    \bibitem{Huang-Li-Xin-2012}
Huang X D, Li J, Xin Z P. Global well-posedness of classical solutions with large oscillations and vacuum to the three-dimensional isentropic compressible Navier-Stokes
equations. Commun. Pure Appl. Math. 65(4), 549-585(2012)
	\bibitem{Ma-2015}
	Huang X D, Matsumura A. A characterization on breakdown of smooth spherically symmetric solutions of the isentropic system of compressible Navier-Stokes equations. Osaka J. Math. 52(1), 271-283(2015)
	\bibitem{hmn}
	Huang X D, Meng W L, Ni A C.
	Free boundary value problem for the radial symmetric compressible isentropic Navier-Stokes equations with density-dependent viscosity.
	J. Math. Anal. Appl. 548(1), Paper No. 129377, 27 pp.(2025)
    \bibitem{HMZ-SV-2025}
    Huang X D, Meng W L, Zhang X Y. Global classical solutions to the 2D viscous Saint-Venent  system in the whole space for arbitrary large initial data. in preparation.
\bibitem{jiang-2006}
	Jiang S, Xin Z P, Zhang P. Global weak solutions to 1D compressible isentropic Navier-Stokes equations with density-dependent viscosity.
	Methods Appl. Anal., 12(3), 239-251(2005)
\bibitem{Jiang-Zhang-2001}
Jiang S, Zhang P. On spherically symmetric solutions of the compressible isentropic
Navier-Stokes equations. Comm. Math. Phys. 215(3), 559-
581(2001)
\bibitem{jiu-2013}
	Jiu Q S, Wang Y, Xin Z P. Global well-posedness of the Cauchy problem of 2D compressible Navier-Stokes equations in weighted spaces. J. Differential Equations 255(3), 351-404 (2013)
	\bibitem{jiu-2014}
	Jiu Q S, Wang Y, Xin Z P. Global well-posedness of 2D compressible Navier-Stokes
	equations with large data and vacuum. J. Math. Fluid Mech. 16(3), 483-521(2014)
	\bibitem{jiu-pre}
	Jiu Q S, Wang Y, Xin Z P. Global classical solution to two-dimensional compressible Navier-Stokes equations with large data in $R^2$. Phys. D 376/377, 180-194(2018)
\bibitem{Kale-1968}
Kanel'Ya. I. A model system of equations for the one-dimensional motion of a gas.
Differ. Uravn. 4, 721-734(1968)
\bibitem{Kawashima-Nishida-1981}
Kawashima S, Nishida T. Global solutions to the initial value problem for the equations
of one-dimensional motion of viscous polytropic gases. J. Math. Kyoto Univ. 21(4), 825-837(1981)
\bibitem{Kazhikhov-Shelukhin-1977}
 Kazhikhov A V, Shelukhin V V. Unique global solution with respect to time of initial
boundary value problems for one-dimensional equations of a viscous gas. J. Appl. Math. Mech. 41(2), 282-291(1977)
\bibitem{vaigant-1995}
	Kazhikhov A V, Vaigant V A. On the existence of global solutions of two-dimensional Navier-Stokes equations of a compressible viscous fluid. Sibirsk. Mat. Zh., 36(6), 1283-1316(1995), ii; translation in Siberian Math. J., 36(6), 1108-1141(1995)
\bibitem{Li-2008}
	Li H L, Li J, Xin Z P. Vanishing of vacuum states and blow-up phenomena of the compressible Navier-Stokes equations.
	Comm. Math. Phys. 281(2), 401-444(2008)
\bibitem{Li-Wang-Xin-2019}
Li H L, Wang Y X, Xin Z P. Non-existence of classical solutions with finite energy to the Cauchy problem of the compressible Navier-Stokes equations. Arch. Ration. Mech. Anal. 232(2), no. 2, 557–590(2019)
  \bibitem{Li-Wang-Xin}
 Li H L, Wang Y X, Xin Z P. On the vacuum free boundary problem of the viscous Saint-Venant system for shallow water in two dimensions. Math. Ann. 391(3), 3555–3639(2025)
 \bibitem{li-2016}
	Li H L, Zhang X W. Global strong solutions to radial symmetric compressible Navier-Stokes equations with free boundary.
	J. Differential Equations 261(11), 6341-6367(2016)
\bibitem{Li-2015}
	Li J, Xin Z P. Global existence of weak solutions to the barotropic compressible Navier-Stokes flows with degenerate viscosities. arXiv:1504.06826(2015)
\bibitem{li-2019}
    Li J, Xin Z P. Global well-posedness and large time asymptotic behavior of classical solutions to the compressible Navier-Stokes equations with vacuum. Ann. PDE 5(1), no. 1, Paper No. 7, 37 pp(2019)
\bibitem{ya-2015}
	Li Y C, Pan R H, Zhu S G. On classical solutions for viscous polytropic fluids with degenerate viscosities and vacuum.
	Arch. Ration. Mech. Anal. 234(3), 1281-1334(2019)
\bibitem{ya-2014}
	Li Y C, Pan R H, Zhu S G. On classical solutions to 2D shallow water equations with degenerate viscosities.
	J. Math. Fluid Mech. 19(1), 151-190(2017)
\bibitem{Lions-1998}
Lions P L. Mathematical Topics in Fluid Dynamics. vol. 2, Compressible Models.
Oxford: Clarendon Press(1998)
\bibitem{Liu-Xin-Yang-1998}
Liu T P, Xin Z P, Yang T. Vacuum states for compressible flow. Discrete Contin. Dyn. Syst. 4(1), 
1–32(1998)
\bibitem{Luo-2012}
Luo Z. Local existence of classical solutions to the two-dimensional viscous compressible
flows with vacuum. Commun. Math. Sci. 10(2): 527-554(2012)
\bibitem{Matsumura-Nishida-1980}
Matsumura A, Nishida T. The initial value problem for the equations of motion of
viscous and heat-conductive gases. J. Math. Kyoto Univ. 20(1), 67-104(1980)
\bibitem{Mellet-2007}
	Mellet A, Vasseur A F. On the barotropic compressible Navier-Stokes equations.
	Comm. Partial Differential Equations, 32(1-3), 431-452(2007)
\bibitem{Nash-1962}
 Nash J F. Le problème de Cauchy pour les équations différentielles d’un fluide général. (French) Bull. Lond. Math. Soc. 90, 487-497(1962)
\bibitem{Serre-1986-1}
Serre D.  Solutions faibles globales des \'{e}quations de Navier-Stokes pour un fluide compressible. C. R. Acad. Sci. Paris S\'{e}r. I Math. 303(13), 639-642(1986)
\bibitem{Serre-1986-2}
Serre D.  Sur l'\'{e}quation monodimensionnelle d'un fluide visqueux, compressible et conducteur de chaleur. C. R. Acad. Sci. Paris Sér. I Math. 303, 703-706(1986)
	\bibitem{Vasseur-2016}
	Vasseur A F, Yu C.
	Existence of global weak solutions for 3D degenerate compressible Navier-Stokes equations.
	Invent. Math. 206(3), 935-974(2016)
\bibitem{Xin-1998}
Xin Z P. Blowup of smooth solutions to the compressible Navier–Stokes equation with
compact density. Commun. Pure Appl. Math. 51(3), 229–240(1998)
\bibitem{Xin-Yan-2013}
Xin Z P, Yan W. On blow-up of classical solutions to the compressible Navier-Stokes equations,
Comm. Math. Phys. 321(2), 529–541(2013)
\bibitem{xin-2006}
	Xin Z P, Yuan H J. Vacuum state for spherically symmetric solutions of the compressible Navier-Stokes equations.
	J. Hyperbolic Differ. Equ., 3(3), 403-442(2006)
\bibitem{Xin-Zhang-Zhu}
Xin Z P, Zhang J W, Zhu S G. Global-in-time well-posedness of classical solutions to the vacuum free boundary problem for the viscous saint-venant system with large data. Arxiv: 25045.10175v1.
\bibitem{xin-2018}
	Xin Z P, Zhu S G. Global well-posedness of regular solutions to the three-dimensional isentropic compressible Navier-Stokes equations with degenerate viscosities and vacuum. Adv. Math. 393, No. 108072, 69 pp(2021)
	\bibitem{yzz}
	Yang A, Zhao X, Zhou W S.
	Global classical solutions of free boundary problem of compressible Navier-Stokes equations with degenerate viscosity.
	J. Differential Equations 416(3), 1837–1860(2025)
	\bibitem{zpx-2014}
	Zhang P X, Zhao J N. The existence of local solutions for the compressible Navier-Stokes equations with the density-dependent viscosities. Commun. Math. Sci. 12(7), 1277-1302(2014) 
	\bibitem{zhang-2025}
	Zhang X Y. Spherically symmetric strong solution of compressible flow with large data
	and density-dependent viscosities, J. Math. Anal. Appl. 549(2), 2025, paper no. 129488(2025)
	\bibitem{zhu-20152}
	Zhu S G. Existence results for viscous polytropic fluids with degenerate viscosity coefficients and vacuum.
	J. Differential Equations 259(1), 84-119(2015)
	\bibitem{zhu-2015}
	Zhu S G. Well-posedness and singularity formation of the compressible isentropic
	Navier-Stokes equations. Ph.D Thesis, Shanghai Jiao Tong University(2015)
	\bibitem{xin-2021}
	Zhu S G, Xin Z P. Well-posedness of the three-dimensional isentropic compressible Navier-Stokes equations with degenerate viscosities and far field vacuum.
	J. Math. Pures Appl. 152, 94-144(2021)
	
\end{thebibliography}
\end{document}